\documentclass[11pt]{amsart}

\usepackage{amssymb,mathrsfs, amsmath,amscd,amsfonts,graphicx,color}
\usepackage[width=\linewidth]{caption}
\usepackage[active]{srcltx}
\usepackage{hyperref}
\usepackage[all]{xy}

\newtheorem{theorem}{Theorem}[section]
\newtheorem{lemma}[theorem]{Lemma}
\newtheorem{corollary}[theorem]{Corollary}
\newtheorem{proposition}[theorem]{Proposition}
\newtheorem{definition}[theorem]{Definition}

\theoremstyle{remark}
\newtheorem{remark}[theorem]{Remark}
\newtheorem{example}[theorem]{Example}

\newcommand{\fleche}[4]{                     % fonction
            \begin{array}{ccll} #1 & \rightarrow & #2 \\   %
                         #3 &\mapsto & #4          %
            \end{array}}

\newcommand{\foncIso}[5]{                     % fonction
            \begin{array}{rcll}#1 :& #2 & \overset{\sim}{\longrightarrow} & #3 \\   %
                         &#4 &\longmapsto & #5          %
            \end{array}}

\newcommand{\smallblacktriangleright}{%
  \vcenter{\hbox{\,\scalebox{0.75}{$\blacktriangleright$}\,}}%
}

\newcommand{\smallblacktriangleleft}{%
  \vcenter{\hbox{\,\scalebox{0.75}{$\blacktriangleleft$}\,}}%
}

\newcommand{\smallblacksquare}{%
  \vcenter{\hbox{\:\scalebox{0.4}{$\blacksquare$}\:}}%
}

\newcommand{\OO}{\mathcal{O}}

\newcommand{\UU}{\mathcal{U}}

\DeclareMathOperator{\Hom}{Hom}
\DeclareMathOperator{\End}{End}

\allowdisplaybreaks %To allow displayed equations to be cut on different pages.

\title[Braided categories of bimodules]{Braided categories of bimodules from stated skein TQFTs}

\author[F. Costantino, M. Faitg]{Francesco Costantino\textsuperscript{1} and Matthieu Faitg\textsuperscript{2}}

\begin{document}

\maketitle

\begin{center}
\textsuperscript{1,2}Institut de Math\'ematiques de Toulouse, Universit\'e Paul Sabatier,\\ 118 route de Narbonne, F-31062 Toulouse, France.

\medskip

\textsuperscript{1}francesco.costantino@math.univ-toulouse.fr, \textsuperscript{2}matthieu.faitg@math.univ-toulouse.fr
\end{center}

\begin{abstract}
For each braided category $\mathcal{C}$ we show that, under mild hypotheses, there is an associated category of ``half braided algebras'' and their bimodules internal to $\mathcal{C}$ which is not only monoidal but even braided and balanced. 
We use this in the case where $\mathcal{C}$ is the category of modules over a ribbon Hopf algebra to interpret stated skeins as a TQFT, namely a braided balanced functor from a category of cobordisms to this category of algebras and their bimodules. Although our construction works in full generality, we relate in the special case of finite-dimensional ribbon factorizable Hopf algebras the stated skein functor to the Kerler--Lyubashenko TQFT by interpreting the former as the ``endomorphisms'' of the latter. 
\end{abstract}

\bigskip

{\small \noindent {\em Keywords:} Stated skein algebras, TQFT, Categories of bimodules.

\smallskip

\noindent {\em 2020 Mathematics Subject Classification.} Primary: 18M15, 57K16. Secondary: 16D90}

\tableofcontents

\section{Introduction}
The links between algebra and topology have been strikingly fostered by the discovery of quantum invariants of knots and more in  general by topological quantum field theories. Braided categories are the natural setup for representing the braid groups, while ribbon categories yield a multitude of link invariants. The standard sources of braided categories are modules over a Hopf algebra $\mathcal{U}$, which must be quasi-triangular (\textit{i.e.} have an R-matrix $R\in \UU\otimes \UU$ satisfying suitable conditions). For the ribbon case, such algebras must further be endowed with a so-called ``ribbon element'' with special properties.

\indent Another key example is the study of TQFTs which, roughly, are monoidal functors from the category of surfaces and their cobordisms to a target monoidal category of algebraic nature. There are different versions of TQFTs, varying according to the choice of the source category and the target one. They are usually built from a ribbon Hopf algebra $\UU$ (with extra assumptions) as basic algebraic input, and a standard target is the category $\UU\text{-}\mathrm{Mod}$ of $\UU$-modules (or even $\UU\text{-}\mathrm{mod}$, \textit{i.e.} finite-dimensional modules).

\indent An important example of such a TQFT is the one constructed by Kerler and Lyubashenko, which has played a central role in topology. Denote by $\mathrm{Cob}$ the category of connected surfaces with one boundary component and their connected cobordisms with annular side boundary, and let $\mathrm{Cob}^{\sigma}$ be the refinement of $\mathrm{Cob}$ where surfaces are endowed with lagrangian subspaces of their homology and cobordisms are decorated by integers.

\begin{theorem}[\cite{KL}, following \cite{BD}]\label{thmKLintro}
If $\UU$ is finite dimensional and factorizable then there exists a braided monoidal functor $\mathrm{KL}: \mathrm{Cob}^{\sigma}\to \UU\text{-}\mathrm{mod}$ sending the once punctured torus to the adjoint representation of $\UU$. 
\end{theorem}
\noindent We recall that factorizability is a non-degeneracy condition of the R-matrix. It is important to remark that the adjoint representation of $\UU$ is a very special object of $\UU\text{-}\mathrm{mod}$, namely it is a Hopf algebra object in it. 
That this is no accident has been understood from a topological point of view by the works of Kerler \cite{kerler}, Habiro (see in \cite{As11}), Bobtcheva--Piergallini \cite{BP} and more recently by Beliakova-Bobtcheva-De Renzi-Piergallini (see their paper \cite{BBDP} for a full account) who showed that the category of surfaces and their cobordisms is freely generated by a Hopf algebra object (the punctured torus) with extra properties which make it a so-called ``BP-Hopf algebra''.  

\smallskip

\indent Before explaining what is done in this paper, and how our results are related to the KL-TQFT, we make a technical comment: we will in general use a coribbon Hopf algebra $\OO$ instead of a ribbon Hopf algebra $\UU$ and right $\OO$-comodules instead of left $\UU$-modules. This switch from $\UU$ to $\OO$ is motivated by technical categorical reasons explained in \S\ref{sub:ComodVsMod}, and also because previous papers like \cite{CL,CL3Man} use comodules instead of modules. One can think of $\OO$ as a dual to $\UU$, the axioms of a coribbon Hopf algebra being completely dual to those of a ribbon Hopf algebra (in particular there is a co-R-matrix $\mathcal{R} : \OO \otimes \OO \to k$). When $\UU$ is finite-dimensional these two settings are equivalent. 

\smallskip

In the recent years a new kind of TQFT has emerged based on the choice of a coribbon Hopf algebra $\mathcal{O}$ and which associates to surfaces their ``stated skein algebras'' (as defined in \cite{LeTriDe,CL} for $\mathcal{O} = \mathcal{O}_q(\mathrm{SL}_2)$ and in \S\ref{subsecStatedSkMod} for general $\mathcal{O}$) and to cobordisms some bimodules over the boundary algebras. The target category has thus algebras as objects and isomorphism classes of bimodules over these algebras as morphisms. The composition is obtained by tensoring over the skein algebras of the glueing surfaces. The first TQFT of this kind has been formalised for the Hopf algebra $\mathcal{O}_q(\mathrm{SL}_2)$ in \cite{CL3Man}, for the cobordism category consisting of all ``marked surfaces'' and cobordisms with the symmetric monoidal structure coming from disjoint union (hence different from the category $\mathrm{Cob}$ introduced above) and with target the category consisting of algebras and bimodules in $\mathrm{Vect}_k$. However this construction lacks the richness of the Kerler-Lyubashenko TQFT because its target category is symmetric, while the target category of the KL TQFT is non-trivially braided. 

\smallskip

\indent The main purposes of this paper are:
\begin{enumerate}
\item to promote the stated skein module construction into a TQFT on $\mathrm{Cob}$ (Thm.\ref{thmStSkTQFTintro}),
\item to formalize in the general case of braided categories the algebraic properties of its target category, which consists of certain algebras and bimodules in $\mathrm{Comod}\text{-}\OO$ (Thm.\,\ref{teo:main1}),
\item and to compare our construction with the Kerler--Lyubashenko TQFT (Thm.\,\ref{thmEndKLintro}).
\end{enumerate}

\indent We start with the purely algebraic part of our results (item 2 above). The goal is to provide a complete algebraic setup which in particular appears as the target category for these TQFTs valued in categories of algebras and their bimodules. This setup can be constructed in a purely algebraic manner, without any reference to topology and also for more general ambient categories than $\mathrm{Comod}\text{-}\OO$, as we now explain.

\indent Let $\mathcal{C}$ be a braided category. The first main piece of the construction is the notion of {\em half-braided algebra}, which is an algebra object $A$ in $\mathcal{C}$ endowed with a half-braiding $t : A \otimes - \Rightarrow - \otimes A$ and satisfying a compatibility condition (Def.\,\ref{subsectionHbAlg}). Hence $(A,t)$ is an object in the Drinfeld center $\mathcal{Z}(\mathcal{C})$, but we stress that the compatibility condition does {\em not} mean that $(A,t)$ is an algebra in $\mathcal{Z}(\mathcal{C})$. The second main player is the notion of {\em hb-compatible bimodule}: a bimodule over two half-braided algebras is naturally endowed with two half-braidings coming from the left and right actions (see \eqref{halfBraidingInTermOfAction}) and we say that the bimodule is hb-compatible if these two half-braidings are equal. These two types of objects organize into a category, which moreover has striking properties:

\begin{theorem}[see Thm.\,\ref{thBraided} and \ref{thmBalanceBim}]\label{teo:main1}
Assume that the braided category $\mathcal{C}$ has coequalizers.
\\1. There is a monoidal category $\mathrm{Bim}^{\mathrm{hb}}_\mathcal{C}$ whose objects are half-braided algebras in $\mathcal{C}$ and whose morphisms are isomorphism classes of hb-compatible  bimodules. The composition is given by tensoring over middle algebras. The monoidal product is the ``braided tensor product'' (\S\ref{sectionPreliminariesModules}).
\\2. The category $\mathrm{Bim}^{\mathrm{hb}}_\mathcal{C}$ is braided, with the braiding provided in Thm.\,\ref{thBraided}.
\\3. Under extra assumptions on $\mathcal{C}$ (including LFP and cp-ribbon), $\mathrm{Bim}^{\mathrm{hb}}_\mathcal{C}$ is balanced (Thm.\,\ref{thmBalanceBim}).
\end{theorem}
\noindent We recall that a {\em balance} on a braided monoidal category $(\mathcal{C},\otimes,c)$ is a natural automorphism $b : \mathrm{Id}_{\mathcal{C}} \Rightarrow \mathrm{Id}_{\mathcal{C}}$ such that $b_{V \otimes W} = c_{W,V} \circ c_{V,W} \circ (b_V \otimes b_W)$ for all $V,W \in \mathcal{C}$. For the meaning of the other assumptions in item 3 (LFP and cp-ribbon), see the introduction of \S\ref{sectionLlinear} and \S\ref{sec:balanceBim}.

The novelty of the previous theorem relies mainly in points $2)$ and $3)$, namely the fact that the category $\mathrm{Bim}^{\mathrm{hb}}_\mathcal{C}$ is not only monoidal but it is braided and, under mild conditions, even balanced. Also, all the objects are naturally objects of $\mathcal{C}$ so that the whole construction is $\mathcal{C}$-ambient. These properties allow to use $\mathrm{Bim}^{\mathrm{hb}}_\mathcal{C}$ as target for topological constructions, as will be explained below.

\smallskip

\indent We remark that the notion of half-braided algebra is dual to the notion in \cite[Def.\,1.1]{schauenburg} (which could be called ``half-braided coalgebra'') defined by Peter Schauenburg. After finishing the writing of this paper we have been informed by David Jordan (whom we warmly thank) of the work of Theo Johnson-Freyd and David Reutter \cite{JR} where items 1 and 2 of Theorem \ref{teo:main1} were proved in the context of $2$-fusion categories in order to provide a concrete model for the $2$-categorical Drinfeld center of a connected fusion category. The proof we present here is based on standard $1$-categorical techniques and requires almost no restriction on $\mathcal{C}$, which in particular is not assumed to be semisimple nor finite. 

\smallskip

The definition of a half-braided algebra in $\mathcal{C}$ might seem unfamiliar at first sight, so we provide an equivalent definition under mild hypotheses on $\mathcal{C}$ (see \eqref{assumptionsCategoryC}), including in particular that $\mathcal{C}$ is locally finitely presentable (LFP, see Def.\,\ref{def:lfp}). A key and main example of such categories are the category of comodules over a coribbon Hopf algebra $\mathcal{O}$. Under these assumptions there exists a special object called the coend of $\mathcal{C}$ and denoted by $\mathscr{L}$, as recalled in \S\ref{sub:coend}. We show that a half-braided algebra in $\mathcal{C}$ is equivalent to a {\em $\mathscr{L}$-linear algebra} (Def.\,\ref{defLlinearAlgebra}), which is the datum of an algebra $A \in \mathcal{C}$ together with a special morphism $\mathscr{L} \to A$ satisfying a ``braided-commutativity'' relation (also this fact was observed in \cite{JR}). The definition of a hb-compatible bimodule can then be rephrased in $\mathscr{L}$-linear terms, yielding the notion of {\em $\mathscr{L}$-compatible bimodule}. $\mathscr{L}$-linear structures appear naturally in the topological examples explained below.%; it is from these examples that we had the idea to define $\mathscr{L}$-linear algebras and their compatible bimodules, and then to reformulate these notions in terms of half-braidings for more general categories $\mathcal{C}$.

\indent We note that what we call a $\mathscr{L}$-linear structure coincides with what is called a {\em quantum moment map} by Safronov \cite{safronov}. It does not exactly coincide with the original definition by Lu \cite{Lu} and Varagnolo-Vasserot \cite{VV} of a quantum moment map for a module-algebra over a Hopf algebra; the relation between the two definitions is explained in Section \ref{sub:LlinQMM}.

\smallskip

We now discuss the topological motivation underlying the above definitions. 
Half-braided algebras and their compatible bimodules are directly inspired from {\em stated skein modules} of surfaces and $3$-manifolds, initially defined for $\mathcal{O}=\mathcal{O}_q(\mathrm{SL}_2)$ \cite{LeTriDe,CL,CL3Man}, later for $\mathcal{O}=\OO_q(\mathrm{SL}_n)$ in \cite{LeSi} and which will be extended to a general categorical framework in the forthcoming paper \cite{CKL}. When the $3$-manifold is a thickened surface $\Sigma \times [-1,1]$, its stated skein module has a natural algebra structure and one speaks of the {\em stated skein algebra} of $\Sigma$. As shown in \cite{Ha} stated skein algebras are the combinatorial counterpart of factorization homology associated to a ribbon category (see \cite{BZBJ,BZBJ2}). Also, as shown in \cite[\S 6.2]{BFR}, stated skein algebras of surfaces with one boundary component are isomorphic to the so-called ``graph algebras". 

In this paper we consider the case of a general coribbon Hopf algebra $\OO$ and use a stated skein theory whose relations are based on the Reshetikhin--Turaev functor for the category $\mathcal{C}$ of right $\OO$-comodules (Def.\,\ref{defStSkMod}). In this way, we prove that to each object of the category $\mathrm{Cob}$ (defined above Thm.\,\ref{thmKLintro}) is associated an half-algebra in $\mathcal{C}$ which is the stated skein algebra of the surface and that to each morphism of $\mathrm{Cob}$ is associated a hb-compatible bimodule which is the stated skein module of the 3-manifold over the stated skein algebras of its boundary. Moreover this association defines a functor, which we call {\em stated skein functor} (Lem.\,\ref{lem:cutting}). The $\mathscr{L}$-linear structure on stated skein algebras appears very naturally, while the $\mathscr{L}$-compatibility of stated skein modules is easily seen topologically (Prop.\,\ref{prop:skeinLlinear}). Furthermore, the main new property is that this functor ``intertwines'' the braided balanced structure of $\mathrm{Cob}$ recalled in Section \ref{subsec:categoryCob} and the algebraically-defined braided balanced structure of $\mathrm{Bim}^{\mathrm{hb}}_\mathcal{C}$ in Thm.\,\ref{teo:main1}.
This is resumed in the following theorem:
\begin{theorem}[Thm.\,\ref{teo:monoidalfunctor}]\label{thmStSkTQFTintro}
Let $\mathcal{O}$ be a coribbon Hopf algebra and $\mathcal{C}=\mathrm{Comod}\text{-}\mathcal{O}$. The stated skein functor is a braided and balanced monoidal functor $\mathcal{S}_\mathcal{O}:\mathrm{Cob} \to \mathrm{Bim}^{\mathrm{hb}}_{\mathcal{C}}$. 
\end{theorem}
We often write $\mathrm{Bim}^{\mathscr{L}}_{\mathcal{C}}$ instead of $\mathrm{Bim}^{\mathrm{hb}}_{\mathcal{C}}$ when we work with the $\mathscr{L}$-linear version of the notions of half-braided algebras and their compatible bimodules.

\indent The TQFT construction from Thm.\ref{thmStSkTQFTintro} is similar to that of \cite{CL3Man}  but essentially differs from it in that it applies to any coribbon Hopf algebra and, more importantly, deals with a different category of cobordisms which is braided but non-symmetric. This is reflected in the fact that the target category of the stated skein functor is non-trivially braided.
This is why in the present paper we decided to restrict to the category $\mathrm{Cob}$, but we remark that the extension of the definition of stated skein algebras and bimodules to the category used in \cite{CL3Man} would be straightforward.
\smallskip

\begin{remark}
There is another construction of a related categorical TQFT associated to $\mathrm{Cob}$ due essentially to \cite{GJS} which we roughly resume as follows. It is the functor $Z:\mathrm{Cob}\to \mathrm{Bimod}$ where $\mathrm{Bimod}$ is the category of presentable cocomplete $k$-linear categories with enough compact projectives and their ``bimodules'' (a bimodule on $C,D$ is a right exact functor $C\times D^{\mathrm{op}}\to \mathrm{Vect}$, a.k.a. a {\em profunctor}); the composition in $\mathrm{Bimod}$ is defined by a suitable coend. The value of $Z$ on a surface $\Sigma$ is the free cocompletion of the so-called ``skein category'' $\mathrm{Sk}(\Sigma)$ of $\Sigma$ with coefficients in $\mathcal{O}$-comod. 
By the works of Cooke \cite{Coo} and Ha\"ioun \cite{Ha} this category is equivalent to $\mathcal{S}_\mathcal{O}(\Sigma)$-modules (at least for $\mathcal{O}=\mathcal{O}_q(\mathfrak{sl}_2)$). 
On the level of morphisms of $\mathrm{Cob}$, by the Eilenberg-Watts equivalence and the above equivalence, one has that a right exact functor $\mathrm{Sk}(\Sigma)\times \mathrm{Sk}(\Sigma')^{\mathrm{op}}\to \mathrm{Vect}$ is given by a $\bigl(\mathcal{S}_\mathcal{O}(\Sigma), \mathcal{S}_\mathcal{O}(\Sigma')\bigr)$-bimodule. Therefore the TQFTs $\mathcal{S}_\mathcal{O}$ and $Z$ are directly comparable and we expect them to be equivalent. This could be proved by showing the isomorphism of the two bimodules $\mathcal{S}_\mathcal{O}(M)$ and $Z(M)$ for each of the generating morphisms $M$ of the torus with one boundary component (which is a Hopf algebra generator of $\mathrm{Cob}$: see Theorem \ref{teo:presentation}). We will not dwelve into these details in this paper and leave this for future research. 
\end{remark}

Finally we relate our construction with the famous Kerler-Lyubashenko TQFT. As a preliminary remark, we note that easy but important examples of half-braided algebras are internal endomorphism spaces $\underline{\End}(V)=V\otimes V^*$, where $V\in \mathcal{C}$ is dualizable (see \S\ref{sub:end}). Internal Hom spaces $\underline{\Hom}(V,W) = W \otimes V^*$ are then hb-compatible bimodules over the internal End spaces. We prove that these basic facts give rise to a braided monoidal functor $\underline{\End}:\overline{\mathcal{C}}\to \mathrm{Bim}^{\mathrm{hb}}_\mathcal{C}$, where $\overline{\mathcal{C}}$ is the full subcategory formed by dualisable objects which contain $\boldsymbol{1}$ as a direct summand. It turns out that this functor takes values in the set of invertible objects of $\mathrm{Bim}^{\mathrm{hb}}_\mathcal{C}$ (see Prop.\,\ref{propEmbedCIntoBim}).

\indent Now recall that the source category of the KL TQFT is $\mathrm{Cob}^\sigma$, which is similar to $\mathrm{Cob}$ but whose objects and morphisms are equipped with extra data. Hence there is a natural forgetful functor $\mathrm{Cob}^\sigma \to \mathrm{Cob}$. Its target category is $\mathcal{C} = \UU$-mod, the finite-dimensional $\UU$-modules, where $\UU$ is a given factorizable, ribbon finite-dimensional Hopf algebra (Thm.\,\ref{thmKLintro}). We note that in this finite-dimensional case, working over $\mathrm{comod}\text{-}\OO$ is equivalent to working over $\UU\text{-}\mathrm{mod}$, where $\OO$ is dual to $\UU$ as a Hopf algebra; thus the stated skein functor $\mathcal{S}_{\OO}$ in Thm.\,\ref{thmStSkTQFTintro} can be rephrased in terms of $\UU$-mod and hence denoted by $\mathcal{S}_{\UU}$. We show that in a proper sense ``stated skeins are the endomorphisms of the vector spaces associated to surfaces by the Kerler-Lyubashenko TQFT'': 
\begin{theorem}[See Thm.\,\ref{teo:commutativediagram} for a precise statement] \label{thmEndKLintro}
Suppose that the Hopf algebra $\UU$ is ribbon, finite-dimensional and factorizable and let $\mathcal{C}=\UU\text{-}\mathrm{mod}$. 
Then the following diagram of braided monoidal functors commutes:
 \begin{equation*}
\xymatrix@C=4em{
\mathrm{Cob}^{\sigma} \ar[r]^-{\mathrm{KL}} \ar[d]_-{\mathrm{Forget}} & \overline{\mathcal{C}}\ar[d]^-{\underline{\End}}\\
\mathrm{Cob} \ar[r]_{\mathcal{S}_\mathcal{O}=\mathcal{S}_{\mathcal{U}}}  & \mathrm{Bim}^{\mathrm{hb}}_{\mathcal{C}}
} \end{equation*} 
\end{theorem}

We stress that, while the Kerler-Lyubashenko functor is defined only for finite-dimensional factorizable ribbon Hopf algebras, the stated skein functor is defined for general (co)-ribbon Hopf algebras. The above theorem might hint at the existence of Kerler-Lyubashenko functors at this level of generality which would still complete the commuting diagram. This is matter for future investigations. 

\medskip

\noindent \textbf{Acknowledgements.} ~We wish to thank St\'ephane Baseilhac, David Jordan, Benjamin Ha\"ioun, Julien Korinman, Thang Le, Jules Martel and Philippe Roche  for many interesting conversations. 
We also thank the anonymous referee for the careful reading of the text and his/her questions which helped us improving its quality. The authors are extremely grateful to CIMI Labex ANR 11-LABX-0040 at IMT Toulouse who funded this collaboration within the program ANR-11-IDEX-0002-02.

\section{Preliminaries}
In all this paper $\mathcal{C} = (\mathcal{C}, \otimes, \boldsymbol{1}, c)$ is a braided monoidal category, assumed strict for simplicity. Further assumptions on $\mathcal{C}$ will be made later. The braiding is a natural isomorphism $c_{X,Y} : X \otimes Y \overset{\sim}{\longrightarrow} Y \otimes X$ for all $X, Y \in \mathcal{C}$ satisfying the hexagon identities (which are triangles for strict $\mathcal{C}$). We will use the usual diagrammatic notation for braided monoidal categories:
\begin{center}
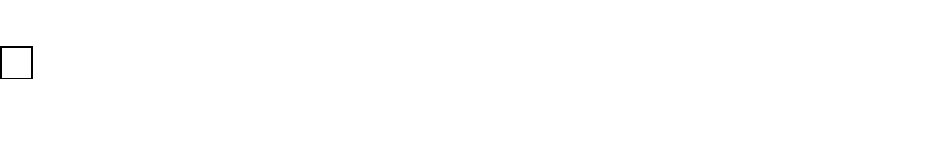
\end{center}
Note that we read diagrams {\em from bottom to top}. Later we will work with rigid objects (\textit{i.e.} which have left and right duals $X^*$ and $^*X$ \cite[\S 2.10]{EGNO}); in this case the duality morphisms will be represented by
\begin{equation}\label{diagramsForDuality}
%% Creator: Inkscape 1.1.2 (0a00cf5339, 2022-02-04), www.inkscape.org
%% Accompanies image file 'diagramsForDuality.pdf' (pdf, eps, ps)
\begingroup%
  \makeatletter%
  \providecommand\color[2][]{%
    \errmessage{(Inkscape) Color is used for the text in Inkscape, but the package 'color.sty' is not loaded}%
    \renewcommand\color[2][]{}%
  }%
  \providecommand\transparent[1]{%
    \errmessage{(Inkscape) Transparency is used (non-zero) for the text in Inkscape, but the package 'transparent.sty' is not loaded}%
    \renewcommand\transparent[1]{}%
  }%
  \providecommand\rotatebox[2]{#2}%
  \newcommand*\fsize{\dimexpr\f@size pt\relax}%
  \newcommand*\lineheight[1]{\fontsize{\fsize}{#1\fsize}\selectfont}%
  \ifx\svgwidth\undefined%
    \setlength{\unitlength}{344.49776042bp}%
    \ifx\svgscale\undefined%
      \relax%
    \else%
      \setlength{\unitlength}{\unitlength * \real{\svgscale}}%
    \fi%
  \else%
    \setlength{\unitlength}{\svgwidth}%
  \fi%
  \global\let\svgwidth\undefined%
  \global\let\svgscale\undefined%
  \makeatother%
  \begin{picture}(1,0.07701122)%
    \lineheight{1}%
    \setlength\tabcolsep{0pt}%
    \put(0,0){\includegraphics[width=\unitlength,page=1]{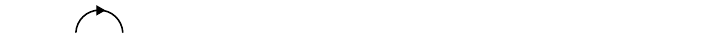}}%
    \put(0.0962157,0.00404886){\color[rgb]{0,0,0}\makebox(0,0)[lt]{\lineheight{1.25}\smash{\begin{tabular}[t]{l}$_{X^*}$\end{tabular}}}}%
    \put(0.15981162,0.00376877){\color[rgb]{0,0,0}\makebox(0,0)[lt]{\lineheight{1.25}\smash{\begin{tabular}[t]{l}$_X$\end{tabular}}}}%
    \put(-0.0005196,0.03135637){\color[rgb]{0,0,0}\makebox(0,0)[lt]{\lineheight{1.25}\smash{\begin{tabular}[t]{l}$\mathrm{ev}_X =$\end{tabular}}}}%
    \put(0,0){\includegraphics[width=\unitlength,page=2]{diagramsForDuality.pdf}}%
    \put(0.44086559,0.06476812){\color[rgb]{0,0,0}\makebox(0,0)[lt]{\lineheight{1.25}\smash{\begin{tabular}[t]{l}$_{X^*}$\end{tabular}}}}%
    \put(0.38216153,0.06487){\color[rgb]{0,0,0}\makebox(0,0)[lt]{\lineheight{1.25}\smash{\begin{tabular}[t]{l}$_X$\end{tabular}}}}%
    \put(0.25326742,0.03168006){\color[rgb]{0,0,0}\makebox(0,0)[lt]{\lineheight{1.25}\smash{\begin{tabular}[t]{l}$\mathrm{coev}_X =$\end{tabular}}}}%
    \put(0,0){\includegraphics[width=\unitlength,page=3]{diagramsForDuality.pdf}}%
    \put(0.6404863,0.00404886){\color[rgb]{0,0,0}\makebox(0,0)[lt]{\lineheight{1.25}\smash{\begin{tabular}[t]{l}$_X$\end{tabular}}}}%
    \put(0.70056806,0.00376877){\color[rgb]{0,0,0}\makebox(0,0)[lt]{\lineheight{1.25}\smash{\begin{tabular}[t]{l}$_{^*\!X}$\end{tabular}}}}%
    \put(0.54383661,0.03135637){\color[rgb]{0,0,0}\makebox(0,0)[lt]{\lineheight{1.25}\smash{\begin{tabular}[t]{l}$\widetilde{\mathrm{ev}}_X =$\end{tabular}}}}%
    \put(0,0){\includegraphics[width=\unitlength,page=4]{diagramsForDuality.pdf}}%
    \put(0.98513632,0.06476813){\color[rgb]{0,0,0}\makebox(0,0)[lt]{\lineheight{1.25}\smash{\begin{tabular}[t]{l}$_X$\end{tabular}}}}%
    \put(0.92126016,0.06474489){\color[rgb]{0,0,0}\makebox(0,0)[lt]{\lineheight{1.25}\smash{\begin{tabular}[t]{l}$_{^*\!X}$\end{tabular}}}}%
    \put(0.7975381,0.03168007){\color[rgb]{0,0,0}\makebox(0,0)[lt]{\lineheight{1.25}\smash{\begin{tabular}[t]{l}$\widetilde{\mathrm{coev}}_X =$\end{tabular}}}}%
    \put(0,0){\includegraphics[width=\unitlength,page=5]{diagramsForDuality.pdf}}%
  \end{picture}%
\endgroup%

\end{equation}

\subsection{Drinfeld center of \texorpdfstring{$\mathcal{C}$}{a monoidal category}}\label{subsecZC}
\indent Let $V$ be an object in $\mathcal{C}$. Recall that a {\em half-braiding for $V$} is a natural isomorphism $t : V \otimes - \overset{\sim}{\implies} - \otimes V$ such that for all $X, Y \in \mathcal{C}$
\begin{equation}\label{axiomHalfBraiding}
t_{X \otimes Y} = (\mathrm{id}_X \,\otimes\, t_Y) \circ (t_X \,\otimes\, \mathrm{id}_Y).
\end{equation}
Such half-braiding might not exist or be not at all unique, depending on $V$. In particular we have $t_{\boldsymbol{1}} = \mathrm{id}_V$ (because $t_{\boldsymbol{1}} = t_{\boldsymbol{1} \otimes \boldsymbol{1}} = t_{\boldsymbol{1}} \circ t_{\boldsymbol{1}}$). The {\em Drinfeld center of $\mathcal{C}$} is the category $\mathcal{Z}(\mathcal{C})$ whose objects are pairs $(V,t)$ where $V \in \mathcal{C}$ and $t$ is a half-braiding for $V$. A morphism $f \in \Hom_{\mathcal{Z}(\mathcal{C})}\bigl( (V,t), (W,u) \bigr)$ is a morphism $f \in \Hom_{\mathcal{C}}(V,W)$ such that $(\mathrm{id}_X \otimes f) \circ t_X = u_X \circ (f \otimes \mathrm{id}_X)$. The category $\mathcal{Z}(\mathcal{C})$ becomes monoidal by defining
\begin{equation}\label{defTensorProductInZC}
(V_1,t^1) \otimes (V_2,t^2) = \bigl( V_1 \otimes V_2, t^{1,2} \bigr) \quad \text{with } \: t^{1,2}_X = (t^1_X \otimes \mathrm{id}_{V_2}) \circ (\mathrm{id}_{V_1} \otimes t^2_X).
\end{equation}
Due to the naturality of the half-braidings we have
\begin{equation}\label{braidingOnZC}
t^1_{V_2} \in \Hom_{\mathcal{Z}(\mathcal{C})}\bigl((V_1,t^1) \otimes (V_2,t^2), (V_2,t^2) \otimes (V_1,t^1) \bigr)
\end{equation}
and this fact allows to define a braiding $T$ on $\mathcal{Z}(\mathcal{C})$:
\[ T_{(V_1,t^1),(V_2,t^2)} = t^1_{V_2} : (V_1,t^1) \otimes (V_2,t^2) \overset{\sim}{\longrightarrow} (V_2,t^2) \otimes (V_1,t^1). \] 
The braiding $T$ will be of great importance in the sequel.

\subsection{Algebras and (bi)modules}\label{sectionPreliminariesModules}
In this subsection we recall the notion of algebra object in a braided category, and of left or right module over the algebra. We also define similarly bimodules and recall the definition of braided tensor product of two algebras and of two modules. All this material is standard. 

\indent An {\em algebra} $\mathbf{A}$ in $\mathcal{C}$ is a triple $(A, m, \eta)$ where $A \in \mathcal{C}$, $m \in \Hom_{\mathcal{C}}(A \otimes A, A)$ and $\eta \in \Hom_{\mathcal{C}}(\boldsymbol{1}, A)$ are such that
\[ m \circ (m \otimes \mathrm{id}_A) = m \circ (\mathrm{id}_A \otimes m), \qquad m \circ (\eta \otimes \mathrm{id}_A) = m \circ (\mathrm{id}_A \otimes \eta) = \mathrm{id}_A. \]
The morphism $m$ is called the multiplication and $\eta$ is called the unit. A {\em morphism of algebras} $f : \mathbf{A} \to \mathbf{A}' = (A', m', \eta')$ is $f \in \Hom_{\mathcal{C}}(A,A')$ such that $f \circ m = m' \circ (f \otimes f)$ and $f \circ \eta = \eta'$. Following \cite[Lem. 9.2.12]{Majid}, the {\em braided tensor product} of two algebras $\mathbf{A}_1 = (A_1, m_1, \eta_1)$ and $\mathbf{A}_2 = (A_2, m_2, \eta_2)$ is the algebra
\begin{equation}\label{defBraidedTensorProductOfAlgebras}
\mathbf{A}_1 \,\widetilde{\otimes}\, \mathbf{A}_2 = \bigl( A_1 \otimes A_2, \: (m_1 \otimes m_2) \circ (\mathrm{id}_{A_1} \otimes c_{A_2,A_1} \otimes \mathrm{id}_{A_2}),\: \eta_1 \otimes \eta_2 \bigr)
\end{equation}
where we recall that $c$ denotes the braiding in $\mathcal{C}$.

\smallskip

\indent Let $\mathbf{A} = (A,m,\eta)$ be an algebra in $\mathcal{C}$. A {\em left $\mathbf{A}$-module} is a couple $\mathbf{M} = (M, \smallblacktriangleright)$ where $M$ is an object in $\mathcal{C}$ and $\smallblacktriangleright \in \Hom_{\mathcal{C}}(A \otimes M, M)$ is such that
\begin{equation}\label{defLeftModuleInAMonoidalCat}
\smallblacktriangleright \circ (m \otimes \mathrm{id}_M) = \smallblacktriangleright \circ (\mathrm{id}_A \otimes \smallblacktriangleright), \qquad \smallblacktriangleright \circ (\eta \otimes \mathrm{id}_M) = \mathrm{id}_M.
\end{equation}
A {\em morphism of left $\mathbf{A}$-modules} $f : \mathbf{M} \to \mathbf{M}' = (M',\smallblacktriangleright')$ is $f \in \Hom_{\mathcal{C}}(M,M')$ such that $f \circ \smallblacktriangleright = \smallblacktriangleright' \circ (\mathrm{id}_A \otimes f)$. If $\mathbf{M}_1 = (M_1, \smallblacktriangleright_{\!1})$ is a left $\mathbf{A}_1$-module and $\mathbf{M}_2 = (M_2, \smallblacktriangleright_{\!2})$ is a left $\mathbf{A}_2$-module, the {\em braided tensor product} of $\mathbf{M}_1$ and $\mathbf{M}_2$ is the left $(\mathbf{A}_1 \,\widetilde{\otimes}\, \mathbf{A}_2)$-module
\begin{equation}\label{defBraidedTensorProductOfLeftModules}
\mathbf{M}_1 \,\widetilde{\otimes}\, \mathbf{M}_2 = \bigl( M_1 \otimes M_2, \; (\smallblacktriangleright_{\!1} \otimes \smallblacktriangleright_{\!2}) \circ (\mathrm{id}_{A_1} \otimes c_{A_2, M_1} \otimes \mathrm{id}_{M_2}) \bigr).
\end{equation}
The notions of right $\mathbf{A}$-module, morphism of right $\mathbf{A}$-modules and braided tensor product of right $\mathbf{A}$-modules are defined similarly; for the latter, given right $\mathbf{A}_i$-modules $\mathbf{N}_i = (N_i, \smallblacktriangleleft_i)$ for $i=1,2$ we have
\begin{equation}\label{defBraidedTensorProductOfRightModules}
\mathbf{N}_1 \,\widetilde{\otimes}\, \mathbf{N}_2 = \bigl( N_1 \otimes N_2, \: (\smallblacktriangleleft_{\!1} \otimes \smallblacktriangleleft_{\!2}) \circ (\mathrm{id}_{N_1} \otimes c_{N_2, A_1} \otimes \mathrm{id}_{A_2}) \bigr).
\end{equation}
\indent Finally, a $(\mathbf{A}', \mathbf{A})$-bimodule is a triple $\mathbf{B} = (B, \smallblacktriangleright, \smallblacktriangleleft)$ such $(B, \smallblacktriangleright)$ is a left $\mathbf{A}'$-module, $(B, \smallblacktriangleleft)$ is a right $\mathbf{A}$-module and we have
\begin{equation}\label{defBimoduleInAMonoidalCat}
\smallblacktriangleright \circ (\mathrm{id}_{A'} \otimes \smallblacktriangleleft) = \smallblacktriangleleft \circ (\smallblacktriangleright \otimes \mathrm{id}_A).
\end{equation}
If $\mathbf{B}_1 = (B_1, \smallblacktriangleright_{\!1}, \smallblacktriangleleft_{\!1})$ is a $(\mathbf{A}'_1, \mathbf{A}_1)$-bimodule and $\mathbf{B}_2 = (B_2, \smallblacktriangleright_{\!2}, \smallblacktriangleleft_{\!2})$ is a $(\mathbf{A}'_2, \mathbf{A}_2)$-bimodule, the {\em braided tensor product} of $\mathbf{B}_1$ and $\mathbf{B}_2$ is the $(\mathbf{A}'_1 \,\widetilde{\otimes}\, \mathbf{A}'_2, \mathbf{A}_1 \,\widetilde{\otimes}\, \mathbf{A}_2)$-bimodule defined by combining the left and right actions \eqref{defBraidedTensorProductOfLeftModules} and \eqref{defBraidedTensorProductOfRightModules}.

\subsection{The Morita category \texorpdfstring{$\mathrm{Bim}_{\mathcal{C}}$}{of algebras and bimodules in C}}\label{sectionMoritaCategoryInGeneral}
The goal of this subsection is to define a category whose objects are algebras and morphisms are isomorphism classes of bimodules in $\mathcal{C}$. The main non-trivial statement is that this category is monoidal. 

Recall from \cite[\S III.3]{MLCat} that a coequalizer of a pair of morphisms $\xymatrix{
X \ar@<.7ex>[r]^-{f_1} \ar@<-.7ex>[r]_-{f_2} & Y }$ in $\mathcal{C}$, if it exists, is an object $C \in \mathcal{C}$ together with a morphism $q : Y \to C$ such that
\begin{enumerate}
\item $q \circ f_1 = q \circ f_2$
\item if $r : Y \to D$ satisfies $r \circ f_1 = r \circ f_2$, then there is a unique $\widetilde{r} : C \to D$ such that $r = \widetilde{r} \circ q$.
\end{enumerate}
A coequalizer is unique up to a unique isomorphism, and denoted by $(C,q)= \mathrm{coeq}(f_1, f_2)$. Note the following simple fact: if $g,g' \in \Hom_{\mathcal{C}}(C,Z)$ are such that $g \circ q = g' \circ q$, then $g = g'$.

\smallskip

\indent From now on, we make the following two assumptions on the monoidal category $\mathcal{C}$:
\begin{itemize}
\item $\mathcal{C}$ has all coequalizers.
\item For any $X \in \mathcal{C}$, the functors $- \otimes X$ and $X \otimes -$ preserve coequalizers.
\end{itemize}
Explicitly, the second item means that if $(C,q) = \mathrm{coeq}(f_1,f_2)$ then $(C \otimes X, q \otimes \mathrm{id}_X) = \mathrm{coeq}(f_1 \otimes \mathrm{id}_X, f_2 \otimes \mathrm{id}_X)$ and similarly for $X \otimes -$. With these assumptions we can define the composition operation for bimodules as follows. For $i = 1,2,3$, let $\mathbf{A}_i = (A_i, m_i, \eta_i)$ be algebras in $\mathcal{C}$ and let $\mathbf{B}_1 = (B_1, \smallblacktriangleright_{\! 1}, \smallblacktriangleleft_{\! 1})$ be an $(\mathbf{A}_2, \mathbf{A}_1)$-bimodule and $\mathbf{B}_2 = (B_2, \smallblacktriangleright_{\!2}, \smallblacktriangleleft_{\!2})$ be an $(\mathbf{A}_3, \mathbf{A}_2)$-bimodule. We let
\begin{equation}\label{defCompositionOfBimodules}
\mathbf{B}_2 \circ \mathbf{B}_1 = \mathrm{coeq}\biggl(\xymatrix@C=5em{
B_2 \otimes A_2 \otimes B_1  \ar@<.7ex>[r]^-{\smallblacktriangleleft_{\!2} \,\otimes\, \mathrm{id}_{B_1}} \ar@<-.7ex>[r]_-{\mathrm{id}_{B_2} \,\otimes\, \smallblacktriangleright_{\!1}}
& B_2 \otimes B_1 }\biggr).
\end{equation}
and denote by $\pi \in \Hom_{\mathcal{C}}(B_2 \otimes B_1, \mathbf{B}_2 \circ \mathbf{B}_1)$ the morphism given by the definition of a coequalizer. A more usual notation for $\mathbf{B}_2 \circ \mathbf{B}_1$ is $\mathbf{B}_2 \otimes_{A_2} \mathbf{B}_1$, but it is less suited to our categorical purposes.

\smallskip

There is a natural structure of $(\mathbf{A}_3, \mathbf{A}_1)$-bimodule on $\mathbf{B}_2 \circ \mathbf{B}_1$, as we now recall. Consider the following pair of arrows in $\mathcal{C}$
\begin{equation}\label{arrowsForDefiningLeftActionOnCoequalizer}
\xymatrix@C=8em{
A_3 \otimes B_2 \otimes A_2 \otimes B_1  \ar@<.7ex>[r]^-{\mathrm{id}_{A_3} \otimes \,\smallblacktriangleleft_{\!2} \,\otimes\, \mathrm{id}_{B_1}} \ar@<-.7ex>[r]_-{\mathrm{id}_{A_3 \otimes B_2} \,\otimes\, \smallblacktriangleright_{\! 1}}
& A_3 \otimes B_2 \otimes B_1. }
\end{equation}
The morphism $\pi \circ (\smallblacktriangleright_{\!2} \otimes \mathrm{id}_{B_2}) : A_3 \otimes B_2 \otimes B_1 \to \mathbf{B}_2 \circ \mathbf{B}_1$ coequalizes \eqref{arrowsForDefiningLeftActionOnCoequalizer}. Since $A_3 \otimes (\mathbf{B}_2 \circ \mathbf{B}_1)$ together with $\mathrm{id}_{A_3} \otimes \pi$ is the coequalizer of \eqref{arrowsForDefiningLeftActionOnCoequalizer}, there exists a unique morphism $\lambda : A_3 \otimes (\mathbf{B}_2 \circ \mathbf{B}_1) \to \mathbf{B}_2 \circ \mathbf{B}_1$ in $\mathcal{C}$ such that the diagram
\begin{equation}\label{defLeftActionOnCoequalizer}
\xymatrix@C=5em{
A_3 \otimes B_2 \otimes B_1 \ar[r]^-{\smallblacktriangleright_{\!2} \otimes \mathrm{id}_{B_1}} \ar[d]_{\mathrm{id}_{A_3} \otimes \pi} & B_2 \otimes B_1 \ar[d]^{\pi}\\
A_3 \otimes (\mathbf{B}_2 \circ \mathbf{B}_1) \ar[r]_-{\lambda} & \mathbf{B}_2 \circ \mathbf{B}_1.
}
\end{equation}
commutes. It is easy to show that $\lambda$ is a left action:
\begin{center}
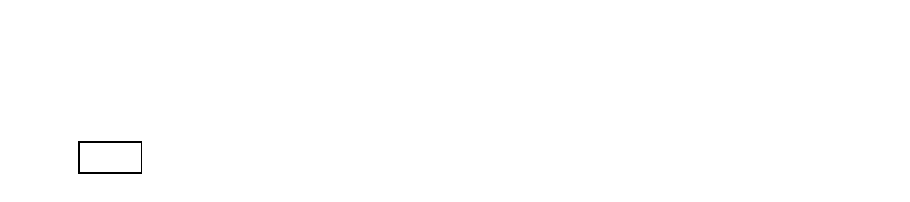
\end{center}
Hence the equality $\lambda \circ (m_3 \otimes \mathrm{id}_{\mathbf{B}_2 \circ \mathbf{B}_1}) = \lambda \circ (\mathrm{id}_{A_3} \otimes \lambda)$ follows from the universal property of $\mathrm{id}_{A_3 \otimes A_3} \otimes \pi$ as the coequalizer of
\[ \xymatrix@C=8em{
A_3 \otimes A_3 \otimes B_2 \otimes A_2 \otimes B_1  \ar@<.7ex>[r]^-{\mathrm{id}_{A_3 \otimes A_3} \otimes \,\smallblacktriangleleft_{\!2} \,\otimes\, \mathrm{id}_{B_1}} \ar@<-.7ex>[r]_-{\mathrm{id}_{A_3 \otimes A_3 \otimes B_2} \,\otimes\, \smallblacktriangleright_{\!1}}
& A_3 \otimes A_3 \otimes B_2 \otimes B_1. } \]
By similar arguments, a right action $\rho : (\mathbf{B}_2 \circ \mathbf{B}_1) \otimes A_1 \to \mathbf{B}_2 \circ \mathbf{B}_1$ is defined by the commutative diagram
\begin{equation}\label{defRightActionOnCoequalizer}
\xymatrix@C=5em{
B_2 \otimes B_1 \otimes A_1 \ar[r]^-{\mathrm{id}_{B_1} \otimes \, \smallblacktriangleleft_{\!1}} \ar[d]_-{\pi \otimes \mathrm{id}_{A_1}} & B_2 \otimes B_1 \ar[d]^{\pi}\\
(\mathbf{B}_2 \circ \mathbf{B}_1) \otimes A_1 \ar[r]_-{\exists! \, \rho} & \mathbf{B}_2 \circ \mathbf{B}_1.
}
\end{equation}
Moreover, one can check the actions $\lambda$ and $\rho$ commute in the sense of \eqref{defBimoduleInAMonoidalCat}. Hence $\mathbf{B}_2 \circ \mathbf{B}_1$ has a natural structure of $(\mathbf{A}_3, \mathbf{A}_1)$-bimodule. In the sequel when we write $\mathbf{B}_2 \circ \mathbf{B}_1$ we implicitly mean this bimodule and not just the object defined in \eqref{defCompositionOfBimodules}.

\smallskip

\indent Let moreover $\mathbf{B}_3$ be an $(\mathbf{A}_4, \mathbf{A}_3)$-bimodule. It is a tedious but easy exercise to check that there are isomorphisms of bimodules
\[ \mathbf{B}_3 \circ (\mathbf{B}_2 \circ \mathbf{B}_1) \cong (\mathbf{B}_3 \circ \mathbf{B}_2) \circ \mathbf{B}_1 \quad \text{and} \quad \mathbf{A}_2 \circ \mathbf{B}_1 \cong \mathbf{B}_1 \cong \mathbf{B}_1 \circ \mathbf{A}_1 \]
where the algebras $\mathbf{A}_1$, $\mathbf{A}_2$ are viewed as bimodules over themselves by multiplication.
As a result we can make the following definition:
\begin{definition}\label{defBimC}
The category $\mathrm{Bim}_{\mathcal{C}}$ has for objects the algebras in $\mathcal{C}$ and $\Hom_{\mathrm{Bim}_{\mathcal{C}}}(\mathbf{A}_1, \mathbf{A}_2)$ is the collection of isomorphism classes of $(\mathbf{A}_2, \mathbf{A}_1)$-bimodules.\footnote{Note the switch!} The identity morphism $\mathrm{id}_{\mathbf{A}}$ is the algebra $\mathbf{A}$ viewed as a $(\mathbf{A}, \mathbf{A})$-bimodule. The composition is the operation $\circ$ defined above.
\end{definition}

\indent Recall the braided tensor product of algebras and bimodules defined in \S\ref{sectionPreliminariesModules}, which uses the braiding in $\mathcal{C}$. The following is well known: 

\begin{lemma}\label{theoCompatibilityCompositionAndBraidedTensorProduct}
Let $\mathbf{B}_1 \in \Hom_{\mathrm{Bim}_{\mathcal{C}}}(\mathbf{A}_1, \mathbf{A}_2)$, $\mathbf{B}_2 \in \Hom_{\mathrm{Bim}_{\mathcal{C}}}(\mathbf{A}_2, \mathbf{A}_3)$, $\mathbf{B}'_1 \in \Hom_{\mathrm{Bim}_{\mathcal{C}}}(\mathbf{A}'_1, \mathbf{A}'_2)$ and $\mathbf{B}'_2 \in \Hom_{\mathrm{Bim}_{\mathcal{C}}}(\mathbf{A}'_2, \mathbf{A}'_3)$. Then
\[ (\mathbf{B}_2 \circ \mathbf{B}_1) \,\widetilde{\otimes}\, (\mathbf{B}'_2 \circ \mathbf{B}'_1) = (\mathbf{B}_2 \,\widetilde{\otimes}\, \mathbf{B}'_2) \circ (\mathbf{B}_1 \,\widetilde{\otimes}\, \mathbf{B}'_1) \text{ in } \Hom_{\mathrm{Bim}_{\mathcal{C}}}(\mathbf{A}_1 \,\widetilde{\otimes}\, \mathbf{A}'_1, \mathbf{A}_3 \,\widetilde{\otimes}\, \mathbf{A}'_3). \]
\end{lemma}

\noindent Said differently, we can see the braided tensor product as a bifunctor $\widetilde{\otimes} : \mathrm{Bim}_{\mathcal{C}} \times \mathrm{Bim}_{\mathcal{C}} \to \mathrm{Bim}_{\mathcal{C}}$. As a result:

\begin{corollary}\label{coroBimCMonoidal}
$\bigl( \mathrm{Bim}_{\mathcal{C}}, \,\widetilde{\otimes} \bigr)$ is a strict monoidal category. Its unit object is $\boldsymbol{1}$ (the unit object of $\mathcal{C}$) endowed with the multiplication $\boldsymbol{1} \otimes \boldsymbol{1} \overset{=}{\longrightarrow} \boldsymbol{1}$ and the unit $\mathrm{id}_{\boldsymbol{1}}$.
\end{corollary}

\subsection{Twisting of bimodules by algebra morphisms}\label{subsecTwistingBimod}
 Let $\mathbf{A}_1, \mathbf{A}'_1, \mathbf{A}_2, \mathbf{A}'_2$ be algebras in the braided monoidal category $\mathcal{C}$ and $\mathbf{B} = (B,\smallblacktriangleright,\smallblacktriangleleft)$ be an $(\mathbf{A}_2,\mathbf{A}_1)$-bimodule. Suppose we are given morphisms of algebras $f_1 : \mathbf{A}'_1 \to \mathbf{A}_1$ and $f_2 : \mathbf{A}'_2 \to \mathbf{A}_2$. Then we obtain the $(\mathbf{A}'_2,\mathbf{A}'_1)$-bimodule
\[ f_2 \smallblacktriangleright \mathbf{B} \smallblacktriangleleft f_1 = \bigl( B, \: \smallblacktriangleright \circ (f_2 \otimes \mathrm{id}_B), \: \smallblacktriangleleft \circ (\mathrm{id}_B \otimes f_1) \bigr). \]
This construction satisfies straightforward properties:
\begin{itemize}
\item Let $\mathbf{A}_1, \mathbf{A}'_1, \mathbf{A}''_1, \mathbf{A}_2, \mathbf{A}'_2, \mathbf{A}''_2$ be algebras in $\mathcal{C}$, let $f_1 : \mathbf{A}'_1 \to \mathbf{A}_1$, $f_1' : \mathbf{A}''_1 \to \mathbf{A}'_1$, $f_2 : \mathbf{A}'_2 \to \mathbf{A}_2$, $f_2' : \mathbf{A}''_2 \to \mathbf{A}'_2$ be morphisms of algebras and let $\mathbf{B}$ be an $(\mathbf{A}_2,\mathbf{A}_1)$-bimodule. Then
\begin{equation}\label{twistingCompMor}
f_2' \smallblacktriangleright (f_2 \smallblacktriangleright \mathbf{B} \smallblacktriangleleft f_1) \smallblacktriangleleft f_1' = (f_2 \circ f_2') \smallblacktriangleright \mathbf{B} \smallblacktriangleleft (f_1 \circ f_1').
\end{equation}
\item Let $\mathbf{A}_1, \mathbf{A}'_1, \mathbf{A}_2,  \mathbf{A}_3, \mathbf{A}'_3$ be algebras in $\mathcal{C}$, let $f_1 : \mathbf{A}'_1 \to \mathbf{A}_1$, $f_3 : \mathbf{A}'_3 \to \mathbf{A}_3$ be morphisms of algebras and let $\mathbf{B}_1$ be an $(\mathbf{A}_2,\mathbf{A}_1)$-bimodule and $\mathbf{B}_2$ be an $(\mathbf{A}_3,\mathbf{A}_2)$-bimodule. Then
\begin{equation}\label{twistingCompBim}
(f_3 \smallblacktriangleright \mathbf{B}_2) \circ (\mathbf{B}_1 \smallblacktriangleleft f_1) = f_3 \smallblacktriangleright (\mathbf{B}_2 \circ \mathbf{B}_1) \smallblacktriangleleft f_1.
\end{equation}
\item Let $\mathbf{A}_1, \mathbf{A}'_1, \mathbf{A}_2, \mathbf{A}'_2,  \mathbf{A}_3, \mathbf{A}'_3, \mathbf{A}_4, \mathbf{A}'_4$ be algebras in $\mathcal{C}$, let $f_i : \mathbf{A}'_i \to \mathbf{A}_i$ be a morphism of algebras for each $i=1,\ldots,4$ and let $\mathbf{B}_1$ be an $(\mathbf{A}_2,\mathbf{A}_1)$-bimodule and $\mathbf{B}_2$ be an $(\mathbf{A}_4,\mathbf{A}_3)$-bimodule. Then
\begin{equation}\label{twistingCompBrTens}
(f_2 \smallblacktriangleright \mathbf{B}_1 \smallblacktriangleleft f_1) \,\widetilde{\otimes}\, ( f_4 \smallblacktriangleright \mathbf{B}_2  \smallblacktriangleleft f_3) = (f_2 \otimes f_4) \smallblacktriangleright (\mathbf{B}_1 \,\widetilde{\otimes}\, \mathbf{B}_2) \smallblacktriangleleft (f_1 \otimes f_3).
\end{equation}
\end{itemize}

\begin{lemma}\label{lemmaTwistingRegBimod}
Let $\mathbf{A}, \mathbf{A}'$ be algebras in $\mathcal{C}$ and $f : \mathbf{A}' \to \mathbf{A}$ be an isomorphism of algebras. Then the bimodules $f \smallblacktriangleright \mathbf{A} \in \Hom_{\mathrm{Bim}_{\mathcal{C}}}(\mathbf{A},\mathbf{A}')$ and $\mathbf{A} \smallblacktriangleleft f \in \Hom_{\mathrm{Bim}_{\mathcal{C}}}(\mathbf{A}',\mathbf{A})$ are inverse to each other. It follows that $\mathbf{A}$ and $\mathbf{A}'$ are isomorphic in $\mathrm{Bim}_{\mathcal{C}}$.
\end{lemma}
\begin{proof}
Note that $f$ yields isomorphisms of bimodules
\begin{equation}\label{twistingLeftToRight}
f : (\mathbf{A}' \smallblacktriangleleft f^{-1}) \overset{\sim}{\longrightarrow} (f \smallblacktriangleright \mathbf{A}) \quad \text{ and } \quad f : (f^{-1} \smallblacktriangleright \mathbf{A}') \overset{\sim}{\longrightarrow} (\mathbf{A} \smallblacktriangleleft f).
\end{equation}
It follows that $(\mathbf{A}' \smallblacktriangleleft f^{-1}) = (f \smallblacktriangleright \mathbf{A})$ and $(f^{-1} \smallblacktriangleright \mathbf{A}') = (\mathbf{A} \smallblacktriangleleft f)$ as morphisms in $\mathrm{Bim}_{\mathcal{C}}$. Using \eqref{twistingCompMor} and \eqref{twistingCompBim} we thus have
\[ (f \smallblacktriangleright \mathbf{A}) \circ (\mathbf{A} \smallblacktriangleleft f) = f \smallblacktriangleright ( \mathbf{A} \circ \mathbf{A}) \smallblacktriangleleft f = f \smallblacktriangleright  \mathbf{A} \smallblacktriangleleft f = (\mathbf{A}' \smallblacktriangleleft f^{-1}) \smallblacktriangleleft f = \mathbf{A}' \]
and 
\begin{align*}
(\mathbf{A} \smallblacktriangleleft f) \circ (f \smallblacktriangleright \mathbf{A}) &= (f^{-1} \smallblacktriangleright \mathbf{A}') \circ (\mathbf{A}' \smallblacktriangleleft f^{-1})\\
&= f^{-1} \smallblacktriangleright (\mathbf{A}' \circ \mathbf{A}') \smallblacktriangleleft f^{-1} = f^{-1} \smallblacktriangleright \mathbf{A}' \smallblacktriangleleft f^{-1} = (\mathbf{A} \smallblacktriangleleft f) \smallblacktriangleleft f^{-1} = \mathbf{A}.
\end{align*}
It follows that $f \smallblacktriangleright \mathbf{A}$ and $\mathbf{A} \smallblacktriangleleft f$ are inverse each other.
\end{proof}

\section{The Morita category of half-braided algebras}\label{sectionMoritaCategoryHBAlgebras}
The goal of this section is to define a suitable extra structure on the algebras and bimodules of the previous section to ensure that the category of algebras and bimodules in $\mathcal{C}$ is not only monoidal as recalled in Corollary \ref{coroBimCMonoidal} but even braided. 
On the level of algebras the suitable notion is what we call a ``half-braided algebra'': it is an algebra endowed with a half-brading satisfying a suitable relation which, we warn the reader, is NOT that of an algebra in the Drinfeld center $\mathcal{Z}(\mathcal{C})$. 
After defining these objects we study their bimodules and show that they are automatically endowed with two natural half-braidings. 
When these coincide we say that the bimodule is {\em hb-compatible}. 
We end the section showing that the category of half-braided algebras and their hb-compatible bimodules (up to isomorphisms) is a braided monoidal category. 

\subsection{Half-braided algebras}\label{subsectionHbAlg}
Here we introduce half-braided algebras and a braided tensor product for them which turns out to be ``commutative''. This material is not new: P. Schauenburg defined the dual concept (\textit{i.e.} half-braided coalgebra) in \cite[Def.\,1.1]{schauenburg} under the name ``central coalgebra'', inspired by the example of the coend $\int^X X^* \otimes X$ studied in \cite{NS}. He proved that the braided tensor product of two central coalgebras is again a central coalgebra and that this operation is commutative up to isomorphism \cite[Prop.\,1.8]{schauenburg}\footnote{Warning: in Schauenburg's papers, diagrams must be read from top to bottom while we use the opposite convention. Also the ``crossing diagram'' for the braiding in $\mathcal{C}$ is the opposite to the one used here.}; for convenience of the reader we will repeat the dual statements and their proofs in Propositions \ref{propBraidedTensorProductAlgebra} and \ref{propIsoBetweenA1A2AndA2A1} below. Also, the concept of half-braided algebra is used in \cite[eq.\,(2.6)]{sch2} under the name ``cocentral algebra''.
\begin{definition}\label{defHBAlgebra}
1. A half-braided algebra in $\mathcal{C}$ is a quadruple $\mathbb{A} = (A, t, m, \eta)$ such that $(A,t) \in \mathcal{Z}(\mathcal{C})$, $(A, m, \eta)$ is an algebra in $\mathcal{C}$ and for all $X \in \mathcal{C}$ we have
\smallskip
\begin{equation}\label{axiomsHalfBraidedAlgebra}
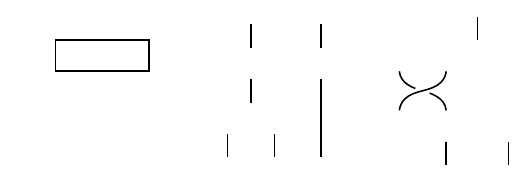
\end{equation}
2. Let $\mathbb{A} = (A,t,m,\eta)$ and $\mathbb{A}' = (A', t', m', \eta')$ be half-braided algebras. We say that $f \in \Hom_{\mathcal{C}}(A,A')$ is a morphism of half-braided algebras from $\mathbb{A}$ to $\mathbb{A}'$ if $f$ is a morphism of algebras and $f \in \Hom_{\mathcal{Z}(\mathcal{C})}\bigl( (A,t), (A',t') \bigr)$ \textit{i.e.} $f$ commutes with the half-braidings.
\end{definition}

\begin{remark}
A half-braided algebra in $\mathcal{C}$ is not an algebra in $\mathcal{Z}(\mathcal{C})$, which would mean that $m \in \Hom_{\mathcal{Z}(\mathcal{C})}\bigl( (A,t) \otimes (A,t), (A,t) \bigr)$ or explicitly (using \eqref{defTensorProductInZC}):
\[ (\mathrm{id}_X \otimes m) \circ (t_X \otimes \mathrm{id}_A) \circ (\mathrm{id}_A \otimes t_X) = t_X \circ (m \otimes \mathrm{id}_X). \]
Instead the two conditions in \eqref{axiomsHalfBraidedAlgebra} are respectively equivalent to
\[ m \in \Hom_{\mathcal{Z}(\mathcal{C})}\bigl( (A,t) \otimes (A,c^{-1}_{-,A}),\, (A,t) \bigr) \quad \text{ and } \quad m \in \Hom_{\mathcal{Z}(\mathcal{C})}\bigl( (A,c_{A,-}) \otimes (A,t),\, (A,t) \bigr). \]
\end{remark}

For two half-braided algebras $\mathbb{A}_1 = (A_1, t^1, m_1, \eta_1)$ and $\mathbb{A}_2 = (A_2, t^2, m_2, \eta_2)$ we define 
\begin{equation}\label{defBraidedTensorProductOfAlgebras2}
\begin{array}{l}
\mathbb{A}_1 \,\widetilde{\otimes}\, \mathbb{A}_2\\[.3em]
=\bigl( A_1 \otimes A_2, \: (t^1_- \otimes \mathrm{id}_{A_2}) \circ (\mathrm{id}_{A_1} \otimes t^2_-), \: (m_1 \otimes m_2) \circ (\mathrm{id}_{A_1} \otimes c_{A_2,A_1} \otimes \mathrm{id}_{A_2}),\: \eta_1 \otimes \eta_2 \bigr).
\end{array}
\end{equation}
This simply combines the monoidal product \eqref{defTensorProductInZC} in $\mathcal{Z}(\mathcal{C})$ and the braided tensor product \eqref{defBraidedTensorProductOfAlgebras} of algebras in $\mathcal{C}$.
\begin{proposition}\label{propBraidedTensorProductAlgebra}
$\mathbb{A}_1 \,\widetilde{\otimes}\, \mathbb{A}_2$ is a half-braided algebra.
\end{proposition}
\begin{proof}
We just need to check the conditions \eqref{axiomsHalfBraidedAlgebra}:
\begin{center}
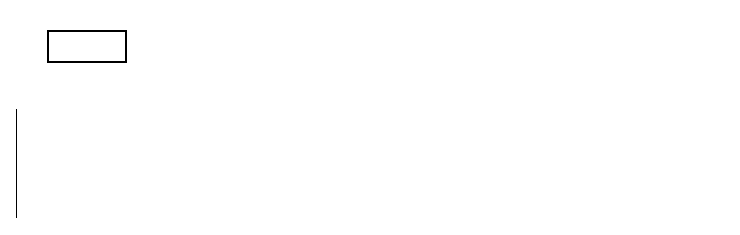
\end{center}
For the first equality we used isotopy and for the second equality we used the first condition in \eqref{axiomsHalfBraidedAlgebra} for $m_1$, $t_1$ and $m_2$, $t_2$ respectively. The second equality in \eqref{axiomsHalfBraidedAlgebra} is checked similarly.
\end{proof}

\begin{remark}
The operation $\widetilde{\otimes}$ is associative.
\end{remark}

\indent We remind the reader that our goal is to show that a suitable category whose objects are half-braided algebras is not only monoidal but even braided. Let $\mathbb{A}_1 = (A_1, t^1, m_1, \eta_1)$ and $\mathbb{A}_2 = (A_2, t^2, m_2, \eta_2)$ be half-braided algebras. The braiding $T$ in $\mathcal{Z}(\mathcal{C})$, recalled in \S\ref{subsecZC}, gives an isomorphism
\begin{equation}\label{isoBrZCforHBAlg}
T_{\mathbb{A}_1, \mathbb{A}_2} = t^1_{A_2} : (A_1,t^1) \otimes (A_2,t^2) \overset{\sim}{\longrightarrow} (A_2,t^2) \otimes (A_1,t^1)
\end{equation}
in $\mathcal{Z}(\mathcal{C})$. The following property will be the key ingredient to build the braiding in a category of bimodules over half-braided algebras: 

\begin{proposition}\label{propIsoBetweenA1A2AndA2A1}
$T_{\mathbb{A}_1, \mathbb{A}_2}$ is an isomorphism of half-braided algebras $\mathbb{A}_1 \,\widetilde{\otimes}\, \mathbb{A}_2 \to \mathbb{A}_2 \,\widetilde{\otimes}\, \mathbb{A}_1$.
\end{proposition}
\begin{proof}
We already know that $T_{\mathbb{A}_1, \mathbb{A}_2}$ is an isomorphism in $\mathcal{Z}(\mathcal{C})$, \textit{i.e.} an iso in $\mathcal{C}$ which commutes with the half-braidings of $\mathbb{A}_1 \,\widetilde{\otimes}\, \mathbb{A}_2$ and $\mathbb{A}_2 \,\widetilde{\otimes}\, \mathbb{A}_1$. It remains to check that it is an algebra morphism:
\begin{center}
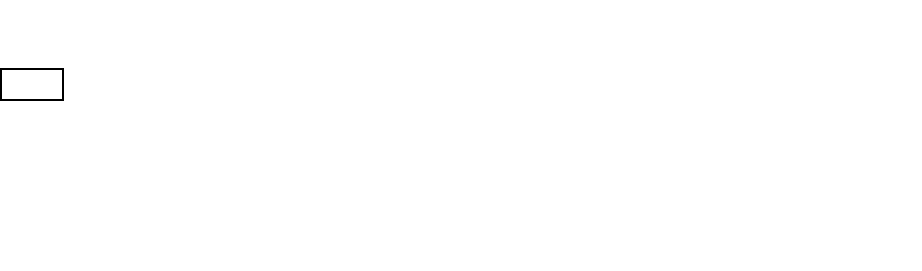
\end{center}
The first equality uses the second equality in \eqref{axiomsHalfBraidedAlgebra}, the second uses the first equality in \eqref{axiomsHalfBraidedAlgebra}, the third uses the defining property of half-braidings \eqref{axiomHalfBraiding} and the fourth is by naturality.  Finally by the naturality of $t^1$ we get
\[ t^1_{A_2} \circ (\eta_1 \otimes \eta_2) = (\eta_2 \otimes \mathrm{id}_{A_1}) \circ t^1_{\boldsymbol{1}} \circ \eta_1 = (\eta_2 \otimes \mathrm{id}_{A_1}) \circ \eta_1 = \eta_2 \otimes \eta_1. \]
Hence $t^1_{A_2}$ preserves the units.
\end{proof}

\begin{remark}
The braiding $c$ in $\mathcal{C}$ gives an isomorphism $c_{A_1,A_2} : A_1 \otimes A_2 \to A_2 \otimes A_1$ in $\mathcal{C}$ which is an isomorphism of algebras but not of half-braided algebras because in general $c_{A_1,A_2} \not\in \mathrm{Hom}_{\mathcal{Z}(\mathcal{C})}\bigl( (A_1,t^1) \otimes (A_2,t^2), (A_2,t^2) \otimes (A_1,t^1) \bigr)$.
\end{remark}

\subsection{(Bi)modules over half-braided algebras}\label{subsectionHbCohBimod}
As in Subsection \ref{sectionMoritaCategoryInGeneral}, the morphisms of our monoidal category are going to be bimodules, but now over {\em half-braided} algebras. In this subsection we show that a bimodule over half-braided algebras is automatically endowed with two half-braidings and when they coincide we call it {\em hb-compatible}. The main result of this section is to show that the monoidal product of hb-compatible bimodules is again hb-compatible. This endows what will be our final category with a monoidal structure.

\indent Let $\mathbb{A} = (A,t,m,\eta)$ and $\mathbb{A}' = (A',t',m',\eta')$ be half-braided algebras in $\mathcal{C}$. In particular $\mathbb{A}$ and $\mathbb{A}'$ are algebras. Let $\mathbf{M} = (M, \smallblacktriangleright)$ be a left $\mathbb{A}'$-module and $\mathbf{N} = (N, \smallblacktriangleleft)$ be a right $\mathbb{A}$-module (\S \ref{sectionPreliminariesModules}). For all $X \in \mathcal{C}$ define
\begin{equation}\label{halfBraidingInTermOfAction}
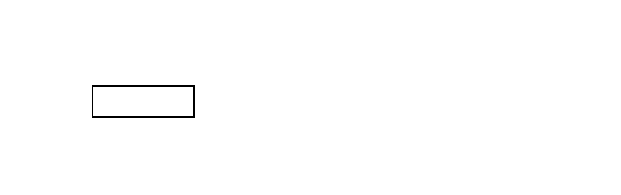
\end{equation}
The notation $\mathrm{hbl}$ (resp. $\mathrm{hbr}$) is an abbreviation of ``half-braiding on left (resp. right) module'', which makes sense thanks to the next proposition:
\begin{proposition}\label{propPropertiesOfModulesOverHBAlgebras}
1. $\mathrm{hbl}^{\mathbf{M}}$ and $\mathrm{hbr}^{\mathbf{N}}$ are half-braidings.
\\2. For all $X \in \mathcal{C}$ we have
\begin{equation}\label{axiomsHalfBraidedLeftModule}
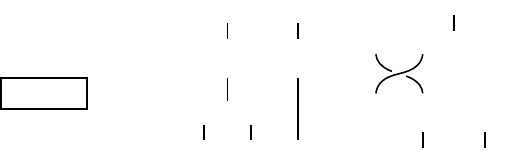
\end{equation}
and
\begin{equation}\label{axiomsHalfBraidedRightModule}
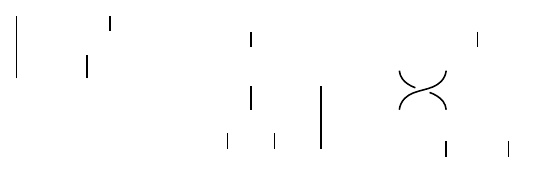
\end{equation}
\end{proposition}
\begin{proof}
1. Naturality is obvious, from the naturality of the half-braiding $t$. For all $X,Y \in \mathcal{C}$, $\mathrm{hbl}^{\mathbf{M}}_{X \otimes Y}$ can be rewritten as follows:
\begin{center}
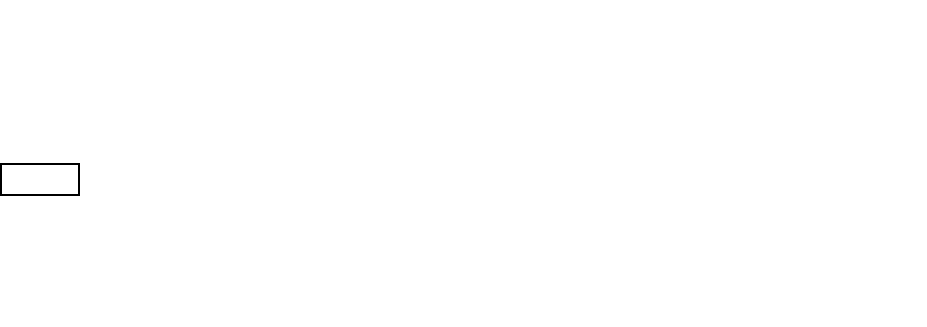
\end{center}
To obtain the first term we used \eqref{axiomHalfBraiding}, then for the first equality we used that $\eta'$ is the unit of $m$', for the second equality we used the first equality in \eqref{axiomsHalfBraidedAlgebra} and for the third equality we used that $\smallblacktriangleright$ is an action and naturality of the braiding. The last term is equal to $(\mathrm{id}_X \otimes \mathrm{hbl}^{\mathbf{M}}_Y) \circ (\mathrm{hbl}^{\mathbf{M}}_X \otimes \mathrm{id}_Y)$, as desired. A similar computation holds for $\mathrm{hbr}^{\mathbf{N}}$. Finally one can check that the inverses are given by
\begin{equation}\label{INVERSEhalfBraidingInTermOfAction}
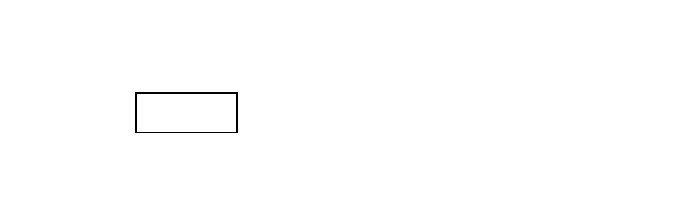
\end{equation}
\noindent 2. The first equality in \eqref{axiomsHalfBraidedLeftModule} is obtained as follows
\begin{center}
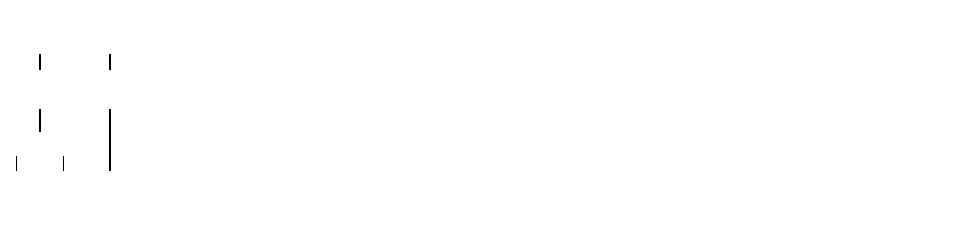
\end{center}
For the other equality in \eqref{axiomsHalfBraidedLeftModule} note that
\begin{center}
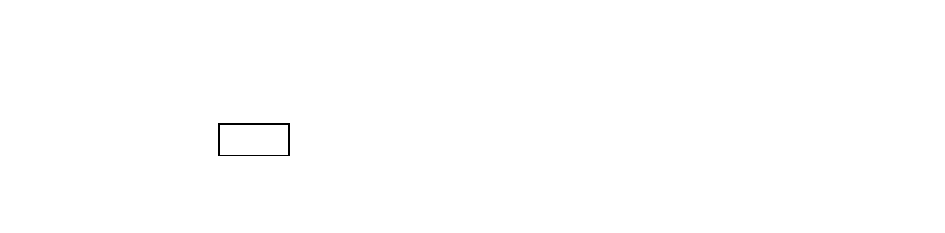
\end{center}
Equation \eqref{axiomsHalfBraidedRightModule} is proven similarly.
\end{proof}

\begin{remark}\label{remarkRegularBimoduleIsCoherent}
We can consider $\mathbb{A}$ as a left or right $\mathbb{A}$-module by multiplication. We get from \eqref{axiomsHalfBraidedAlgebra} that $\mathrm{hbl}^{\mathbb{A}} = \mathrm{hbr}^{\mathbb{A}} = t$. In this case the equalities \eqref{axiomsHalfBraidedLeftModule} and \eqref{axiomsHalfBraidedRightModule} simply reduce to \eqref{axiomsHalfBraidedAlgebra}.
\end{remark}

\indent The half-braidings $\mathrm{hbl}^{\mathbf{M}}$, $\mathrm{hbr}^{\mathbf{N}}$ are compatible with the braided tensor product of modules (defined in \S \ref{sectionPreliminariesModules}), in the following sense. For $i = 1,2$, let $\mathbf{M}_i = (M_i,\smallblacktriangleright_i)$ be a left module over a half-braided algebra $\mathbb{A}'_i$ and let $\mathbf{N}_i = (N_i, \smallblacktriangleleft_i)$ be a right module over a half-braided algebra $\mathbb{A}_i$.

\begin{lemma}\label{lemmaHalfBraidingOnModuleCompatibleWithTensorProductZC}
With these notations, we have for all $X \in \mathcal{C}$
\begin{align*}
\mathrm{hbl}^{\mathbf{M}_1 \,\widetilde{\otimes}\, \mathbf{M}_2}_X &= \bigl( \mathrm{hbl}^{\mathbf{M}_1}_X \otimes \mathrm{id}_{M_2} \bigr) \circ \bigl( \mathrm{id}_{M_1} \otimes \mathrm{hbl}^{\mathbf{M}_2}_X \bigr),\\
\mathrm{hbr}^{\mathbf{N}_1 \,\widetilde{\otimes}\, \mathbf{N}_2}_X &= \bigl( \mathrm{hbr}^{\mathbf{N}_1}_X \otimes \mathrm{id}_{N_2} \bigr) \circ \bigl( \mathrm{id}_{N_1} \otimes \mathrm{hbr}^{\mathbf{N}_2}_X \bigr).
\end{align*}
\end{lemma}
\begin{proof}
Easy graphical computations left to the reader.
\end{proof}

\indent Now let $\mathbb{A}$, $\mathbb{A}'$ be half-braided algebras and $\mathbf{B}$ be a $(\mathbb{A}', \mathbb{A})$-bimodule as defined in \S \ref{sectionPreliminariesModules}. Since $\mathbf{B}$ is both a left $\mathbb{A}'$-module and a right $\mathbb{A}$-module we have the two half-braidings $\mathrm{hbl}^{\mathbf{B}}$ and $\mathrm{hbr}^{\mathbf{B}}$ which come respectively from the left and right actions, recall \eqref{halfBraidingInTermOfAction}.

\begin{definition}\label{defHbCoherentBimodule}
We say that the bimodule $\mathbf{B}$ is hb-compatible if $\mathrm{hbl}^{\mathbf{B}} = \mathrm{hbr}^{\mathbf{B}}$. In such a case we denote this half-braiding by $\mathrm{hb}^{\mathbf{B}}$.
\end{definition}

\begin{example}\label{coherentRegularBimod}
The regular bimodule of a half-braided algebra $\mathbb{A}$ is hb-compatible, by Remark \ref{remarkRegularBimoduleIsCoherent}.
\end{example}

Here is an example of a non-compatible bimodule:
\begin{example} Let $\mathbb{A} = (A, t, m, \eta)$ be a half-braided algebra. Let $\mathbf{B}$ be the $(\mathbb{A}, \mathbb{A})$-bimodule $\bigl( A \otimes A, m \otimes \mathrm{id}_A, \mathrm{id}_A \otimes m \bigr)$. In this case, the definition \eqref{halfBraidingInTermOfAction} together with \eqref{axiomsHalfBraidedAlgebra} give
\begin{align*}
\mathrm{hbl}^{\mathbf{B}}_X &= \bigl( t_X \otimes \mathrm{id}_A \bigr) \circ \bigl( \mathrm{id}_A \otimes c_{X,A}^{-1} \bigr),\\
\mathrm{hbr}^{\mathbf{B}}_X &= \bigl( c_{A,X} \otimes \mathrm{id}_A \bigr) \circ \bigl( \mathrm{id}_A \otimes t_X \bigr).
\end{align*}
In general these two half-braidings are not equal, and thus $\mathbf{B}$ is not hb-compatible in general. For instance take  $A = \underline{\End}(V) = V \otimes V^*$ with the half-braided algebra structure described in Example \ref{lemmaMatrixAlgebra} below. Then 
\begin{center}
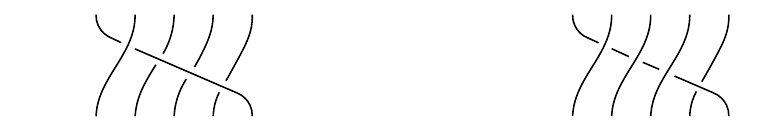
\end{center}
Assume that the functors $V \otimes -$ and $- \otimes V^*$ are faithful (which for instance is true for any $V$ in the category of finite-dimensional modules over a Hopf $k$-algebra, with $k$ a field). Then $\mathrm{hbl}^{\mathbf{B}}$ and $\mathrm{hbr}^{\mathbf{B}}$ are equal if and only if for all $X\in \mathcal{C}$ we have $c_{X, V^* \otimes V}^{-1} = c_{V^* \otimes V, X}$, \textit{i.e.} $\End(V)$ is in the M\"uger center of $\mathcal{C}$. There are many braided monoidal categories with trivial M\"uger center, which is equivalent to factorizability, see e.g. \cite{shimizu}.
\end{example}

\begin{proposition}\label{propHbCoherenceStableByBraidedTensorProduct}
If $\mathbf{B}_1$ is a hb-compatible $(\mathbb{A}'_1, \mathbb{A}_1)$-bimodule and $\mathbf{B}_2$ is a hb-compatible $(\mathbb{A}'_2, \mathbb{A}_2)$-bimodule then $\mathbf{B}_1 \,\widetilde{\otimes}\, \mathbf{B}_2$ is a hb-compatible $(\mathbb{A}'_1 \,\widetilde{\otimes}\, \mathbb{A}'_2, \mathbb{A}_1 \,\widetilde{\otimes}\, \mathbb{A}_2)$-bimodule.
\end{proposition}
\begin{proof}
This follows from Lemma \ref{lemmaHalfBraidingOnModuleCompatibleWithTensorProductZC}:
\[ \mathrm{hbl}^{\mathbf{B}_1 \,\widetilde{\otimes}\, \mathbf{B}_2}_X = \bigl( \mathrm{hbl}^{\mathbf{B}_1}_X \otimes \mathrm{id}_{B_2} \bigr) \circ \bigl( \mathrm{id}_{B_1} \otimes \mathrm{hbl}^{\mathbf{B}_2}_X \bigr)= \bigl( \mathrm{hbr}^{\mathbf{B}_1}_X \otimes \mathrm{id}_{B_2} \bigr) \circ \bigl( \mathrm{id}_{B_1} \otimes \mathrm{hbr}^{\mathbf{B}_2}_X \bigr) = \mathrm{hbr}^{\mathbf{B}_1 \,\widetilde{\otimes}\, \mathbf{B}_2}_X \]
for all $X \in \mathcal{C}$.
\end{proof}

\indent Recall from \S\ref{subsecTwistingBimod} that bimodules can be twisted by morphisms of algebras.
\begin{lemma}\label{lemmaTwistingCohBim}
Let $f_1 : \mathbb{A}'_1 \to \mathbb{A}_1$ and $f_2 : \mathbb{A}'_2 \to \mathbb{A}_2$ be morphisms of half-braided algebras, and $\mathbf{B}$ be a hb-compatible $(\mathbb{A}_2, \mathbb{A}_1)$-bimodule. Then $f_2 \smallblacktriangleright \mathbf{B} \smallblacktriangleleft f_1$ is also hb-compatible.
\end{lemma}
\begin{proof}
Straightforward computations reveal that
 \[ \mathrm{hbl}^{f_2 \smallblacktriangleright \mathbf{B} \smallblacktriangleleft f_1} = \mathrm{hbl}^{\mathbf{B}} \quad \text{ and } \quad \mathrm{hbr}^{f_2 \smallblacktriangleright \mathbf{B}  \smallblacktriangleleft f_1} = \mathrm{hbr}^{\mathbf{B}}.\qedhere \]
\end{proof}

\subsection{The monoidal category of hb-compatible bimodules}\label{sectionBraidingOnBimodules}
Until now we showed that hb-compatible bimodules can be tensored. In order to later be able to consider them as morphisms, we now need to define their ``composition'' which, as in Subsection \ref{sectionMoritaCategoryInGeneral}, is nothing but the tensor product over a middle half-braided algebra. We end this section defining the monoidal category $\mathrm{Bim}^{\mathrm{hb}}_{\mathcal{C}}$ of half-braided algebras and their hb-compatible bimodules. 

\smallskip

\indent Recall from \S\ref{sectionMoritaCategoryInGeneral} the operation $\circ$ on bimodules and the definition of the category $\mathrm{Bim}_{\mathcal{C}}$.

\begin{proposition}\label{propCompoOfHBCoherentBimodules}
If $\mathbf{B}_1$ is a hb-compatible $(\mathbb{A}_2, \mathbb{A}_1)$-bimodule and $\mathbf{B}_2$ is a hb-compatible $(\mathbb{A}_3, \mathbb{A}_2)$-bimodule then $\mathbf{B}_2 \circ \mathbf{B}_1$ is a hb-compatible $(\mathbb{A}_3, \mathbb{A}_1)$-bimodule.
\end{proposition}
\begin{proof}
Write $\mathbf{A}_i = (A_i, t^i, m_i, \eta_i$) for $i=1,2,3$ and $\mathbf{B}_j = (B_j, \smallblacktriangleright_{\!j}, \smallblacktriangleleft_{\!j})$ for $j=1,2$. Let $\pi : B_2 \otimes B_1 \to \mathbf{B}_2 \circ \mathbf{B}_1$ be the coequalizer \eqref{defCompositionOfBimodules} and let $X$ be an object in $\mathcal{C}$. Consider the following morphisms
\begin{align}
\begin{split}\label{EqAlphaEqualsBeta}
&\alpha : B_2 \otimes B_1 \otimes X \xrightarrow{\mathrm{id}_{B_2} \otimes c_{X,B_1}^{-1}} B_2 \otimes X \otimes B_1 \xrightarrow{\mathrm{hb}^{\mathbf{B}_2}_X \otimes \mathrm{id}_{B_1}} X \otimes B_2 \otimes B_1 \xrightarrow{\mathrm{id}_X \otimes \pi} X \otimes (\mathbf{B}_2 \circ \mathbf{B}_1),\\
&\beta : B_2 \otimes B_1 \otimes X \xrightarrow{\mathrm{id}_{B_2} \otimes \mathrm{hb}^{\mathbf{B}_1}_X} B_2 \otimes X \otimes B_1 \xrightarrow{c_{B_2,X} \otimes \mathrm{id}_{B_1}} X \otimes B_2 \otimes B_1 \xrightarrow{\mathrm{id}_X \otimes \pi} X \otimes (\mathbf{B}_2 \circ \mathbf{B}_1).
\end{split}
\end{align}
We have
\begin{center}
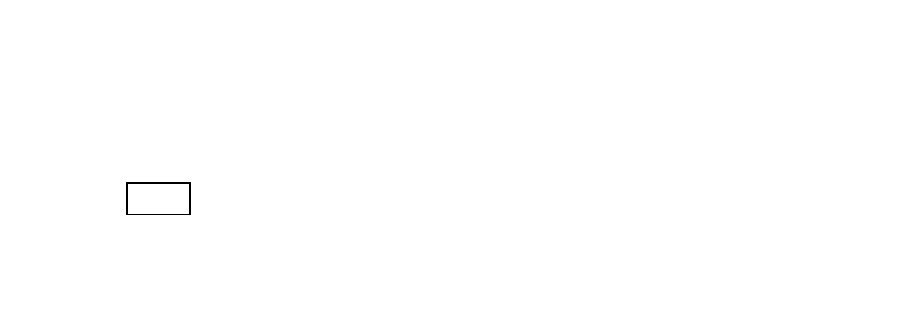
\end{center}
For the first equality we used the definition of $\alpha$, for the second equality we used \eqref{axiomsHalfBraidedRightModule}, for the third we used that $\pi$ coequalizes $\smallblacktriangleleft_{\!2} \otimes \mathrm{id}_{B_1}$ and $\mathrm{id}_{B_2} \otimes \smallblacktriangleright_{\!1}$, for the fourth we used \eqref{axiomsHalfBraidedLeftModule} and for the last we used the definition of $\beta$. Composing with $\mathrm{id}_{B_2} \otimes \eta_2 \otimes \mathrm{id}_{B_1 \otimes X}$ we get $\alpha = \beta$. Now let $\lambda$ and $\rho$ be the left and right actions on $\mathbf{B}_2 \circ \mathbf{B}_1$, as defined in \eqref{defLeftActionOnCoequalizer} and \eqref{defRightActionOnCoequalizer}. Then $\mathrm{hbl}^{\mathbf{B}_2 \circ \mathbf{B}_1}_X \circ (\pi \otimes \mathrm{id}_X)$ is equal to
\begin{center}
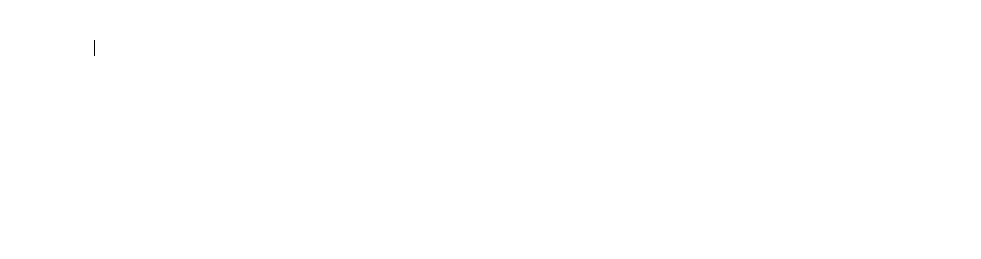
\end{center}
and the last term is $\mathrm{hbr}^{\mathbf{B}_2 \circ \mathbf{B}_1}_X \circ (\pi \otimes \mathrm{id}_X)$. The first equality is by definition of $\lambda$ and naturality of the braiding, the second is by the definition \eqref{halfBraidingInTermOfAction} of $\mathrm{hb}^{\mathbf{B}_2}_X = \mathrm{hbl}^{\mathbf{B}_2}_X$, the third uses the equality of the morphisms $\alpha$ and $\beta$ in \eqref{EqAlphaEqualsBeta}, the fourth is by the definition \eqref{halfBraidingInTermOfAction} of $\mathrm{hb}^{\mathbf{B}_1}_X = \mathrm{hbr}^{\mathbf{B}_1}_X$ and the fifth is by definition of $\rho$ and naturality of the braiding. Since $\pi \otimes \mathrm{id}_X$ is a coequalizer it follows that $\mathrm{hbl}^{\mathbf{B}_2 \circ \mathbf{B}_1}_X = \mathrm{hbr}^{\mathbf{B}_2 \circ \mathbf{B}_1}_X$, which means that $\mathbf{B}_2 \circ \mathbf{B}_1$ is hb-compatible.
\end{proof}

\noindent As a result we can make the following definition:
\begin{definition}\label{defBimChb}
Let $\mathrm{Bim}^{\mathrm{hb}}_{\mathcal{C}}$ be the subcategory of $\mathrm{Bim}_{\mathcal{C}}$ such that:
\begin{itemize}
\item its objects are the half-braided algebras in $\mathcal{C}$,
\item $\Hom_{\mathrm{Bim}^{\mathrm{hb}}_{\mathcal{C}}}(\mathbb{A}_1, \mathbb{A}_2)$ consists of the isomorphisms classes of hb-compatible $(\mathbb{A}_2, \mathbb{A}_1)$-bimodules.\footnote{Note the switch!}
\end{itemize}
\end{definition}
\noindent Recall from Corollary \ref{coroBimCMonoidal} that the category $\mathrm{Bim}_{\mathcal{C}}$ is strict monoidal, thanks to the braided tensor product $\widetilde{\otimes}$ of algebras and bimodules. By Propositions \ref{propBraidedTensorProductAlgebra} and \ref{propHbCoherenceStableByBraidedTensorProduct}, the category $\mathrm{Bim}^{\mathrm{hb}}_{\mathcal{C}}$ is stable under $\widetilde{\otimes}$ and hence is strict monoidal as well. Its unit object is
\[ \bigl(\boldsymbol{1}, \:\boldsymbol{1} \otimes - \overset{=}{\longrightarrow} - \otimes \boldsymbol{1},\: \boldsymbol{1} \otimes \boldsymbol{1} \overset{=}{\longrightarrow} \boldsymbol{1},\: \mathrm{id}_{\boldsymbol{1}}\bigr). \]

\smallskip

\subsection{The braiding on \texorpdfstring{$\mathrm{Bim}^{\mathrm{hb}}_{\mathcal{C}}$}{the compatible Morita category}}\label{sub:braiding}
\indent We now define a braiding on $\bigl(\mathrm{Bim}^{\mathrm{hb}}_{\mathcal{C}}, \widetilde{\otimes}\bigr)$. Let $\mathbb{A}_1 = (A_1, t^1, m_1, \eta_1)$ and $\mathbb{A}_2 = (A_2, t^2, m_2, \eta_2)$ be half-braided algebras. Recall from Proposition \ref{propIsoBetweenA1A2AndA2A1} the isomorphism of half-braided algebras given by the half-braiding of $\mathbb{A}_1$, \textit{i.e.}
\[ T_{\mathbb{A}_1, \mathbb{A}_2} = t^1_{A_2} : \mathbb{A}_1 \,\widetilde{\otimes}\, \mathbb{A}_2 \overset{\sim}{\longrightarrow} \mathbb{A}_2 \,\widetilde{\otimes}\, \mathbb{A}_1 \] 
Consider the $\bigl( \mathbb{A}_2 \,\widetilde{\otimes}\, \mathbb{A}_1, \mathbb{A}_1 \,\widetilde{\otimes}\, \mathbb{A}_2 \bigr)$-bimodule
\begin{equation}\label{eq:braidingbimodule}
 \mathcal{B}_{\mathbb{A}_1, \mathbb{A}_2} = \bigl( \mathbb{A}_2 \,\widetilde{\otimes}\, \mathbb{A}_1 \bigr) \smallblacktriangleleft T_{\mathbb{A}_1, \mathbb{A}_2}
\end{equation}
where we use the process of twisting a bimodule by algebra morphisms as explained in \S\ref{subsecTwistingBimod}. More precisely, here we twist the regular bimodule $\mathbb{A}_2 \,\widetilde{\otimes}\, \mathbb{A}_1$ on the right by the algebra isomorphism $T_{\mathbb{A}_1, \mathbb{A}_2}$. Explicitly, the left action on $\mathcal{B}_{\mathbb{A}_1, \mathbb{A}_2}$ is simply the multiplication in $\mathbb{A}_2 \,\widetilde{\otimes}\, \mathbb{A}_1$ while the right action is 
\[ \bigl( \mathbb{A}_2 \,\widetilde{\otimes}\, \mathbb{A}_1 \bigr) \otimes \bigl( \mathbb{A}_1 \,\widetilde{\otimes}\, \mathbb{A}_2 \bigr) \xrightarrow{\mathrm{id}_{\mathbb{A}_2 \,\widetilde{\otimes}\, \mathbb{A}_1} \,\otimes\, T_{\mathbb{A}_1, \mathbb{A}_2}} \bigl( \mathbb{A}_2 \,\widetilde{\otimes}\, \mathbb{A}_1 \bigr) \otimes \bigl( \mathbb{A}_2 \,\widetilde{\otimes}\, \mathbb{A}_1 \bigr) \xrightarrow{\mathrm{mult.}} \mathbb{A}_2 \,\widetilde{\otimes}\, \mathbb{A}_1. \]
It follows from Example \ref{coherentRegularBimod} and Lemma \ref{lemmaTwistingCohBim} that $\mathcal{B}_{\mathbb{A}_1, \mathbb{A}_2}$ is hb-compatible. In other words
\[ \mathcal{B}_{\mathbb{A}_1, \mathbb{A}_2} \in \Hom_{\mathrm{Bim}^{\mathrm{hb}}_{\mathcal{C}}}\bigl(\mathbb{A}_1 \,\widetilde{\otimes}\, \mathbb{A}_2, \mathbb{A}_2 \,\widetilde{\otimes}\, \mathbb{A}_1\bigr) \]
where we identify this bimodule with its isomorphism class. Thanks to \eqref{twistingLeftToRight} we also have
\begin{equation}\label{eq:braidingbimoduleBis}
\mathcal{B}_{\mathbb{A}_1,\mathbb{A}_2} = T_{\mathbb{A}_1, \mathbb{A}_2}^{-1} \smallblacktriangleright \bigl( \mathbb{A}_1 \,\widetilde{\otimes}\, \mathbb{A}_2 \bigr).
\end{equation}

\begin{theorem}\label{thBraided}
The family of morphisms $(\mathcal{B}_{\mathbb{A}_1, \mathbb{A}_2})_{\mathbb{A}_1, \mathbb{A}_2 \in \mathrm{Ob}(\mathrm{Bim}^{\mathrm{hb}}_{\mathcal{C}})}$ is a braiding on $\bigl(\mathrm{Bim}^{\mathrm{hb}}_{\mathcal{C}}, \widetilde{\otimes}\bigr)$.
\end{theorem}
\begin{proof} We check the three axioms of a braiding:

\smallskip

\indent {\em Naturality of $\mathcal{B}$.} Let $\mathbf{B}_1 \in \Hom_{\mathrm{Bim}^{\mathrm{hb}}_{\mathcal{C}}}(\mathbb{A}_1, \mathbb{A}'_1)$ and $\mathbf{B}_2 \in \Hom_{\mathrm{Bim}^{\mathrm{hb}}_{\mathcal{C}}}(\mathbb{A}_2, \mathbb{A}'_2)$. Write $\mathbb{A}_i = (A_i, t^i, m_i,\eta_i)$, $\mathbb{A}'_i = (A'_i, {t'}{^i}, m'_i,\eta'_i)$ and $\mathbf{B}_i = (B_i,\smallblacktriangleright_i, \smallblacktriangleleft_i)$. Then
\begin{align*}
\mathcal{B}_{\mathbb{A}'_1, \mathbb{A}'_2} \circ (\mathbf{B}_1 \,\widetilde{\otimes}\, \mathbf{B}_2) &\overset{\eqref{eq:braidingbimoduleBis}}{=} \bigl[ T_{\mathbb{A}'_1, \mathbb{A}'_2}^{-1} \smallblacktriangleright ( \mathbb{A}'_1 \,\widetilde{\otimes}\, \mathbb{A}'_2 ) \bigr] \circ (\mathbf{B}_1 \,\widetilde{\otimes}\, \mathbf{B}_2)\\
&\overset{\eqref{twistingCompBim}}{=} T_{\mathbb{A}'_1, \mathbb{A}'_2}^{-1} \smallblacktriangleright \bigl[ ( \mathbb{A}'_1 \,\widetilde{\otimes}\, \mathbb{A}'_2 ) \circ (\mathbf{B}_1 \,\widetilde{\otimes}\, \mathbf{B}_2) \bigr] = T_{\mathbb{A}'_1, \mathbb{A}'_2}^{-1} \smallblacktriangleright (\mathbf{B}_1 \,\widetilde{\otimes}\, \mathbf{B}_2).
\end{align*}
Similarly, $(\mathbf{B}_2 \,\widetilde{\otimes}\, \mathbf{B}_1) \circ \mathcal{B}_{\mathbb{A}_1, \mathbb{A}_2} = (\mathbf{B}_2 \,\widetilde{\otimes}\, \mathbf{B}_1) \smallblacktriangleleft T_{\mathbb{A}_1,\mathbb{A}_2}$. Let $\mathrm{hb}^{\mathbf{B}_1} : B_1 \otimes - \overset{\sim}{\implies} - \otimes B_1$ be the half-braiding on the compatible bimodule $\mathbf{B}_1$ (Def.\,\ref{defHbCoherentBimodule}), which is indifferently $\mathrm{hbl}^{\mathbf{B}_1}$ or $\mathrm{hbr}^{\mathbf{B}_1}$ from \eqref{halfBraidingInTermOfAction}. We claim that $\mathrm{hb}^{\mathbf{B}_1}_{B_2} : B_1 \otimes B_2 \overset{\sim}{\longrightarrow} B_2 \otimes B_1$ is an isomorphism of $(\mathbb{A}'_2 \,\widetilde{\otimes}\, \mathbb{A}'_1, \mathbb{A}_1 \,\widetilde{\otimes}\, \mathbb{A}_2)$-bimodules $T_{\mathbb{A}'_1, \mathbb{A}'_2}^{-1} \smallblacktriangleright (\mathbf{B}_1 \,\widetilde{\otimes}\, \mathbf{B}_2) \overset{\sim}{\longrightarrow} (\mathbf{B}_2 \,\widetilde{\otimes}\, \mathbf{B}_1) \smallblacktriangleleft T_{\mathbb{A}_1,\mathbb{A}_2}$. Indeed, it intertwines the left action:
\begin{center}
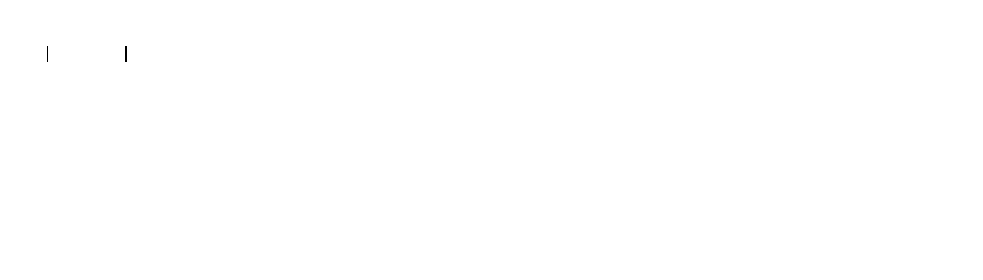
\end{center}
The bottom of the first term is the left action on $T_{\mathbb{A}'_1, \mathbb{A}'_2}^{-1} \smallblacktriangleright (\mathbf{B}_1 \,\widetilde{\otimes}\, \mathbf{B}_2)$ because $T_{\mathbb{A}'_1,\mathbb{A}'_2} = {t'}^1_{\!\!A'_2}$ by definition. For the first equality we used \eqref{axiomsHalfBraidedLeftModule}, for the second equality we used isotopy and naturality of ${t'}{^1}$, for the third we used the half-braiding property \eqref{axiomHalfBraiding} for ${t'}{^1}$ and for the fourth we used \eqref{axiomsHalfBraidedLeftModule}. The fact that $\mathrm{hb}^{\mathbf{B}_1}_{B_2}$ intertwines the right action is proven similarly, but now using $\mathrm{hbr}^{\mathbf{B}_1}$ instead of $\mathrm{hbl}^{\mathbf{B}_1}$ and \eqref{axiomsHalfBraidedRightModule} instead of \eqref{axiomsHalfBraidedLeftModule}.\footnote{Here we see that it is important for $\mathbf{B}_1$ to be hb-compatible, because for computations we need to use both $\mathrm{hbl}^{\mathbf{B}_1}$ and $\mathrm{hbr}^{\mathbf{B}_1}$. The hb-compatibility assumption is thus important to ensure that the claimed braiding is {\em natural}.}

\smallskip

\indent {\em Invertibility of $\mathcal{B}$.} By Lemma \ref{lemmaTwistingRegBimod}, the inverse of $\mathcal{B}_{\mathbb{A}_1, \mathbb{A}_2}$ in $\mathrm{Bim}^{\mathrm{hb}}_{\mathcal{C}}$ is $T_{\mathbb{A}_1, \mathbb{A}_2} \smallblacktriangleright \bigl( \mathbb{A}_2 \,\widetilde{\otimes}\, \mathbb{A}_1 \bigr)$.

\smallskip

\indent {\em Braiding property of $\mathcal{B}$.} We have
\begin{align*}
&(\mathbb{A}_2 \,\widetilde{\otimes}\, \mathcal{B}_{\mathbb{A}_1, \mathbb{A}_3}) \circ (\mathcal{B}_{\mathbb{A}_1, \mathbb{A}_2} \,\widetilde{\otimes}\, \mathbb{A}_3)\\
=\:\,& \bigl[ \mathbb{A}_2 \,\widetilde{\otimes}\, \bigl( T_{\mathbb{A}_1, \mathbb{A}_3}^{-1} \smallblacktriangleright (\mathbb{A}_1 \,\widetilde{\otimes}\, \mathbb{A}_3) \bigr) \bigr] \circ \bigl[ \bigl( ( \mathbb{A}_2 \,\widetilde{\otimes}\, \mathbb{A}_1 ) \smallblacktriangleleft T_{\mathbb{A}_1, \mathbb{A}_2} \bigr) \,\widetilde{\otimes}\, \mathbb{A}_3 \bigr] \quad \text{by \eqref{eq:braidingbimoduleBis} and \eqref{eq:braidingbimodule}}\\
=\:\,& \bigl[ \bigl( \mathrm{id}_{\mathbb{A}_2} \otimes T_{\mathbb{A}_1, \mathbb{A}_3}^{-1} \bigr) \smallblacktriangleright \bigl( \mathbb{A}_2 \,\widetilde{\otimes}\, \mathbb{A}_1 \,\widetilde{\otimes}\, \mathbb{A}_3 \bigr) \bigr] \circ \bigl[ \bigl( \mathbb{A}_2 \,\widetilde{\otimes}\, \mathbb{A}_1 \,\widetilde{\otimes}\, \mathbb{A}_3 \bigr) \smallblacktriangleleft \bigl( T_{\mathbb{A}_1, \mathbb{A}_2} \otimes \mathrm{id}_{\mathbb{A}_3} \bigr) \bigr] \quad \text{by \eqref{twistingCompBrTens}}\\
=\:\,& \bigl( \mathrm{id}_{\mathbb{A}_2} \otimes T_{\mathbb{A}_1, \mathbb{A}_3}^{-1} \bigr) \smallblacktriangleright \bigl[ \bigl( \mathbb{A}_2 \,\widetilde{\otimes}\, \mathbb{A}_1 \,\widetilde{\otimes}\, \mathbb{A}_3 \bigr) \circ \bigl( \mathbb{A}_2 \,\widetilde{\otimes}\, \mathbb{A}_1 \,\widetilde{\otimes}\, \mathbb{A}_3 \bigr) \bigr] \smallblacktriangleleft  \bigl( T_{\mathbb{A}_1, \mathbb{A}_2} \otimes \mathrm{id}_{\mathbb{A}_3} \bigr) \quad \text{by \eqref{twistingCompBim}}\\
=\:\,& \bigl( \mathrm{id}_{\mathbb{A}_2} \otimes T_{\mathbb{A}_1, \mathbb{A}_3}^{-1} \bigr) \smallblacktriangleright \bigl( \mathbb{A}_2 \,\widetilde{\otimes}\, \mathbb{A}_1 \,\widetilde{\otimes}\, \mathbb{A}_3 \bigr) \smallblacktriangleleft  \bigl( T_{\mathbb{A}_1, \mathbb{A}_2} \otimes \mathrm{id}_{\mathbb{A}_3} \bigr)\\
=\:\,& \bigl[ \bigl( \mathbb{A}_2 \,\widetilde{\otimes}\, \mathbb{A}_3 \,\widetilde{\otimes}\, \mathbb{A}_1 \bigr) \smallblacktriangleleft \bigl( \mathrm{id}_{\mathbb{A}_2} \otimes T_{\mathbb{A}_1, \mathbb{A}_3} \bigr) \bigr] \smallblacktriangleleft  \bigl( T_{\mathbb{A}_1, \mathbb{A}_2} \otimes \mathrm{id}_{\mathbb{A}_3} \bigr) \quad \text{by \eqref{twistingLeftToRight}}\\
=\:\,& \bigl( \mathbb{A}_2 \,\widetilde{\otimes}\, \mathbb{A}_3 \,\widetilde{\otimes}\, \mathbb{A}_1 \bigr) \smallblacktriangleleft \bigl[ \bigl( \mathrm{id}_{\mathbb{A}_2} \otimes T_{\mathbb{A}_1, \mathbb{A}_3} \bigr) \circ \bigl( T_{\mathbb{A}_1, \mathbb{A}_2} \otimes \mathrm{id}_{\mathbb{A}_3} \bigr) \bigr] \quad \text{by \eqref{twistingCompMor}}\\
=\:\,& \bigl( \mathbb{A}_2 \,\widetilde{\otimes}\, \mathbb{A}_3 \,\widetilde{\otimes}\, \mathbb{A}_1 \bigr) \smallblacktriangleleft T_{\mathbb{A}_1, \mathbb{A}_2 \,\widetilde{\otimes}\, \mathbb{A}_3} \quad \text{because $T$ is the braiding in }\mathcal{Z}(\mathcal{C})\\
=\:\,& \mathcal{B}_{\mathbb{A}_1, \mathbb{A}_2 \,\widetilde{\otimes}\, \mathbb{A}_3} \quad \text{by \eqref{eq:braidingbimodule}.}
\end{align*}
The equality $\mathcal{B}_{\mathbb{A}_1 \,\widetilde{\otimes}\, \mathbb{A}_2, \mathbb{A}_3} = (\mathcal{B}_{\mathbb{A}_1, \mathbb{A}_3} \,\widetilde{\otimes}\, \mathbb{A}_2) \circ (\mathbb{A}_1 \,\widetilde{\otimes}\, \mathcal{B}_{\mathbb{A}_2, \mathbb{A}_3})$ is obtained similarly.
\end{proof}

\subsection{A braided functor \texorpdfstring{$\overline{\mathcal{C}} \to \mathrm{Bim}_{\mathcal{C}}^{\mathrm{hb}}$}{to the Morita category}}\label{sub:end}
In this subsection we note that internal End and Hom objects in $\mathcal{C}$ are examples of half-braided algebras and hb-compatible bimodules respectively. We then prove that these objects assemble to define a functor $\overline{\mathcal{C}} \to \mathrm{Bim}_{\mathcal{C}}^{\mathrm{hb}}$, where $\overline{\mathcal{C}}$ is a full subcategory of $\mathcal{C}$ defined by a mild condition.

\smallskip

\indent Let $V, W \in \mathcal{C}$ and assume that $V$ has a left dual $V^*$. We recall the notations
\[ \underline{\Hom}(V,W) = W \otimes V^*, \qquad \underline{\End}(V) = \underline{\Hom}(V,V) = V \otimes V^*. \]
The next lemma is a slight adaptation of \cite[Ex.\,1.4]{schauenburg}, where it was given for coalgebras.
\begin{lemma}\label{lemmaMatrixAlgebra}
Let $V,W \in \mathcal{C}$ be objects which have left duals $V^*, W^*$.
\\1. Define
\begin{align*}
&m : \underline{\End}(V) \otimes \underline{\End}(V) \xrightarrow{\mathrm{id}_V \,\otimes\, \mathrm{ev}_V \,\otimes\, \mathrm{id}_{V^*}} \underline{\End}(V), \qquad \eta : \boldsymbol{1} \xrightarrow{\mathrm{coev}_V} \underline{\End}(V),\\
&t_X : \underline{\End}(V) \otimes X \xrightarrow{\mathrm{id}_V \,\otimes\, c_{X,V^*}^{-1}} V \otimes X \otimes V^* \xrightarrow{c_{V,X} \,\otimes\, \mathrm{id}_{V^*}} X \otimes \underline{\End}(V) \qquad (\forall \, X \in \mathcal{C}).
\end{align*}
Then $\bigl( \underline{\End}(V), t,m,\eta \bigr)$ is a half-braided algebra.

\noindent 2. Define
\begin{align*}
&\smallblacktriangleright : \underline{\End}(W) \otimes \underline{\Hom}(V,W) \xrightarrow{\mathrm{id}_W \,\otimes\, \mathrm{ev}_W \,\otimes\, \mathrm{id}_{V^*}} \underline{\Hom}(V,W),\\
&\smallblacktriangleleft : \underline{\Hom}(V,W) \otimes \underline{\End}(V) \xrightarrow{\mathrm{id}_W \,\otimes\, \mathrm{ev}_V \,\otimes\, \mathrm{id}_{V^*}} \underline{\Hom}(V,W).
\end{align*}
Then $\bigl( \underline{\Hom}(V,W), \smallblacktriangleright, \smallblacktriangleleft \bigr)$ is a hb-compatible $\bigl( \underline{\End}(W), \underline{\End}(V) \bigr)$bimodule.
\end{lemma}
\begin{proof}
1. Straightforward verification.
\\2. It is easy to compute the half-braidings $\mathrm{hbl}^{\underline{\Hom}(V,W)}$ and $\mathrm{hbr}^{\underline{\Hom}(V,W)}$ from \eqref{halfBraidingInTermOfAction}:
\[ \forall \, X \in \mathcal{C}, \quad \mathrm{hbl}^{\underline{\Hom}(V,W)}_X = \mathrm{hbr}^{\underline{\Hom}(V,W)}_X = (c_{W,X} \otimes \mathrm{id}_{V^*}) \circ (\mathrm{id}_W \otimes c^{-1}_{X,V^*}). \]
They agree, which is the definition of hb-compatibility. 
\end{proof}

The next lemma describes how the algebras $\underline{\End}$ and the bimodules $\underline{\Hom}$ behave with respect to the braided tensor product $\widetilde{\otimes}$.
\begin{lemma}\label{lemmaIntHomAndBrProd}
Let $U,V,W,Z \in \mathcal{C}$ be objects which have left duals.
\\1. There is an isomorphism of half-braided algebras
\[ I_{U,V} : \underline{\End}(U) \,\widetilde{\otimes}\, \underline{\End}(V) \overset{\sim}{\longrightarrow} \underline{\End}(U \otimes V) \]
given by $(U \otimes U^*) \otimes (V \otimes V^*) \xrightarrow{\mathrm{id}_U \otimes c^{-1}_{V \otimes V^*,U^*}} (U \otimes V) \otimes (V^* \otimes U^*)$.
\\2. There is an isomorphism of $\bigl( \underline{\End}(W) \,\widetilde{\otimes}\, \underline{\End}(Z), \underline{\End}(U) \,\widetilde{\otimes}\, \underline{\End}(V) \bigr)$-bimodules
\[ _{W,Z}J_{U,V} : \underline{\Hom}(U,W) \,\widetilde{\otimes}\, \underline{\Hom}(V,Z) \overset{\sim}{\longrightarrow} \bigl[ I_{W,Z} \smallblacktriangleright \underline{\Hom}(U \otimes V, W \otimes Z)\smallblacktriangleleft I_{U,V} \bigr] \]
given by $(W \otimes U^*) \otimes (Z \otimes V^*) \xrightarrow{\mathrm{id}_W \otimes c^{-1}_{Z \otimes V^*,U^*}} (W \otimes Z) \otimes (V^* \otimes U^*)$.
\end{lemma}
\begin{proof}
Straightforward verifications (with diagrammatic calculus in $\mathcal{C}$) left to the reader.
\end{proof}
\noindent Thanks to the properties of the braiding $c$, the isomorphisms $I_{U,V}$ in Lemma \ref{lemmaIntHomAndBrProd} satisfy
\begin{equation}\label{monStructIsoEnd}
I_{U, V \otimes W} \circ \bigl( \mathrm{id}_{\underline{\End}(U)} \otimes I_{V,W} \bigr) = I_{U \otimes V, W} \circ \bigl( I_{U,V} \otimes \mathrm{id}_{\underline{\End}(W)} \bigr).
\end{equation}

\indent For the algebras $\underline{\End}(V)$, the braiding $\mathcal{B}$ on $\mathrm{Bim}^{\mathrm{hb}}_{\mathcal{C}}$ defined in \eqref{eq:braidingbimodule} takes a simple form:
\begin{lemma}\label{lemmaBraidingForInternalEnd}
The $\bigl( \underline{\End}(W) \,\widetilde{\otimes}\, \underline{\End}(V), \underline{\End}(V) \,\widetilde{\otimes}\, \underline{\End}(W) \bigr)$-bimodules $\mathcal{B}_{\underline{\End}(V), \underline{\End}(W)}$ and \\$\underline{\Hom}(V,W) \,\widetilde{\otimes}\, \underline{\Hom}(W,V)$ are isomorphic.
\end{lemma}
\begin{proof}
One can check that
\[ W \otimes W^* \otimes V \otimes V^* \xrightarrow{\mathrm{id}_W \otimes c^{-1}_{V,W^*} \otimes \mathrm{id}_{V^*}} W \otimes V \otimes W^* \otimes V^* \xrightarrow{\mathrm{id}_{W} \otimes c_{V \otimes W^*,V^*}} W \otimes V^* \otimes V \otimes W^* \]
is an isomorphism of bimodules $\mathcal{B}_{\underline{\End}(V), \underline{\End}(W)} \overset{\sim}{\to} \underline{\Hom}(V,W) \,\widetilde{\otimes}\, \underline{\Hom}(W,V)$.
\end{proof}

\indent The description of the composition rule for the bimodules $\underline{\Hom}$ requires an extra technical assumption. Let $S,V$ be objects in $\mathcal{C}$. Recall that $S$ is called a {\em direct summand} of $V$ if there exist morphisms $\iota : S \to V$, $\pi : V \to S$ such that $\pi \circ \iota = \mathrm{id}_S$.
\begin{lemma}\label{lemmaCompositionInternalHoms}
Let $V \in \mathcal{C}$ and assume that $\boldsymbol{1}$ is a direct summand of $V$. Then we have an isomorphism of $\bigl( \underline{\End}(W), \underline{\End}(U) \bigr)$-bimodules
\[ \underline{\Hom}(V,W) \circ \underline{\Hom}(U,V) \cong \underline{\Hom}(U,W). \]
\end{lemma}
\begin{proof}
By definition, see \eqref{defCompositionOfBimodules}, $\underline{\Hom}(V,W) \circ \underline{\Hom}(U,V)$ is the coequalizer of
\begin{equation}\label{coeqForInternalHom}
\xymatrix@C=10em{
{ \underbrace{W \otimes V^*}_{\underline{\Hom}(V,W)} \otimes \underbrace{V \otimes V^*}_{\underline{\End}(V)} \otimes \underbrace{V \otimes U^*}_{\underline{\Hom}(U,V)} \ar@<.7ex>[r]^-{\mathrm{id}_W \,\otimes\, \mathrm{ev}_V \,\otimes\, \mathrm{id}_{V^* \otimes V \otimes U^*}} \ar@<-.7ex>[r]_-{\mathrm{id}_{W \otimes V^* \otimes V} \,\otimes\, \mathrm{ev}_V \,\otimes\, \mathrm{id}_{U^*}} }
& { \underbrace{W \otimes V^*}_{\underline{\Hom}(V,W)} \otimes \underbrace{V \otimes U^*}_{\underline{\Hom}(U,V)} }.}
\end{equation}
It is clear that the morphism $\mathrm{id}_W \otimes \mathrm{ev}_V \otimes \mathrm{id}_{U^*} : \underline{\Hom}(V,W) \otimes \underline{\Hom}(U,V) \to \underline{\Hom}(U,W)$ coequalizes these two arrows. We show that it is universal. Let $r : \underline{\Hom}(V,W) \otimes \underline{\Hom}(U,V) \to X$ be another morphism which coequalizes the two arrows in \eqref{coeqForInternalHom}. By assumption there exist morphisms $\iota : \boldsymbol{1} \to V$, $\pi : V \to \boldsymbol{1}$ such that $\pi \circ \iota = \mathrm{id}_{\boldsymbol{1}}$. Denote by $\pi^* : \boldsymbol{1} = \boldsymbol{1}^* \to V^*$ the transpose of $\pi$  and observe that $\mathrm{ev}_V \circ (\pi^* \otimes \iota) = \mathrm{id}_{\boldsymbol{1}}$. Then we have
\begin{center}
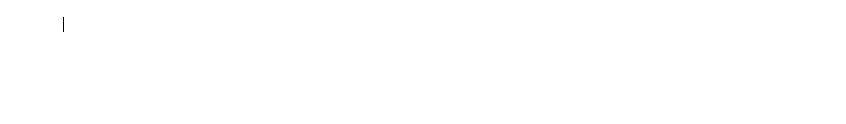
\end{center}
where $r' = r \circ \bigl( \mathrm{id}_W \otimes \pi^* \otimes \iota \otimes \mathrm{id}_{U^*} \bigr) : W \otimes U^* \to X$. The second equality uses the coequalizing property of $r$. Hence $r$ factorizes through $\mathrm{id}_W \otimes \mathrm{ev}_V \otimes \mathrm{id}_{U^*}$. Moreover, if there is another factorization $r = r'' \circ \bigl( \mathrm{id}_W \otimes \mathrm{ev}_V \otimes \mathrm{id}_{U^*} \bigr)$ then it suffices to precompose with $\mathrm{id}_W \otimes \pi^* \otimes \iota \otimes \mathrm{id}_{U^*}$ to get $r' = r''$. Finally, it is left as an exercise for the reader that the bimodule structure \eqref{defLeftActionOnCoequalizer}--\eqref{defRightActionOnCoequalizer} on the coequalizer of \eqref{coeqForInternalHom} agrees with the bimodule structure on $\underline{\Hom}(U,W)$.
\end{proof}

Let $\overline{\mathcal{C}}$ be the full (monoidal) subcategory of $\mathcal{C}$ consisting of those objects $V \in \mathcal{C}$ which have a left dual $V^*$ and have $\boldsymbol{1}$ as a direct summand. Thanks to Lemma \ref{lemmaCompositionInternalHoms}, we have a functor
\begin{equation}\label{internalEndFunctor}
F : \overline{\mathcal{C}} \to \mathrm{Bim}_{\mathcal{C}}^{\mathrm{hb}}, \qquad V \mapsto \underline{\End}(V), \qquad \bigl[ f : V \to W \bigr] \mapsto \underline{\Hom}(V,W).
\end{equation}
Note that $F$ has the same value on all morphisms in $\Hom_{\mathcal{C}}(V,W)$, namely the bimodule $\underline{\Hom}(V,W)$ which is an element of $\Hom_{\mathrm{Bim}_{\mathcal{C}}^{\mathrm{hb}}}\bigl( \underline{\End}(V), \underline{\End}(W) \bigr)$.

\smallskip

\indent In order to put a monoidal structure on $F$, we define
\begin{equation}\label{monStructEmbCinBim}
F^{(2)}_{V,W} = \bigl( \underline{\End}(V \otimes W) \smallblacktriangleleft I_{V,W} \bigr) \in \Hom_{\mathrm{Bim}_{\mathcal{C}}^{\mathrm{hb}}}\bigl( \underline{\End}(V) \,\widetilde{\otimes}\, \underline{\End}(W), \underline{\End}(V \otimes W) \bigr)
\end{equation}
where the half-braided algebra isomorphism $I_{V,W} : \underline{\End}(V) \,\widetilde{\otimes}\, \underline{\End}(W) \to \underline{\End}(V \otimes W)$ is used to endow $\underline{\End}(V \otimes W)$ with a right action of $\underline{\End}(V) \,\widetilde{\otimes}\, \underline{\End}(W)$, see \S\ref{subsecTwistingBimod}.

\begin{proposition}\label{propEmbedCIntoBim}
1. The pair $(F,F^{(2)})$ is a braided strong monoidal functor $\overline{\mathcal{C}} \to \mathrm{Bim}_{\mathcal{C}}^{\mathrm{hb}}$.
\\2. For any $V \in \overline{\mathcal{C}}$, the object $F(V)$ is isomorphic to $F(\boldsymbol{1}) = \boldsymbol{1}$ in $\mathrm{Bim}_{\mathcal{C}}^{\mathrm{hb}}$.
\end{proposition}
\begin{proof}
1. The monoidal structure property reads
\[ F^{(2)}_{U \otimes V, W} \circ \bigl( F^{(2)}_{U,V} \,\widetilde{\otimes}\, \underline{\End}(W) \bigr) = F^{(2)}_{U, V \otimes W} \circ \bigl( \underline{\End}(U) \,\widetilde{\otimes}\, F^{(2)}_{V,W} \bigr) \]
and it easily follows from \eqref{monStructIsoEnd}, thanks to \eqref{twistingCompMor}, \eqref{twistingCompBim} and \eqref{twistingCompBrTens}. Let us prove the braided functor property (recall that isomorphic bimodules are {\em equal} as morphisms in $\mathrm{Bim}_{\mathcal{C}}^{\mathrm{hb}}$):
\begin{align*}
&F^{(2)}_{W,V} \circ \mathcal{B}_{F(V),F(W)} = \bigl[ \underline{\End}(W \otimes V) \smallblacktriangleleft I_{W,V} \bigr] \circ \mathcal{B}_{\underline{\End}(V), \End(W)}\\
=\:\,& \bigl[ I_{W,V}^{-1} \smallblacktriangleright \bigl( \underline{\End}(W) \,\widetilde{\otimes}\, \underline{\End}(V) \bigr) \bigr] \circ \mathcal{B}_{\underline{\End}(V), \End(W)} = I_{W,V}^{-1} \smallblacktriangleright \mathcal{B}_{\underline{\End}(V), \End(W)}\\
=\:\,&  I_{W,V}^{-1} \smallblacktriangleright \bigl[ \underline{\Hom}(V,W) \,\widetilde{\otimes}\, \underline{\Hom}(W,V) \bigr] = \underline{\Hom}(V \otimes W, W \otimes V) \smallblacktriangleleft I_{V,W}\\
=\:\,& \underline{\Hom}(V \otimes W, W \otimes V) \circ \bigl[ \underline{\End}(V \otimes W) \smallblacktriangleleft I_{V,W} \bigr] = F(c_{V,W}) \circ F^{(2)}_{V,W}.
\end{align*}
where the first equality is by definition of $F^{(2)}$ and $F$, the second uses \eqref{twistingLeftToRight}, the third uses \eqref{twistingCompBim} and the fact that $\underline{\End}(W) \,\widetilde{\otimes}\, \underline{\End}(V)$ is an identity morphism in $\mathrm{Bim}^{\mathrm{hb}}_{\mathcal{C}}$, the fourth uses Lemma \ref{lemmaBraidingForInternalEnd}, the fifth uses item 2 in Lemma \ref{lemmaIntHomAndBrProd} and \eqref{twistingCompMor}, the sixth uses \eqref{twistingCompBim} and the fact that $\underline{\End}(V \otimes W)$ is an identity morphism in $\mathrm{Bim}^{\mathrm{hb}}_{\mathcal{C}}$ and the seventh is by definition of $F^{(2)}$ and $F$.

2. We have the bimodules
\[ \underline{\Hom}(\boldsymbol{1},V) \in \Hom_{\mathrm{Bim}_{\mathcal{C}}^{\mathrm{hb}}}\bigl( \boldsymbol{1}, \underline{\End}(V) \bigr) \quad \text{and} \quad \underline{\Hom}(V, \boldsymbol{1}) \in \Hom_{\mathrm{Bim}_{\mathcal{C}}^{\mathrm{hb}}}\bigl(\underline{\End}(V), \boldsymbol{1} \bigr) \]
which are inverse to each other by Lemma \ref{lemmaCompositionInternalHoms}.
\end{proof}

\section{The \texorpdfstring{$\mathscr{L}$}{coend}-linear Morita category}\label{sectionLlinear}
In this section we give another definition, which might look simpler, of half-braided algebras and hb-compatible bimodules. However, these equivalent definitions make sense only when there is more structure on the ambient braided monoidal category $\mathcal{C}$. Thinking about the example $\mathcal{C} = \mathrm{Comod}\text{-}\OO$ (the category of {\em all} comodules over a Hopf $k$-algebra $\OO$ with $k$ a field), note that:
\begin{itemize}
\item every comodule is the union of its finite-dimensional subcomodules \cite[Th. 5.1.1]{Mon}.
\item every finite-dimensional comodule is rigid (\textit{i.e.} has left and right duals).
\end{itemize}
The notion of {\em locally finitely presentable} category generalizes the first feature, where the role of finite-dimensional comodules will be played by so-called {\em compact objects}. The notion of {\em cp-rigidity} generalizes the second feature.

\smallskip

We now recall precisely these notions; a reference for LFP categories is \cite[\S 1.A]{AR}.

\indent Let $\mathcal{C}$ be a category. We denote the colimit of a functor $F : \mathcal{I} \to \mathcal{C}$ by $\mathrm{colim}\,F$ or $\underset{X \in \mathcal{I}}{\mathrm{colim}}\:F(X)$, when it exists. By definition \cite[\S III.3]{MLCat}, it is a pair $(C, \phi)$ where $C \in \mathcal{C}$ and $\phi = \bigl( \phi_X : F(X) \to C \bigr)_{X \in \mathcal{I}}$ satisfies $\phi_{Y} \circ F(f) = \phi_X$ for all $f \in \Hom_{\mathcal{I}}(X,Y)$ and which is universal for this property, meaning that if $(D,\psi)$ is another such pair then there exists a unique $u \in \Hom_{\mathcal{C}}(C,D)$ such that $\psi_X = u \circ \phi_X$ for all $X \in \mathcal{I}$. In particular we have the following principle, often used in the sequel: for all $D \in \mathcal{C}$ and $u,u' \in \Hom_{\mathcal{C}}(C,D)$, if $u \circ \phi_X = u' \circ \phi_X$ for all $X \in \mathcal{I}$ then $u=u'$.

\indent A {\em filtered colimit} is the colimit of a functor $F : \mathcal{I} \to \mathcal{C}$ where the category $\mathcal{I}$ is filtered \cite[\S IX.1]{MLCat}. We do not recall the definition of filtered, as it will not appear explicitly in the sequel.

\begin{definition}\label{defCompact}
An object $K \in \mathcal{C}$ is called compact (a.k.a. finitely presentable) if the functor $\Hom_{\mathcal{C}}(K,-)$ preserves filtered colimits. We denote by $\mathcal{C}_{\mathrm{cp}}$ the full subcategory of compact objects.
\end{definition}

\noindent Compact objects satisfy almost by definition the following factorization property, which will be of great importance in later proofs:
\begin{lemma}\label{lemmaFactoCompactObjects}
Let $(C, \phi) = \mathrm{colim} \, \bigl( F : \mathcal{I} \to \mathcal{C} \bigr)$ be a filtered colimit, let $K \in \mathcal{C}$ be a compact object and $f \in \Hom_{\mathcal{C}}(K,C)$. Then there exist $X \in \mathcal{I}$ and $g \in \Hom_{\mathcal{C}}\bigl(K, F(X) \bigr)$ such that $f = \phi_X \circ g$.
\end{lemma}
\begin{proof}
Since the functor $\Hom_{\mathcal{C}}(K,-)$ takes values in the category of sets, we can use the explicit description of its colimit (see e.g. \cite[Th. 3.4.12]{riehl}):
\[ \underset{X \in \mathcal{I}}{\mathrm{colim}} \: \Hom_{\mathcal{C}}\bigl(K,F(X)\bigr) = \left(\coprod_{X \in \mathcal{I}} \Hom_{\mathcal{C}}\bigl(K,F(X)\bigr)\right)\big/\!\sim \]
where the elements of the disjoint union $\coprod$ are by definition pairs $(X,g)$ where $X \in \mathcal{I}$ and $g \in \Hom_{\mathcal{C}}\bigl(K, F(X)\bigr)$, and $\sim$ is the equivalence relation induced by $(X,g) \sim (Y, F(\alpha) \circ g)$ for all $X,Y \in \mathcal{I}$ and $\alpha \in \Hom_{\mathcal{I}}(X,Y)$. By definition of a compact object, the canonical morphism
\[ \fleche{\underset{X \in \mathcal{I}}{\mathrm{colim}} \: \Hom_{\mathcal{C}}\bigl(K,F(X)\bigr)}{\Hom_{\mathcal{C}}(K,\mathrm{colim}\,F) = \Hom_{\mathcal{C}}(K,C)}{(X,g)}{\phi_X \circ g} \]
is an isomorphism. But in the category of sets an isomorphism means a bijection and hence $f$ has a preimage in $\coprod_{X \in \mathcal{I}} \Hom_{\mathcal{C}}\bigl(K,F(X)\bigr)$, which is the desired couple $(X,g)$.
\end{proof}

\begin{definition}\label{def:lfp}
The category $\mathcal{C}$ is locally finitely presentable (LFP for short) if:
\begin{itemize}
\item $\mathcal{C}$ is cocomplete (\textit{i.e.} has all small colimits),
\item the subcategory $\mathcal{C}_{\mathrm{cp}}$ is essentially small (\textit{i.e.} is equivalent to a small category),
\item any object in $\mathcal{C}$ is a filtered colimit of objects in $\mathcal{C}_{\mathrm{cp}}$.
\end{itemize}
\end{definition}

\indent In all the sequel we work with a category $\mathcal{C}$ such that:
\begin{equation}\label{assumptionsCategoryC}
\left\{\!
\begin{array}{l}
\bullet\:\:\:(\mathcal{C}, \otimes, \boldsymbol{1}, c) \text{ is a braided monoidal category (assumed strict for simplicity),}\\
\bullet\:\:\:\mathcal{C} \text{ is a LFP category,}\\
\bullet\:\:\:\boldsymbol{1} \text{ is a compact object,}\\
\bullet\:\:\:\text{for any } V \in \mathcal{C} \text{ the functors } V \otimes - \text{ and } - \otimes \,V \text{ are cocontinuous,}\\
\bullet\:\:\:\text{any compact object } K \text{ is rigid, \textit{i.e} it has left and right duals } K^* \text{ and } ^*K.
\end{array}
\right.
\end{equation}
The definition of left and right duals may be found e.g. in \cite[\S 2.10]{EGNO}.
\begin{lemma}\label{lemmaCompactStableByMonoidalProduct}
Under the assumptions \eqref{assumptionsCategoryC}, any rigid object is compact. In particular, if $K$ and $Q$ are compact then $K \otimes Q$ is compact.
\end{lemma}
\begin{proof}
Let $V \in \mathcal{C}$ be a rigid object. If $\underset{X \in \mathcal{I}}{\mathrm{colim}} \: F(X)$ is a filtered colimit in $\mathcal{C}$ we have
\begin{equation}\label{argumentDualizabilityImpliesCompact}
\begin{array}{l}
\underset{X \in \mathcal{I}}{\mathrm{colim}} \: \Hom_{\mathcal{C}}\bigl(V, F(X)\bigr) \cong \underset{X \in \mathcal{I}}{\mathrm{colim}} \: \Hom_{\mathcal{C}}\bigl(\boldsymbol{1}, F(X) \otimes V^* \bigr) \cong \Hom_{\mathcal{C}}\bigl(\boldsymbol{1}, \underset{X \in \mathcal{I}}{\mathrm{colim}} \: (F(X) \otimes V^*) \bigr)\\[.7em]
\qquad\qquad\qquad\qquad\qquad\cong \Hom_{\mathcal{C}}\bigl(\boldsymbol{1}, (\mathrm{colim}\,F) \otimes V^*\bigr) \cong \Hom_{\mathcal{C}}\bigl(V, \mathrm{colim}\,F \bigr).
\end{array}
\end{equation}
We used the adjunction $(- \otimes V) \dashv (- \otimes V^*)$ \cite[Prop. 2.10.8]{EGNO} and the hypotheses (\ref{assumptionsCategoryC}) on $\boldsymbol{1}$ being compact and $\cdot \otimes V^*$ to preserve colimits. The second claim in the lemma uses that a monoidal product of rigid objects is rigid and hence compact.
\end{proof}

\begin{example}\label{exampleVectIsLFP}
Let $\mathrm{Vect}_k$ be the category of vector spaces over a field $k$. The vector space $k$ is compact since $\Hom_k(k,-) \cong \mathrm{Id}$ and it follows from \eqref{argumentDualizabilityImpliesCompact} that any finite-dimensional $k$-vector space is a compact object (the converse is also true). As a result $\mathrm{Vect}_k$ is LFP since any vector space is the union of its finite-dimensional subspaces. Also it is easy to see from the explicit description of a colimit in $\mathrm{Vect}_k$ that the functors $V \otimes -$ and $- \otimes V$ are cocontinuous. Hence $\mathrm{Vect}_k$ satisfies \eqref{assumptionsCategoryC}.
More in general if $\OO$ is a Hopf algebra over $k$, as remarked in  \cite{BZBJ}, then the category of all right $\OO$-comodules is LFP : we will explore this example in detail in Section \ref{sec:Hcomod}.
\end{example}

\subsection{The coend \texorpdfstring{$\mathscr{L}$}{}}\label{sub:coend} Since a coend is a particular kind of colimit \cite[\S IX.6]{MLCat}, the axioms of a LFP category imply that
\begin{equation}\label{defCoendL}
\mathscr{L} = \int^{K \in \mathcal{C}_{\mathrm{cp}}} K^* \otimes K
\end{equation}
exists in $\mathcal{C}$. We denote by $i_K : K^* \otimes K \to \mathscr{L}$ the components of the universal dinatural transformation. In this subsection we discuss properties of $\mathscr{L}$ which will be used in \S \ref{sectionDefLlinearAlg}.

\smallskip

\indent There is a well-known bialgebra structure on the coend $\mathscr{L}$, given in \cite{majBrGr} and \cite[\S 2]{lyu} (the former reference uses a different definition of $\mathscr{L}$ from reconstruction theory, see \cite[\S 3.4]{FGR} for a detailed explanation of the equivalence between the two definitions). Namely, the unit $\eta_{\mathscr{L}} : \boldsymbol{1} \to \mathscr{L}$ is just $i_{\boldsymbol{1}}$ while the other structure morphisms $m_{\mathscr{L}} : \mathscr{L}^{ \otimes 2} \to \mathscr{L}$, $\Delta_{\mathscr{L}} : \mathscr{L} \to \mathscr{L}^{ \otimes 2}$, $\varepsilon_{\mathscr{L}} : \mathscr{L} \to \boldsymbol{1}$ are defined using the universal property of a coend:
\begin{equation}\label{defStructureCoend}
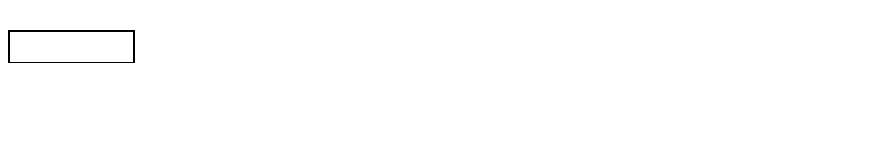
\end{equation}
\smallskip
for all $K, Q \in \mathcal{C}_{\mathrm{cp}}$, where we use the diagrams for duality morphisms defined in \eqref{diagramsForDuality}. Note that the definition of $\eta$ (resp. $m$) makes sense because $\boldsymbol{1} \in \mathcal{C}_{\mathrm{cp}}$ by assumption (resp. thanks to Lemma \ref{lemmaCompactStableByMonoidalProduct}).

\smallskip

\indent By \eqref{assumptionsCategoryC} the functor $- \otimes V$ is cocontinuous for all $V \in \mathcal{C}$, so it follows that the coend $\int^{K \in \mathcal{C}_{\mathrm{cp}}} (K^* \otimes K \otimes V)$ is equal to $\mathscr{L} \otimes V$. More precisely, the dinatural transformation $(i_K \otimes \mathrm{id}_V)_{K \in \mathcal{C}_{\mathrm{cp}}}$ is universal. Thus using the universal property of a coend we define a morphism
\begin{equation}\label{halfBrSigma}
\sigma_V : \mathscr{L} \otimes V \to V \otimes \mathscr{L} 
\end{equation}
by declaring that
\begin{equation}\label{defHalfBraidingSigmaOnCoend}
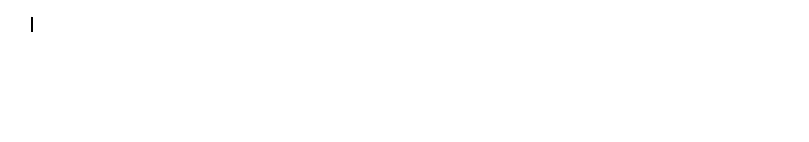
\end{equation}
for all $K \in \mathcal{C}_{\mathrm{cp}}$. The natural transformation $\sigma$ was studied in detail in \cite{NS} (take $\omega = \mathrm{Id}_{\mathcal{C}}$ there), except that they use the ``reconstruction theory formalism'' mentioned above \eqref{defStructureCoend}.

\begin{lemma}\label{lemmaPropertiesHalfBraidingSigma}
$\sigma$ is a natural isomorphism and has the following properties for all $X,Y \in \mathcal{C}$
\begin{enumerate}
\item braided commutativity: $m_{\mathscr{L}} \circ \sigma_{\mathscr{L}} = m_{\mathscr{L}}$ ,
\item half-braiding: $\sigma_{X \otimes Y} = (\mathrm{id}_X \otimes \sigma_Y) \circ (\sigma_X \otimes \mathrm{id}_Y)$,
\item $\sigma_X \circ (m_{\mathscr{L}} \otimes \mathrm{id}_X) = (\mathrm{id}_X \otimes m_{\mathscr{L}}) \circ (\sigma_X \otimes \mathrm{id}_{\mathscr{L}}) \circ (\mathrm{id}_{\mathscr{L}} \otimes \sigma_X)$ \\or in other words $m_{\mathscr{L}} \in \Hom_{\mathcal{Z}(\mathcal{C})}\bigl( (\mathscr{L}, \sigma)^{\otimes 2}, (\mathscr{L}, \sigma) \bigr)$,
\item \begin{enumerate}
\item $(\mathrm{id}_X \otimes \Delta_{\mathscr{L}}) \circ \sigma_X = (\sigma_X \otimes \mathrm{id}_{\mathscr{L}}) \circ (\mathrm{id}_{\mathscr{L}} \otimes c_{\mathscr{L},X}) \circ (\Delta_{\mathscr{L}} \otimes \mathrm{id}_X)$,
\item $(\mathrm{id}_X \otimes \Delta_{\mathscr{L}}) \circ \sigma_X = (c_{X,\mathscr{L}}^{-1} \otimes \mathrm{id}_{\mathscr{L}}) \circ (\mathrm{id}_{\mathscr{L}} \otimes \sigma_X) \circ (\Delta_{\mathscr{L}} \otimes \mathrm{id}_X)$,
\end{enumerate}
\end{enumerate}
where we recall that $c$ denotes the braiding in $\mathcal{C}$.
\end{lemma}
\begin{proof}
The first item is easily checked, while the next items are particular cases of respectively Proposition 5, Theorem 8 and Lemma 6 in \cite{NS} (and can also be checked by straightforward diagrammatic computations with the coend definition of $\mathscr{L}$ that we use here).
\end{proof}

\indent Since $\mathscr{L}$ is an algebra object in $\mathcal{C}$, we can consider (bi)modules over $\mathscr{L}$, as defined in \S\ref{sectionPreliminariesModules}. Given a left module $(V,\lambda)$, use $\sigma$ from \eqref{defHalfBraidingSigmaOnCoend} to define
\begin{equation}\label{defLambdaR}
\lambda^{\mathrm{R}} : V \otimes \mathscr{L} \xrightarrow{\sigma_V^{-1}} \mathscr{L} \otimes V \overset{\lambda}{\longrightarrow} V.
\end{equation}
Then $(V,\lambda,\lambda^{\mathrm{R}})$ is a $\mathscr{L}$-bimodule. This is not hard to check: prove first that $\lambda$ and $\lambda^{\mathrm{R}}$ commute and then that $\lambda^{\mathrm{R}}$ is a right action. Said differently, any $\mathscr{L}$-module is automatically a $\mathscr{L}$-bimodule.

\begin{remark}
The category $\mathscr{L}\text{-}\mathrm{mod}_\mathcal{C}$ is called the Harish-Chandra category in \cite{GJS} and the half-braiding $\sigma$ is exactly what is called the field-goal transform in \cite[\S 1.5]{GJS}.
\end{remark}

\subsection{\texorpdfstring{$\mathscr{L}$}{Coend}-linear algebras}
\label{sectionDefLlinearAlg}
We now reformulate the general theory of Section \ref{sectionMoritaCategoryHBAlgebras} under the isomorphism between the Drinfeld center $\mathcal{Z}(\mathcal{C})$ and the category $\mathscr{L}\text{-}\mathrm{mod}_{\mathcal{C}}$ of $\mathscr{L}$-modules in $\mathcal{C}$ (see Appendix \ref{subsectionIsoZCLModC}). This yields more natural-looking definitions: in this subsection we prove that half-braided algebras (Def.\,\ref{defHBAlgebra}) become algebras endowed with a morphism from $\mathscr{L}$ satisfying a ``braided-commutativity'' property. In the next subsection we will prove that hb-compatible bimodules (Def.\,\ref{defHbCoherentBimodule}) become bimodules which have an analogous ``braided-commutativity'' relating the left and right actions of $\mathscr{L}$.

\begin{definition}\label{defLlinearAlgebra}
1. A $\mathscr{L}$-linear algebra is a quadruple $(A,m,\eta,\mathfrak{d})$ where $(A,m,\eta)$ is an algebra in $\mathcal{C}$ and $\mathfrak{d} : \mathscr{L} \to A$ is a morphism of algebras such that the diagram
\begin{equation}\label{QMM_GJS}
\xymatrix@C=3em{
\mathscr{L} \otimes A \ar[rr]^-{\mathfrak{d} \, \otimes \mathrm{id}_A} \ar[d]_-{\sigma_A} && A \otimes A\ar[d]^-m\\
A \otimes \mathscr{L} \ar[r]_-{\mathrm{id}_A \otimes \, \mathfrak{d}} & A \otimes A \ar[r]_-{m} & A
} \end{equation}
commutes, with $\sigma : \mathscr{L} \otimes - \overset{\sim}{\implies} - \otimes \mathscr{L}$ the half-braiding defined in \eqref{halfBrSigma}.
\\2. A morphism of $\mathscr{L}$-linear algebras $f : (A,m,\eta,\mathfrak{d}) \to (A',m',\eta',\mathfrak{d}')$ is $f \in \Hom_{\mathcal{C}}(A,A')$ such that $f$ is a morphism of algebras $(A,m,\eta) \to (A',m',\eta')$ and $f \circ \mathfrak{d} = \mathfrak{d}'$.
\end{definition}
\begin{remark}
A morphism of algebras $\mathfrak{d}$ as in the previous definition is called a quantum moment map in \cite[Def.\,3.1]{safronov} and \cite[\S 2.6]{GJS}. The relation with an older definition of a quantum moment map will be discussed in detail in \S\ref{sub:LlinQMM}.
\end{remark}

\indent Given a $\mathscr{L}$-linear algebra $(A,m,\eta,\mathfrak{d})$ we can define $\lambda  = m \circ (\mathfrak{d} \otimes \mathrm{id}_A) : \mathscr{L} \otimes A \to A$. Then $(A,\lambda) \in \mathscr{L}\text{-}\mathrm{mod}_{\mathcal{C}}$, \textit{i.e} it is a left $\mathscr{L}$-module in $\mathcal{C}$, and it is an exercise to check that the diagrams 
\[ \xymatrix@C=4em{
\mathscr{L} \otimes A \otimes A \ar[r]^-{\lambda \,\otimes\, \mathrm{id}_A}\ar[d]_-{\mathrm{id}_{\mathscr{L}} \,\otimes\, m}& A \otimes A \ar[d]^-m\\
 \mathscr{L} \otimes A \ar[r]_-{\lambda}& A }
\qquad\qquad
\xymatrix@C=4em{
A \otimes A \otimes \mathscr{L} \ar[r]^-{\mathrm{id}_A \,\otimes\, \lambda^{\mathrm{R}}}\ar[d]_-{m \,\otimes\,\mathrm{id}_{\mathscr{L}}}& A \otimes A \ar[d]^-m\\
 A \otimes \mathscr{L} \ar[r]_-{\lambda^{\mathrm{R}}}& A }
 \]
commute\footnote{The fact that $\mathfrak{d}$ is a morphism of algebras is equivalent to the commutation of the first diagram, while the property \eqref{QMM_GJS} is equivalent to the commutation of the second diagram.}, where $\lambda^{\mathrm{R}} = \lambda \circ \sigma_A^{-1}$ is the right action of $\mathscr{L}$ associated to $\lambda$, see \eqref{defLambdaR}. Conversely, if $A$ is in $\mathscr{L}\text{-}\mathrm{mod}_{\mathcal{C}}$, with action $\lambda : \mathscr{L} \otimes A \to A$ such that these two diagrams commute, then $\mathfrak{d} = \lambda \circ (\mathrm{id}_{\mathscr{L}} \otimes \eta)$ is a morphism which satisfies the condition \eqref{QMM_GJS}. This justifies the name ``$\mathscr{L}$-linear algebra''.

\smallskip

\indent Using the monoidal isomorphism $\Upsilon : \mathcal{Z}(\mathcal{C}) \overset{\sim}{\longrightarrow} \mathscr{L}\text{-}\mathrm{mod}_{\mathcal{C}}$ provided in Appendix \ref{subsectionIsoZCLModC} (Prop.\,\ref{propIsoZCandLModC}) we can prove the equivalence between half-braided algebras (Def. \ref{defHBAlgebra}) and $\mathscr{L}$-linear algebras (Def.\,\ref{defLlinearAlgebra}):

\begin{proposition}\label{propLLinearAlgIntoHBAlg}
1. Let $(A,t,m,\eta)$ be a half-braided algebra and define $\mathfrak{d}(t) : \mathscr{L} \to A$ by
\begin{equation}\label{QMMinTermsOfHB}
\mathfrak{d}(t) \circ i_K : K^* \otimes K \xrightarrow{\mathrm{id}_{K^*} \,\otimes\, \eta \,\otimes \,\mathrm{id}_K} K^* \otimes A \otimes K \xrightarrow{\mathrm{id}_{K^*}\, \otimes \, t_K} K^* \otimes\, K \otimes A \xrightarrow{\mathrm{ev}_K \,\otimes\, \mathrm{id}_A} A
\end{equation}
for all $K \in \mathcal{C}_{\mathrm{cp}}$. Then $\bigl (A,m,\eta,\mathfrak{d}(t) \bigr)$ is a $\mathscr{L}$-linear algebra.
\\2. Let $(A,m,\eta,\mathfrak{d})$ be a $\mathscr{L}$-linear algebra and define $t(\mathfrak{d}) : A \otimes - \overset{\sim}{\implies} - \otimes A$ by
\[ t(\mathfrak{d})_K : A \otimes K \xrightarrow{\mathrm{coev}_K \,\otimes\, c_{K,A}^{-1}} K \otimes K^* \otimes K \otimes A \xrightarrow{\mathrm{id}_K \,\otimes\, (\mathfrak{d} \,\circ\, i_K) \,\otimes\, \mathrm{id}_A} K \otimes A \otimes A \xrightarrow{\mathrm{id}_K \,\otimes\,m} K \otimes A \]
for all $K \in \mathcal{C}_{\mathrm{cp}}$. Then $\bigl(A,t(\mathfrak{d}),m,\eta \bigr)$ is a half-braided algebra.
\\3. These two constructions are inverse each other.
\\4. Morphisms of half-braided algebras are morphisms of $\mathscr{L}$-linear algebras and conversely.
\end{proposition}
\begin{proof}
Recall from Appendix \ref{subsectionIsoZCLModC} that a half-braiding is uniquely characterized by its values on $\mathcal{C}_{\mathrm{cp}}$. We use the isomorphism $\Upsilon : \mathcal{Z}(\mathcal{C}) \to \mathscr{L}\text{-}\mathrm{mod}_{\mathcal{C}}$ from Prop.\,\ref{propIsoZCandLModC}.

1. Write $\Upsilon(A,t) = \bigl( A,\lambda(t) \bigr)$ with $\lambda(t) : \mathscr{L} \otimes A \to A$ the left action defined by \eqref{actionFromHalfBraiding}. The definition is such that $\mathfrak{d}(t) = \lambda(t) \circ (\mathrm{id}_{\mathscr{L}} \otimes \eta)$. We have
\begin{center}
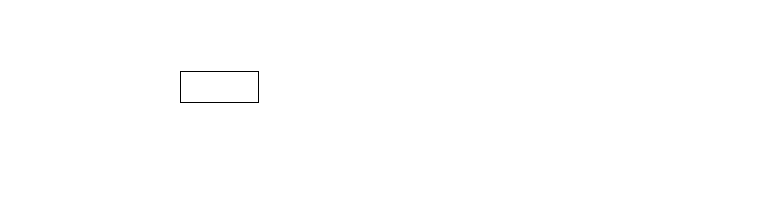
\end{center}
where the first equality is by definition of $\lambda(t)$, the second is by unitality of $m$, the third is by the first equality in \eqref{axiomsHalfBraidedAlgebra} and the fourth is by definition of $\lambda(t)$ and naturality of the braiding. Hence
\begin{equation}\label{QMMfromLambdaT}
\lambda(t) = m \circ \bigl( \lambda(t) \otimes \mathrm{id}_A \bigr) \circ (\mathrm{id}_{\mathscr{L}} \otimes \eta \otimes \mathrm{id}_A) = m \circ (\mathfrak{d}(t) \otimes \mathrm{id}_A)
\end{equation}
and this together with the fact that $\lambda(t)$ is a left action easily implies that $\mathfrak{d}$ is a morphism of algebras. To obtain the property \eqref{QMM_GJS}, note that
\begin{center}
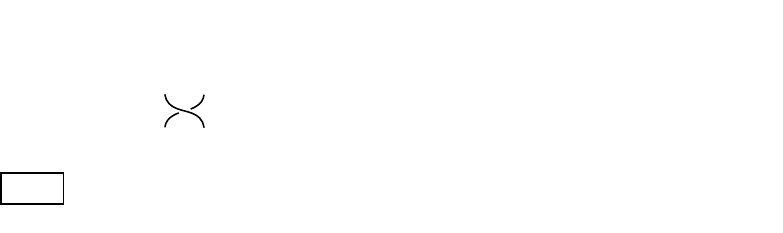
\end{center}
where the first equality is just by definition of $\sigma$ in \eqref{defHalfBraidingSigmaOnCoend} and of $\mathfrak{d}(t)$, the second is by isotopy, the third uses the second equality in \eqref{axiomsHalfBraidedAlgebra} and the fourth uses that $m$ is unital. By definition of $\lambda(t)$, it follows that
\[ m \circ \bigl( \mathrm{id}_A \otimes \mathfrak{d}(t) \bigr) \circ \sigma_A = \lambda(t) \overset{\eqref{QMMfromLambdaT}}{=} m \circ (\mathfrak{d}(t) \otimes \mathrm{id}_A). \]

2. Let $\lambda(\mathfrak{d}) = m \circ (\mathfrak{d} \otimes \mathrm{id}_A) : \mathscr{L} \otimes A \to A$, which is a left action. The definition is such that $\Upsilon^{-1}\bigl( A,\lambda(\mathfrak{d}) \bigr) = \bigl( A,t(\mathfrak{d}) \bigr)$ and hence $t(\mathfrak{d})$ is a half-braiding. The properties \eqref{axiomsHalfBraidedAlgebra} are proved by diagrammatic computations similar to the first item, and thus left to the reader.

3. Obvious, as we have already explained that the two constructions are related by the isomorphism $\Upsilon$.

4. Easily seen from the formulas in items 1 and 2.
\end{proof}

\begin{example}\label{LlinearStructEndV}
Recall the algebra $\underline{\End}(V) = V \otimes V^*$ from Lemma \ref{lemmaMatrixAlgebra}. For all $K \in \mathcal{C}_{\mathrm{cp}}$, let
\begin{align*}
d_K : K^* \otimes K \xrightarrow{\mathrm{id}_{K^* \otimes K} \otimes \mathrm{coev}_V} K^* \otimes K \otimes V \otimes V^* &\xrightarrow{\mathrm{id}_{K^*} \otimes (c_{V,K} \circ c_{K,V}) \otimes \mathrm{id}_{V^*}} K^* \otimes K \otimes V \otimes V^*\\
&\xrightarrow{\mathrm{ev}_K \otimes \mathrm{id}_{V\otimes V^*}} V \otimes V^*.
\end{align*}
This is a dinatural family, hence there exists a morphism $\mathfrak{d} : \mathscr{L} \to \underline{\End}(V)$ such that $d_K = \mathfrak{d} \circ i_K$ for all $K \in \mathcal{C}_{\mathrm{cp}}$. It is straightforward to check that $\mathfrak{d}$ satisfies the diagram \eqref{QMM_GJS} and thus defines a structure of $\mathscr{L}$-linear algebra on $\underline{\End}(V)$. Through the correspondence of Prop.\,\ref{propLLinearAlgIntoHBAlg}, we recover the half-braided structure described in Lemma \ref{lemmaMatrixAlgebra}. We note that the dinatural family $d_K$ is the one used to define a categorical version of the Drinfeld morphism, see \cite[\S 4.4]{FGR}.
\end{example}

\begin{example}
The braided commutativity of the product in $\mathscr{L}$ (Lemma \ref{lemmaPropertiesHalfBraidingSigma}) is equivalent to the commutation of \eqref{QMM_GJS} with $\mathfrak{d} = \mathrm{id}_{\mathscr{L}}$. Hence $\mathscr{L}$ is a $\mathscr{L}$-linear algebra. Through the correspondence of Prop.\,\ref{propLLinearAlgIntoHBAlg}, we get a half-braiding $\tau = t(\mathrm{id}_{\mathscr{L}}) : \mathscr{L} \otimes - \overset{\sim}{\implies} - \otimes \mathscr{L}$ which is given on a compact object $K$ by
\begin{equation}\label{defTauOnCompacts}
\forall \, Q \in \mathcal{C}_{\mathrm{cp}}, \quad \tau_K \circ (i_Q \otimes \mathrm{id}_K) = (\mathrm{id}_K \otimes i_{Q \otimes K}) \circ (\mathrm{coev}_K \otimes \mathrm{id}_{Q^* \otimes Q \otimes K})
\end{equation}
Here we use that the dinatural transformation $\bigl( i_Q \otimes \mathrm{id}_K : Q^* \otimes Q \otimes K \to \mathscr{L} \otimes K \bigr)_{Q \in \mathcal{C}_{\mathrm{cp}}}$ is universal and hence the values \eqref{defTauOnCompacts} uniquely define $\tau_K$. This half-braiding has interesting properties regarding the coproduct in $\mathscr{L}$, namely $\tau_{\mathscr{L}} \circ \Delta_{\mathscr{L}} = \Delta_{\mathscr{L}}$ (``braided co-commutativity'') and $\Delta_{\mathscr{L}} \in \Hom_{\mathcal{Z}(\mathcal{C})}\bigl( (\mathscr{L}, \tau), (\mathscr{L}, \tau)^{\otimes 2} \bigr)$. Proofs are left to the reader.
\end{example}

\indent For two $\mathscr{L}$-linear algebras $\mathscr{A}_1 = (A_1,m_1,\eta_1,\mathfrak{d}_1)$ and $\mathscr{A}_2 = (A_2,m_2,\eta_2,\mathfrak{d}_2)$ we define
\begin{equation}\label{defBraidedProductLlinAlg}
\mathscr{A}_1 \,\widetilde{\otimes}\, \mathscr{A}_2 = \bigl( A_1 \otimes A_2, \: (m_1 \otimes m_2) \circ (\mathrm{id}_{A_1} \otimes c_{A_2,A_1} \otimes \mathrm{id}_{A_2}), \: \eta_1 \otimes \eta_2, (\mathfrak{d}_1 \otimes \mathfrak{d}_2) \circ \Delta_{\mathscr{L}} \bigr)
\end{equation}
where $\Delta_{\mathscr{L}} : \mathscr{L} \to \mathscr{L} \otimes \mathscr{L}$ is the coproduct defined in \eqref{defStructureCoend}.

\begin{proposition}\label{propBraidedTensorProductOfLlinearAlg}
$\mathscr{A}_1 \,\widetilde{\otimes}\, \mathscr{A}_2$ is a $\mathscr{L}$-linear algebra.
\end{proposition}
\begin{proof}
Through the equivalence of Proposition \ref{propLLinearAlgIntoHBAlg}, the braided tensor product defined in \eqref{defBraidedProductLlinAlg} corresponds to the braided tensor product of half-braided algebras defined in \eqref{defBraidedTensorProductOfAlgebras2}. Hence, by Proposition \ref{propBraidedTensorProductAlgebra}, $\mathscr{A}_1 \,\widetilde{\otimes}\, \mathscr{A}_2$ is $\mathscr{L}$-linear. A direct proof is also easy.
\end{proof}

\subsection{Bimodules over \texorpdfstring{$\mathscr{L}$}{coend}-linear algebras}\label{sub:equivalence}
In the previous subsection we saw thanks to the isomorphism $\Upsilon : \mathcal{Z}(\mathcal{C}) \to \mathscr{L}\text{-}\mathrm{mod}_{\mathcal{C}}$ that a half-braided algebra becomes a $\mathscr{L}$-linear algebra, whose definition looks more familiar. Now we will define the notion of $\mathscr{L}$-compatible bimodule and show that it correponds to the notion of hb-compatible bimodule from \S\ref{subsectionHbCohBimod}.

\smallskip

\indent Let $\mathscr{A} = (A,m,\eta,\mathfrak{d})$ and $\mathscr{A}' = (A',m',\eta',\mathfrak{d}')$ be $\mathscr{L}$-linear algebras and consider bimodules over them in the usual sense (\S \ref{sectionPreliminariesModules}). Thanks to the morphisms of algebras $\mathfrak{d} : \mathscr{L} \to A$ and $\mathfrak{d}' : \mathscr{L} \to A'$, they are in particular $\mathscr{L}$-bimodules.

\begin{definition}\label{defLcoherentBimodule}
We say that an $(\mathscr{A}',\mathscr{A})$-bimodule $\mathbf{B} = (B, \smallblacktriangleright, \smallblacktriangleleft)$ is $\mathscr{L}$-compatible if the diagram
\begin{equation*}
\xymatrix@C=3em{
\mathscr{L} \otimes B \ar[rr]^-{\mathfrak{d}' \, \otimes \mathrm{id}_B} \ar[d]_-{\sigma_B} && A' \otimes B\ar[d]^-{\smallblacktriangleright}\\
B \otimes \mathscr{L} \ar[r]_-{\mathrm{id}_B \otimes \, \mathfrak{d}} & B \otimes A \ar[r]_-{\smallblacktriangleleft} & B
} \end{equation*}
commutes, with $\sigma : \mathscr{L} \otimes - \overset{\sim}{\implies} - \otimes \mathscr{L}$ the half-braiding defined in \eqref{halfBrSigma}.
\end{definition}
\noindent This can be seen as a generalization of Definition \ref{defLlinearAlgebra}, because a $\mathscr{L}$-linear algebra is a $\mathscr{L}$-compatible bimodule over itself.

\smallskip

In Proposition \ref{propLLinearAlgIntoHBAlg} we have established a bijection
\begin{equation}\label{bijectionHBAlgLlinAlg}
\widetilde{\Upsilon} : \bigl\{ \text{half-braided algebras in } \mathcal{C} \bigr\} \overset{\sim}{\longrightarrow} \bigl\{ \mathscr{L}\text{-linear algebras in } \mathcal{C} \bigr\}
\end{equation}
which is based on the monoidal isomorphism $\Upsilon : \mathcal{Z}(\mathcal{C}) \to \mathscr{L}\text{-}\mathrm{mod}_{\mathcal{C}}$ from Appendix \ref{subsectionIsoZCLModC}. Note that as {\em algebras} in $\mathcal{C},$ $\mathbb{A}$ and $\widetilde{\Upsilon}(\mathbb{A})$ are the same.

\begin{proposition}\label{propRelatingHBCoherentAndLCoherent}
Let $\mathbb{A}$ and $\mathbb{A}'$ be half-braided algebras in $\mathcal{C}$.
An $(\mathbb{A}', \mathbb{A})$-bimodule is hb-compatible (Def.\,\ref{defHbCoherentBimodule}) if and only if it is $\mathscr{L}$-compatible as an $\bigl(\widetilde{\Upsilon}(\mathbb{A}'), \widetilde{\Upsilon}(\mathbb{A})\bigr)$-bimodule.
\end{proposition}
\begin{proof}
Write $\mathbb{A} = (A,t,m,\eta)$, $\mathbb{A}' = (A',t',m',\eta')$ and $\widetilde{\Upsilon}(\mathbb{A}) = \bigl(A, m, \eta, \mathfrak{d}(t) \bigr)$, $\widetilde{\Upsilon}(\mathbb{A}') = \bigl(A', m', \eta',$ $\mathfrak{d}(t') \bigr)$, employing the same notations as in Prop.\,\ref{propLLinearAlgIntoHBAlg}. Let $\mathbf{B} = (B, \smallblacktriangleright, \smallblacktriangleleft)$ be a $(\mathbb{A}, \mathbb{A}')$-bimodule. Then we have the half-braidings $\mathrm{hbl}^{\mathbf{B}}, \mathrm{hbr}^{\mathbf{B}} : B \otimes - \overset{\sim}{\implies} - \otimes B$ defined in \eqref{halfBraidingInTermOfAction} and this gives objects $(B, \mathrm{hbl}^{\mathbf{B}})$, $(B, \mathrm{hbr}^{\mathbf{B}})$ in $\mathcal{Z}(\mathcal{C})$. Through the isomorphism $\Upsilon : \mathcal{Z}(\mathcal{C}) \to \mathscr{L}\text{-}\mathrm{mod}_{\mathcal{C}}$, write $\Upsilon(B, \mathrm{hbl}^{\mathbf{B}}) = \bigl( B, \lambda(\mathrm{hbl}^{\mathbf{B}}) \bigr)$, $\Upsilon(B, \mathrm{hbr}^{\mathbf{B}}) = \bigl( B, \lambda(\mathrm{hbr}^{\mathbf{B}}) \bigr)$. For any $K \in \mathcal{C}_{\mathrm{cp}}$ we have
\begin{center}
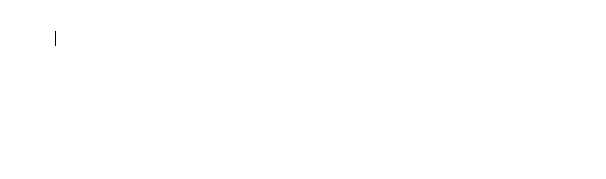
\end{center}
where the first equality uses \eqref{actionFromHalfBraiding}, the second uses \eqref{halfBraidingInTermOfAction} and the third uses \eqref{QMMinTermsOfHB}. On the other hand,
\begin{center}
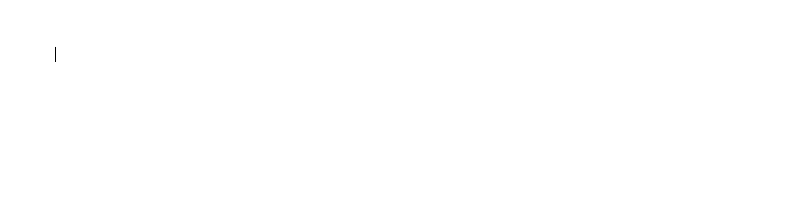
\end{center}
where the first equality uses \eqref{actionFromHalfBraiding}, the second uses \eqref{halfBraidingInTermOfAction}, the third is by isotopy and the fourth uses \eqref{QMMinTermsOfHB} and \eqref{defHalfBraidingSigmaOnCoend}. It follows from universality of $i$ that $\lambda(\mathrm{hbl}^{\mathbf{B}}) = \smallblacktriangleright \circ \bigl( \mathfrak{d}(t') \otimes \mathrm{id}_B \bigr)$ and $\lambda(\mathrm{hbr}^{\mathbf{B}}) = \smallblacktriangleleft \circ \bigl( \mathrm{id}_B \otimes \mathfrak{d}(t) \bigr) \circ \sigma_B$, which proves the Proposition.
\end{proof}

\indent Recall the composition of bimodules from \eqref{defCompositionOfBimodules} and their braided tensor product from \S \ref{sectionPreliminariesModules}.
\begin{proposition}\label{propMonoidalProductOfLCoherentBimodules}
1. Let $\mathbf{B}_1$ be a $\mathscr{L}$-compatible $(\mathscr{A}_2, \mathscr{A}_1)$-bimodule and $\mathbf{B}_2$ be a $\mathscr{L}$-compatible $(\mathscr{A}_3, \mathscr{A}_2)$-bimodule. Then $\mathbf{B}_2 \circ \mathbf{B}_1$ is a $\mathscr{L}$-compatible $(\mathscr{A}_3, \mathscr{A}_1)$-bimodule.
\\2. Let $\mathbf{B}_1$ be a $\mathscr{L}$-compatible $(\mathscr{A}'_1, \mathscr{A}_1)$-bimodule and $\mathbf{B}_2$ be a $\mathscr{L}$-compatible $(\mathscr{A}'_2, \mathscr{A}_2)$-bimodule. Then $\mathbf{B}_1 \,\widetilde{\otimes}\, \mathbf{B}_2$ is a $\mathscr{L}$-compatible $(\mathscr{A}'_1 \,\widetilde{\otimes}\, \mathscr{A}'_2, \mathscr{A}_1 \,\widetilde{\otimes}\,\mathscr{A}_2)$-bimodule.
\end{proposition}
\begin{proof}
This follows immediately from Proposition \ref{propCompoOfHBCoherentBimodules} and Proposition \ref{propHbCoherenceStableByBraidedTensorProduct} thanks to the equivalence obtained in Proposition \ref{propRelatingHBCoherentAndLCoherent}. A direct proof of item 2 which does not resort on the equivalence is also easy.
\end{proof}

We denoted by $\mathrm{Bim}_{\mathcal{C}}$ the category whose objects are algebras in $\mathcal{C}$ and morphisms are isomorphism classes of bimodules, see Definition \ref{defBimC}. The subcategory $\mathrm{Bim}^{\mathrm{hb}}_{\mathcal{C}}$ of $\mathrm{Bim}_{\mathcal{C}}$ consisting of half-braided algebras as objects and isomorphism classes of hb-compatible bimodules as morphisms (Definition \ref{defBimChb}) can now be restated in $\mathscr{L}$-linear terms:
\begin{corollary}\label{cor:correspondenceLmodvshalfbraided}
There is a subcategory $\mathrm{Bim}^{\mathscr{L}}_{\mathcal{C}}$ of $\mathrm{Bim}_{\mathcal{C}}$ such that
\begin{itemize}
\item its objects are the $\mathscr{L}$-linear algebras
\item $\Hom_{\mathrm{Bim}^{\mathscr{L}}_{\mathcal{C}}}(\mathscr{A}_1, \mathscr{A}_2)$ consists of the isomorphism classes of $\mathscr{L}$-compatible $(\mathscr{A}_2, \mathscr{A}_1)$-bimodules.\footnote{Note the swicth!}
\end{itemize}
It is a strict monoidal category with the braided tensor product $\widetilde{\otimes}$. Moreover the bijection \eqref{bijectionHBAlgLlinAlg} extends to a functor
\[ \widetilde{\Upsilon} : \mathrm{Bim}^{\mathrm{hb}}_{\mathcal{C}} \to \mathrm{Bim}^{\mathscr{L}}_{\mathcal{C}}, \quad \mathbb{A} \mapsto \widetilde{\Upsilon}(\mathbb{A}), \quad \mathbf{B} \mapsto \mathbf{B} \]
which is a strict monoidal isomorphism. It follows that $\mathrm{Bim}^{\mathscr{L}}_{\mathcal{C}}$ is braided.
\end{corollary}
\begin{proof}
By Proposition \ref{propMonoidalProductOfLCoherentBimodules}, $\mathrm{Bim}^{\mathscr{L}}_{\mathcal{C}}$ is a monoidal subcategory of $\mathrm{Bim}_{\mathcal{C}}$. The functor $\widetilde{\Upsilon}$ is a strict monoidal isomorphism because so is $\Upsilon$. The last claim is due to Theorem \ref{thBraided}.
\end{proof}

\subsection{Balance on \texorpdfstring{$\mathrm{Bim}_{\mathcal{C}}^{\mathscr{L}}$}{the coend-linear Morita category}}\label{sec:balanceBim}
Recall from Lemma \ref{lemmaCompactStableByMonoidalProduct} that under the assumptions \eqref{assumptionsCategoryC}, the subcategory $\mathcal{C}_{\mathrm{cp}}$ of compact objects coincides with the subcategory of rigid (\textit{i.e.} dualizable) objects. In this section we assume that $\mathcal{C}$ is {\em cp-ribbon}, which means that the subcategory $\mathcal{C}_{\mathrm{cp}}$ is ribbon. Said explicitly, we assume that there is a natural isomorphism $\theta : \mathrm{Id}_{\mathcal{C}_{\mathrm{cp}}} \overset{\sim}{\implies} \mathrm{Id}_{\mathcal{C}_{\mathrm{cp}}}$ called {\em twist} such that
\begin{equation}\label{axiomsTwist}
\forall \, K, Q \in \mathcal{C}_{\mathrm{cp}}, \quad \theta_{K \otimes Q} = c_{Q,K} \circ c_{K,Q} \circ (\theta_K \otimes \theta_Q) \quad \text{ and } \quad \theta_{K^*} = (\theta_K)^*.
\end{equation}
In this case the right dual $^*\!K$ can be realized as the left dual $K^*$ by letting (see e.g. \cite[\S XIV.3]{kassel})
\begin{equation}\label{dualityMorphByMeansOfTwist}
\widetilde{\mathrm{ev}}_K = \mathrm{ev}_K \circ c_{K,K^*} \circ (\theta_K \otimes \mathrm{id}_{K^*}), \qquad \widetilde{\mathrm{coev}}_K = (\mathrm{id}_{K^*} \otimes \theta_K) \circ c_{K,K^*} \circ \mathrm{coev}_K.
\end{equation}
With these definitions we have
\begin{equation}\label{twistAsLoop}
\forall \, K \in \mathcal{C}_{\mathrm{cp}}, \quad \theta_K = \bigl( \mathrm{ev}_K \otimes \mathrm{id}_K \bigr) \circ \bigl( \mathrm{id}_{K^*} \otimes c_{K,K} \bigr) \circ \bigl( \widetilde{\mathrm{coev}}_K \otimes \mathrm{id}_K \bigr)
\end{equation}
Moreover, thanks to Lemma \ref{lemmaExtensionNat}, we can extend $\theta$ to a natural isomorphism $\mathrm{Id}_{\mathcal{C}} \overset{\sim}{\implies} \mathrm{Id}_{\mathcal{C}}$, still denoted by $\theta$. Recall that this uses a filtered colimit presentation: if $V = \mathrm{colim} \bigl( K: \mathcal{I} \to \mathcal{C}_{\mathrm{cp}} \bigr)$ with universal cocone $\phi = \bigl( \phi_X : K(X) \to V \bigr)_{X \in \mathcal{I}}$ , there exists a unique isomorphism $\theta_V : V \to V$ characterized by
\begin{equation}\label{defExtensionTwist}
\forall \, X \in \mathcal{I}, \quad \theta_V \circ \phi_X = \phi_X \circ \theta_{K(X)}.
\end{equation}
Then $\theta$ is a {\em balance} on $\mathcal{C}$, meaning that
\[ \forall \, V,W \in \mathcal{C}, \quad \theta_{V \otimes W} = c_{W,V} \circ c_{V,W} \circ (\theta_V \otimes \theta_W). \]
The proof is similar to the one of item 1 in Lemma \ref{lemmaHBonCompacts} and is thus left to the reader.

\smallskip

For any half-braided algebra $\mathbb{A} = (A,t,m,\eta)$ in $\mathcal{C}$, define
\begin{equation}\label{defBalOnHBAlg}
\mathrm{bal}_{\mathbb{A}} : A \xrightarrow{\: \eta \,\otimes\, \theta_A \:} A \otimes A \xrightarrow{\: t_A \:} A \otimes A \xrightarrow{\: m \:} A.
\end{equation}
This morphism has striking properties:
\begin{proposition} \label{propPropertiesBalA}
1. $\mathrm{bal}_{\mathbb{A}}$ is an automorphism of half-braided algebra, whose inverse is
\[ \mathrm{bal}_{\mathbb{A}}^{-1} : A \xrightarrow{\: \theta_A \,\otimes\, \eta \:} A \otimes A \xrightarrow{\: t_A^{-1} \:} A \otimes A \xrightarrow{\: m \:} A. \]
2. For any half-braided algebras $\mathbb{A}_1 = (A_1,t^1,m_1,\eta_1)$ and $\mathbb{A}_2 = (A_2, t^2, m_2, \eta_2)$ we have
\[ \mathrm{bal}_{\mathbb{A}_1 \, \widetilde{\otimes}\, \mathbb{A}_2} = T_{\mathbb{A}_2, \mathbb{A}_1} \circ T_{\mathbb{A}_1, \mathbb{A}_2} \circ \bigl( \mathrm{bal}_{\mathbb{A}_1} \otimes \mathrm{bal}_{\mathbb{A}_2} \bigr) \]
where we recall from \eqref{isoBrZCforHBAlg} that $T$ is the braiding in $\mathcal{Z}(\mathcal{C})$, \textit{i.e.} $T_{\mathbb{A}_2, \mathbb{A}_1} = t^2_{A_1}$ and $T_{\mathbb{A}_1, \mathbb{A}_2} = t^1_{A_2}$.
\\3. Let $\Theta : \mathrm{Id}_{\mathcal{Z}(\mathcal{C})} \overset{\sim}{\implies} \mathrm{Id}_{\mathcal{Z}(\mathcal{C})}$ be the balance on $\mathcal{Z}(\mathcal{C})$ defined in Appendix \ref{appBalanceZC}. Then $\mathrm{bal}_{\mathbb{A}} = \Theta_{(A,t)}$ where $\mathbb{A} = (A,t,m,\eta)$.
\end{proposition}
\begin{proof}
Let us prove that $\mathrm{bal}_{\mathbb{A}}$ is a morphism of algebras: $m \circ \bigl( \mathrm{bal}_{\mathbb{A}} \otimes \mathrm{bal}_{\mathbb{A}} \bigr)$ is equal to
\begin{center}
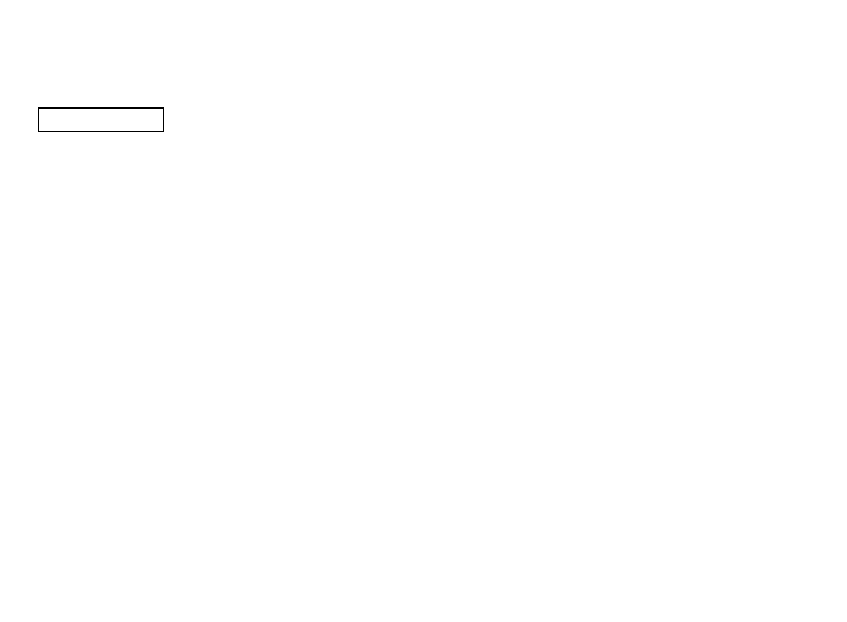
\end{center}
and the last term is $\mathrm{bal}_{\mathbb{A}} \circ m$, as desired. The first equality is by associativity of $m$, the second uses the half-braided algebra axioms \eqref{axiomsHalfBraidedAlgebra}, the third is by unitality of $m$ and naturality of the braiding, the fourth uses \eqref{axiomsHalfBraidedAlgebra}, the fifth uses the half-braiding axiom \eqref{axiomHalfBraiding} and associativity of $m$, the sixth is by naturality of $t$, the seventh is by the balancing property of $\theta$ and the eighth is by naturality of $\theta$. Now let us prove item 3. Let $V \in \mathcal{C}$ and write $V = \mathrm{colim}\bigl( K : \mathcal{I} \to \mathcal{C}_{\mathrm{cp}} \bigr)$ with universal cocone $\phi$. Then for all $X \in \mathcal{I}$,
\begin{center}
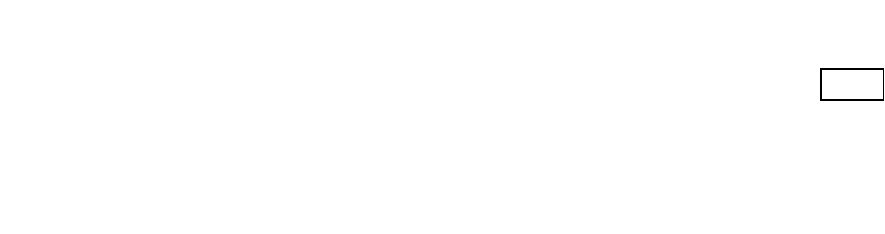
\end{center}
which implies that $\mathrm{bal}_{\mathbb{A}} = \Theta_{(A,t)}$ by the universal property of $\phi$. The first equality uses the definition of $\mathrm{bal}_{\mathbb{A}}$ together with \eqref{defExtensionTwist} and \eqref{twistAsLoop}, the second equality is by naturality of the braiding, the third equality is by naturality of $t$, the fourth equality is by \eqref{axiomsHalfBraidedAlgebra} and the last equality uses the unitality of $m$ and the definition of $\Theta$ in \eqref{defBalZC}. Hence all the other claimed properties of $\mathrm{bal}_{\mathbb{A}}$ immediately follow from the properties of $\Theta$ given in Prop.\,\ref{propBalanceZC}.
\end{proof}
\indent We continue to denote by $\mathbb{A} = (A,t,m,\eta)$ a half-braided algebra in $\mathcal{C}$. Recall from \eqref{subsecTwistingBimod} that an algebra morphism can be used to change the action on a bimodule. Consider the $(\mathbb{A}, \mathbb{A})$-bimodule
\begin{equation}\label{defBALbimod}
\mathrm{BAL}_{\mathbb{A}} = \mathbb{A} \smallblacktriangleleft \mathrm{bal}_{\mathbb{A}}
\end{equation}
which is the object $A$ endowed with the left action of $\mathbb{A}$ simply given by the multiplication $m : A \otimes A \to A$ while the right action is $m \circ \mathrm{bal}_{\mathbb{A}}$. By Lemma \ref{lemmaTwistingCohBim} and item 1 in Prop.\,\ref{propPropertiesBalA} the bimodule $\mathrm{BAL}_{\mathbb{A}}$ is hb-compatible, or in other words $\mathrm{BAL}_{\mathbb{A}} \in \Hom_{\mathrm{Bim}_{\mathcal{C}}^{\mathrm{hb}}}(\mathbb{A}, \mathbb{A})$.

\begin{theorem}\label{thmBalanceBim}
Assume that the category $\mathcal{C}$ satisfies \eqref{assumptionsCategoryC} and is moreover cp-ribbon. Then the family $\bigl( \mathrm{BAL}_{\mathbb{A}} \bigr)_{\mathbb{A} \in \mathrm{Ob}(\mathrm{Bim}_{\mathcal{C}}^{\mathrm{hb}})}$ is a balance on $\mathrm{Bim}_{\mathcal{C}}^{\mathrm{hb}}$, meaning that it is natural and
\[ \forall \, \mathbb{A}_1, \mathbb{A}_2, \quad  \mathrm{BAL}_{\mathbb{A}_1 \,\widetilde{\otimes}\, \mathbb{A}_2} = \mathcal{B}_{\mathbb{A}_2, \mathbb{A}_1} \circ \mathcal{B}_{\mathbb{A}_1, \mathbb{A}_2} \circ \bigl( \mathrm{BAL}_{\mathbb{A}_1} \,\widetilde{\otimes}\, \mathrm{BAL}_{\mathbb{A}_2} \bigr) \]
where $\mathcal{B}$ is the braiding \eqref{eq:braidingbimodule} on $\mathrm{Bim}_{\mathcal{C}}^{\mathrm{hb}}$.
\end{theorem}
\begin{proof}
Recall that equality of morphisms in $\mathrm{Bim}_{\mathcal{C}}^{\mathrm{hb}}$ means that the bimodules are isomorphic. Let us first prove naturality. So let $\mathbf{B} \in \Hom_{\mathrm{Bim}_{\mathcal{C}}^{\mathrm{hb}}}(\mathbb{A}_1, \mathbb{A}_2)$, \textit{i.e.} $\mathbf{B}$ is a hb-compatible $(\mathbb{A}_2, \mathbb{A}_1)$-bimodule. Write $\mathbf{B} = (B, \smallblacktriangleright, \smallblacktriangleleft)$. By \eqref{twistingLeftToRight} and \eqref{twistingCompBim} we have
\[ \mathrm{BAL}_{\mathbb{A}_2} \circ \mathbf{B} = \bigl( \mathbb{A}_2 \smallblacktriangleleft \mathrm{bal}_{\mathbb{A}_2} \bigr) \circ \mathbf{B} = \bigl( \mathrm{bal}_{\mathbb{A}_2}^{-1} \smallblacktriangleright \mathbb{A}_2 \bigr) \circ \mathbf{B} = \mathrm{bal}_{\mathbb{A}_2}^{-1} \smallblacktriangleright \bigl( \mathbb{A}_2 \circ \mathbf{B} \bigr) = \mathrm{bal}_{\mathbb{A}_2}^{-1} \smallblacktriangleright \mathbf{B} \]
and similarly $\mathbf{B} \circ \mathrm{BAL}_{\mathbb{A}_1} = \mathbf{B} \smallblacktriangleleft \mathrm{bal}_{\mathbb{A}_1}$. We thus have to show that $\mathbf{B}$ is isomorphic to $\mathrm{bal}_{\mathbb{A}_2} \smallblacktriangleright \mathbf{B} \smallblacktriangleleft \mathrm{bal}_{\mathbb{A}_1}$ as bimodules. Since $\mathbf{B}$ is a hb-compatible bimodule, it is endowed with the half-braiding $\mathrm{hb}_{\mathbf{B}} : B \otimes - \Rightarrow - \otimes B$ which is indifferently $\mathrm{hbl}_{\mathbf{B}}$ or $\mathrm{hbr}_{\mathbf{B}}$, recall \eqref{halfBraidingInTermOfAction}. Hence $\mathbf{B}$ can be seen as an object in $\mathcal{Z}(\mathcal{C})$ and we have the isomorphism $\Theta_{\mathbf{B}} \in \Hom_{\mathcal{Z}(\mathcal{C})}(\mathbf{B},\mathbf{B})$ defined in \eqref{defBalZC} in App.\,\ref{appBalanceZC}. We claim that $\Theta_{\mathbf{B}}$ is an isomorphism of bimodules $\mathbf{B} \overset{\sim}{\to} \mathrm{bal}_{\mathbb{A}_2} \smallblacktriangleright \mathbf{B} \smallblacktriangleleft \mathrm{bal}_{\mathbb{A}_1}$. A bit of preparation is in order to show this. Write as usual $\mathbb{A}_2 = (A_2, t^2, m_2, \eta_2)$; we can in particular see $\mathbb{A}_2$ as an object in $\mathcal{Z}(\mathcal{C})$ and thus we have the morphism $\Theta_{\mathbb{A}_2} \in \Hom_{\mathcal{Z}(\mathcal{C})}(\mathbb{A}_2,\mathbb{A}_2)$. Take filtered colimit presentations $A_2 = \mathrm{colim}\bigl( K : \mathcal{I} \to \mathcal{C}_{\mathrm{cp}} \bigr)$, $B = \mathrm{colim}\bigl( Q : \mathcal{J} \to \mathcal{C}_{\mathrm{cp}} \bigr)$ with universal cocones $\phi = \bigl( \phi_X : K(X) \to A_2 \bigr)_{X \in \mathcal{I}}$ and $\psi = \bigl( \psi_Y : Q(Y) \to B \bigr)_{Y \in \mathcal{J}}$ respectively. Then by \eqref{colimPresOfMonProd} below, $A_2 \otimes B$ is the colimit of the functor $(X,Y) \mapsto K(X) \otimes Q(Y)$ with universal cocone $\phi \otimes \psi$. Let $(X,Y) \in \mathcal{I} \times \mathcal{J}$ be arbitrary; applying Lemma \ref{lemmaFactoCompactObjects} to the morphism $\smallblacktriangleright \circ (\phi_X \otimes \psi_Y) : K(X) \otimes Q(Y) \to B$, we get that there exists $Z \in \mathcal{J}$ and $g \in \Hom_{\mathcal{C}}\bigl( K(X) \otimes Q(Y), Q(Z) \bigr)$ such that $\smallblacktriangleright \circ (\phi_X \otimes \psi_Y) = \psi_Z \circ g$. We can finally make the following computation:
\begin{center}
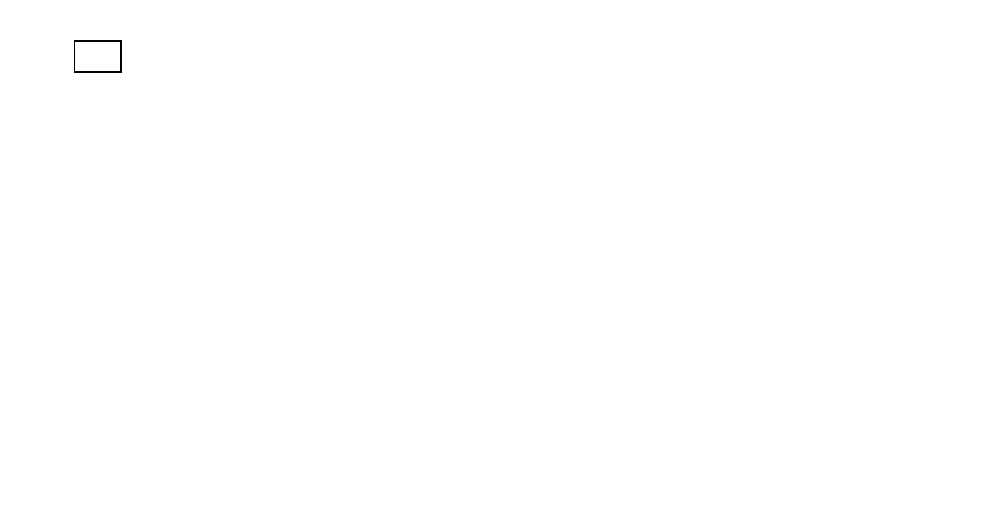
\end{center}
where the first equality is by definition of $g$, the second is by definition of $\Theta_{\mathbf{B}}$ in \eqref{defBalZC}, the third is by naturality of $\mathrm{hb}^{\mathbf{B}}$ and cups/caps, the fourth is by definition of $g$, the fifth is by \eqref{axiomsHalfBraidedLeftModule}, the sixth is by the half-braiding axiom \eqref{axiomHalfBraiding}, the seventh is by definition of $\Theta_{\mathbb{A}_2}$ in \eqref{defBalZC} and naturality of the braiding, the eighth is by \eqref{axiomsHalfBraidedLeftModule} and the last is by naturality of the braiding and definition of $\Theta_{\mathbf{B}}$. By the universal property of $\phi \otimes \psi$ and by item 3 in Prop.\,\ref{propPropertiesBalA} we conclude that $\Theta_{\mathbf{B}} \circ \smallblacktriangleright = \smallblacktriangleright \circ \bigl( \mathrm{bal}_{\mathbb{A}_2} \otimes \mathrm{id}_B \bigr) \circ \bigl( \mathrm{id}_{A_2} \otimes \Theta_{\mathbf{B}} \bigr)$, which proves that $\Theta_{\mathbf{B}}$ intertwines the left actions. The proof for the right actions follows the same lines, except that it uses \eqref{axiomsHalfBraidedRightModule} (so in particular we see that it is important to assume that $\mathbf{B}$ is a hb-compatible bimodule).
\\ The balance property is easily deduced from the balance property of $\mathrm{bal}_{\mathbb{A}}$ (item 2 in Prop.\,\ref{propPropertiesBalA}):
\begin{align*}
\mathrm{BAL}_{\mathbb{A}_1 \,\widetilde{\otimes}\, \mathbb{A}_2} &= \bigl( \mathbb{A}_1 \,\widetilde{\otimes}\, \mathbb{A}_2 \bigr) \smallblacktriangleleft \mathrm{bal}_{\mathbb{A}_1 \,\widetilde{\otimes}\, \mathbb{A}_2} = \bigl( \mathbb{A}_1 \,\widetilde{\otimes}\, \mathbb{A}_2 \bigr) \smallblacktriangleleft \bigl( T_{\mathbb{A}_2, \mathbb{A}_1} \circ T_{\mathbb{A}_1, \mathbb{A}_2} \circ ( \mathrm{bal}_{\mathbb{A}_1} \otimes \mathrm{bal}_{\mathbb{A}_2} ) \bigr)\\
&= \mathcal{B}_{\mathbb{A}_2, \mathbb{A}_1} \smallblacktriangleleft \bigl( T_{\mathbb{A}_1, \mathbb{A}_2} \circ ( \mathrm{bal}_{\mathbb{A}_1} \otimes \mathrm{bal}_{\mathbb{A}_2} ) \bigr) \quad \text{ \footnotesize by \eqref{twistingCompMor} and \eqref{eq:braidingbimodule}}\\
&= \bigl( \mathcal{B}_{\mathbb{A}_2, \mathbb{A}_1} \circ (\mathbb{A}_2 \,\widetilde{\otimes}\, \mathbb{A}_1) \bigr) \smallblacktriangleleft \bigl( T_{\mathbb{A}_1, \mathbb{A}_2} \circ ( \mathrm{bal}_{\mathbb{A}_1} \otimes \mathrm{bal}_{\mathbb{A}_2} ) \bigr) \quad \text{ \footnotesize (trick, nothing is changed)}\\
&= \bigl( \mathcal{B}_{\mathbb{A}_2, \mathbb{A}_1} \circ \mathcal{B}_{\mathbb{A}_1, \mathbb{A}_2} \bigr) \smallblacktriangleleft \bigl(\mathrm{bal}_{\mathbb{A}_1} \otimes \mathrm{bal}_{\mathbb{A}_2}\bigr) \quad \text{\footnotesize by \eqref{twistingCompMor}, \eqref{twistingCompBim} and \eqref{eq:braidingbimodule}}\\
&= \bigl( \mathcal{B}_{\mathbb{A}_2, \mathbb{A}_1} \circ \mathcal{B}_{\mathbb{A}_1, \mathbb{A}_2} \circ (\mathbb{A}_1 \,\widetilde{\otimes}\, \mathbb{A}_2) \bigr) \smallblacktriangleleft \bigl(\mathrm{bal}_{\mathbb{A}_1} \otimes \mathrm{bal}_{\mathbb{A}_2}\bigr) \quad \text{ \footnotesize (trick, nothing is changed)}\\
&= \mathcal{B}_{\mathbb{A}_2, \mathbb{A}_1} \circ \mathcal{B}_{\mathbb{A}_1, \mathbb{A}_2} \circ \bigl( (\mathbb{A}_1 \smallblacktriangleleft \mathrm{bal}_{\mathbb{A}_1}) \,\widetilde{\otimes}\, (\mathbb{A}_2 \smallblacktriangleleft \mathrm{bal}_{\mathbb{A}_2}) \bigr) \quad \text{\footnotesize by \eqref{twistingCompBim} and \eqref{twistingCompBrTens}}\\
&=\mathcal{B}_{\mathbb{A}_2, \mathbb{A}_1} \circ \mathcal{B}_{\mathbb{A}_1, \mathbb{A}_2} \circ \bigl( \mathrm{BAL}_{\mathbb{A}_1} \,\widetilde{\otimes}\, \mathrm{BAL}_{\mathbb{A}_2} \bigr) \quad \text{\footnotesize by definition.} \qedhere
\end{align*}
\end{proof}

\section{The case \texorpdfstring{$\mathcal{C} = \mathrm{Comod}\text{-}\OO$}{of categories of comodules}}\label{sec:Hcomod}

We now specialize the concepts and results of the previous section (\S\ref{sectionLlinear}) to the case where $\mathcal{C}$ is the category of right comodules over a coquasitriangular Hopf algebra $\OO$. It is an important example of a category which satisfies the assumptions \eqref{assumptionsCategoryC}, as already noted e.g in \cite{BZBJ}. This choice of ambient category $\mathcal{C}$ is motivated by the topological examples of $\mathscr{L}$-linear algebras and $\mathscr{L}$-compatible bimodules obtained through the stated skein functor which will be presented in \S\ref{sectionStatedSkein}.

\subsection{The category \texorpdfstring{$\mathrm{Comod}\text{-}\OO$}{of comodules} and its coend}\label{subsecComodH}
Let $k$ be a field. If $C = (C,\Delta,\varepsilon)$ is a $k$-coalgebra we write
\[ \Delta(c) = c_{(1)} \otimes c_{(2)} \]
and then $c_{(1)} \otimes\ldots \otimes c_{(n)}$ means iterated coproduct (Sweedler's notation). Recall that a right $C$-comodule is a $k$-vector space $V$ endowed with a linear map $\delta_V : V \to V \otimes C$ satisfying the usual coaction conditions. We write
\[ \delta_V(v) = v_{[0]} \otimes v_{[1]} \]
and then $v_{[0]} \otimes \ldots \otimes v_{[n]}$ means iterated coaction:
\[ v_{[0]} \otimes v_{[1]} \otimes v_{[2]} = v_{[0][0]} \otimes v_{[0][1]} \otimes v_{[1]} = v_{[0]} \otimes v_{[1](1)} \otimes v_{[1](2)}, \quad \textit{etc} \]
We denote by $\mathrm{Comod}\text{-}C$ the category of right $C$-comodules (not necessarily finite-dimensional). The Hom spaces $\Hom^C(V,W)$ consist of $k$-linear maps $f : V \to W$ such that
\[ \forall\, v \in V, \quad f(v)_{[0]} \otimes f(v)_{[1]} = f(v_{[0]}) \otimes v_{[1]}. \]
\indent The next properties are well-known (see e.g. \cite[Cor.\,26]{wis}, \cite[Prop.\,1]{porst}). However it seems hard to find self-contained proofs in the literature, especially for item 3, so we provide a detailed proof for convenience.
\begin{lemma}\label{propLFPComod}
1. $\mathrm{Comod}\text{-}C$ is cocomplete.
\\2. For any $V \in \mathrm{Comod}\text{-}C$ the endofunctors $V \otimes -$ and $- \otimes V$ are cocontinuous.
\\3. A $C$-comodule is a compact object in $\mathrm{Comod}\text{-}\,C$ if and only if it is finite-dimensional.
\\4. $\mathrm{Comod}\text{-}C$ is a LFP category.
\end{lemma}
\begin{proof}
1. $\mathrm{Comod}\text{-}C$ has coproducts (\textit{i.e.} direct sums of any family of comodules indexed by any set) and cokernels (\textit{i.e.} quotients), hence it has small colimits.
\\2. These functors commute with direct sums indexed by any set and commute with cokernels, hence they commute with small colimits.
\\3. Let $V$ be a finite-dimensional $C$-comodule. Fix a basis $v_1, \ldots, v_n$ of $V$. For any $C$-comodule $Y$, we can describe $\Hom^C(V,Y)$ as a finite limit in $\mathrm{Vect}_k$ thanks to the following construction which describes it as a ``multiple coequalizer''. Let $\mathcal{J}$ be the following finite category
\[ \xymatrix{
& \lozenge \ar@<0.7ex>[dl]^{\beta_1} \ar@<-0.7ex>[dl]_{\alpha_1} \ar@<0.7ex>[dr]^{\beta_n} \ar@<-0.7ex>[dr]_{\alpha_n}  &\\
\square_1 & \ldots & \square_n
} \]
which has $n+1$ ``formal'' objects denoted by $\lozenge$, $\square_1, \ldots, \square_n$, has exactly two ``formal'' arrows $\alpha_i,\beta_i : \lozenge \to \square_i$ for each $i \in \{1, \ldots, n\}$ and we did not draw the identity morphisms. Consider the functor $D^Y : \mathcal{J} \to \mathrm{Vect}_k$ given by $D^Y(\lozenge) = \mathrm{Hom}_k(V,Y)$, $D^Y(\square_i) = Y \otimes C$ for all $i$ and
\begin{align*}
&D^Y(\alpha_i) : \Hom_k(V,Y) \to Y \otimes C, \quad f \mapsto f(v_i)_{[0]} \otimes f(v_i)_{[1]}\\
&D^Y(\beta_i) : \Hom_k(V,Y) \to Y \otimes C, \quad f \mapsto f\bigl( (v_i)_{[0]} \bigr) \otimes (v_i)_{[1]}
\end{align*}
Note that a $k$-linear map $f : V \to Y$ is in $\Hom^C(V,Y)$ if and only if $D^Y(\alpha_i)(f) = D^Y(\beta_i)(f)$ for all $i$. It follows that $\Hom^C(V,Y) = \mathrm{lim}_{J \in \mathcal{J}} \, D^Y(J)$, the universal cone $\bigl( e_J : \Hom^C(V,Y) \to D^Y(J) \bigr)_{J \in \mathcal{J}}$ being given by $e_{\lozenge} : \Hom^C(V,Y) \hookrightarrow \Hom_k(V,Y)$ the inclusion and $e_{\square_i} = \alpha_i = \beta_i : \Hom^C(V,Y) \to Y \otimes C$. Now let $W = \mathrm{colim}_{I \in \mathcal{I}} \, F(I)$ be a filtered colimit in $\mathrm{Comod}\text{-}C$ and consider the bifunctor $\mathcal{I} \otimes \mathcal{J} \to \mathrm{Vect}_k$ given by $(I,J) \mapsto D^{F(I)}(J)$. Since filtered colimits commute with finite limits in $\mathrm{Vect}_k$ (see e.g. \cite[\S IX.2]{MLCat}) we obtain 
\[ \underset{I \in \mathcal{I}}{\mathrm{colim}} \, \Hom^C(V, F(I)) \cong \underset{I \in \mathcal{I}}{\mathrm{colim}} \, \underset{J \in \mathcal{J}}{\mathrm{lim}} \, D^{F(I)}(J) \cong \underset{J \in \mathcal{J}}{\mathrm{lim}} \, \underset{I \in \mathcal{I}}{\mathrm{colim}} \, D^{F(I)}(J) \cong \underset{J \in \mathcal{J}}{\mathrm{lim}} \, D^W(J) \cong \Hom^C(V,W). \]
For the third isomorphism we used that for all $J \in \mathcal{J}$ the functor $\mathrm{Comod}\text{-}C \to \mathrm{Vect}_k$ given by $Y \mapsto D^Y(J)$ is cocontinuous. Indeed for $J = \lozenge$ it is the functor $Y \mapsto \Hom_k(V, Y)$ which is cocontinuous by Example \ref{exampleVectIsLFP} while for $J = \square_i$ it is the functor $Y \mapsto Y \otimes C$ which is also cocontinuous beacuse the tensor product commutes with cokernels and arbitrary direct sums in $\mathrm{Vect}_k$.
\\For the converse, recall that any comodule is the union of its finite-dimensional subcomodules \cite[Th. 5.1.1]{Mon}. It follows that for any $V \in \mathrm{Comod}\text{-}C$, if we let $\mathcal{S}_V$ be the category whose objects are finite-dimensional subcomodules of $V$ and whose morphisms are inclusions of subcomodules, we have $V = \underset{X \in \mathcal{S}_V}{\mathrm{colim}} \, X$. Let $K$ be a compact comodule. Then $\Hom^C(K,K) = \underset{X \in \mathcal{S}_K}{\mathrm{colim}} \, \Hom^C(K,X)$. It follows from Lemma \ref{lemmaFactoCompactObjects} that there exists $X \in \mathcal{S}_K$ and $g : K \to X$ such that $\mathrm{id}_K = \iota_X \circ g$, where $\iota_X : X \to K$ is the inclusion. As a result $K = \iota_X(\mathrm{im} \, g)$ is finite-dimensional.
\\4. This follows from items 1 and 3 and the fundamental theorem on comodules \cite[Th. 5.1.1]{Mon}.
\end{proof}

\indent Now let $\OO = (\OO,m,1,\Delta,\varepsilon,S)$ be a Hopf $k$-algebra with invertible antipode $S$, where $k$ is a field. The category $\mathrm{Comod}\text{-}\OO$ is monoidal: if $V$ and $W$ are right $\OO$-comodules we define
\begin{equation}\label{monoidalProductOfComodules}
(v \otimes w)_{[0]} \otimes (v \otimes w)_{[1]} = v_{[0]} \otimes w_{[0]} \otimes v_{[1]}w_{[1]}.
\end{equation}
If $V$ is a {\em finite-dimensional} $\OO$-comodule, then it has a left (resp. right) dual, which is the dual vector space $V^*$ with the coaction
\begin{equation*}
\textstyle f_{[0]} \otimes f_{[1]} = f\bigl( v_{i \, [0]} \bigr) \: v^i \otimes S\bigl(v_{i\,[1]}\bigr) \qquad (\text{resp. } f_{[0]} \otimes f_{[1]} = f\bigl( v_{i \, [0]} \bigr) \: v^i \otimes S^{-1}\bigl(v_{i\,[1]}\bigr)\:)
\end{equation*}
where $(v_i)$ is a basis of $V$ with dual basis $(v^i)$ and summation on $i$ is understood. In other words, $f_{[0]}(v)\,f_{[1]} = f(v_{[0]})\,S(v_{[1]})
$ (resp. $f_{[0]}(v)\,f_{[1]} = f(v_{[0]})\,S^{-1}(v_{[1]})$) for all $v \in V$.

\smallskip

\indent If $\OO$ is {\em coquasitriangular}, which means that it is endowed with a convolution invertible element $\mathcal{R} : \OO \otimes \OO \to k$
satisfying the usual axioms which can be found in \cite[\S 2.2]{Majid}, then the category $\mathrm{Comod}\text{-}\OO$ inherits a braiding:
\[ \foncIso{c_{X,Y}}{X \otimes Y}{Y \otimes X}{x \otimes y}{\mathcal{R}\bigl(x_{[1]} \otimes y_{[1]}\bigr) \, y_{[0]} \otimes x_{[0]}}. \]
Note that $c^{-1}_{X,Y}(y \otimes x) = \mathcal{R}\bigl(S(x_{[1]}) \otimes y_{[1]}\bigr) \, x_{[0]} \otimes y_{[0]}$ because $\mathcal{R}^{-1} = \mathcal{R} \circ (S \otimes \mathrm{id})$.
\begin{lemma}
If $\OO$ is a coquasitriangular Hopf algebra, then $\mathrm{Comod}\text{-}\OO$ satisfies the properties \eqref{assumptionsCategoryC}.
\end{lemma}
\begin{proof}
Immediate from Lemma \ref{propLFPComod} and the facts recalled above on comodules over Hopf algebras.
\end{proof}

\indent We denote by $\mathrm{comod}\text{-}\OO$ the full subcategory of finite-dimensional comodules, \textit{i.e} the subcategory $\mathcal{C}_{\mathrm{cp}}$ of compact objects in $\mathcal{C} = \mathrm{Comod}\text{-}\OO$. The description of the coend \eqref{defCoendL} and of its structure \eqref{defStructureCoend} in this case is due to Majid, who called it the {\em transmutation} of $\OO$:

\begin{proposition}{\em \cite[Th.\,4.1]{majBrGr}}\label{descriptionCoendComodH}~~The coend $\mathscr{L} = \int^{X \in \mathrm{comod}\text{-}\OO} X^* \otimes X$ is the vector space $\OO$ endowed with the adjoint coaction defined by
\begin{equation}\label{defAdjointCoaction}
\mathrm{cad} : \mathscr{L} \to \mathscr{L} \otimes \OO, \quad \varphi \mapsto \varphi_{(2)} \otimes S(\varphi_{(1)})\varphi_{(3)}.
\end{equation}
The universal dinatural transformation is
\begin{equation}\label{univDinatTransfoTransmutation}
i_X : X^*  \otimes X \to \mathscr{L}, \quad f \otimes x \mapsto f(x_{[0]})x_{[1]}
\end{equation}
for any finite-dimensional $\OO$-comodule $X$. The product $m_{\mathscr{L}}$, which we denote by $\odot$, is given by
\begin{equation}\label{braidedProductCoend}
\forall \, \varphi,\psi \in \mathscr{L}, \quad \varphi \odot \psi = \mathcal{R}\bigl( S(\varphi_{(1)})\varphi_{(3)} \otimes S(\psi_{(1)}) \bigr) \, \varphi_{(2)}\psi_{(2)}.
\end{equation}
The unit $1_{\mathscr{L}}$, coproduct $\Delta_{\mathscr{L}}$ and counit $\varepsilon_{\mathscr{L}}$ are just $1_{\OO}$, $\Delta_{\OO}$ and $\varepsilon_{\OO}$. There is also an antipode $S_{\mathscr{L}}$ given by $S_{\mathscr{L}}(\varphi) = \mathcal{R}\bigl( S^2(\varphi_{(3)}) S(\varphi_{(1)}) \otimes \varphi_{(4)} \bigr) \, S(\varphi_{(2)})$.
\end{proposition}

\subsection{\texorpdfstring{$\mathscr{L}$}{Coend}-linear algebras and their bimodules for \texorpdfstring{$\mathrm{Comod}\text{-}\OO$}{categories of comodules}}\label{subsec:BimodHcomod}
\indent Recall the description of the coend $\mathscr{L}$ in Prop.\,\ref{descriptionCoendComodH}, and especially the fact that $\mathscr{L}$ is identified with $\OO$ as a coalgebra.

\smallskip

\indent A straightforward computation using \eqref{univDinatTransfoTransmutation} reveals that for any $V \in \mathrm{Comod}\text{-}\OO$, the half-braiding $\sigma_V : \mathscr{L} \otimes V \to V \otimes \mathscr{L}$ from \eqref{defHalfBraidingSigmaOnCoend} and its inverse are given by
\begin{align}
\begin{split}\label{halfBraidingSigmaComodH}
\sigma_V(\varphi \otimes v) &= \mathcal{R}\bigl(v_{[1]} \otimes \varphi_{(1)}\bigr) \, \mathcal{R}\bigl(\varphi_{(3)} \otimes v_{[2]}\bigr) \, v_{[0]} \otimes \varphi_{(2)}\\
\!\!\sigma_V^{-1}(v \otimes \varphi) &=\mathcal{R}\bigl( v_{[2]} \otimes S(\varphi_{(1)}) \bigr) \, \mathcal{R}^{-1}\bigl( \varphi_{(3)} \otimes v_{[1]} \bigr) \, \varphi_{(2)} \otimes v_{[0]}
\end{split}
\end{align}
for all $\varphi \in \mathscr{L}$ and $v \in V$. Hence, a {\em $\mathscr{L}$-linear algebra in $\mathrm{Comod}\text{-}\OO$} (Def. \ref{defLlinearAlgebra}) is a pair $\mathscr{A} = (A, \mathfrak{d})$ with:
\begin{itemize}
\item $A$ is both a unital associative $k$-algebra and a $\OO$-comodule, and we have
\begin{equation}\label{axiomComodAlg}
\forall \, x,y \in A, \quad (x y)_{[0]} \otimes (x y)_{[1]} = x_{[0]} y_{[0]} \otimes x_{[1]}y_{[1]} \quad \text{ and } \quad (1_A)_{[0]} \otimes (1_A)_{[1]} = 1_A \otimes 1_{\OO},
\end{equation}
\textit{i.e.} $A$ is a right {\em $\OO$-comodule-algebra} (which means algebra object in $\mathrm{Comod}\text{-}\OO$).
\item $\mathfrak{d} : \mathscr{L} \to A$ is a morphism of $\OO$-comodule-algebras which satisfies
\begin{equation}\label{LlinearAxiomComodH}
\forall \,\varphi \in \mathscr{L}, \:\: \forall \, a \in A, \quad \mathfrak{d}(\varphi)a = \mathcal{R}\bigl(a_{[1]} \otimes \varphi_{(1)}\bigr) \, \mathcal{R}\bigl(\varphi_{(3)} \otimes a_{[2]}\bigr) \, a_{[0]} \mathfrak{d}(\varphi_{(2)}).
\end{equation}
This can also be written as $\mathcal{R}^{-1}\bigl( \varphi_{(2)} \otimes a_{[1]} \bigr) \mathfrak{d}(\varphi_{(1)})a_{[0]} = \mathcal{R}\bigl(a_{[1]} \otimes \varphi_{(1)} \bigr) a_{[0]} \mathfrak{d}(\varphi_{(2)})$.
\end{itemize}
Explicitly, the requirement that $\mathfrak{d}$ is a morphism of $\OO$-comodule algebras means
\[ \forall \, \varphi,\psi \in \mathscr{L}, \quad \mathfrak{d}(\varphi)_{[0]} \otimes \mathfrak{d}(\varphi)_{[1]} = \mathfrak{d}(\varphi_{(2)}) \otimes S(\varphi_{(1)}) \varphi_{(3)} \quad \text{ and } \quad \mathfrak{d}(\varphi)\mathfrak{d}(\psi) = \mathfrak{d}(\varphi \odot \psi) \]
because of the $\OO$-coaction \eqref{defAdjointCoaction} on $\mathscr{L}$, and where $\odot$ is the product \eqref{braidedProductCoend} in $\mathscr{L}$.

\begin{remark}
Other point of views on $\mathscr{L}$-linear algebras for $\mathcal{C} = \mathrm{Comod}\text{-}\OO$ are given in \S\ref{sub:LlinQMM} (module-algebras endowed with a quantum moment map in the sense of \cite{Lu,VV}) and in Appendix \ref{appendixLlinearComodH} (algebras which have a Yetter--Drinfeld module structure).
\end{remark}

\indent Let us spell out the equivalence between half-braided algebras and $\mathscr{L}$-linear algebras (Prop.\,\ref{propLLinearAlgIntoHBAlg}) in the present case of $\mathcal{C} = \mathrm{Comod}\text{-}\OO$. The fastest way is to use Lemma \ref{lemmaIsoZCvsLModCForComodH}, with the left action $\smallblacksquare$ of $\mathscr{L}$ on $A$ given by $\varphi \smallblacksquare a = \mathfrak{d}(\varphi)a$. Thus, if $(A,t)$ is a half-braided algebra in $\mathrm{Comod}\text{-}\OO$\footnote{By Def.\,\ref{defHBAlgebra}, it means that $A$ is a right $\OO$-comodule-algebra and $t : A \otimes - \overset{\sim}{\implies} - \otimes A$ is a half-braiding in $\mathrm{Comod}\text{-}\OO$ which satisfies \eqref{axiomsHalfBraidedAlgebra}.} then we define
\[ \mathfrak{d} : \mathscr{L} \to A, \quad \varphi \mapsto (\varepsilon_{\OO} \otimes \mathrm{id}_A) \circ t_{\OO}(1 \otimes \varphi) \]
with  $\varepsilon_{\OO}$ the counit of $\OO$ and $\OO$ is viewed as a comodule over itself thanks to coproduct; then $(A,\mathfrak{d})$ is a $\mathscr{L}$-linear algebra. Conversely, if $(A,\mathfrak{d})$ is an $\mathscr{L}$-linear algebra, then for all $X \in \mathrm{Comod}\text{-}\OO$ we define
\begin{equation}\label{HBfromLlinComodH}
t_X : A \otimes X \to X \otimes A, \quad a \otimes x \mapsto \mathcal{R}^{-1}(x_{[2]} \otimes a_{[1]}) \, x_{[0]} \otimes \mathfrak{d}(x_{[1]}) a_{[0]}
\end{equation}
and then $(A,t)$ is a half-braided algebra.

\smallskip

\indent For $i=1,2$, let $\mathscr{A}_i = (A_i, \mathfrak{d}_i)$ be a $\mathscr{L}$-linear algebra. An {\em $(\mathscr{A}_2, \mathscr{A}_1)$-bimodule in $\mathrm{Comod}\text{-}\OO$} is a triple $\mathbf{B} = (B, \smallblacktriangleright, \smallblacktriangleleft)$ such that $B$ is a right $\OO$-comodule and the actions $\smallblacktriangleright : A_2 \otimes B \to B$, $\smallblacktriangleleft : B \otimes A_1 \to B$ commute and are morphisms of $\OO$-comodules. The bimodule $\mathbf{B}$ is {\em $\mathscr{L}$-compatible} (Def. \ref{defLcoherentBimodule}) if
\begin{equation}\label{LCoherenceForComodH}
\forall \,\varphi \in \mathscr{L}, \:\: \forall \, b \in B, \quad \mathfrak{d}_2(\varphi) \smallblacktriangleright b = \mathcal{R}\bigl(b_{[1]} \otimes \varphi_{(1)}\bigr) \, \mathcal{R}\bigl(\varphi_{(3)} \otimes b_{[2]}\bigr) \, b_{[0]} \smallblacktriangleleft \mathfrak{d}_1(\varphi_{(2)})
\end{equation}
which can also be written as $\mathcal{R}^{-1}\bigl( \varphi_{(2)} \otimes b_{[1]} \bigr) \mathfrak{d}_2(\varphi_{(1)}) \smallblacktriangleright b_{[0]} = \mathcal{R}\bigl(b_{[1]} \otimes \varphi_{(1)} \bigr) b_{[0]} \smallblacktriangleleft \mathfrak{d}_1(\varphi_{(2)})$.

\smallskip

\indent We are ready to describe the structure of the braided monoidal category $\mathrm{Bim}^{\mathscr{L}}_{\mathcal{C}}$ from Corollary \ref{cor:correspondenceLmodvshalfbraided} in the case $\mathcal{C} = \mathrm{Comod}\text{-}\OO$. This category is moreover balanced (\S\ref{sec:balanceBim}) if $\OO$ is coribbon (see Equation \eqref{eq:coribbon}).
\smallskip

\indent \textbullet ~ \textbf{Objects} are $\mathscr{L}$-linear algebras in  $\mathrm{Comod}\text{-}\OO$.

\smallskip

\indent \textbullet ~ \textbf{Morphisms} $\mathscr{A}_1 \to \mathscr{A}_2$ are $\mathscr{L}$-compatible $(\mathscr{A}_2, \mathscr{A}_1)$-bimodules in $\mathrm{Comod}\text{-}\OO$, considered up to isomorphism.

\smallskip

\indent \textbullet ~ \textbf{Composition of morphisms.} Recall that the composition of bimodules was defined in general in \eqref{defCompositionOfBimodules} as a coequalizer in $\mathcal{C}$. For $\mathcal{C} = \mathrm{Comod}\text{-}\OO$, coequalizers are usual quotients and we recover the tensor product of bimodules over the ``middle algebra''. More precisely, if $\mathbf{B}_1$ is a $(\mathscr{A}_2,\mathscr{A}_1)$-bimodule and  $\mathbf{B}_2$ is a $(\mathscr{A}_3,\mathscr{A}_2)$-bimodule in $\mathcal{C}$ then
\[ \mathbf{B}_2 \circ \mathbf{B}_1 = \mathbf{B}_2 \underset{\mathscr{A}_2}{\otimes} \mathbf{B}_1 = \frac{B_2 \otimes B_1}{\mathrm{span}_k\bigl\{ (w \smallblacktriangleleft a) \otimes v - w \otimes (a \smallblacktriangleright v) \, \big| \, v \in B_1, \: w \in B_2, \: a \in A_2 \bigr\}} \]
where $B_1$, $B_2$ and $A_2$ are the underlying $\OO$-comodules.

\smallskip

\indent \textbullet ~ \textbf{Monoidal product of objects.} Let $\mathscr{A}_1 = (A_1,\mathfrak{d}_1)$ and $\mathscr{A}_2 = (A_2,\mathfrak{d}_2)$ be $\mathscr{L}$-linear algebras. Then $\mathscr{A}_1 \,\widetilde{\otimes}\, \mathscr{A}_2$ is the comodule $A_1 \otimes A_2$ whose product and morphism $\mathfrak{d}_{1,2} : \mathscr{L} \to A_1 \otimes A_2$ are given by
\begin{align}
(x \otimes b)(a \otimes y) &= \mathcal{R}(b_{[1]} \otimes a_{[1]}) \,xa_{[0]} \otimes b_{[0]}y\label{brProdComodH}\\
\mathfrak{d}_{1,2}(\varphi) &= \mathfrak{d}_1(\varphi_{(1)}) \otimes \mathfrak{d}_2(\varphi_{(2)})\nonumber
\end{align}
for all $x \otimes b, a \otimes y \in A_1 \otimes A_2$ and $\varphi \in \mathscr{L}$.

\smallskip

\indent \textbullet ~ \textbf{Monoidal product of morphisms.} Let $\mathbf{B}_i$ be a $(\mathscr{A}'_i,\mathscr{A}_i)$-bimodule for $i=1,2$ with underlying comodules $B_i$, $A_i$, $A'_i$. Then the $(\mathscr{A}'_1 \,\widetilde{\otimes}\, \mathscr{A}'_2, \mathscr{A}_1 \,\widetilde{\otimes}\, \mathscr{A}_2)$-bimodule $\mathbf{B}_1 \,\widetilde{\otimes}\, \mathbf{B}_2$ is the comodule $B_1 \otimes B_2$ endowed with the actions
\begin{align*}
(a' \otimes b') \smallblacktriangleright (v \otimes w) &= \mathcal{R}(b'_{[1]} \otimes v_{[1]}) \, (a' \smallblacktriangleright v_{[0]}) \otimes (b'_{[0]} \smallblacktriangleright w),\\
(v \otimes w) \smallblacktriangleleft (a \otimes b) &= \mathcal{R}(w_{[1]} \otimes a_{[1]}) \, (v \smallblacktriangleleft a_{[0]}) \otimes (w_{[0]} \smallblacktriangleleft b)
\end{align*}
for all $a \otimes b \in A_1 \otimes A_2$, $a' \otimes b' \in A'_1 \otimes A'_2$ and $v \otimes w \in B_1 \otimes B_2$.

\smallskip

\indent \textbullet ~ \textbf{Braiding.} Let $\mathscr{A}_1 = (A_1,\mathfrak{d}_1)$, $\mathscr{A}_2 = (A_2,\mathfrak{d}_2)$ be $\mathscr{L}$-linear algebras in $\mathrm{Comod}\text{-}\OO$. Thanks to \eqref{HBfromLlinComodH}, we see that the isomorphism $T_{\mathscr{A}_1,\mathscr{A}_2}$ defined in \eqref{isoBrZCforHBAlg} for general half-braided algebras is

\begin{equation}\label{eq:twist} \foncIso{T_{\mathscr{A}_1,\mathscr{A}_2}}{\mathscr{A}_1 \,\widetilde{\otimes}\, \mathscr{A}_2}{\mathscr{A}_2 \,\widetilde{\otimes}\, \mathscr{A}_1}{a \otimes b}{\mathcal{R}^{-1}(b_{[2]} \otimes a_{[1]}) \, b_{[0]} \otimes \mathfrak{d}_1(b_{[1]}) a_{[0]}} 
\end{equation}
Other expressions are
\[ T_{\mathscr{A}_1,\mathscr{A}_2}(a \otimes b) \overset{\eqref{LlinearAxiomComodH}}{=} \mathcal{R}\bigl(a_{[1]} \otimes b_{[1]} \bigr) b_{[0]} \otimes a_{[0]} \mathfrak{d}_1(b_{[2]}) \overset{\eqref{brProdComodH}}{=} (1 \otimes a) \bigl( b_{[0]} \otimes \mathfrak{d}_1(b_{[1]})  \bigr). \]
According to \S\ref{sub:braiding}, the braiding $\mathcal{B}_{\mathscr{A}_1, \mathscr{A}_2} \in \mathrm{Hom}_{\mathrm{Bim}^{\mathscr{L}}_{\mathcal{C}}}\bigl( \mathscr{A}_1 \,\widetilde{\otimes}\, \mathscr{A}_2, \mathscr{A}_2 \,\widetilde{\otimes}\, \mathscr{A}_1 \bigr)$ is $A_2 \otimes A_1$ as a $\OO$-comodule. The left action of $b \otimes a \in \mathscr{A}_2 \,\widetilde{\otimes}\, \mathscr{A}_1$ on $y \otimes x \in A_2 \otimes A_1$ is the multiplication \eqref{brProdComodH} in $\mathscr{A}_2 \,\widetilde{\otimes}\, \mathscr{A}_1$:
\[ (b \otimes a) \smallblacktriangleright (y \otimes x) = (b \otimes a)(y \otimes x). \]
The right action of $a \otimes b \in \mathscr{A}_1 \,\widetilde{\otimes}\, \mathscr{A}_2$ on $y \otimes x \in A_2 \otimes A_1$ is
\[ (x \otimes y) \smallblacktriangleleft (a \otimes b) = (x \otimes y)\,T_{\mathscr{A}_1,\mathscr{A}_2}(a \otimes b) \]
where again we use the multiplication \eqref{brProdComodH} in $\mathscr{A}_2 \,\widetilde{\otimes}\, \mathscr{A}_1$. Note that the vector $1_{\mathscr{A}_2} \otimes 1_{\mathscr{A}_1}$ is free for the left action, and thus the bimodule structure is fully summarized by the formula
\begin{equation}\label{eq:braidingtwist} (1_{\mathscr{A}_2} \otimes 1_{\mathscr{A}_1}) \smallblacktriangleleft (a \otimes b) = \mathcal{R}\bigl(a_{[1]} \otimes b_{[1]} \bigr) \, \bigl( b_{[0]} \otimes a_{[0]} \mathfrak{d}_1(b_{[2]}) \bigr) \smallblacktriangleright (1_{\mathscr{A}_2} \otimes 1_{\mathscr{A}_1}). 
\end{equation}

\smallskip

\indent \textbullet ~ \textbf{Balance.} Assume that $\OO$ is moreover {\em coribbon}, \textit{i.e.} there is a linear form $\mathsf{v} : \OO \to k$ which is convolution-invertible\footnote{Meaning that there exists $\mathsf{v}^{-1} : \OO \to k$ such that $\mathsf{v}(\varphi_{(1)})\,\mathsf{v}^{-1}(\varphi_{(2)}) = \mathsf{v}(\varphi_{(2)})\,\mathsf{v}^{-1}(\varphi_{(1)}) = \varepsilon_{\OO}(\varphi)$ for all $\varphi \in \OO$.} and satisfies
\begin{equation}\label{eq:coribbon}
 \begin{array}{c}
\mathsf{v}(\varphi_{(1)})\varphi_{(2)} = \varphi_{(1)}\mathsf{v}(\varphi_{(2)}), \qquad \mathsf{v} \circ S = \mathsf{v},\\[.5em]
\mathsf{v}(\varphi\psi) = \mathcal{R}^{-1}\bigl( \varphi_{(3)} \otimes \psi_{(3)} \bigr) \mathcal{R}^{-1}\bigl( \psi_{(2)} \otimes \varphi_{(2)} \bigr) \mathsf{v}(\varphi_{(1)}) \, \mathsf{v}(\psi_{(1)})
\end{array} 
\end{equation}
for all $\varphi,\psi \in \OO$. Then $\mathrm{Comod}\text{-}\OO$ is cp-ribbon in the sense of \S\ref{sec:balanceBim}, for the twist $\theta$ given by
\[ \forall \, X \in \mathrm{Comod}\text{-}\OO, \:\: \forall \, x \in X, \quad \theta_X(x) = x_ {[0]} \, \mathsf{v}^{-1}(x_{[1]}). \]
Let $\mathscr{A} = (A, \mathfrak{d})$ be a $\mathscr{L}$-linear algebra. We easily compute the automorphism $\mathrm{bal}_{\mathscr{A}} : \mathscr{A} \to \mathscr{A}$ from \eqref{defBalOnHBAlg} thanks to the canonical half-braiding on $\mathscr{A}$ given in \eqref{HBfromLlinComodH}:
\begin{align*}
\mathrm{bal}_{\mathscr{A}}(a) &= m \circ t_A\bigl( 1_A \otimes a_{[0]} \bigr) \, \mathsf{v}^{-1}(a_{[1]})\\
 &= \mathcal{R}^{-1}\bigl( a_{[2]} \otimes (1_A)_{[1]} \bigr) \, a_{[0]} \mathfrak{d}(a_{[1]}) \, (1_A)_{[0]} \, \mathsf{v}^{-1}(a_{[3]}) = a_{[0]} \mathfrak{d}(a_{[1]}) \, \mathsf{v}^{-1}(a_{[2]})
\end{align*}
for all $a \in \mathscr{A}$. As a result, the $(\mathscr{A}, \mathscr{A})$-bimodule $\mathrm{BAL}_{\mathscr{A}}$ defined in general in \eqref{defBALbimod} is the $\OO$-comodule $A$ endowed with the following actions:
\[ \forall \, a \in \mathscr{A}, \:\: \forall \, x \in A, \quad a \smallblacktriangleright x = ax, \quad x \smallblacktriangleleft a = xa_{[0]} \mathfrak{d}(a_{[1]}) \, \mathsf{v}^{-1}(a_{[2]}). \]
Note that the vector $1_{\mathscr{A}}$ is free for the left action, and thus the bimodule structure is fully summarized by the formula
\begin{equation}\label{eq:balanceBimodHComod}
1_{\mathscr{A}} \smallblacktriangleleft a = \mathsf{v}^{-1}(a_{[2]})\, a_{[0]} \mathfrak{d}(a_{[1]}) \smallblacktriangleright 1_{\mathscr{A}}.
\end{equation}

\subsection{Right comodules vs. left modules}\label{sub:ComodVsMod}
We have seen in \S\ref{subsecComodH} that the category $\mathcal{C} = \mathrm{Comod}\text{-}\OO$ is well-behaved: it is LFP, compact objects are finite-dimensional comodules and hence it is cp-rigid. In contrast, the category of left modules over a Hopf algebra does not have such good categorical properties. This is why we have decided to work with comodules instead of modules. However the reader is probably more familiar with modules. In this short subsection we clarify the relations between the two settings, assuming that $\OO$ is dual to a given Hopf algebra.

\indent So, let $\mathcal{U}$ be a Hopf algebra over the field $k$. For the Hopf algebra $\mathcal{O}$, we take here the {\em restricted dual} (a.k.a. {\em finite dual}) of $\UU$, which can be defined as the subspace of $\mathcal{U}^*$ generated by the matrix coefficients of finite-dimensional $\mathcal{U}$-modules. It is a Hopf algebra whose structure morphisms are dual to those of $\mathcal{U}$ \cite[\S 9.1]{Mon}:
\begin{equation}\label{usualProductCoproductOnRestrictedDual}
\forall \, \varphi, \psi \in \OO, \:\: \forall \, x,y \in \mathcal{U}, \quad (\varphi \psi)(x) = \varphi(x_{(1)})\,\psi(x_{(2)}), \quad \langle \Delta_{\OO}(\varphi), x \otimes y \rangle = \varphi(xy).
\end{equation}
If $\mathcal{U} = U_q(\mathfrak{g})$ one usually takes for $\mathcal{O}$ the subspace spanned by matrix coefficients of type $1$ modules (instead of all finite-dimensional modules), which is denoted by $\mathcal{O}_q(G)$ and called quantized algebra of functions on $G$.

\smallskip

\indent Every right $\OO$-comodule $V$ is automatically a left $\UU$-module as follows:
\begin{equation}\label{moduleFromComodule}
\forall \, h \in \UU, \:\: \forall \, v \in V, \quad h \cdot v = v_{[0]} \, v_{[1]}(h).
\end{equation}
Recall that a $\UU$-module is called {\em locally finite} if the $\UU$-orbit of each vector is finite-dimensional. We denote by $\UU\text{-}\mathrm{Mod}^{\mathrm{lf}}$ the full subcategory of such modules; it is actually a monoidal subcategory.
\begin{lemma}\label{corrComodMod}
The construction \eqref{moduleFromComodule} defines an isomorphism of monoidal categories
\[ F : \mathrm{Comod}\text{-}\mathcal{O} \,\overset{\sim}{\longrightarrow}\, \mathcal{U}\text{-}\mathrm{Mod}^{\mathrm{lf}}. \]
In particular it restricts to an isomorphism $\mathrm{comod}\text{-}\mathcal{O} \,\overset{\sim}{\longrightarrow}\, \mathcal{U}\text{-}\mathrm{mod}$ between full subcategories of finite-dimensional objects.
\end{lemma}
\begin{proof}
First it is readily seen that a $k$-linear map which is a morphism of $\OO$-comodules is a morphism of $\UU$-modules for \eqref{moduleFromComodule}. Hence we can define $F(f) = f$ on morphisms and this gives a functor $\mathrm{Comod}\text{-}\mathcal{O} \to \mathcal{U}\text{-}\mathrm{Mod}$. Because of the fundamental theorem on comodules \cite[Th.\,5.1.1]{Mon}, this functor actually takes values in $\mathcal{U}\text{-}\mathrm{Mod}^{\mathrm{lf}}$ and we denote it by $F$. It is readily seen to be strict monoidal. We have to construct its inverse $\overline{F}$.  So let $V$ be a locally finite $\UU$-module with basis $(v_i)$ and dual basis $(v^i)$. For any $v \in V$ define
\[ \textstyle \delta(v) = \sum_i v_i \otimes v^i(? \cdot v) \]
where $v^i(? \cdot v)$ is the linear form $\UU \to k$ defined by $h \mapsto v^i(h \cdot v)$. The sum is finite because of local finiteness. Moreover $v^i(? \cdot v) \in \OO$ by definition (it is a matrix coefficient of the finite-dimensional module $\UU\cdot v$). Finally, $\delta(v)$ is independent of the choice of the basis (repeated index in covariant and contravariant positions). Hence we get a $k$-linear map $\delta : V \to V \otimes \OO$. It is easy to check that $\delta$ is a coaction, using that the coproduct in $\mathcal{O}$ is given by $f(?\cdot v) \mapsto \sum_i f(? \cdot v_i) \otimes v^i(? \cdot v)$; note again that the sum is finite by local finiteness. We have also to check that $\overline{F}$ is a functor, \textit{i.e.} a $k$-linear map $f : V \to V'$ which is a morphism of $\UU$-modules is a morphism of $\OO$-comodules $(V,\delta) \to (V',\delta')$; but this is because $f(v_i) \otimes v^i(? \cdot v) = v_i \otimes (v^i \circ f)(?\cdot v) = v_i \otimes v^i\bigl(? \cdot f(v)\bigr)$ by $\UU$-linearity. Finally, it is readily seen that $F$ and $\overline{F}$ are inverse each other. The last claim in the Lemma is obvious.
\end{proof}

\noindent In particular, when $\mathcal{U}$ is finite dimensional then every module is locally finite and hence
\[ \dim(\mathcal{U}) < \infty \:\: \implies \:\: \quad \mathrm{Comod}\text{-}\mathcal{O} \,\cong\, \mathcal{U}\text{-}\mathrm{Mod}. \]

\indent If $\UU$ is quasitriangular (resp. ribbon), with $R$-matrix $R \in \UU \otimes \UU$ (resp. ribbon element $\nu \in \UU$), then $\OO$ is coquasitriangular (resp. coribbon) with the co-$R$-matrix $\mathcal{R} : \OO \otimes \OO \to k$ (resp. coribbon element $\mathsf{v} : \OO \to k$) given by
\begin{equation}\label{coRmatrixFromRmatrix}
\forall \, \varphi, \psi \in \OO, \quad \mathcal{R}(\varphi \otimes \psi) = (\varphi \otimes \psi)(R) \quad \bigl(\text{resp.}\:\: \mathsf{v}(\varphi) = \varphi(\nu)\bigr).
\end{equation}
With these choices of coquasitriangular (resp. corribon) structure, the isomorphism in Lemma \ref{corrComodMod} becomes a braided (resp. ribbon) functor. Indeed, recall that the braiding and twist in $\UU\text{-}\mathrm{Mod}$ are defined by $c_{X,Y}(x \otimes y) = R''_i \cdot y \otimes R'_i \cdot x$ and $\theta_X(x) = \nu^{-1} \cdot x$, where we write $R = R'_i \otimes R''_i$ with implicit summation on $i$.

\smallskip

\indent Assume that $(\UU,R)$ is quasitriangular with  $R = R'_i \otimes R''_i$. Introduce the left and right coregular actions $\triangleright$, $\triangleleft$ of $\UU$ on $\OO$ defined by
\begin{equation}\label{coregActions}
\forall \, h,x \in \UU, \:\: \forall \, \varphi \in \OO, \quad (h \triangleright \varphi)(x) = \varphi(xh), \quad (\varphi \triangleleft h)(x) = \varphi(hx).
\end{equation}
Recall from Proposition \ref{descriptionCoendComodH} that $\mathscr{L} = \int^{X \in \mathrm{comod}\text{-}\OO} X^* \otimes X$ is $\OO$ as a vector space. Through formula \eqref{moduleFromComodule}, the $\UU$-module structure on $\mathscr{L}$ obtained from the $\OO$-comodule structure \eqref{defAdjointCoaction} is the coadjoint action:
\begin{equation}\label{coadActionOnCoend}
\forall \, h \in \UU, \:\: \forall \, \varphi \in \mathscr{L}, \quad \mathrm{coad}(h)(\varphi) = h_{(2)} \triangleright \varphi \triangleleft S(h_{(1)}). 
\end{equation}
If the co-$R$-matrix $\mathcal{R}$ on $\OO$ is dual to $R$ as in \eqref{coRmatrixFromRmatrix}, the product in $\mathscr{L}$ can be written as
\begin{equation}\label{productCoendMod}
\forall \, \varphi, \psi \in \mathscr{L}, \quad \varphi \odot \psi = \bigl(R'_i \triangleright \varphi \triangleleft R'_j\bigr) \, \bigl(\psi \triangleleft S(R''_i)R''_j\bigr)
\end{equation}
and the coproduct of $\mathscr{L}$ is simply $\Delta_{\mathscr{L}} = \Delta_{\OO}$. By Lemma \ref{corrComodMod} any $V \in \UU\text{-}\mathrm{mod}$ (finite-dimensional) is the same thing as an $\OO$-comodule, and through this identification the universal dinatural transformation $i$ of $\mathscr{L}$ is given by matrix coefficients: 
\[ i_V : V^* \otimes V \to \mathscr{L}, \quad f \otimes v \mapsto f(? \cdot v) \in \OO. \]
In this way $\mathscr{L}$ is identified with the more familiar coend $\int^{X \in \UU\text{-}\mathrm{mod}} X^* \otimes X \in \UU\text{-}\mathrm{Mod}^{\mathrm{lf}}$ of the category of finite-dimensional $\UU$-modules, as appearing e.g. in \cite[\S 3.3]{lyuMCG}.

\subsection{Module-algebras and quantum moment maps}\label{sub:LlinQMM}
Here we explain the relation between $\mathscr{L}$-linear algebras in $\mathrm{Comod}\text{-}\OO$ (see \eqref{axiomComodAlg}--\eqref{LlinearAxiomComodH}) and module-algebras endowed with a so-called quantum moment map. The notion of $\mathscr{L}$-compatible bimodule is also rephrased in terms of quantum moment maps.

\smallskip

\indent As in the previous section, we assume that $\mathcal{O}$ arises as the restricted dual of a quasitriangular Hopf algebra $(\UU,R)$. Recall that a {\em $\UU$-module-algebra} is a left $\UU$-module $A$ endowed with an associative product with unit $1_A$ such that
\[ \forall \, h \in \UU, \:\: \forall \, a,b \in A, \qquad h\cdot(ab) = (h_{(1)}\cdot a)(h_{(2)} \cdot b) \quad \text{and}\quad h \cdot 1_A = \varepsilon_{\UU}(h)1_A. \]
Through the correspondence in Lemma \ref{corrComodMod}, $\OO$-comodule-algebras are equivalent to $\UU$-module-algebras which are in $\UU\text{-}\mathrm{Mod}^{\mathrm{lf}}$. In particular, when $\UU$ is finite-dimensional, $\OO$-comodule-algebras and $\UU$-module-algebras are the same thing.

\smallskip

\indent Item 1 in the following definition is taken from \cite[Def.\,1.2]{Lu} and \cite[\S 1.5]{VV}.
\begin{definition} 1. If $A$ is an $\UU$-module-algebra, a quantum moment map (QMM) for $A$ is an algebra morphism $\mu : \UU \to A$ such that
\begin{equation}\label{axiomeQMM}
\forall \, h \in \UU, \:\: \forall \,a \in A, \quad \mu(h)a = (h_{(1)} \cdot a)\mu(h_{(2)}).
\end{equation}
2. Let $(A_1, \mu_1)$, $(A_2,\mu_2)$ be $\UU$-module-algebras endowed with QMMs $\mu_1$, $\mu_2$ and $\mathbf{B} = (B, \smallblacktriangleright, \smallblacktriangleleft)$ be a $(A_2, A_1)$-bimodule in $\UU\text{-}\mathrm{Mod}$. We say that $\mathbf{B}$ is QMM-compatible if 
\begin{equation}\label{QMMcoherence}
\forall \, h \in \UU, \:\: \forall \,b \in B, \quad \mu_2(h) \smallblacktriangleright b = (h_{(1)} \cdot b)\smallblacktriangleleft \mu_1(h_{(2)}).
\end{equation}
\end{definition}
\noindent The condition \eqref{QMMcoherence} can equivalently be expressed as
\begin{equation}\label{axiomeQMM2}
\forall \, h \in \UU, \:\: \forall \,b \in B, \quad h \cdot b = \mu_2(h_{(1)}) \smallblacktriangleright b \smallblacktriangleleft \mu_1\bigl( S(h_{(2)}) \bigr)
\end{equation}
which means that the $\UU$-module structure on $B$ is entirely determined by $\mu_1, \mu_2$ and the $(A_2,A_1)$-bimodule structure. This applies in particular to $A$, viewed as a bimodule over itself. It follows that $\mu$ is $\UU$-linear when $\UU$ is endowed with the adjoint action:
\[ \forall \, h,x \in \UU, \quad \mu\bigl( h_{(1)}x S(h_{(2)}) \bigr) = h \cdot \mu(x). \]

\smallskip

\indent Recall the description \eqref{coadActionOnCoend}--\eqref{productCoendMod} of the coend $\mathscr{L}$ when it is viewed in $\UU\text{-}\mathrm{Mod}$. The notion of {\em $\mathscr{L}$-linear algebra} makes sense in $\UU\text{-}\mathrm{mod}$: it is a pair $(A,\mathfrak{d})$ where $A$ is a $\UU$-module-algebra and $\mathfrak{d} : \mathscr{L} \to A$ is a morphism of $\UU$-module-algebras such that
\begin{equation}\label{QMMGJSforComodAlg}
\forall \, \varphi \in \mathscr{L}, \:\: \forall \, a \in A, \quad \mathfrak{d}(\varphi)\,a = (R'_iR''_j \cdot a)\,\mathfrak{d}\bigl( R'_j \triangleright \varphi \triangleleft R''_i \bigr)
\end{equation}
with the coregular actions $\triangleright, \triangleleft$ of $\UU$ on $\OO$ as defined in \eqref{coregActions}. This last condition is just obtained from \eqref{LlinearAxiomComodH} thanks to \eqref{coRmatrixFromRmatrix}.
In particular the $\UU$-linearity of $\mathfrak{d}$ means that
\begin{equation}\label{dIsULinear}
\forall \, h \in \UU, \:\: \forall \, \varphi \in \mathscr{L}, \quad h \cdot \mathfrak{d}(\varphi) = \mathfrak{d}\bigl( h_{(2)} \triangleright \varphi \triangleleft S(h_{(1)}) \bigr)
\end{equation}
by definition of the coadjoint action of $\UU$ on $\mathscr{L}$ in \eqref{coadActionOnCoend}.

\indent If $(A_1,\mathfrak{d}_1)$ and $(A_2,\mathfrak{d}_2)$ are $\mathscr{L}$-linear algebras in $\UU\text{-}\mathrm{Mod}$ and $\mathbf{B} = (B, \smallblacktriangleright, \smallblacktriangleleft)$ is an $(A_2,A_1)$-bimodule in $\UU\text{-}\mathrm{Mod}$, the notion of {\em $\mathscr{L}$-compatibility} makes sense in $\UU\text{-}\mathrm{Mod}$. Using \eqref{LCoherenceForComodH} and \eqref{coRmatrixFromRmatrix}, it reads
\begin{equation}\label{LcoherenceUmod}
\forall \, \varphi \in \mathscr{L}, \:\: \forall \, b \in B, \quad \mathfrak{d}_2(\varphi) \smallblacktriangleright b = (R'_iR''_j \cdot b)\smallblacktriangleleft \mathfrak{d}_1\bigl( R'_j \triangleright \varphi \triangleleft R''_i \bigr).
\end{equation}

\indent Given an $\UU$-module-algebra $A$, our goal is to relate quantum moment maps for $A$ with $\mathscr{L}$-linear structures on $A$ and $\mathscr{L}$-compatible bimodules with QMM-compatible bimodules. This will use the following map, known as {\em Drinfeld map}:
\begin{equation}\label{DrinfeldMap}
\Phi : \mathscr{L} \to \UU, \quad \varphi \mapsto (\varphi \otimes \mathrm{id}_{\UU})(R^{fl}R)
\end{equation}
where $R^{fl} = R''_i \otimes R'_i$ is the flip of $R$. The $R$-matrix property $R\Delta = \Delta^{\mathrm{op}}R$ easily implies that $\Phi$ is $\UU$-linear when $\mathscr{L}$ is endowed with the coadjoint action \eqref{coadActionOnCoend} and $\UU$ is endowed with the adjoint action:
\begin{equation}\label{equivarianceDrinfeld}
\forall \, h \in \UU, \:\: \forall \, \varphi \in \mathscr{L}, \quad h_{(1)}\Phi(\varphi)S(h_{(2)}) = \Phi\bigl( h_{(2)} \triangleright \varphi \triangleleft S(h_{(1)}) \bigr).
\end{equation}
A computation using the Yang--Baxter equation for $R$ reveals that the linear map $\Phi$ is a morphism of algebras. If moreover we define
\[ \widetilde{\Delta} : \UU \to \UU \otimes \UU, \quad h \mapsto  h_{(1)}R''_iS(R''_j) \otimes R'_j h_{(2)} R'_i \]
then another straightforward computation based on the Yang--Baxter equation proves that
\begin{equation}\label{PhiMorphismBialgebras}
(\Phi \otimes \Phi) \circ \Delta_{\mathscr{L}} = \widetilde{\Delta} \circ \Phi.
\end{equation}

\begin{remark}
The structures $(\mathscr{L},m_{\mathscr{L}}, \Delta_{\mathscr{L}})$ and $(\UU,m_{\UU},\widetilde{\Delta})$ are not bialgebras in the usual sense (\textit{i.e.} in Vect) but in $\mathcal{C} = \UU\text{-}\mathrm{Mod}$. It means that $\Delta_{\mathscr{L}} : \mathscr{L} \to \mathscr{L} \otimes \mathscr{L}$ and $\widetilde{\Delta} : \UU \to \UU \otimes \UU$ are coassociative but are morphisms of algebras when the target is endowed with the braided product structure \eqref{defBraidedTensorProductOfAlgebras} in $\UU\text{-}\mathrm{Mod}$. The map $\Phi$ is then a morphism of bialgebras in $\UU\text{-}\mathrm{Mod}$.
\end{remark}

\indent We say that the Hopf algebra $\UU$ is {\em factorizable} if the Drinfeld map $\Phi$ is an isomorphism. 
\begin{proposition}\label{prop:momentmap}
Let $A_1,A_2$ be a $\UU$-module-algebras and $\mathbf{B} = (B, \smallblacktriangleright, \smallblacktriangleleft)$ be an $(A_2,A_1)$-bimodule in  $\UU\text{-}\mathrm{Mod}$.
\\ 1. Assume that we have QMMs $\mu_i : \UU \to A_i$ for $i=1,2$ and define $\mathfrak{d}_i = \mu_i \circ \Phi$. Then $\mathfrak{d}_i$ gives $A_i$ a structure of $\mathscr{L}$-linear algebra. Moreover, if $\mathbf{B}$ is QMM-compatible then it is $\mathscr{L}$-compatible for these structures.
\\2. Conversely assume that $\UU$ is factorizable and that we have $\mathscr{L}$-linear structures $\mathfrak{d}_i : \mathscr{L} \to A_i$ for $i=1,2$. Then $\mu_i = \mathfrak{d}_i \circ \Phi^{-1}$ is a QMM for $A$. Moreover, if $\mathbf{B}$ is $\mathscr{L}$-compatible then it is QMM-compatible for these QMMs.
\end{proposition}
\begin{proof}
1. First $\mathfrak{d}_i$ is a morphism of $\UU$-module-algebras because so are $\Phi : (\mathscr{L},\mathrm{coad}) \to (\UU,\mathrm{ad})$ and $\mu_i : (\UU,\mathrm{ad}) \to (A,\cdot)$. Let us directly prove the $\mathscr{L}$-compatibility condition for bimodules \eqref{LcoherenceUmod}, as the $\mathscr{L}$-linearity condition for algebras \eqref{QMMGJSforComodAlg} is a particular case of it (considering an algebra as a bimodule over itself). For any $\varphi \in \mathscr{L}$, note that \eqref{PhiMorphismBialgebras} can be rewritten as
\begin{equation}\label{MorphDrinfeldBialgebraMorphism}
\Phi(\varphi)_{(1)} \otimes \Phi(\varphi)_{(2)} = \Phi(\varphi_{(1)})R''_iR''_j \otimes R'_i \Phi(\varphi_{(2)}) S(R'_j)
\end{equation}
where $\Phi(\varphi)_{(1)} \otimes \Phi(\varphi)_{(2)} = \Delta_{\UU}\bigl( \Phi(\varphi) \bigr)$, $\varphi_{(1)} \otimes \varphi_{(2)} = \Delta_{\mathscr{L}}(\varphi) = \Delta_{\OO}(\varphi)$ and we use that $(S \otimes \mathrm{id})(R) = (\mathrm{id} \otimes S^{-1})(R) = R^{-1}$. Now for all $b \in B$, we compute
\begin{align*}
&\mathfrak{d}_2(\varphi) \smallblacktriangleright b = \mu_2\bigl( \Phi(\varphi) \bigr) \smallblacktriangleright b \overset{\eqref{QMMcoherence}}{=} \bigl( \Phi(\varphi)_{(1)} \cdot b \bigr) \smallblacktriangleleft \mu_1\bigl( \Phi(\varphi)_{(2)} \bigr) \\
\overset{\eqref{MorphDrinfeldBialgebraMorphism}}{=}\:& \bigl( \Phi(\varphi_{(1)})R''_iR''_j \cdot b \bigr) \smallblacktriangleleft \mu_1\bigl( R'_i \Phi(\varphi_{(2)}) S(R'_j) \bigr)
= \bigl( \Phi(\varphi_{(1)})R''_i \cdot b \bigr) \smallblacktriangleleft \mu_1\bigl( R'_{i\,(1)} \Phi(\varphi_{(2)}) S(R'_{i\,(2)}) \bigr)\\
\overset{\eqref{equivarianceDrinfeld}}{=}\:\:&\bigl( \Phi(\varphi_{(1)})R''_i \cdot b \bigr) \smallblacktriangleleft \mathfrak{d}_1\bigl( R'_{i\,(2)} \triangleright \varphi_{(2)} \triangleleft S(R'_{i\,(1)}) \bigr) = \bigl( \Phi(\varphi_{(1)})R''_iR''_j \cdot b \bigr) \smallblacktriangleleft \mathfrak{d}_1\bigl( R'_j \triangleright \varphi_{(2)} \triangleleft S(R'_i) \bigr).
\end{align*}
In the unlabelled equalities we used that $(\Delta \otimes \mathrm{id})(R) = R_{13}R_{23}$. Now note by definition of $\Phi$ in \eqref{DrinfeldMap} and of the coproduct $\Delta_{\mathscr{L}}(\varphi) = \Delta_{\OO}(\varphi) = \varphi_{(1)} \otimes \varphi_{(2)}$ in \eqref{usualProductCoproductOnRestrictedDual} that
\[ \Phi(\varphi_{(1)}) \otimes \varphi_{(2)} = \varphi_{(1)}\bigl( R''_k R'_l \bigr) \,R'_k R''_l \otimes \varphi_{(2)} = R'_k R''_l \otimes \bigl( \varphi \triangleleft R''_k R'_l \bigr). \]
Applying this equality in the result of the above computation we get
\begin{align*}
&\mathfrak{d}_2(\varphi) \smallblacktriangleright b = \bigl( R'_k R''_lR''_iR''_j \cdot b \bigr) \smallblacktriangleleft \mathfrak{d}_1\bigl( R'_j \triangleright \varphi \triangleleft R''_k R'_lS(R'_i) \bigr) = \bigl( R'_k R''_j \cdot b \bigr) \smallblacktriangleleft \mathfrak{d}_1\bigl( R'_j \triangleright \varphi \triangleleft R''_k  \bigr)
\end{align*}
because $(S \otimes \mathrm{id})(R) = R^{-1}$. This is the desired property.
\\2. Since $\Phi$ and $\mathfrak{d}_i$ are algebra morphisms, so is $\mu_i$. We directly prove the QMM-compatibility condition for bimodules \eqref{QMMcoherence}, as the QMM axiom for algebras \eqref{axiomeQMM} is a particular case of it (considering an algebra as a bimodule over itself). For any $\varphi \in \mathscr{L}$ and $b \in B$ we have
\begin{align*}
&\mu_2\bigl( \Phi(\varphi) \bigr) \smallblacktriangleright b = \mathfrak{d}_2(\varphi) \smallblacktriangleright b \overset{\eqref{LcoherenceUmod}}{=} (R'_iR''_j \cdot b) \smallblacktriangleleft \mathfrak{d}_1\bigl( R'_j \triangleright \varphi \triangleleft R''_i \bigr) = (R'_iR''_j \cdot b) \smallblacktriangleleft \mu_1\bigl( \Phi(R'_j \triangleright \varphi \triangleleft R''_i) \bigr)\\
\overset{\eqref{DrinfeldMap}}{=}\:& (R'_iR''_j \cdot b) \smallblacktriangleleft \mu_1(R'_k R''_l) \, \varphi\bigl( R''_iR''_kR'_lR'_j \bigr) = \bigl(R'_{i\,(1)} R''_{j\,(1)} \cdot b\bigr) \smallblacktriangleleft \mu_1\bigl(R'_{i\,(2)} R''_{j\,(2)}\bigr) \, \varphi\bigl( R''_iR'_j \bigr)\\
\overset{\eqref{DrinfeldMap}}{=}\:&\bigl( \Phi(\varphi)_{(1)} \cdot b \bigr) \smallblacktriangleleft \mu_1\bigl( \Phi(\varphi)_{(2)} \bigr).
\end{align*}
For the last unlabelled equality we used $(\Delta \otimes \mathrm{id})(R) = R_{13}R_{23}$ and $(\mathrm{id} \otimes \Delta)(R) = R_{13}R_{12}$. Since $\Phi$ is an isomorphism, any $h \in \UU$ can be written as $h = \Phi(\varphi)$ and thus the QMM-compatibility condition \eqref{QMMcoherence} holds.
\end{proof}

To finish this section, we discuss an example of $\UU$-module-algebras which will be important in \S\ref{subsec:StSkVsKL} below, namely internal End algebras (\S\ref{sub:end}). For finite-dimensional $\UU$-modules $V,W$, the internal Hom space $\underline{\Hom}(V,W) = W \otimes V^*$ is identified to the space $\Hom_k(V,W)$ of all $k$-linear maps $f : V \to W$, endowed with the $\UU$-action defined by
\begin{equation}\label{UactionOnIntHom}
\forall \, h \in \UU, \:\: \forall \, v \in V, \quad (h \cdot f)(v) = h_{(1)} \cdot f\bigl( S(h_{(2)}) \cdot v \bigr).
\end{equation}
Let us compute the algebra morphism $\mathfrak{d} : \mathscr{L} \to V \otimes V^*$ defined in Example \ref{LlinearStructEndV} for the present choice $\mathcal{C} = \mathrm{Comod}\text{-}\OO = \UU\text{-}\mathrm{Mod}^{\mathrm{lf}}$. Let $(\OO,\triangleright)$ be the $\UU$-module given by coregular action; note that it is in $\UU\text{-}\mathrm{Mod}^{\mathrm{lf}}$ because the subspace of matrix coefficients of a given module is stable under $\triangleright$. For any $\varphi \in \OO$, let $F_{\varphi} = \UU \triangleright \varphi$ be the submodule generated by $\varphi$ and note that $i_{F_{\varphi}}\bigl( \langle -,1_{\UU} \rangle \otimes \varphi \bigr)(h) = \langle h \triangleright \varphi, 1_{\UU} \rangle = \varphi(h)$ for all $h \in \UU$. Hence, using implicit summation,
\begin{align*}
\mathfrak{d}(\varphi) &= \mathfrak{d}\bigl( i_{F_\varphi}\bigl( \langle -,1_{\UU} \rangle \otimes \varphi\bigr) \bigr)\\
&= \bigl( \mathrm{ev}_{F_\varphi} \otimes \mathrm{id}_{V \otimes V^*} \bigr) \circ \bigl( \mathrm{id}_{F_\varphi^*} \otimes (c_{V,F_\varphi} \circ c_{F_\varphi,V}) \otimes \mathrm{id}_{V^*} \bigr)(\langle -,1_{\UU} \rangle \otimes \varphi \otimes v_i \otimes v^i)\\
&= \bigl\langle R''_lR'_j \triangleright \varphi, 1_{\UU} \bigr\rangle \, R'_lR''_j \cdot v_i \otimes v^i = \varphi\bigl( R''_lR'_j \bigr) \, R'_lR''_j \cdot v_i \otimes v^i = \Phi(\varphi)\cdot v_i \otimes v^i.
\end{align*}
where $(v_i)$ and $(v^i)$ are dual bases for $V$. Thus through the identification $V \otimes V^* \cong \End_k(V)$, we get the algebra morphism $\mathfrak{d} : \mathscr{L} \to \End_k(V)$ defined by
\begin{equation}\label{LlinearStructIntEnd}
\forall \, \varphi \in \mathscr{L}, \:\: \forall \, v \in V, \quad \mathfrak{d}(\varphi)(v) = \Phi(\varphi) \cdot v.
\end{equation}
In this way $\End_k(V)$ becomes a $\mathscr{L}$-linear algebra in $\UU\text{-}\mathrm{mod}$. If $\UU$ is factorizable, the QMM associated to $\mathfrak{d}$ (item 2 in Prop.\,\ref{prop:momentmap}) $\mu_V : \UU \to \underline{\End}(V)$ is simply the representation morphism:
\begin{equation}\label{QMMonIntEnd}
\forall \, h \in \UU, \:\: \forall \, v \in V, \quad \mu_V(h)(v) = h \cdot v.
\end{equation}
Although it could be deduced from Lemma \ref{lemmaMatrixAlgebra}(2) and Prop.\,\ref{prop:momentmap}(2) that the $\bigl( \underline{\End}(W), \underline{\End}(V)\bigr)$-bimodule $\underline{\Hom}(V,W)$ is QMM-compatible, this fact is actually obvious from the definitions. Indeed, the definition of the $\UU$-action on $\underline{\Hom}(V,W)$ in \eqref{UactionOnIntHom} can be rewritten as $h \cdot f = \mu_W(h_{(1)}) \circ f \circ \mu_V\bigl( S(h_{(2)}) \bigr)$, which is exactly the QMM-compatibility condition \eqref{axiomeQMM2}.

\smallskip

\indent From the relation between $\mathscr{L}$-compatibility and QMM-compatibility, we deduce the following fact, which will be a key-point in the proof of Theorem \ref{teo:commutativediagram}.

\begin{lemma}\label{lemmaUniqueBimodCoh}
Assume that $\UU$ is factorizable. Let $V,W$ be finite-dimensional $\UU$-modules and let $\mathbf{B} = (B, \smallblacktriangleright, \smallblacktriangleleft)$ be a non-zero finite-dimensional $\bigl(\underline{\End}(W), \underline{\End}(V)\bigr)$-bimodule in $\UU\text{-}\mathrm{mod}$ which is $\mathscr{L}$-compatible. Then $\mathbf{B}$ is isomorphic as a bimodule in $\UU\text{-}\mathrm{mod}$ to a direct sum of copies of $\underline{\Hom}(V,W)$.
\end{lemma}
\begin{proof}
Let us forget the ambient category $\UU\text{-}\mathrm{mod}$ for an instant, and work in $\mathrm{vect}_k$. Then $\underline{\End}(V)$ is simply the endomorphism algebra $\End_k(V)$ and similarly for  $\underline{\End}(W)$. There are isomorphism of categories
\[ \bigl(\End_k(W), \End_k(V)\bigr)\text{-}\mathrm{bimod} \cong \bigl( \End_k(W) \otimes \End_k(V)^{\mathrm{op}} \bigr)\text{-}\mathrm{mod} \cong \End_k(W \otimes V^*)\text{-}\mathrm{mod} \]
where the first is by definition of a bimodule while the second uses the transpose isomorphism $\End_k(V)^{\mathrm{op}} \cong \End_k(V^*)$. It is a basic fact that modules over $\End_k(X)$ are direct sums of copies of $X$ for any finite-dimensional vector space $X$, see e.g. \cite[Th.\,3.3.1]{RepTh}. Since $W \otimes V^* \cong \Hom_k(V,W)$, we conclude that there is an isomorphism of bimodules $\omega : \mathbf{B} \overset{\sim}{\longrightarrow} \Hom_k(V,W)^{\oplus N}$ for some $N \geq 1$. It now remains to show that $\omega$ is actually a morphism in $\UU\text{-}\mathrm{mod}$. Let $\mu_V$ and $\mu_W$ be the QMMs \eqref{QMMonIntEnd}. Since $\mathbf{B}$ is $\mathscr{L}$-compatible, it is also QMM-compatible by Prop.\,\ref{prop:momentmap}(2). As a result, for all $h \in \UU$, 
\[ \omega(h \cdot b) \overset{\eqref{axiomeQMM2}}{=} \omega\bigl[ \mu_W(h_{(1)}) \smallblacktriangleright b \smallblacktriangleleft \mu_V\bigl( S(h_{(2)}) \bigr) \bigr] = \mu_W(h_{(1)}) \circ \omega(b) \circ \mu_V(S(h_{(2)})) \bigr) \]
which shows that $\omega$ is $\UU$-linear when each summand $\Hom_k(V,W)$ is endowed with the $\UU$-action \eqref{UactionOnIntHom}; but this exactly the definition of $\underline{\Hom}(V,W)$.
\end{proof}

\section{Main example: stated skein algebras and modules}\label{sectionStatedSkein}
Let $\OO = (\OO,\cdot,1,\Delta,\varepsilon,S, \mathcal{R})$ be a Hopf algebra over a field $k$, and endowed with a dual quasitriangular structure $\mathcal{R} : \OO \otimes \OO \to k$, and a coribbon element $\mathsf{v}: \OO \to k$ as in Equation \eqref{eq:coribbon}. 
We denote by $\mathcal{C}= \mathrm{Comod}\text{-}\OO$ the category of $\OO$-comodules and by $\mathcal{C}^{\mathrm{fin}} = \mathrm{comod}\text{-}\OO$ the subcategory of finite dimensional $\OO$-comodules. See  \S\ref{subsecComodH} for our notations regarding comodules.
If the reader is more familiar with left modules over an algebra than right comodules over a coalgebra, we refer to Subsection \ref{sub:ComodVsMod} for a discussion of this point. 

\smallskip

\indent For each oriented and connected compact $2$-dimensional surface $S$ with one boundary component and a marked point on it, we will define a ``stated skein algebra'' $\mathcal{S}_{\OO}(S)$ associated to $\mathcal{C}$. This definition is a special case of the general definition provided in \cite{CKL} and coincides with that given in \cite{BFR}. 
Then we consider the category of cobordisms $\mathrm{Cob}$ between the above surfaces and we show that to each such cobordism is associated a bimodule over the stated skein algebras of its boundary components. Also this construction is a special case of that given in \cite{CKL}.

\smallskip

\indent We then exhibit a monoidal product on $\mathrm{Cob}$ and a braiding on it: these structures where first defined by Kerler and Lyubashenko in a slightly different context (\textit{i.e.} for unmarked surfaces) \cite{KL, kerler}. The application of the stated skein functor $\mathcal{S}_{\OO}$ to these structures is actually the initial motivation for the previous sections: one of the main results of this section (Thm.\,\ref{teo:monoidalfunctor}) is that the stated skein functor is a braided and balanced monoidal functor from $\mathrm{Cob}$ to the category $\mathrm{Bim}^{\mathscr{L}}_{\mathcal{C}}$ of half-braided algebras in $\mathcal{C}$ and their $\mathscr{L}$-compatible bimodules (\S\ref{sectionLlinear}).

\smallskip

We finally relate this result to the TQFT built by Kerler and Lyubashenko and show that when $\OO$ is a co-factorizable Hopf algebra, the stated skein functor is basically the functor associating to a surface the algebra of linear endomorphisms of the state space of that surface for the Kerler-Lyubashenko TQFT (Thm.\,\ref{teo:commutativediagram}). 
 
\subsection{A category of cobordisms}\label{subsec:categoryCob}
By a marked surface we shall mean a connected, compact, oriented $2$-dimensional manifold $S$ with a single boundary component denoted $\partial S$, oriented as induced by $S$ and endowed with a marked point $*\in \partial S$. A diffeomorphism of marked surfaces is an orientation preserving diffeomorphism sending the marked point to the marked point. 

A marked $3$-manifold is an oriented, connected $3$-manifold $M$ with non empty boundary decomposed as $\partial M=\partial^- M\cup \partial^{\mathrm{s}} M\cup \partial^+ M$ so that $\partial^-M\cap \partial^+M=\varnothing$, and the ``side boundary'' $\partial^{\mathrm{s}}M$ is diffeomorphic to $S^1\times [-1,1]$ so that $\partial^\pm M\cap \partial^{\mathrm{s}}M=S^1\times\{\pm 1\}$, i.e. the boundary of $\partial^{\mathrm{s}}M$ is the union of the boundaries of $\partial^\pm M$ which are two circles; furthermore $\partial^{\mathrm{s}} M$ is endowed with the datum of an embedded oriented arc $\mathcal{N}$ intersecting $\partial^- M$ and $\partial^+ M$ exactly in its endpoints, which are then forming the  markings of $\partial^\pm M$ and considered up to isotopy relative to its boundary.

\begin{definition}
$ \mathrm{Cob}$ is the category whose objects are marked surfaces. If $S_-,S_+\in Ob(\mathrm{Cob})$, a morphism from $S_-$ to $S_+$ is a diffeomorphism class of a marked $3$-manifolds with parametrized boundary : $(M,\mathcal{N},\phi_{\pm})$ where 
$\phi_{\pm}:S_\pm\to \partial^{\pm} M$ is a diffeomorphism whose sign is $\pm$ and sending the marked points to $\partial \mathcal{N}$. 
A diffeomorphism of such tuple is an orientation preserving diffeomorphism which is the identity on $\partial^\pm M$ and sends the marking to the marking. 
The composition of $(M_2,\mathcal{N}_2,\psi_\pm)\in \Hom_{\mathrm{Cob}}(S_0,S_+)$ and $(M_1,\mathcal{N}_1,\phi_\pm)\in  \Hom_{\mathrm{Cob}}(S_-,S_0)$ is the marked $3$-manifold obtained by identifying the two copies of $S_0$ via the diffeomorphism given by $\phi_0\circ \psi^{-1}_0$ and whose marking is given by the oriented segment obtained by gluing $\mathcal{N}_1$ and $\mathcal{N}_2$.  
\end{definition}

\begin{example}
If $(S,*)$ is a marked surface then the identity morphism on $S$ is $\mathrm{Id}_S=(S\times[-1,1], \{*\} \times [-1,1],\mathrm{id},\mathrm{id})$. 
\end{example}

\begin{remark}\label{remarkKL}
The above notion of marked manifold is a special case of the notion introduced in \cite{CL3Man}. 
Also the category $\mathrm{Cob}$ is a subcategory of the category defined in \cite{CL3Man}; furthermore it is basically the same as the category $\mathrm{Cob}_0$ of \cite{KL} and the 3Cob of \cite{BBDP} with the only difference that our manifolds carry a marking in their boundary which shall be used in the next subsection. Since we consider cobordisms up to diffeomorphisms which are not the identity on the side boundary but are only required to preserve the marking, our category is indeed equivalent to that considered in \cite{BBDP}. For instance in Figure \ref{fig:balancingcob} we exhibit two examples of cobordisms which are equivalent. 
\end{remark}

\indent We now discuss the main features of $\mathrm{Cob}$:

\smallskip

\indent \textbullet ~ {\bf The category $\mathrm{Cob}$ is monoidal:} The tensor product of two marked surfaces is obtained by glueing them to the complement of two open discs in a marked disc along the unmarked components. Similarly the tensor product of two marked $3$-manifolds yielding morphisms in $\mathrm{Cob}$ is given by glueing them with the thickening of the complement of two open discs in a marked disc by identifying their side boundaries to the thickenings of the boundaries of the two discs. See the left part of Figure \ref{fig:cobmonoidal}. 
\begin{figure}[htbp]
\centering
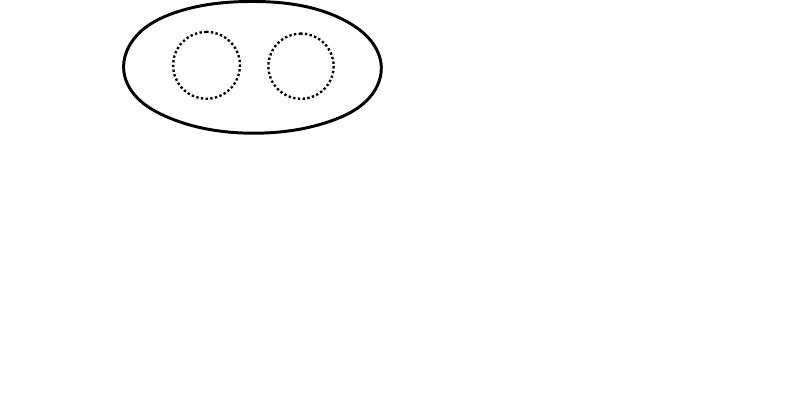
\caption{On the left we schematize the tensor product of two marked surfaces $S_1$, $S_2$ and of two marked manifolds $M_1$, $M_2$; the oriented arc depicted in front is the marking $\mathcal{N}$ of $M_1 \otimes M_2$ going from the marked point of the bottom surface to the marked point of the top surface. On the right we exhibit the braiding  $c_{S_1,S_2}$; the marking of the two surfaces is running vertical on the dotted cylinders and that of the cobordism is depicted in front. }\label{fig:cobmonoidal}
\end{figure}

\smallskip

\indent \textbullet ~ {\bf The category $\mathrm{Cob}$ is braided:} if $S_1,S_2$ are marked surfaces, then $c_{S_1,S_2}$ is the marked three manifold $(S_1\otimes S_2)\times [-1,1]$ endowed with parametrisations $\phi_-=\mathrm{id} : S_1\otimes S_2\to (S_1\otimes S_2)\times \{-1\}$ and $\phi_+:S_2\otimes S_1\to (S_1\otimes S_2)\times \{1\}$ obtained by exchanging the two surfaces via a half twist along $\partial (S_2\otimes S_1)$ as shown in the right-hand side of Figure \ref{fig:cobmonoidal}. 
This structure is identical to that already detailed in \cite{KL}.

\smallskip

\indent \textbullet ~ {\bf The category $\mathrm{Cob}$ is balanced:}
For each surface $S$ let $\tau_S$ be the cobordism whose underlying manifold is $S\times[-1,1]$ but whose marking performs a full  positive twist of $\partial S$ while going from $*\times \{-1\}$ to $*\times \{1\}$ (see Figure \ref{fig:balancingcob}). Equivalently, $\tau_S$ is the $3$-manifold $S\times[-1,1]$ marked with $*\times [-1,1]$ (where $*\in \partial S$ is the marked point) and with a surgery along a knot parallel to $\partial S$ at height $\{0\}$ and with framing $+1$ with respect to the vertical framing. It is easy to verify that these two cobordisms are the same in the category and, using the latter presentation that $\tau$ is indeed a balancing {\it i.e.} that $\tau_{S_1\otimes S_2}=(c_{S_2,S_1}\circ c_{S_1,S_2})\circ (\tau_{S_1}\otimes \tau_{S_2})$: this is explained in Figure \ref{fig:balancingtwist}.
\begin{figure}[htbp]
\centering
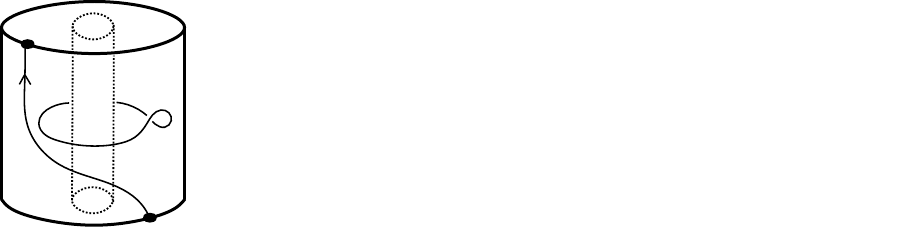
\caption{On the left we represent a neighbourhood of the side boundary of a $3$-manifold $M$, and make a $+1$ surgery along the knot parallel to the core of the side boundary. This gives a manifold diffeomorphic to the one depicted at its right. Similarly the last two cobordisms are diffeomorphic. When the manifold $M$ is $S\times [-1,1]$ then the cobordism depicted on the left is the balancing $\tau_S$, the one on the right is $\tau_S^{-1}$. }\label{fig:balancingcob}
\end{figure}
\begin{figure}[htbp]
\centering
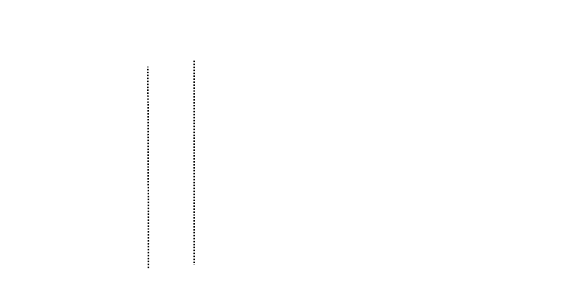
\caption{Graphical proof of the fact that $\tau$ is a balancing: on the left the manifold contains a $+1$-surgery knot around the core of the side boundary. On the right a slam-dunk move is applied; equivalently the whole cylinders representing $S_1$ and $S_2$ are slid over the knot so that the knot becomes contained in a ball and can be removed. During the sliding $S_1$ and $S_2$ acquire a full twist encoded by the two unknots with framing $+1$ in the bottom of the right figure. }\label{fig:balancingtwist}
\end{figure}

\subsection{Stated skein modules of marked manifolds}\label{subsecStatedSkMod}
The purpose of this section is to exhibit a braided monoidal functor $\mathrm{Cob} \to \mathrm{Bim}_{\mathcal{C}}^{\mathrm{hb}}$ from the notion of stated skein modules of $3$-manifolds, which we now recall. 

A {\em ribbon graph} is an oriented surface decomposed into finitely many {\em coupons} and {\em strips} (or {\em bands}, or {\em edges}) with disjoint interior so that: 
\begin{enumerate}
\item Both coupons and bands are homeomorphic to rectangles $]-1,1[\times [-1,1]$, and we say that $]-1,1[\times \{1\}$ (resp. $]-1,1[\times \{-1\}$) is the ``top base'' (resp. ``bottom base''). 
\item The interiors of two coupons or strips do not intersect. A strip can be glued to one or two coupons along its two bases. The ``cores'' of the strips, namely the arcs $\{0\}\times [-1,1]$ are oriented (arbitrarily). 
\item Each connected component of the surface contains at least one strip but we allow components containing no coupons: they are necessarily either a single strip or of the additional form $]-1,1[\times S^1$. 
\end{enumerate}
The ``boundary'' of a ribbon graph is the set of boundary components of the strips which are not glued to the coupons; they are open oriented arcs. We will think of a ribbon graph as the oriented graph formed by the cores of the strips connected to the coupons considered as vertices.

\indent Let $\mathcal{C}^{\mathrm{fin}}$ be the full subcategory of finite-dimensional right $\OO$-comodules. A $\mathcal{C}^{\mathrm{fin}}$-ribbon graph $\Gamma$ in a marked $3$-manifold $(M,\mathcal{N})$ is an embedded ribbon graph $\Gamma\subset M$ whose edges are marked (``colored'') by objects of $\mathcal{C}^{\mathrm{fin}}$ and whose coupons are decorated by morphisms of $\mathcal{C}$ (in the usual way: see \cite{RT}), and such that $\partial \Gamma\subset \mathcal{N}$ as oriented arcs. By a slight abuse of language we shall say that $\Gamma\cap \mathcal{N}$ is formed by points (corresponding to the intersection of the cores of the strips of $\Gamma$ with $\mathcal{N}$). 
\begin{definition} 
Suppose that $\Gamma\cap \mathcal{N}=\{p_1,\ldots, p_k\}$ so that $p_1>p_2>\ldots >p_k$ (in the order induced by the orientation of $\mathcal{N}$, i.e. the orientation goes from $p_k$ to $p_1$) and let the sign $\epsilon_i$ of $p_i$ be $+$ if the core of $\Gamma$ at $p_i$ is outgoing $M$ and $-$ else. Recall that each $p_i$ belongs to a strip which is carrying a color $V_i\in \mathcal{C}^{\mathrm{fin}}$. 
A state for $\Gamma$ is the datum of a vector $v_1\otimes \cdots \otimes v_k\in  V_1^{\epsilon_i}\otimes \cdots \otimes V_k^{\epsilon_k}$ where we denote $V^+=V$ and $V^-=V^{*}$. A stated ribbon graph is a pair $(\Gamma,s)$ where $s$ is a state for $\Gamma$. 
\end{definition}

\begin{definition}\label{defStSkMod}
The stated skein module $\mathcal{S}_{\OO}(M,\mathcal{N})$ is the quotient of the ${k}$-vector space spanned by all stated $\mathcal{C}^{\mathrm{fin}}$-ribbon graphs in $(M,\mathcal{N})$ by the subvector space generated by the relations:
\begin{enumerate}
\item $(\Gamma,s)-(\Gamma',s')$ if $\Gamma$ and $\Gamma'$ are isotopic via an isotopy of framed ribbon graphs (so that their boundaries are embedded in the marking $\mathcal{N}$ of $M$) and $s=s'$.
\item $(\Gamma,s)$ is multilinear in the state vector, i.e. if the $i^{th}$ vector of $s$ is $\lambda'v'_i+\lambda'' v''_i$ then $(\Gamma,s)=\lambda'(\Gamma,s')+\lambda''(\Gamma,s'')$ where $s'$ is identical to $s$ except for the $i^{th}$ vector which is $v'_i$ and similarly for $s''$, for all $\lambda',\lambda'' \in k$. 
\item $(\Gamma,s)-(\Gamma',f(s))$ where $\Gamma$ and $\Gamma'$ are almost identical except that $\Gamma'$ is the ribbon graph obtained from $\Gamma$ by ``pushing off'' a $\mathcal{C}^{\mathrm{fin}}$-ribbon tangle $T$ whose image through the Reshetikhin-Turaev functor $\mathrm{RT}$ is a morphism $f\in \Hom_\mathcal{C}\bigl( V_1^{\epsilon_1}\otimes\cdots  \otimes V_l^{\epsilon_l}, W_1^{\eta_1}\otimes\cdots \otimes W_k^{\eta_k} \bigr)$ and the state $f(s)$ on it is given by $f(v_1\otimes \cdots \otimes v_l)$, see Figures \ref{fig:SkeinRelation} and \ref{skeinRmatrix}. 
\end{enumerate}
\end{definition}
\begin{figure}[htbp]
\centering
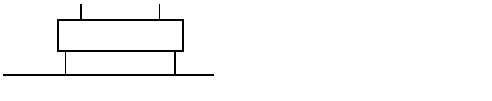
\caption{The basic skein relation: in the drawing we picked a ribbon graph represented by a single coupon; in general one can push off the surface any graph $G$ and use $f=\mathrm{RT}(G)$ to compute the state of the resulting graph. Here the strands are decorated by objects of $\mathcal{C}^{\mathrm{fin}}$ and oriented (not indicated on the picture). The horizontal oriented segment represents the marking $\mathcal{N}$ of $M$. We use the convention that marking each boundary point by a state $s_i$ (as on the left-hand side) is the same thing as globally marking the boundary points by one state $s_1 \otimes \ldots \otimes s_k$ (as on the right-hand side); this can be formalized by the use of identity coupons. 
}\label{fig:SkeinRelation}
\end{figure}

\begin{figure}[htbp]
\centering
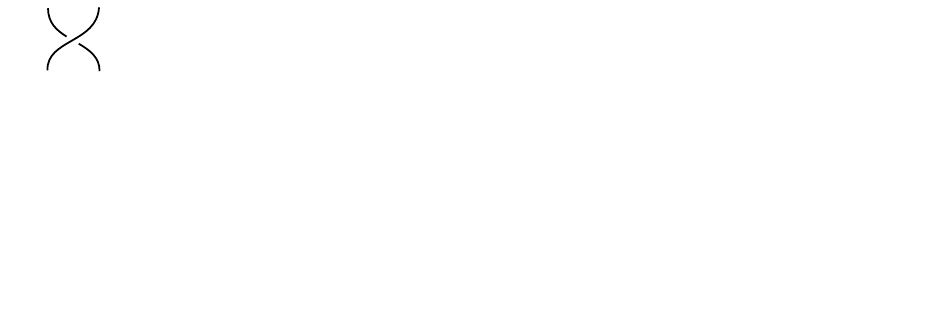
\caption{Some useful skein relations, obtained by applying the RT functor for $\mathrm{comod}\text{-}\OO$. The co-pivotal element $\mathsf{g} : \OO \to k$ is defined by $\mathsf{g}(\varphi) = \mathsf{u}(\varphi_{(1)})\mathsf{v}^{-1}(\varphi_{(2)})$, where $\mathsf{u} : \OO \to k$ is the co-Drinfeld element given by $\mathsf{u}(\varphi) = \mathcal{R}\bigl( \varphi_{(2)} \otimes S(\varphi_{(1)}) \bigr)$ as defined e.g. in \cite[Prop.\,2.2.4]{Majid}.}
\label{skeinRmatrix}
\end{figure}

\begin{remark}
The dimension of $\mathcal{S}_{\OO}(M,\mathcal{N})$ might be infinite in general. 
\end{remark}

Let $(\Gamma,s)$ be a stated $\mathcal{C}^{\mathrm{fin}}$-ribbon graph in $(M,\mathcal{N})$ and write the state $s$ as an implicit sum $s^1 \otimes \ldots \otimes s^l$ (we use upper indices for notational convenience below). Then we define $\Delta(\Gamma,s) \in \mathcal{S}_{\OO}(M,\mathcal{N})\otimes \OO$ by
\begin{equation}\label{HComodStatedSkein}
\Delta(\Gamma,s) = \bigl( \Gamma, s^1_{[0]} \otimes s^2_{[0]} \otimes \ldots \otimes s^l_{[0]} \bigr) \otimes s^1_{[1]}s^2_{[1]} \ldots s^l_{[1]}
\end{equation}
where $s^i_{[0]} \otimes s^i_{[1]} \in V_i^{\epsilon_i} \otimes \OO$ is the coaction in the comodule $V_i^{\epsilon_i}$, written in Sweedler's notation. In the special case where $\Gamma$ has no boundary (so no state vector) then we define $\Delta(\Gamma,\varnothing)=(\Gamma,\varnothing)\otimes 1$.
\begin{lemma}
Given a marked $3$-manifold  $(M,\mathcal{N})$, the above defined $k$-linear map $$\Delta:\mathcal{S}_{\OO}(M,\mathcal{N})\to \mathcal{S}_{\OO}(M,\mathcal{N})\otimes \OO$$  endows $\mathcal{S}_{\OO}(M,\mathcal{N})$ with the structure of a right $\OO$-comodule (not necessarily finite dimensional). 
 \end{lemma}
 \begin{proof}
 The relations (1) and (3) of Definition \ref{defStSkMod} yield equal values for $\Delta$ by the requirement that the coupons are decorated by morphisms of right $\OO$-comodules. 
\end{proof}

\indent Stated skein modules naturally have interesting algebraic properties:

\smallskip

\indent \textbf{Stated skein algebras:} The stated skein module of a thickened marked surface $(S\times [-1,1],\{*\}\times[-1,1])$ can be naturally endowed with the structure of an associative unital algebra as follows. If $\Gamma_1$ and $\Gamma_2$ are stated skeins in $(S\times [-1,1],\{*\}\times[-1,1])$ then $\Gamma_1\cdot \Gamma_2$ is the stated ribbon graph $\Gamma_1\sqcup \Gamma_2$ where $\Gamma_1$ is suitably isotoped in $S\times ]0,1[$ and $\Gamma_2$ in $S\times ]-1,0[$. This algebra structure is compatible with that of right $\OO$-comodule described above so that $\mathcal{S}_{\OO}(S)$ is a right $\OO$-comodule algebra, i.e. an algebra object in $\mathcal{C}$. 

\smallskip

\indent \textbf{Stated skein bimodules:} The stated skein module of a marked $3$-manifold $\mathbf{M} = (M,\mathcal{N})$ is naturally a right module over $\mathcal{S}_{\OO}(\partial^-M)$ and a left module over $\mathcal{S}_{\OO}(\partial^+M)$ and, again, this structure is compatible with that of right $\OO$-comodule in the sense of Subsection \ref{sectionPreliminariesModules}.
These actions are defined exactly as above by ``pushing skeins'' from the bottom or from the top respectively, so they commute.
Said briefly, if ${\bf M}=(M,\mathcal{N},\psi_\pm)\in \Hom_{\mathrm{Cob}}(S_-,S_+)$ then $\mathcal{S}_{\OO}({\bf M})$ is a $\bigl(\mathcal{S}_{\OO}(S_-),\mathcal{S}_{\OO}(S_+) \bigr)$-bimodule in $\mathrm{Comod}\text{-}\OO$. We do not detail these statements here as all the arguments are identical to those given in \cite{CL3Man}.

\smallskip

There is yet another structure on stated skein algebras. Recall the coend $\mathscr{L} = \int^{X \in \mathcal{C}^{\mathrm{fin}}}X^* \otimes X$ described in \S\ref{sub:coend} and \S\ref{subsecComodH} and its universal dinatural transformation $\bigl( i_X : X^* \otimes X \to \mathscr{L} \bigr)_{X \in \mathcal{C}^{\mathrm{fin}}}$. For any marked surface $S$, there is a morphism of $\OO$-comodules $\mathfrak{d} : \mathscr{L} \to \mathcal{S}_{\OO}(S)$ defined in Fig.\,\ref{QMMstatedSkein}.
\begin{figure}[h!]
\centering
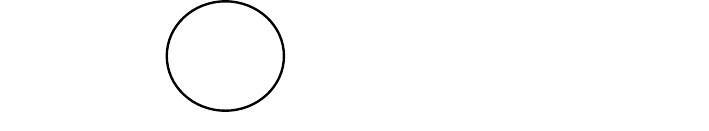
\caption{The $\mathscr{L}$-linear structure $\mathfrak{d} : \mathscr{L} \to \mathcal{S}_{\OO}(S)$ on stated skein algebras (as proved in Prop.\,\ref{prop:skeinLlinear}). The two points visible in the boundary are the extremities of the single marking $\mathcal{N}$ ending to the marked points of $S \times \{-1\}$ and $S \times \{1\}$. We slightly bend $\mathcal{N}$ and project on $S$ in order to be able to draw diagrams.}\label{QMMstatedSkein}
\end{figure}
For the second equality in Fig.\,\ref{QMMstatedSkein}, we use that $\mathscr{L}$ is $\OO$ as a vector space and $F_{\varphi}$ denotes any finite-dimensional subcomodule of the regular comodule $\OO$ (defined by the coproduct) which contains $\varphi$. To establish the equivalence between the two definitions in Fig.\,\ref{QMMstatedSkein}, consider the morphism of comodules $\delta_f : X \to \OO$ given by $\delta_f(x) = i_X(f \otimes x) = f(x_{[0]})x_{[1]}$. Its transpose $\delta^*_f : \mathrm{im}(\delta_f)^* \to X^*$ satisfies $\delta_f^*(\varepsilon_{\mathcal{O}}) = f$, where the counit $\varepsilon_{\OO}$ is implicitly restricted to the image of $\delta_f$. Hence
\begin{equation}\label{dinatSpecialMonogon}
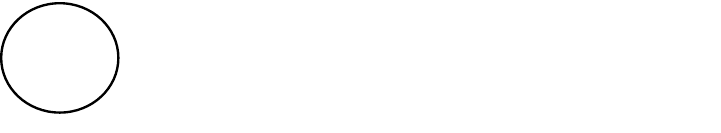
\end{equation}
where the unlabelled strand can be colored by $\mathrm{im}(\delta_f)$ or any finite-dimensional subcomodule of $\OO$ which contains $\delta_f(x)$.

\begin{proposition}\label{prop:skeinLlinear}
1. Let $S$ be a marked surface and endow $\mathcal{S}_{\OO}(S)$ with the morphism $\mathfrak{d}$ defined in Fig.\,\ref{QMMstatedSkein}. Then $\mathcal{S}_{\OO}(S)$ is a $\mathscr{L}$-linear algebra in the sense of Definition \ref{defLlinearAlgebra}.
\\2. Let ${\bf M}=(M,\mathcal{N},\psi_\pm)\in \Hom_{\mathrm{Cob}}(S_-,S_+)$. The $\bigl(\mathcal{S}_{\OO}(S_+),\mathcal{S}_{\OO}(S_-)\bigr)$-bimodule $\mathcal{S}_{\OO}({\bf M})$ is $\mathscr{L}$-compatible in the sense of Definition \ref{defLcoherentBimodule}.
\end{proposition}
\begin{proof}
1. We must first show that $\mathfrak{d}$ is a morphism of algebras:
\begin{center}
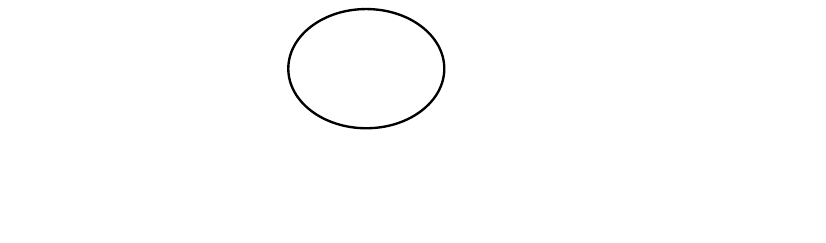
\end{center}
The first equality is by definition of $\mathfrak{d}$ and of the product in $\mathcal{S}_{\OO}(S)$, the second equality uses a skein relation (Fig.\,\ref{fig:SkeinRelation}, recall that $c$ denotes the braiding in $\mathcal{C}$ and in this picture we omit the subscripts of $c$ to save space), the third equality is by definition of $\mathfrak{d}$ and the last equality is by definition of the product $\odot$ in $\mathscr{L}$ as defined in \eqref{defStructureCoend}. Next we must show the $\mathscr{L}$-linear property \eqref{QMM_GJS} for $\mathfrak{d}$, but it is a particular case of the $\mathscr{L}$-compatibility proved in the next item (for the identity cobordism $M = S \times [0,1]$, or said differently by regarding $\mathcal{S}_{\OO}(S)$ as a bimodule over itself).
\\\noindent 2. Let $\mathfrak{d}_{\pm}$ be the $\mathscr{L}$-linear structure on $\mathcal{S}_{\OO}(S_{\pm})$ respectively. We must show that the diagram \eqref{defLcoherentBimodule} commutes for $B = \mathcal{S}_{\OO}(\mathbf{M})$:
\begin{center}
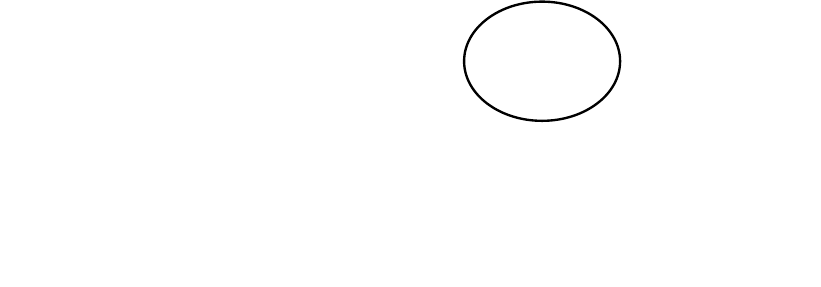
\end{center}
In the first term we use an element in $\mathscr{L}$ presented as the image of the universal transformation $i$ and a generic element in $\mathcal{S}_{\OO}(\mathbf{M})$ which is only represented in a the tubular neighborhood of the boundary of $M$ (what happens in $M$ has no importance). The first equality simply uses the definitions of $\mathfrak{d}_+$ and the left action of $\mathcal{S}_{\OO}(S_+)$ on $\mathcal{S}_{\OO}(\mathbf{M})$, the second equality is a trick (nothing is changed), the third equality uses a skein relation (Fig.\,\ref{fig:SkeinRelation}) and the last equality uses the definition of $\sigma$ in \eqref{defHalfBraidingSigmaOnCoend} together with the definitions of $\mathfrak{d}_-$ and the right action of $\mathcal{S}_{\OO}(S_-)$ on $\mathcal{S}_{\OO}(\mathbf{M})$.
\end{proof}

\begin{remark}
The above proposition shows that if $(\OO,\mathcal{R})$ is {co-factorizable} (i.e. it is the dual of a factorizable Hopf algebra) then stated skein algebras carry a quantum moment map, as explained in Section \ref{sub:LlinQMM} (see Proposition \ref{prop:momentmap}).
\end{remark}
The following fact is a preparation for the proof of Thm.\,\ref{teo:monoidalfunctor} below:
\begin{lemma}\label{lem:cutting}
Let $M$ be a compact $3$-manifold and $\mathcal{N}\subset \partial M$ be an oriented segment. Let also $S\subset M$ be a properly embedded oriented surface intersecting $\mathcal{N}$ transversally once and disconnecting $(M,\mathcal{N})$ into the disjoint union $(M_1,\mathcal{N}_1)\sqcup (M_2,\mathcal{N}_2)$. Then the inclusions $i_j:M_j\hookrightarrow M, j=1,2$ induce an isomorphism of $\OO$-comodules: 
$$(i_1\sqcup i_2)_*:\mathcal{S}_{\OO}(M_2)\otimes_{\mathcal{S}_{\OO}(S)}  \mathcal{S}_{\OO}(M_1) \to \mathcal{S}_{\OO}(M).$$
\end{lemma}
\begin{proof}
It easy to see that the morphism of $\OO$-comodules $(i_1)_*\otimes_k (i_2)_*$ descends to a morphism $(i_1\sqcup i_2)_*:\mathcal{S}_{\OO}({M_1})\otimes_{\mathcal{S}_{\OO}(S)}\mathcal{S}_{\OO}({M_2})\to \mathcal{S}_{\OO}({M})$.
We need to prove that it is surjective and injective. 
This has been proved in \cite{CL3Man} for $\OO=\mathcal{O}_q(SL_2)$ and is proved in an even more general framework in \cite{CKL}, therefore we will limit ourserlves to sketch the proof here. Roughly, the key idea is to show that each stated skein in 
$\mathcal{S}_{\OO}(M)$ is equivalent to a linear combination of tensor products of stated skeins in ${ M_1}$ or in ${M_2}$ and that each skein relation in ${M}$ is equivalent to a linear combination of skein relations in ${M_1}$ or in ${M_2}$. 
These statements can be proved by the ``slicing" operation, defined in \cite{CL3Man} which applies to the general case of stated skeins colored by finite-dimensional $\OO$-comodules. Namely if $(\Gamma,s)\subset M$ is a stated skein which intersects transversally $S$ into finitely many points $p_1,\ldots, p_k$ then consider an arc $\alpha\subset S$ with one endpoint in $\mathcal{N}\cap S$ and the other in $\partial S\setminus \mathcal{N}$ and containing all of $p_1,\ldots p_k$. Then one can ``push the skein $\Gamma$'' along $\alpha$ out of $S$ by the slicing operation described in \cite{CL3Man} (Theorem 2.1: remark that the operation holds for any Hopf algebra $\OO$, not only $\mathcal{O}_q(SL_2)$) to exhibit $s$ as an equivalent linear combination of stated skeins not intersecting $S$. This shows immediately $(i_1\sqcup i_2)_*$ is surjective. Similarly, each isotopy in $M$ can be decomposed into ``small'' isotopies whose support is contained entirely in one of ${M_1}$ or ${M_2}$; finally by isotopying skeins locally around $\mathcal{N}$ into ${M_1}$ (or ${M_2}$) and applying the above slicing operation, one can reduce each skein relation in ${ M}$ to a linear combination of skein relations in ${M_1}$ (or $M_2$). This proves injectivity of $(i_1\sqcup i_2)_*$.
\end{proof}

We explained in \S\ref{subsec:BimodHcomod} that the category $\mathrm{Bim}^{\mathscr{L}}_{\mathcal{C}}$ is monoidal, braided and balanced, as a particular case of more general results. In \S\ref{subsec:categoryCob} we also noted that the category $\mathrm{Cob}$ is monoidal, braided and balanced. Finally in Prop.\,\ref{prop:skeinLlinear} we saw that $\mathcal{S}_{\OO}$ produces $\mathscr{L}$-linear algebras and $\mathscr{L}$-compatible bimodules. One of the main contributions of this paper is the following:
\begin{theorem}\label{teo:monoidalfunctor}
Let $\OO$ be a coribbon Hopf algebra and $\mathcal{C}^{}$ be the category of right $\OO$-comodules.
The stated skein functor $\mathcal{S}_{\OO} :(\mathrm{Cob},\otimes)\to (\mathrm{Bim}^{\mathscr{L}}_{\mathcal{C}},\,\widetilde{\otimes})$ is a strict monoidal functor which is compatible with the braiding and balance on both categories, \textit{i.e.} for each marked surfaces $S, S_1,S_2$ it holds
\[ \mathcal{S}_{\OO}(S_1 \otimes S_2) = \mathcal{S}_{\OO}(S_1) \,\widetilde{\otimes}\, \mathcal{S}_{\OO}(S_2), \quad  \mathcal{S}_{\OO}(c_{S_1,S_2})=\mathcal{B}_{\mathcal{S}_{\OO}(S_1), \mathcal{S}_{\OO}(S_2)}, \quad \mathcal{S}_{\OO}(\tau_S) = \mathrm{BAL}_{\mathcal{S}_{\OO}(S)} \]
with the notations introduced in \S\ref{subsec:BimodHcomod} and \S\ref{subsec:categoryCob}. 
\end{theorem}
\begin{proof}
There are four properties to be checked:

\textbullet~ {\em $\mathcal{S}_{\OO}$ is a functor.} We need to show that if $\mathbf{M}_2=(M_2,\mathcal{N}_2,\psi_\pm)\in \Hom_{\mathrm{Cob}}(S_0,S_+)$ and $\mathbf{M}_1=(M_1,\mathcal{N}_1,\psi'_\pm)\in \Hom_{\mathrm{Cob}}(S_-,S_0)$ then $\mathcal{S}_{\OO}(\mathbf{M}_2 \circ \mathbf{M}_1)$ is isomorphic as a $(\mathcal{S}_{\OO}({ S_+}),\mathcal{S}_{\OO}({S_-}))$-bimodule to $\mathcal{S}_{\OO}( \mathbf{M}_2)\otimes_{\mathcal{S}_{\OO}(S_0)}\mathcal{S}_{\OO}(\mathbf{M}_1)$.
This is basically the content of Lemma \ref{lem:cutting}, one only needs to remark further that the actions of $\mathcal{S}_{\OO}({ S_\pm})$ commute with the isomorphism in the lemma. 

\smallskip

\indent \textbullet~ {\em $\mathcal{S}_{\OO}$ is strict monoidal.} We start by observing that for each morphisms $M_i \in \Hom_{\mathrm{Cob}}(S_j, S_j'),j=1,2$, the $\OO$-comodule $\mathcal{S}_{\OO}(M_1\otimes M_2)$ is isomorphic to $\mathcal{S}_{\OO}(M_1)\otimes_k\mathcal{S}_{\OO}(M_2)$, thanks to the inclusion isomorphism $(i_1 \sqcup i_2)_*$ obtained by applying Lemma \ref{lem:cutting} to a properly embedded disc splitting $M_1\otimes M_2$ into $M_1\sqcup M_2$ and using that $\mathcal{S}_{\OO}(D^2)=k$. See Figure \ref{fig:monProdSkeins}. In Figure \ref{fig:braidedtensorproduct} we prove that the isomorphism $(i_1\sqcup i_2)_*:\mathcal{S}_{\OO}(S_1) \otimes \mathcal{S}_{\OO}(S_2)\to \mathcal{S}_{\OO}(S_1\otimes S_2)$ is moreover an algebra morphism when the source is equipped with the braided tensor product structure $\mathcal{S}_{\OO}(S_1) \,\widetilde{\otimes}\, \mathcal{S}_{\OO}(S_2)$ of $\OO$-comodule algebras in \eqref{brProdComodH}; note that it suffices to check \eqref{brProdComodH} with $x = 1_{A_1}$ and $y = 1_{A_2}$. This shows monoidality of $\mathcal{S}_{\OO}$ on the level of objects. To prove monoidality for morphisms, replace $S_i$, $i=1,2$ in Figure \ref{fig:braidedtensorproduct} by two $3$-manifolds $M_i$ and remark that the isomorphism of $\OO$-comodules $(i_1\sqcup i_2)_*$ is actually an isomorphism of bimodules because each $M_i$ is a cobordism $S_i\to S'_i$ and the splitting along the vertical disc used to prove that $(i_1\sqcup i_2)_*$ is an isomorphism preserves the splitting of $\partial_+ (M_1\otimes M_2)$ (resp.  $\partial_- (M_1\otimes M_2)$)  into $S_1\sqcup S_2$ (resp. $(S'_1\sqcup S'_2)$).

\begin{figure}[h!]
\centering
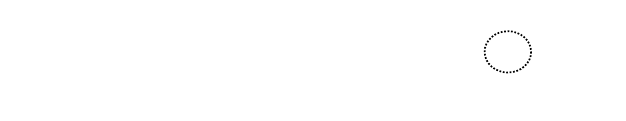
\caption{Tensor product of skeins $a \in \mathcal{S}_{\OO}(M_1)$ and $b \in \mathcal{S}_{\OO}(M_2)$. The skein $a$, which has some state $v$, is represented only in a tubular neighborhood of the side boundary $\partial^{\mathrm{s}}(M_1)$ and similarly for $b$ and $a \otimes b$. Here the cobordisms are seen from above, as in Fig.\,\ref{QMMstatedSkein}. The resulting view of the disc which splits $M_1 \otimes M_2$ into $M_1 \sqcup M_2$ is represented by the dashed grey line.}
\label{fig:monProdSkeins}
\end{figure}

\begin{figure}[h!]
\centering
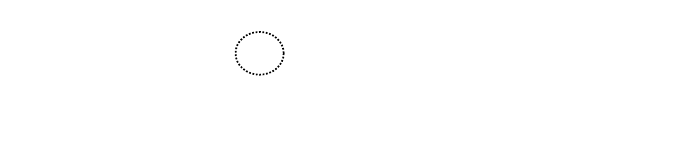
\caption{Verification of eq.\,\eqref{brProdComodH} for stated skein algebras. We denote by $\emptyset_{S_i}$ the empty skein, which is the unit element of $\mathcal{S}_{\OO}(S_i)$. We use the definition of the product in $\mathcal{S}_{\OO}(S_1 \otimes S_2)$, the crossing skein relation in Fig.\,\ref{skeinRmatrix}, the $\OO$-comodule structure \eqref{HComodStatedSkein} of stated skein algebras and the tensor product of skeins in Fig.\,\ref{fig:monProdSkeins}.}\label{fig:braidedtensorproduct}
\end{figure}

\medskip

\textbullet~ {\em $\mathcal{S}_{\OO}$ is braided.} The cobordism $c_{S_1,S_2}$ is diffeomorphic to the thickening of $S_2\otimes S_1$ via a diffeomorphism which is the identity on the positive (i.e. top) boundary: we can then identify $\mathcal{S}_{\OO}(c_{S_1,S_2})$ with $\mathcal{S}_{\OO}(S_2\otimes S_1)=\mathcal{S}_{\OO}(S_2)\,\widetilde{\otimes}\, \mathcal{S}_{\OO}(S_1)$ via this diffeomorphism. It thus suffices to check \eqref{eq:braidingtwist}. Take elements $a \in \mathcal{S}_{\OO}(S_1)$ and $b \in \mathcal{S}_{\OO}(S_2)$, represented as in Fig.\,\ref{fig:monProdSkeins} (thus with $S_i \otimes [-1,1]$ in place of $M_i$). Denote by $\emptyset \in \mathcal{S}_{\OO}(c_{S_1,S_2})$ the empty skein, which through the above identification is the unit element in $\mathcal{S}_{\OO}(S_2\otimes S_1)$. In Figure \ref{fig:braidedfunctor} we use isotopy to express $\emptyset \smallblacktriangleleft (a \otimes b)$ as a {\em left} action on $\emptyset$. The last term in Fig.\,\ref{fig:braidedfunctor} can be written as $\mathcal{R}\bigl(a_{[1]} \otimes b_{[1]}\bigr) \, \bigl( b_{[0]} \otimes a_{[0]} \bigr) \, \mathfrak{d}_1\bigl( b_{[2]} \bigr) \smallblacktriangleright \emptyset$, thanks to the definition of the product and of the $\OO$-comodule structure \eqref{HComodStatedSkein} for stated skein algebras, and also by definition of the morphism $\mathfrak{d}_1 : \mathscr{L} \to \mathcal{S}_{\OO}(S_1)$ given in Fig.\,\ref{QMMstatedSkein}. This proves that \eqref{eq:braidingtwist} is satisfied, as claimed.
\begin{figure}[h!]
\centering
{\def\svgwidth{.98\linewidth}
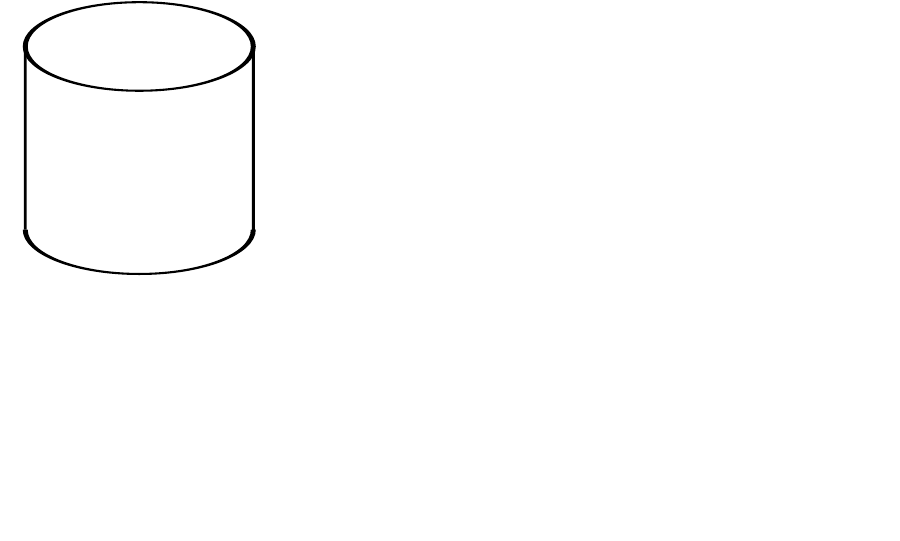
}
\caption{Computation of $\emptyset \smallblacktriangleleft (a \otimes b)$. We can always assume this orientation on the strands of $a$ and $b$ near the side boundary. For the first equality we push the skein up through an isotopy. (Remark that in the two first pictures we should have connected the skeins to the marking depicted in diagonal, but this can be done canonically by vertical segments and we did not for the sake of clarity.) The second equality is by definition of the left action, the third is a trick, the fourth uses skein relations (see Fig.\,\ref{skeinRmatrix}), the fifth uses \eqref{dinatSpecialMonogon} and evaluation of $w^i$.}\label{fig:braidedfunctor}
\end{figure}

\indent \textbullet~ {\em $\mathcal{S}_{\OO}$ is balanced.} The arguments are similar to the ones in the previous item. It suffices to check \eqref{eq:balanceBimodHComod}. Let $a \in \mathcal{S}_{\OO}(S)$ be some skein with state $v$. In Fig.\,\ref{figProofBal} we use isotopy to express $\emptyset \smallblacktriangleleft a$ (represented only in a tubular neighborhood of $\partial S \times [0,1]$) as a {\em left} action by some element in $\mathcal{S}_{\OO}(S)$. The last term in Fig.\,\ref{figProofBal} can be written as $a_{[0]} \, \mathfrak{d}(a_{[1]}) \, \mathsf{v}^{-1}(a_{[2]}) \smallblacktriangleright \emptyset$, thanks to the definition \eqref{HComodStatedSkein} of the coaction on $\mathcal{S}_{\OO}(S)$ and of the $\mathscr{L}$-linear structure $\mathfrak{d}$ (Fig.\,\ref{QMMstatedSkein}). This proves that \eqref{eq:balanceBimodHComod} is satisfied, as claimed.
\begin{figure}[h!]
\centering
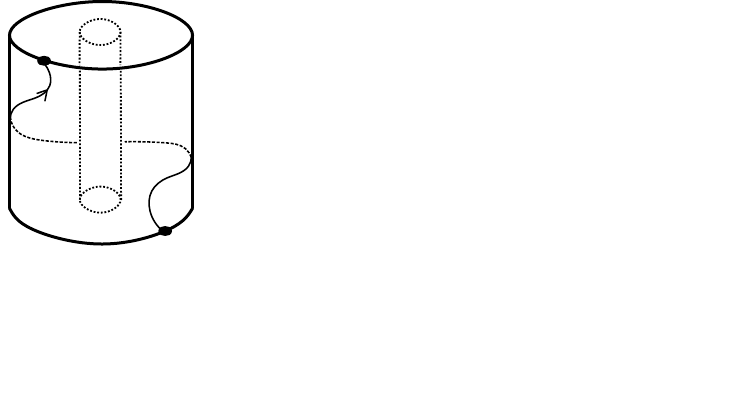
\caption{Computation of $\emptyset \smallblacktriangleleft a$. For the first equality we push the skein up through an isotopy. The second equality is by definition of the left action, the third is a trick where nothing is changed, the fourth uses skein relations (see Fig.\,\ref{skeinRmatrix}) and the fifth uses \eqref{dinatSpecialMonogon}.}
\label{figProofBal}
\end{figure}
\end{proof}

\subsection{A presentation of \texorpdfstring{$\mathrm{Cob}$}{the cobordism category}}\label{sec:BPhopf}
Recall from \cite[Def.\ 2.4.1]{BBDP} that a {\em Bobtcheva--Piergallini Hopf algebra} (BP Hopf algebra for short) in a braided monoidal category $\mathcal{C}$ is a Hopf algebra object $\mathscr{H}$ in $\mathcal{C}$ endowed with an integral $\lambda : \mathscr{H} \to \boldsymbol{1}$, a cointegral $\Lambda : \boldsymbol{1} \to \mathscr{H}$, a copairing $w : \boldsymbol{1} \to \mathscr{H} \otimes \mathscr{H}$ and a so-called ribbon automorphism $\tau : \mathscr{H} \to \mathscr{H}$ which are subject to a list of 12 axioms. Moreover, the BP Hopf algebra $\mathscr{H}$ is called {\em factorizable} if $(\lambda \otimes \mathrm{id}) \circ w = \Lambda$ and is called {\em anomaly free} if $\lambda \circ \tau \circ \eta = \mathrm{id}_{\boldsymbol{1}}$ \cite[Def.\,2.6.1]{BBDP}.

The following was proved by Bobtcheva and Piergallini and reformulated as follows in \cite[Thm.\,C]{BBDP}:
\begin{theorem}\label{teo:presentation}
$\mathrm{Cob}$ is the braided category freely generated by an anomaly free, factorizable, BP-Hopf algebra object, namely the once punctured torus $\Sigma_1$. 
\end{theorem}
Stated differently, for any factorizable and anomaly free BP-Hopf algebra $\mathscr{H}$ in a braided monoidal category $\mathcal{C}$, there is a braided monoidal functor $F : \mathrm{Cob} \to \mathcal{C}$ uniquely determined on objects by the requirement $F(\Sigma_1) = \mathscr{H}$ and sending each morphism in the BP-Hopf structure of $\Sigma_1$ to the corresponding morphism in the BP-Hopf structure of $\mathscr{H}$. 

\begin{remark}
Theorem \ref{teo:presentation} was initially proved only in part by Kerler \cite{kerler} who could not find a complete set of relations satisfied by the punctured torus in $\mathrm{Cob}$. A few years later, Habiro announced a solution to the problem, and his presentation appeared in \cite{As11}.  A first complete proof appeared in \cite{BP}. In \cite{BBDP} a new proof appeared (and the name ``BP-Hopf algebra'' was introduced). 
\end{remark}

\begin{example}\label{exampleBPHopfEnd}
If  $\mathcal{C}$ is a braided monoidal category then $\boldsymbol{1}\in \mathcal{C}$ is a (trivial) anomaly free, factorizable BP-Hopf algebra object, when endowed with $\lambda=\mathrm{id}_{\boldsymbol{1}}=\Lambda=\tau=w$. In particular, $\underline{\End}(\boldsymbol{1}) = \boldsymbol{1}$ is an anomaly free, factorizable BP-Hopf algebra object in $\mathrm{Bim}^{\mathrm{hb}}_{\mathcal{C}}$.\end{example}
Recall the internal End algebras and their internal Hom bimodules discussed in length in \S\ref{sub:end}. Although Example \ref{exampleBPHopfEnd} is especially trivial, it implies the following:

\begin{proposition}\label{propKLFunctors}
Let $\mathcal{C}$ be a braided monoidal category. For all $V \in \mathcal{C}$ which has a left dual $V^*$ and such that $V = V' \oplus \boldsymbol{1}$, there is a braided monoidal functor $F_V : \mathrm{Cob} \to \mathrm{Bim}^{\mathrm{hb}}_{\mathcal{C}}$ uniquely defined on objects by $F_V(\Sigma_1) = \underline{\End}(V)$ and uniquely defined on morphisms by
\[ \begin{array}{c}
F_V(\mathrm{id}_{\Sigma_1}) = F_V(S) = F_V(\tau) = \underline{\Hom}(V,V), \quad F_V(m) = \underline{\Hom}(V \otimes V,V),\\[.3em]
F_V(\eta) = F_V(\Lambda) = \underline{\Hom}(\boldsymbol{1},V), \quad F_V(\Delta) = \underline{\Hom}(V,V \otimes V),\\[.3em]
F_V(\varepsilon) = F_V(\lambda) = \underline{\Hom}(V,\boldsymbol{1}), \quad F_V(w) = \underline{\Hom}(\boldsymbol{1}, V \otimes V).
\end{array} \]
\end{proposition}
\begin{proof}

Let $E = \underline{\Hom}(\boldsymbol{1},V) \in \Hom_{\mathrm{Bim}^{\mathrm{hb}}_{\mathcal{C}}}\bigl( \underline{\End}(V), \boldsymbol{1} \bigr)$, $N = \underline{\Hom}(V,\boldsymbol{1}) \in \Hom_{\mathrm{Bim}^{\mathrm{hb}}_{\mathcal{C}}}\bigl( \boldsymbol{1},  \underline{\End}(V) \bigr)$ and $I = \underline{\Hom}(V,V) = \mathrm{id}_{\underline{\End}(V)} \in \Hom_{\mathrm{Bim}^{\mathrm{hb}}_{\mathcal{C}}}\bigl( \underline{\End}(V),  \underline{\End}(V) \bigr)$. 
To prove the claim, we show that the morphism $E: \underline{\End}(\boldsymbol{1})\to \underline{\End}(V)$ is an isomorphism of BP-Hopf algebras in $\mathrm{Bim}^{\mathrm{hb}}_{\mathcal{C}}$ whose inverse is $N: \underline{\End}(V)\to \underline{\End}(\boldsymbol{1})$.
This is a direct consequence of the following isomorphisms of bimodules, which are implied by Lemma \ref{lemmaIntHomAndBrProd} :
$$I=N\circ E=E\otimes N,\qquad  \underline{\Hom}(V\otimes V,V)=E\otimes N\otimes N, $$
$$\underline{\Hom}(V,V\otimes V)=E\otimes E\otimes N,\qquad \underline{\Hom}(V\otimes V,\boldsymbol{1})=N\otimes N.$$
For instance the isomorphism $\underline{\Hom}(V\otimes V,V)=E\otimes N\otimes N$ exhibits the multiplication of $\underline{\End}(V)$ as $E\circ m_{\underline{\End}(\boldsymbol{1})}\circ (N\otimes N)$. 
\end{proof}

As a consequence of the above proposition, one can find plenty of functors $\mathrm{Cob}\to \mathrm{Bim}_{\mathcal{C}}^{\mathrm{hb}}$, albeit all naturally isomorphic to the trivial one. As we will see later, rather surprisingly, these functors $F_V$ are related to the  
Kerler-Lyubashenko TQFTs (in the case $V = \mathscr{L}$, the coend of $\mathcal{C}$).

\smallskip

As explained in \cite{BD} Theorem 7.4, another important example of functor which can be understood from the point of view of BP-Hopf algebras is the Kerler-Lyubashenko \cite{KL} functor:
\begin{theorem}[Kerler-Lyubashenko for end]
Let $\mathcal{U}$ be a finite dimensional factorizable ribbon Hopf algebra. Then there exists a TQFT (i.e. a braided balanced functor) sending $\Sigma_1\in \mathrm{Cob}^{\sigma}$ to $\mathrm{ad}\in \mathcal{U}\text{-}\mathrm{mod}$, where $\mathrm{ad}$ is the left adjoint representation, considered as a BP-Hopf algebra object. 
\end{theorem}
Here $\mathrm{Cob}^{\sigma}$ is a category of surfaces endowed with lagrangian subspaces of their first homology and extended $3$-manifolds as cobordisms. 
We refer to the account provided in \cite{BBDP} for details. We limit ourselves to recall that there is a ribbon monoidal forgetful functor $\mathrm{Forget}: \mathrm{Cob}^{\sigma}\to \mathrm{Cob}$, and that in \cite{BBDP} an analogue of Theorem \ref{teo:presentation} was proved for $\mathrm{Cob}^\sigma$ in which one only needs to drop the ``anomaly free condition'' on the factorizable BP-Hopf algebra object $\Sigma_1$.

By what was explained in Section \ref{sub:LlinQMM}, the coend $\mathscr{L}$ of $ \mathcal{U}\text{-}\mathrm{mod}$ (namely $\mathcal{U}$ with the coadjoint action) is isomorphic as a braided Hopf algebra object in $\mathcal{U}\text{-}\mathrm{mod}$ to the end, namely to 
$\mathrm{ad}$. The isomorphism is given by the Drinfeld map \eqref{DrinfeldMap} and through it we can pull back the full structure of BP-Hopf algebra object from $\mathrm{ad}$ to $\mathscr{L}$. We will not explicit the values of the integral, cointegral twist and copairing as our next arguments will not depend on them. 
We limit ourselves here to remark that we can redefine the KL TQFT by mapping $\Sigma_1$ to the left coadjoint representation and restate the above theorem as follows:

\begin{theorem}[Kerler-Lyubashenko for coend]\label{teo:KLcoend}
Let $\UU$ be a finite dimensional factorizable ribbon Hopf algebra. Then there exists a TQFT (i.e. a braided balanced functor) sending $\Sigma_1\in \mathrm{Cob}^{\sigma}$ to $\mathscr{L}\in \mathcal{U}\text{-}\mathrm{mod}$, where $\mathscr{L}$ is the left coadjoint representation, considered as a BP-Hopf algebra object. 
\end{theorem}

\subsection{Stated skein vs Kerler-Lyubashenko's functor}\label{subsec:StSkVsKL}
Assume that the coribbon Hopf algebra $\OO$ is {\em finite-dimensional}. Then $(\mathcal{O},\mathcal{R},\mathsf{v})$ arises as the dual of a finite-dimensional ribbon Hopf algebra $(\UU, R,\nu)$, see \S\ref{sub:ComodVsMod} for details. In this subsection we make the crucial hypothesis that $\UU$ is {\em factorizable} (\textit{i.e.} the Drinfeld map \eqref{DrinfeldMap} is an isomorphism). 

We prefer to work with $\UU$ instead of $\mathcal{O}$ in this section, as this will make the discussion a bit simpler. Recall that $\mathcal{C} = \mathrm{Comod}\text{-}\OO \cong \UU\text{-}\mathrm{Mod}$, and this isomorphism is a ribbon functor (\S\ref{sub:ComodVsMod}). It restricts to a ribbon isomorphism between the subcategories $\mathcal{C}^{\mathrm{fin}}$ of finite-dimensional objects. As a result stated skeins can be colored by $\UU\text{-}\mathrm{mod}$ instead of $\mathrm{comod}\text{-}\OO$ and the skein relations are obtained from the Reshetikhin--Turaev functor for $\UU\text{-}\mathrm{mod}$. We thus allow ourselves to write $\mathcal{S}_{\UU}$ instead of $\mathcal{S}_{\OO}$.

\smallskip

For all $V \in \UU\text{-}\mathrm{mod}$, recall the $\mathscr{L}$-linear structure on the internal End algebras $\underline{\End}(V) \in \UU\text{-}\mathrm{mod}$ from \eqref{UactionOnIntHom}--\eqref{LlinearStructIntEnd}, or equivalently the QMM version of this structure in \eqref{QMMonIntEnd}.
\begin{lemma}[``Alekseev isomorphism'']\label{lemmaAlekseevIso}
Under the present assumptions on $\UU$, for any $g \geq 1$ there is an isomorphism of $\mathscr{L}$-linear algebras in $\UU\text{-}\mathrm{mod}$
\[ \mathcal{S}_{\UU}(\Sigma_g) \cong \underline{\End}\bigl( \mathscr{L}^{\otimes g} \bigr) \]
where $\mathscr{L}$ is the coend $\int^{X \in \UU\text{-}\mathrm{mod}} X^* \otimes X$ described in \S\ref{sub:ComodVsMod}.
\end{lemma}
\begin{proof}
It suffices to prove this for $g=1$, because $\mathcal{S}_{\UU}(\Sigma_g) \cong \mathcal{S}_{\UU}(\Sigma_1)^{\widetilde{\otimes}\,g}$ and $\underline{\End}\bigl( \mathscr{L}^{\otimes g} \bigr) \cong \underline{\End}(\mathscr{L})^{\widetilde{\otimes}\, g}$ as $\mathscr{L}$-linear algebras (Lemma \ref{lemmaIntHomAndBrProd}). A direct proof for all $g$ is also possible, but more cumbersome. Let $M$ be a ``solid torus'', viewed as an element $M \in \Hom_{\mathrm{Cob}}(\varnothing, \Sigma_1)$. There is an $\UU$-morphism $J : \mathscr{L} \to \mathcal{S}_{\UU}(M)$ defined as follows, using the universal dinatural transformation $i$ of $\mathscr{L}$: 
\begin{equation}\label{cobCoend}
%% Creator: Inkscape 1.1.2 (0a00cf5339, 2022-02-04), www.inkscape.org
%% Accompanies image file 'cobCoend.pdf' (pdf, eps, ps)
\begingroup%
  \makeatletter%
  \providecommand\color[2][]{%
    \errmessage{(Inkscape) Color is used for the text in Inkscape, but the package 'color.sty' is not loaded}%
    \renewcommand\color[2][]{}%
  }%
  \providecommand\transparent[1]{%
    \errmessage{(Inkscape) Transparency is used (non-zero) for the text in Inkscape, but the package 'transparent.sty' is not loaded}%
    \renewcommand\transparent[1]{}%
  }%
  \providecommand\rotatebox[2]{#2}%
  \newcommand*\fsize{\dimexpr\f@size pt\relax}%
  \newcommand*\lineheight[1]{\fontsize{\fsize}{#1\fsize}\selectfont}%
  \ifx\svgwidth\undefined%
    \setlength{\unitlength}{269.09222169bp}%
    \ifx\svgscale\undefined%
      \relax%
    \else%
      \setlength{\unitlength}{\unitlength * \real{\svgscale}}%
    \fi%
  \else%
    \setlength{\unitlength}{\svgwidth}%
  \fi%
  \global\let\svgwidth\undefined%
  \global\let\svgscale\undefined%
  \makeatother%
  \begin{picture}(1,0.34893728)%
    \lineheight{1}%
    \setlength\tabcolsep{0pt}%
    \put(0,0){\includegraphics[width=\unitlength,page=1]{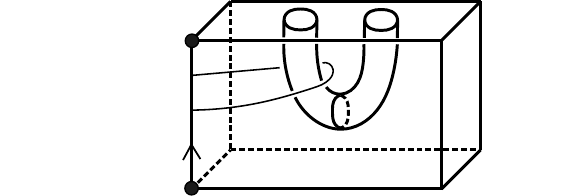}}%
    \put(0.30771623,0.20973747){\color[rgb]{0,0,0}\makebox(0,0)[lt]{\lineheight{1.25}\smash{\begin{tabular}[t]{l}$f$\end{tabular}}}}%
    \put(0,0){\includegraphics[width=\unitlength,page=2]{cobCoend.pdf}}%
    \put(0.30671088,0.14706315){\color[rgb]{0,0,0}\makebox(0,0)[lt]{\lineheight{1.25}\smash{\begin{tabular}[t]{l}$x$\end{tabular}}}}%
    \put(0.46536956,0.14282096){\color[rgb]{0,0,0}\makebox(0,0)[lt]{\lineheight{1.25}\smash{\begin{tabular}[t]{l}$_X$\end{tabular}}}}%
    \put(0.9419705,0.03580435){\color[rgb]{0,0,0}\makebox(0,0)[lt]{\lineheight{1.25}\smash{\begin{tabular}[t]{l}$\varnothing$\end{tabular}}}}%
    \put(0.9429431,0.3257375){\color[rgb]{0,0,0}\makebox(0,0)[lt]{\lineheight{1.25}\smash{\begin{tabular}[t]{l}$\Sigma_1$\end{tabular}}}}%
    \put(0,0){\includegraphics[width=\unitlength,page=3]{cobCoend.pdf}}%
    \put(0.98015428,0.18449239){\color[rgb]{0,0,0}\makebox(0,0)[lt]{\lineheight{1.25}\smash{\begin{tabular}[t]{l}(Cob)\end{tabular}}}}%
    \put(0,0){\includegraphics[width=\unitlength,page=4]{cobCoend.pdf}}%
    \put(-0.00104502,0.16812282){\color[rgb]{0,0,0}\makebox(0,0)[lt]{\lineheight{1.25}\smash{\begin{tabular}[t]{l}$J\bigl(i_X(f \otimes x) \bigr)=$\end{tabular}}}}%
  \end{picture}%
\endgroup%

\end{equation}
for all $X \in \UU\text{-}\mathrm{mod}$, $f \in X^*$ and $x \in X$. We claim that $J$ is an isomorphism of $\UU$-modules. First, using the skein relations (Fig.\,\ref{fig:SkeinRelation}), we see that any element in $\mathcal{S}_{\UU}(M)$ can be transformed to a linear combination of elements of the form \eqref{cobCoend}. Hence $J$ is surjective. Second, it is known by combining \cite[Th.\,6.5]{BFR} and \cite[Th.\,4.8]{F} that $\mathcal{S}_{\UU}(\Sigma_1) \cong \End_k(\mathscr{L})$ as algebras.\footnote{Actually \cite{BFR} uses the opposite convention for the product in stated skein algebras, but the isomorphism remains true within our convention since $\End_k(\mathscr{L})^{\mathrm{op}} \cong \End_k(\mathscr{L}^*) \cong \End_k(\mathscr{L})$ thanks to the transpose. Also note that the main point of the present proof is to show an isomorphism of \fbox{$\mathscr{L}$-linear} algebras. This could be deduced directly from previous works but the different conventions confuses things a bit. For convenience we decided to provide a proof within the present setting, which moreover gives a topological explanation of the Alekseev isomorphism.} Hence, since $\mathcal{S}_{\UU}(M)$ is a left $\mathcal{S}_{\UU}(\Sigma_1)$-module, the dimension of $\mathcal{S}_{\UU}(M)$ is necessarily a multiple of $\dim_k\mathscr{L}$ (see e.g. \cite[Th.\,3.3.1]{RepTh}). By surjectivity of $J$, two alternatives are left: $\dim_k\mathcal{S}_{\UU}(M) = 0$ or $\dim_k\mathcal{S}_{\UU}(M) = \dim_k \mathscr{L}$. The former is not possible because there exists an embedding $i:M \hookrightarrow B^3$ where the marked $3$-ball $B^3$ is the cobordism $\mathrm{Id}_{D^2}$, and we claim that $\mathcal{S}_{\UU}(B^3)=k$. Indeed the Reshetikhin-Turaev functor gives a morphism $\mathrm{RT}:\mathcal{S}_{\UU}(B^3)\to k$ by sending a skein $s$ whose state is $v$ to $RT(s)(v)\in k$. Then, the composition $\mathrm{RT}\circ i_*$  gives a non zero morphism to $k$ (the image of the empty skein is $1$). Hence $J$ identifies $\mathcal{S}_{\UU}(M)$ with $\mathscr{L}$ as $\UU$-modules.
\\By what has been said, we have an isomorphism of $\UU$-module-algebras
\[ \Psi : \mathcal{S}_{\UU}(\Sigma_1) \to \underline{\End}\bigl(   \mathcal{S}_{\UU}(M) \bigr) \cong \underline{\End}(\mathscr{L}), \quad \Psi(\Gamma)(\psi) = \Gamma \smallblacktriangleright \psi \]
where $\smallblacktriangleright$ is the left action of $\mathcal{S}_{\UU}(\Sigma_1)$ on $\mathcal{S}_{\UU}(M)$ and the target is the internal End because $\smallblacktriangleright$ is $\UU$-equivariant. It remains to show that $\Psi$ commutes with the $\mathscr{L}$-linear structure, \textit{i.e.} $\Psi \circ \mathfrak{d}_1 = \mathfrak{d}_2$, where $\mathfrak{d}_1 : \mathscr{L} \to \mathcal{S}_{\UU}(\Sigma_1)$ is defined in Fig.\,\ref{QMMstatedSkein} while $\mathfrak{d}_2 : \mathscr{L} \to \underline{\End}(\mathscr{L})$ is defined in \eqref{LlinearStructIntEnd}. First, it is readily seen that $\mathfrak{d}_1(\varphi) \smallblacktriangleright 1_{\mathscr{L}} = \varphi(1_{\UU})1_{\mathscr{L}}$ for all $\varphi \in \mathscr{L}$, where $1_{\mathscr{L}} \in \mathscr{L}$ is identified with the empty skein $\emptyset \in \mathcal{S}_{\UU}(M)$. Let $\psi \in \mathscr{L} = \mathcal{S}_{\UU}(M)$ and write $\psi = \Gamma \smallblacktriangleright \emptyset$ for some $\Gamma \in \mathcal{S}_{\UU}(\Sigma_1)$. Then for all $\varphi \in \mathscr{L}$ we have (using the $\mathscr{L}$-linearity of stated skein algebras)
\begin{align*}
\Psi\bigl( \mathfrak{d}_1(\varphi) \bigr)(\psi) &= \mathfrak{d}_1(\varphi) \smallblacktriangleright \psi = \mathfrak{d}_1(\varphi) \Gamma \smallblacktriangleright  1_{\mathscr{L}}\overset{\eqref{QMMGJSforComodAlg}}{=}  (R'_iR''_j \cdot \Gamma)\,\mathfrak{d}_1\bigl( R'_j \triangleright \varphi \triangleleft R''_i \bigr) \smallblacktriangleright  1_{\mathscr{L}}\\
&= \varphi(R''_iR'_j) (R'_iR''_j \cdot \Gamma) \smallblacktriangleright  1_{\mathscr{L}} \overset{\eqref{DrinfeldMap}}{=} \bigl( \Phi(\varphi) \cdot \Gamma \bigr) \smallblacktriangleright  1_{\mathscr{L}} = \Phi(\varphi) \cdot \bigl( \Gamma \smallblacktriangleright  1_{\mathscr{L}} \bigr) \overset{\eqref{LlinearStructIntEnd}}{=} \mathfrak{d}_2(\varphi)(\psi)
\end{align*}
where the last unlabelled equality uses that $1_{\mathscr{L}}$ is $\UU$-invariant and that $\smallblacktriangleright$ is $\UU$-equivariant. Since this is true for all $\psi$ we are done.
\end{proof}

\indent Recall from Cor.\,\ref{cor:correspondenceLmodvshalfbraided} that we can write indifferently $\mathrm{Bim}^{\mathrm{hb}}_{\mathcal{C}}$ or $\mathrm{Bim}^{\mathscr{L}}_{\mathcal{C}}$. The following theorem, which is the main result of this section, explains that for finite dimensional factorizable Hopf algebras the stated skeins are nothing else than the endomorphisms of the state spaces of the Kerler-Lyubashenko TQFT: 
\begin{theorem}\label{teo:commutativediagram}
Let $\UU$ be a finite dimensional, factorizable ribbon Hopf algebra, $\mathcal{C}$ be the category of finite dimensional $\UU$-modules and $\overline{\mathcal{C}}$ be the full (monoidal) subcategory of $\mathcal{C}$ consisting of those objects $V \in \mathcal{C}$ which have a left dual $V^*$ and have $\boldsymbol{1}$ as a direct summand.
Then  the image of the functor $KL$ is contained in $\overline{\mathcal{C}}$, and the following diagram of balanced monoidal functors commutes:
 \begin{equation*}
\xymatrix@C=4em{
\mathrm{Cob}^{\sigma} \ar[r]^-{KL} \ar[d]_-{\mathrm{Forget}} & \overline{\mathcal{C}}\ar[d]^-{\underline{\End}}\\
\mathrm{Cob} \ar[r]_-{\mathcal{S}_{\UU} = \mathcal{S}_{\OO}}  & \mathrm{Bim}_{\mathcal{C}}^{\mathscr{L}}
} \end{equation*}
where $\underline{\End}: \overline{\mathcal{C}}\to  \mathrm{Bim}_{\mathcal{C}}^{\mathscr{L}}$ was described in \eqref{internalEndFunctor} and $KL$ in Theorem \ref{teo:KLcoend}. 
\end{theorem}
\begin{proof}
By Lemma \ref{lemmaAlekseevIso} the diagram is commutative on the level of objects. 
Therefore, by Theorem \ref{teo:presentation}, it is sufficient to prove that the image of the structural morphisms of the BP-Hopf algebra $\Sigma_1$ through the functors $\underline{\End}\circ KL$ and $\mathcal{S}_{\UU}\circ \mathrm{Forget}$ are isomorphic bimodules. 
Observe that by Lemma \ref{lemmaUniqueBimodCoh}, it is sufficient to prove that the $k$-dimension of the stated skein module of each cobordism representing one of the structural morphism of $\Sigma_1$ is positive and is less than or equal to the dimension of the corresponding bimodule through $\underline{\End}\circ KL$. 
This is particularly simple to check for the morphisms $\eta,\Lambda:\Hom_{\mathrm{Cob}}(\boldsymbol{1},\Sigma_1)$ and $\epsilon,\lambda:\Hom_{\mathrm{Cob}}(\Sigma_1,\boldsymbol{1})$ which are all represented by a marked $3$-manifold $M$ homeomorphic to a solid torus and with a single oriented arc $\mathcal{N}$ in the boundary, differing only in the orientation of $M$ and in the identification of $\partial M$ with $\Sigma_1$. 
In all these cases, by arguments similar to the ones in the proof of Lemma \ref{lemmaAlekseevIso}, we have that $\mathcal{S}_{\UU}(M)$ is isomorphic to $\mathscr{L}=KL(\Sigma_1)$ as an $\UU$-module and so it has the same dimension as $\underline{\Hom}(\boldsymbol{1},KL(\Sigma_1))\simeq \mathscr{L}$ or $\underline{\Hom}(KL(\Sigma_1),\boldsymbol{1})\simeq \mathscr{L}^*$ (where the isomorphisms are intended as $k$-vector spaces). Thanks to $\mathscr{L}$-coherence (item 2 in Prop.\,\ref{prop:skeinLlinear} ) and to Lemma \ref{lemmaUniqueBimodCoh}, these isomorphisms are also $\UU$-module isomorphisms. We stress here that the functor $\underline{\End}$ is very rough on morphisms : two morphisms $f,g:V\to W$ are send by it to the same bimodule $\underline{\Hom}(V,W)$. This is the reason why we did not describe explicitly the BP-Hopf algebra structural morphisms of $\mathscr{L}$ in Subsection \ref{sec:BPhopf}: the precise definition of these morphisms has no importance for the present proof.

The remaining structural morphisms are the product $m:\Sigma_1\otimes \Sigma_1\to \Sigma_1$, the co-product $\Delta:\Sigma_1\to   \Sigma_1\otimes \Sigma_1$, the antipode and twist $S,\tau:\Sigma_1\to \Sigma_1$ and the pairing $w:\Sigma_1\otimes \Sigma_1\to \boldsymbol{1}$. 
We show the claimed inequality for the product, the coproduct being similar and $\tau,S$ being easier. So we have to prove that $$1\leq \dim_k \mathcal{S}_{\UU}(m)\leq \dim_k \underline{\Hom}\bigl(\mathscr{L}^{\otimes 2},\mathscr{L}\bigr)=(\dim_k \mathscr{L})^3.$$
To show the right inequality, we prove that a generating family of skeins in $\mathcal{S}_{\UU}(m)$ are those of the three kinds shown in the lower right part of Figure \ref{fig:skeingenerators}.
First of all observe that $m$ is split by the vertical plane $\pi$ depicted in the figure in two ``sides", say $m_+$ and $m_-$ and that, up to isotopy, we can suppose that each skein is almost entirely contained in $m_+$ except for some arcs encircling the upper torus only along the leftmost annulus as shown in the upper right part of Figure \ref{fig:skeingenerators}.  Then we can use the skein relations (Fig.\,\ref{fig:SkeinRelation}) in order to ``push all the topology" of the skeins out of the cube: in the upper right part of Figure \ref{fig:skeingenerators} we represent the part to be pushed out by a white box. In this way we are left with linear combinations of skeins which consist only of the three types of parallel strands displayed in the lower left part of Fig.\,\ref{fig:skeingenerators}.
Finally by arguing as in \eqref{dinatSpecialMonogon} the state of each component of the skein can be supposed to be a linear combination of states of the form $\varepsilon_{\UU^*} \otimes \omega_i$ where $\omega_1,\ldots \omega_d$ is a fixed basis of $\UU^*$ and $\varepsilon_{\UU^*} : \UU^* \to k$ is the counit.
This proves that $\dim_k \mathcal{S}_{\UU}(m)\leq \dim_k \underline{\Hom}\bigl(\mathscr{L}^{\otimes 2},\mathscr{L}\bigr)=(\dim_k \mathscr{L})^3.$
To prove that $1\leq \dim_k \mathcal{S}_{\UU}(m)$, it is sufficient to show that  $\mathcal{S}_{\UU}(m)$ is non zero.  For this we use the same argument as in the proof of Lemma \ref{lemmaAlekseevIso}: there exists an embedding $i:m\hookrightarrow B^3$ where the marked $3$-ball $B^3$ is the cobordism $\mathrm{Id}_{D^2}$, and the Reshetikhin-Turaev functor gives a non zero morphism to $\mathcal{S}_{\UU}(B^3) \to k$ (the image of the empty skein is $1$). 
By Lemma \ref{lemmaUniqueBimodCoh} this isomorphism is an $\UU$-module isomorphism and the proof is complete. 
\end{proof}

\begin{figure}[h!]
\centering
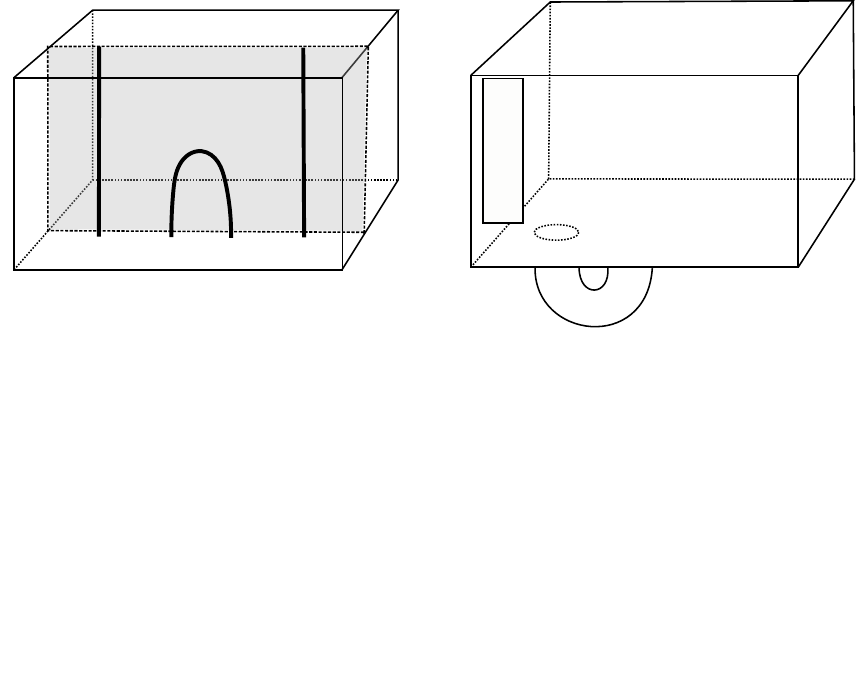
\caption{The cobordism $m$ is the complement of a tubular neighborhood of the thick black path in the genus $2$-handlebody. In the upper left figure, the plane $\pi$ (depicted in grey) divides $m$ into two submanifolds $m_\pm$: up to isotopy we can make sure that each skein is almost entirely contained in $m_+$ except for some small arcs linking as shown on the upper right figure, in which the coupon stands for a complicated ribbon graph $G$ (to get convinced of this statement, one should isotope the links along the black path so to push all the linking in the upper left part of it). Up to pushing $G$ out of the boundary arc thanks to defining relation in Fig.\,\ref{fig:SkeinRelation}, we can reduce to a skein as in the lower left figure. Finally arguing as in \eqref{dinatSpecialMonogon}, we can reduce each of the three families of skeins into a linear combination of skeins as those in the lower right figure.}
\label{fig:skeingenerators}
\end{figure}

\appendix

\section{Natural transformations for LFP categories}
The concept of {\em LFP category} has been recalled at the beginning of \S\ref{sectionLlinear}. We denote by $\mathcal{C}_{\mathrm{cp}}$ the full subcategory of compact objects. The goal of this appendix is to show that a natural transformation $F \Rightarrow G$ is uniquely determined by its values on compact objects, provided that $F$ is cocontinuous. Precisely:

\begin{lemma}\label{lemmaExtensionNat}
Let $\mathcal{C}$ be a LFP category, $\mathcal{D}$ be any category and $F,G : \mathcal{C} \to \mathcal{D}$ be functors. Assume that we are given a natural transformation $\alpha = \bigl( \alpha_K : F(K) \to G(K) \bigr)_{K \in \mathcal{C}_{\mathrm{cp}}}$ defined on the subcategory of compact objects.
\\1. If $F$ is cocontinuous, then there exists a unique natural transformation $\widetilde{\alpha} : F \Rightarrow G$ such that $\widetilde{\alpha}_K = \alpha_K$ for all $K \in \mathcal{C}_{\mathrm{cp}}$.
\\2. If moreover $\alpha$ is an isomorphism, $G$ is cocontinuous and $\mathcal{D}$ is LFP, then $\widetilde{\alpha}$ is an isomorphism.
\end{lemma}
\begin{proof}
1. Let $C \in \mathcal{C}$ be an arbitrary object. By definition of a LFP category there exists a filtered category $\mathcal{I}$ and a functor $K : \mathcal{I} \to \mathcal{C}_{\mathrm{cp}}$ such that $C = \mathrm{colim} \, K$. Denote by $\phi = \bigl( \phi_X : K(X) \to C \bigr)_{X \in \mathcal{I}}$ the universal cocone of $K$. Cocontinuity means that $F(\phi) = \bigl( F(\phi_X) : FK(X) \to F(C) \bigr)_{X \in \mathcal{I}}$ is the universal cocone of $FK$. Note that $\bigl( G(\phi_X) \circ \alpha_{K(X)} : FK(X) \to G(C) \bigr)_{X \in \mathcal{I}}$ is also a cocone of $FK$; indeed for all morphism $f \in \Hom_{\mathcal{I}}(X,Y)$ we have
\[ G(\phi_{Y}) \circ \alpha_{K(Y)} \circ FK(f) = G(\phi_{Y}) \circ GK(f) \circ \alpha_{K(X)} = G(\phi_X) \circ \alpha_{K(X)}. \]
By the universal property of $F(\phi)$, we thus have a factorization
\begin{equation}\label{defExtension}
\xymatrix@C=4em@R=1.5em{
FK(X) \ar[r]^{\alpha_{K(X)}} \ar[d]_{F(\phi_X)} & GK(X) \ar[d]^{G(\phi_X)}\\
F(C) \ar[r]_{\exists!\, \widetilde{\alpha}_C} & G(C)
} \end{equation}
which holds for all $X \in \mathcal{I}$. We must show that $\widetilde{\alpha}_C$ does not depend on the choice of a colimit presentation of $C$. So let $\mathcal{I}'$ be another filtered category and $K' : \mathcal{I}' \to \mathcal{C}_{\mathrm{cp}}$ such that $C = \mathrm{colim}\, K'$, with universal cocone $\phi' = \bigl( \phi'_{X'} : K'(X') \to C \bigr)_{X' \in \mathcal{I}'}$. Similarly to $\widetilde{\alpha}_C$ in \eqref{defExtension}, let $\widetilde{\alpha}'_C : F(V) \to G(V)$ be defined by $\widetilde{\alpha}'_C \circ F(\phi'_{X'}) = G(\phi'_{X'}) \circ \alpha_{K'(X')}$ for all $X' \in \mathcal{I}'$. By Lemma \ref{lemmaFactoCompactObjects}, the morphism $\phi'_{X'} : K'(X') \to C$ can be factored as $\phi'_{X'} = \phi_X \circ g$ for some $X \in \mathcal{I}$ and $g : K'(X') \to K(X)$. But then we can use the naturality of $\alpha$ to obtain:
\begin{align*}
\widetilde{\alpha}'_C \circ F(\phi'_{X'}) = G(\phi'_{X'}) \circ \alpha_{K'(X')} = G(\phi_X) \circ G(g) \circ \alpha_{K'(X')} &= G(\phi_X) \circ \alpha_{K(X)} \circ F(g)\\
&= \widetilde{\alpha}_C \circ F(\phi_X) \circ F(g)= \widetilde{\alpha}_C \circ F(\phi'_{X'}).
\end{align*}
Since this is true for all $X' \in \mathcal{I}'$, we conclude that $\widetilde{\alpha}'_C = \widetilde{\alpha}_C$ because $F(\phi')$ is a universal cocone. We now prove naturality of $\widetilde{\alpha}$. So let $f \in \Hom_{\mathcal{C}}(C_1,C_2)$ and write $C_i = \mathrm{colim} \bigl( K_i : \mathcal{I}_i \to \mathcal{C}_{\mathrm{cp}} \bigr)$ with universal cocone $\phi^i = \bigl( \phi^i_X : K_i(X) \to C_i \bigr)_{X \in \mathcal{I}_i}$. By Lemma \ref{lemmaFactoCompactObjects}, the morphism $f \circ \phi^1_X : K_1(X) \to C_1$ can be factored as $f \circ \phi^1_{X} = \phi^2_{Y} \circ g$ for some $Y \in \mathcal{I}_2$ and $g : K_1(X) \to K_2(Y)$. Then, using this factorization and the naturality of $\alpha$, we get
\begin{align*}
G(f) \circ \widetilde{\alpha}_{C_1} \circ F(\phi^1_X) &\overset{\eqref{defExtension}}{=} G(f) \circ G(\phi^1_X) \circ \alpha_{K_1(X)} = G(\phi^2_Y) \circ G(g) \circ \alpha_{K_1(X)}\\
&\:= G(\phi^2_Y) \circ \alpha_{K_2(Y)} \circ F(g) = \widetilde{\alpha}_{C_2} \circ F(\phi^2_Y) \circ F(g) = \widetilde{\alpha}_{C_2} \circ F(f) \circ F(\phi^1_X).
\end{align*}
Since this is true for all $X \in \mathcal{I}_1$, we conclude that $G(f) \circ \widetilde{\alpha}_{C_1} = \widetilde{\alpha}_{C_2} \circ F(f)$ because $F(\phi^1)$ is a universal cocone.
\\2. By the same arguments which led to \eqref{defExtension}, we have a commutative diagram
\[ \xymatrix@C=4em@R=1.5em{
GK(X) \ar[r]^{\alpha^{-1}_{K(X)}} \ar[d]_{G(\phi_X)} & FK(X) \ar[d]^{F(\phi_X)}\\
G(C) \ar[r]_{\exists!} & F(C)
} \]
which holds for all $X \in \mathcal{I}$. The unique arrow on the bottom is nothing but the inverse of $\widetilde{\alpha}_C$.
\end{proof}

\section{Relating \texorpdfstring{$\mathscr{L}$}{coend}-modules and the Drinfeld center \texorpdfstring{$\mathcal{Z}(\mathcal{C})$}{}}\label{subsectionIsoZCLModC}
In this appendix, $\mathcal{C}$ is a category which satisfies assumptions \eqref{assumptionsCategoryC}. The goal is to show that each $\mathscr{L}$-module is naturally endowed with a half-braiding and that this construction yields an isomorphism of categories between $\mathscr{L}\text{-}\mathrm{mod}_{\mathcal{C}}$ and $\mathcal{Z}(\mathcal{C})$. This fact is well-known for finite braided tensor categories $\mathcal{C}$; the point here is to establish it under the assumptions \eqref{assumptionsCategoryC}. We also provide explicit details in the case $\mathcal{C} = \mathrm{Comod}\text{-}H$.

\smallskip

\indent Recall that we denote by $\mathcal{C}_{\mathrm{cp}}$ the full subcategory of compact objects (Def.\,\ref{defCompact}).

\begin{lemma}\label{lemmaHBonCompacts}
1. A half-braiding is uniquely determined by its values on compact objects. More precisely, if $V \in \mathcal{C}$ and $u = \bigl( u_K : V \otimes K \overset{\sim}{\to} K \otimes V \bigr)_{K \in \mathcal{C}_{\mathrm{cp}}}$ is a natural family of isomorphisms which satisfies the half-braiding property \eqref{axiomHalfBraiding}, then there exists a unique half-braiding $t : V \otimes - \overset{\sim}{\implies} - \otimes V$ such that $t_K = u_K$ for all $K \in \mathcal{C}_{\mathrm{cp}}$.
\\2. Let $(V,t), (V',t') \in \mathcal{Z}(\mathcal{C})$ and $f \in \Hom_{\mathcal{C}}(V,V')$. If
\[ \forall \, K \in \mathcal{C}_{\mathrm{cp}}, \quad t_K \circ (f \otimes \mathrm{id}_K) = (\mathrm{id}_K \otimes f) \circ t'_K \]
then $f \in \Hom_{\mathcal{Z}(\mathcal{C})}\bigl( (V,t),(V',t') \bigr)$, i.e. this equality holds for any $K \in \mathcal{C}$.
\end{lemma}
\begin{proof}
1. Let $C \in \mathcal{C}$ be an arbitrary object of $\mathcal{C}$, which we write as $C = \mathrm{colim} \bigl(K : \mathcal{I} \to \mathcal{C}_{\mathrm{cp}} \bigr)$ with $\mathcal{I}$ filtered. Denote by $\phi = \bigl( \phi_X : K(X) \to C \bigr)_{X \in \mathcal{I}}$ the universal cocone. By assumption on $\mathcal{C}$, the functors $V \otimes -$ and $- \otimes V$ are cocontinuous. Hence, by Lemma \ref{lemmaExtensionNat}, there is a unique natural isomorphism $t : V \otimes - \overset{\sim}{\implies} - \otimes V$ defined by
\begin{equation}\label{defExtendedHalfBraiding}
\xymatrix@R=.4em@C=1.5em{
V \otimes K(X) \ar[rr]^{u_{K(X)}} \ar[dd]_-{\mathrm{id}_V \otimes \phi_X} & & K(X) \otimes V \ar[dd]^-{\phi_X \otimes \mathrm{id}_V}\\
&\circlearrowright&\\
V \otimes C \ar[rr]_{t_C} & & C \otimes V.
}
\end{equation}
for all $X \in \mathcal{I}$. It remains to check that $t$ satisfies the half-braiding property \eqref{axiomHalfBraiding}. So let $C'$ be another object in $\mathcal{C}$, written as $C' = \mathrm{colim} \bigl( K' : \mathcal{I}' \to \mathcal{C}_{\mathrm{cp}} \bigr)$ with universal cocone $\phi'$. Recall that for any functor $T : \mathcal{I} \times \mathcal{I}' \to \mathcal{C}$ we have
\[ \underset{(X,X') \in \mathcal{I} \times \mathcal{I}'}{\mathrm{colim}} T(X,X') \,\cong\, \underset{X \in \mathcal{I}}{\mathrm{colim}} \bigl( \underset{X' \in \mathcal{I}'}{\mathrm{colim}} \:T(X,X') \bigr). \]
Applying this to the functor $T(X,X') = K(X) \otimes K'(X')$, we get
\begin{equation}\label{colimPresOfMonProd}
\underset{(X,X') \in \mathcal{I} \times \mathcal{I}'}{\mathrm{colim}} \: K(X) \otimes K'(X) = (C \otimes C', \, \phi \otimes \phi')
\end{equation}
because $\otimes$ is cocontinuous in each variable. Since $\mathcal{I} \times \mathcal{I}'$ is filtered and the objects $K(X) \otimes K'(X')$ are compact (Lemma \ref{lemmaCompactStableByMonoidalProduct}) we can use $\phi \otimes \phi'$ to compute $t_{C \otimes C'}$ through \eqref{defExtendedHalfBraiding}, thus obtaining:
\begin{align*}
t_{C \otimes C'} \circ (\mathrm{id}_V \otimes \phi_X \otimes \phi'_{X'}) &= (\phi_X \otimes \phi'_{X'} \otimes \mathrm{id}_C) \circ u_{K(X) \,\otimes\, K'(X')}\\
&= (\phi_X \otimes \phi'_{X'} \otimes \mathrm{id}_C) \circ (\mathrm{id}_{K(X)} \otimes u_{K'(X')}) \circ (u_{K(X)} \otimes \mathrm{id}_{K'(X')})\\
&= (\mathrm{id}_{K(X)} \otimes t_{C'}) \circ (\phi_X \otimes \mathrm{id}_C \otimes \phi'_{X'}) \circ (u_{K(X)} \otimes \mathrm{id}_{K'(X')})\\
&= (\mathrm{id}_C \otimes t_{C'}) \circ (t_{C} \otimes \mathrm{id}_{C'}) \circ (\mathrm{id}_C \otimes \phi_X \otimes \phi'_{X'})
\end{align*}
for all $(X,X') \in \mathcal{I} \times \mathcal{I}'$, which gives the desired equality by universality of the cocone $\mathrm{id}_C \otimes \phi \otimes \phi'$.
\\2. Let $C \in \mathcal{C}$ and take again a filtered colimit presentation $(C,\phi) = \mathrm{colim} \bigl(K : \mathcal{I} \to \mathcal{C}_{\mathrm{cp}} \bigr)$. Then using the naturality of $t$ and the assumption we have
\begin{align*}
&t_C \circ (f \otimes \mathrm{id}_C) \circ (\mathrm{id}_V \otimes \phi_X) = t_C \circ (\mathrm{id}_{V'} \otimes \phi_X) \circ (f \otimes \mathrm{id}_{K(X)})\\
=\:& (\phi_X \circ \mathrm{id}_{V'}) \circ t_{K(X)} \circ (f \otimes \mathrm{id}_{K(X)}) = (\phi_X \circ \mathrm{id}_{V'}) \circ (\mathrm{id}_{K(X)} \otimes f) \circ t'_{K(X)}\\
&= (\mathrm{id}_C \otimes f) \circ (\phi \otimes \mathrm{id}_V) \circ t'_{K(X)} = (\mathrm{id}_C \otimes f) \circ t'_C \circ (\mathrm{id}_V \otimes \phi_X)
\end{align*}
which gives the desired equality by universality of $\mathrm{id}_V \otimes \phi_X$ as cocone for $V \otimes -$.
\end{proof}

\indent Note that $\mathscr{L}\text{-}\mathrm{mod}_{\mathcal{C}}$ is a monoidal category, whose monoidal product comes from the coproduct \eqref{defStructureCoend} of the Hopf algebra $\mathscr{L}$ in $\mathcal{C}$.
\begin{proposition}\label{propIsoZCandLModC}
Let $\Upsilon : \mathscr{L}\text{-}\mathrm{mod}_{\mathcal{C}} \to \mathcal{Z}(\mathcal{C})$ be the functor defined by $\Upsilon(V, \lambda) = \bigl(V,\, t(\lambda)\bigr)$ and $\Upsilon(f) = f$, where
\begin{center}
%% Creator: Inkscape 1.1.2 (0a00cf5339, 2022-02-04), www.inkscape.org
%% Accompanies image file 'halfBraidingFromAction.pdf' (pdf, eps, ps)
\begingroup%
  \makeatletter%
  \providecommand\color[2][]{%
    \errmessage{(Inkscape) Color is used for the text in Inkscape, but the package 'color.sty' is not loaded}%
    \renewcommand\color[2][]{}%
  }%
  \providecommand\transparent[1]{%
    \errmessage{(Inkscape) Transparency is used (non-zero) for the text in Inkscape, but the package 'transparent.sty' is not loaded}%
    \renewcommand\transparent[1]{}%
  }%
  \providecommand\rotatebox[2]{#2}%
  \newcommand*\fsize{\dimexpr\f@size pt\relax}%
  \newcommand*\lineheight[1]{\fontsize{\fsize}{#1\fsize}\selectfont}%
  \ifx\svgwidth\undefined%
    \setlength{\unitlength}{197.74159417bp}%
    \ifx\svgscale\undefined%
      \relax%
    \else%
      \setlength{\unitlength}{\unitlength * \real{\svgscale}}%
    \fi%
  \else%
    \setlength{\unitlength}{\svgwidth}%
  \fi%
  \global\let\svgwidth\undefined%
  \global\let\svgscale\undefined%
  \makeatother%
  \begin{picture}(1,0.41767077)%
    \lineheight{1}%
    \setlength\tabcolsep{0pt}%
    \put(0,0){\includegraphics[width=\unitlength,page=1]{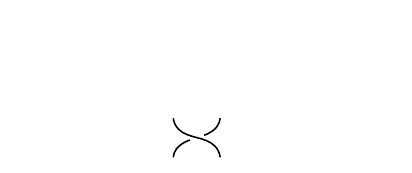}}%
    \put(0.39952617,0.00656583){\color[rgb]{0,0,0}\makebox(0,0)[lt]{\lineheight{1.25}\smash{\begin{tabular}[t]{l}$_V$\end{tabular}}}}%
    \put(0,0){\includegraphics[width=\unitlength,page=2]{halfBraidingFromAction.pdf}}%
    \put(0.44455522,0.39651881){\color[rgb]{0,0,0}\makebox(0,0)[lt]{\lineheight{1.25}\smash{\begin{tabular}[t]{l}$_V$\end{tabular}}}}%
    \put(0,0){\includegraphics[width=\unitlength,page=3]{halfBraidingFromAction.pdf}}%
    \put(0.4431897,0.28183325){\color[rgb]{0,0,0}\makebox(0,0)[lt]{\lineheight{1.25}\smash{\begin{tabular}[t]{l}$\lambda$\end{tabular}}}}%
    \put(0,0){\includegraphics[width=\unitlength,page=4]{halfBraidingFromAction.pdf}}%
    \put(-0.00091627,0.18860618){\color[rgb]{0,0,0}\makebox(0,0)[lt]{\lineheight{1.25}\smash{\begin{tabular}[t]{l}$t(\lambda)_K =$\end{tabular}}}}%
    \put(0,0){\includegraphics[width=\unitlength,page=5]{halfBraidingFromAction.pdf}}%
    \put(0.35983955,0.156312){\color[rgb]{0,0,0}\makebox(0,0)[lt]{\lineheight{1.25}\smash{\begin{tabular}[t]{l}$i_K$\end{tabular}}}}%
    \put(0,0){\includegraphics[width=\unitlength,page=6]{halfBraidingFromAction.pdf}}%
    \put(0.51223366,0.01038053){\color[rgb]{0,0,0}\makebox(0,0)[lt]{\lineheight{1.25}\smash{\begin{tabular}[t]{l}$_K$\end{tabular}}}}%
    \put(0,0){\includegraphics[width=\unitlength,page=7]{halfBraidingFromAction.pdf}}%
    \put(0.21901002,0.39392274){\color[rgb]{0,0,0}\makebox(0,0)[lt]{\lineheight{1.25}\smash{\begin{tabular}[t]{l}$_K$\end{tabular}}}}%
    \put(0.73962303,0.1881944){\color[rgb]{0,0,0}\makebox(0,0)[lt]{\lineheight{1.25}\smash{\begin{tabular}[t]{l}$(\forall \, K \in \mathcal{C}_{\mathrm{cp}})$\end{tabular}}}}%
  \end{picture}%
\endgroup%

\end{center}
(these values uniquely define $t(\lambda)$ by item 1 in Lemma \ref{lemmaHBonCompacts}). Then $\Upsilon$ is an isomorphism of monoidal categories which satisfies $\Upsilon(\mathbf{V} \otimes \mathbf{W}) = \Upsilon(\mathbf{V}) \otimes \Upsilon(\mathbf{W})$.
\end{proposition}
\begin{proof}
Naturality is clear, by dinaturality of $i$. The half-braiding axiom \eqref{axiomHalfBraiding} is checked by a straightforward graphical computation. Moreover $t(\lambda)$ is an isomorphism, with inverse
\begin{center}
%% Creator: Inkscape 1.1.2 (0a00cf5339, 2022-02-04), www.inkscape.org
%% Accompanies image file 'inverseOfHalfBraidingFromAction.pdf' (pdf, eps, ps)
\begingroup%
  \makeatletter%
  \providecommand\color[2][]{%
    \errmessage{(Inkscape) Color is used for the text in Inkscape, but the package 'color.sty' is not loaded}%
    \renewcommand\color[2][]{}%
  }%
  \providecommand\transparent[1]{%
    \errmessage{(Inkscape) Transparency is used (non-zero) for the text in Inkscape, but the package 'transparent.sty' is not loaded}%
    \renewcommand\transparent[1]{}%
  }%
  \providecommand\rotatebox[2]{#2}%
  \newcommand*\fsize{\dimexpr\f@size pt\relax}%
  \newcommand*\lineheight[1]{\fontsize{\fsize}{#1\fsize}\selectfont}%
  \ifx\svgwidth\undefined%
    \setlength{\unitlength}{122.40763816bp}%
    \ifx\svgscale\undefined%
      \relax%
    \else%
      \setlength{\unitlength}{\unitlength * \real{\svgscale}}%
    \fi%
  \else%
    \setlength{\unitlength}{\svgwidth}%
  \fi%
  \global\let\svgwidth\undefined%
  \global\let\svgscale\undefined%
  \makeatother%
  \begin{picture}(1,0.70535541)%
    \lineheight{1}%
    \setlength\tabcolsep{0pt}%
    \put(0,0){\includegraphics[width=\unitlength,page=1]{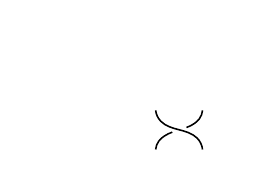}}%
    \put(0.57631902,0.01060667){\color[rgb]{0,0,0}\makebox(0,0)[lt]{\lineheight{1.25}\smash{\begin{tabular}[t]{l}$_V$\end{tabular}}}}%
    \put(0,0){\includegraphics[width=\unitlength,page=2]{inverseOfHalfBraidingFromAction.pdf}}%
    \put(0.64906064,0.67118578){\color[rgb]{0,0,0}\makebox(0,0)[lt]{\lineheight{1.25}\smash{\begin{tabular}[t]{l}$_V$\end{tabular}}}}%
    \put(0,0){\includegraphics[width=\unitlength,page=3]{inverseOfHalfBraidingFromAction.pdf}}%
    \put(0.6468547,0.48591872){\color[rgb]{0,0,0}\makebox(0,0)[lt]{\lineheight{1.25}\smash{\begin{tabular}[t]{l}$\lambda$\end{tabular}}}}%
    \put(0,0){\includegraphics[width=\unitlength,page=4]{inverseOfHalfBraidingFromAction.pdf}}%
    \put(-0.00148017,0.33481397){\color[rgb]{0,0,0}\makebox(0,0)[lt]{\lineheight{1.25}\smash{\begin{tabular}[t]{l}$t(\lambda)^{-1}_K =$\end{tabular}}}}%
    \put(0,0){\includegraphics[width=\unitlength,page=5]{inverseOfHalfBraidingFromAction.pdf}}%
    \put(0.48430242,0.3137826){\color[rgb]{0,0,0}\makebox(0,0)[lt]{\lineheight{1.25}\smash{\begin{tabular}[t]{l}$i_{^*\!K}$\end{tabular}}}}%
    \put(0,0){\includegraphics[width=\unitlength,page=6]{inverseOfHalfBraidingFromAction.pdf}}%
    \put(0.95868454,0.666992){\color[rgb]{0,0,0}\makebox(0,0)[lt]{\lineheight{1.25}\smash{\begin{tabular}[t]{l}$_K$\end{tabular}}}}%
    \put(0,0){\includegraphics[width=\unitlength,page=7]{inverseOfHalfBraidingFromAction.pdf}}%
    \put(0.45377762,0.01060667){\color[rgb]{0,0,0}\makebox(0,0)[lt]{\lineheight{1.25}\smash{\begin{tabular}[t]{l}$_K$\end{tabular}}}}%
    \put(0,0){\includegraphics[width=\unitlength,page=8]{inverseOfHalfBraidingFromAction.pdf}}%
  \end{picture}%
\endgroup%

\end{center}
where we use that $^*(K^*) = K$. The strict monoidality of $\Upsilon$ is easily checked from the definitions of monoidal product in $\mathcal{Z}(\mathcal{C})$ (see \eqref{defTensorProductInZC}) and in $\mathscr{L}\text{-}\mathrm{mod}_{\mathcal{C}}$. Finally we have $\Upsilon^{-1}(V,t) = \bigl( V, \lambda(t) \bigr)$, with $\lambda(t) : \mathscr{L} \otimes V \to V$ defined by
\begin{equation}\label{actionFromHalfBraiding}
%% Creator: Inkscape 1.1.2 (0a00cf5339, 2022-02-04), www.inkscape.org
%% Accompanies image file 'actionFromHalfBraiding.pdf' (pdf, eps, ps)
\begingroup%
  \makeatletter%
  \providecommand\color[2][]{%
    \errmessage{(Inkscape) Color is used for the text in Inkscape, but the package 'color.sty' is not loaded}%
    \renewcommand\color[2][]{}%
  }%
  \providecommand\transparent[1]{%
    \errmessage{(Inkscape) Transparency is used (non-zero) for the text in Inkscape, but the package 'transparent.sty' is not loaded}%
    \renewcommand\transparent[1]{}%
  }%
  \providecommand\rotatebox[2]{#2}%
  \newcommand*\fsize{\dimexpr\f@size pt\relax}%
  \newcommand*\lineheight[1]{\fontsize{\fsize}{#1\fsize}\selectfont}%
  \ifx\svgwidth\undefined%
    \setlength{\unitlength}{216.9578121bp}%
    \ifx\svgscale\undefined%
      \relax%
    \else%
      \setlength{\unitlength}{\unitlength * \real{\svgscale}}%
    \fi%
  \else%
    \setlength{\unitlength}{\svgwidth}%
  \fi%
  \global\let\svgwidth\undefined%
  \global\let\svgscale\undefined%
  \makeatother%
  \begin{picture}(1,0.31153933)%
    \lineheight{1}%
    \setlength\tabcolsep{0pt}%
    \put(0,0){\includegraphics[width=\unitlength,page=1]{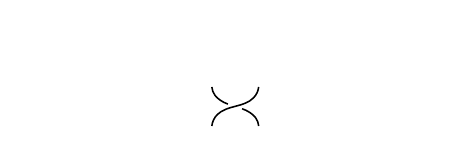}}%
    \put(0.44941617,0.00598428){\color[rgb]{0,0,0}\makebox(0,0)[lt]{\lineheight{1.25}\smash{\begin{tabular}[t]{l}$_K$\end{tabular}}}}%
    \put(0,0){\includegraphics[width=\unitlength,page=2]{actionFromHalfBraiding.pdf}}%
    \put(0.55959481,0.27497635){\color[rgb]{0,0,0}\makebox(0,0)[lt]{\lineheight{1.25}\smash{\begin{tabular}[t]{l}$_V$\end{tabular}}}}%
    \put(0,0){\includegraphics[width=\unitlength,page=3]{actionFromHalfBraiding.pdf}}%
    \put(0.50649678,0.13587969){\color[rgb]{0,0,0}\makebox(0,0)[lt]{\lineheight{1.25}\smash{\begin{tabular}[t]{l}$t_K$\end{tabular}}}}%
    \put(0,0){\includegraphics[width=\unitlength,page=4]{actionFromHalfBraiding.pdf}}%
    \put(0.05027081,0.09061386){\color[rgb]{0,0,0}\makebox(0,0)[lt]{\lineheight{1.25}\smash{\begin{tabular}[t]{l}$i_K$\end{tabular}}}}%
    \put(0.55214106,0.00946109){\color[rgb]{0,0,0}\makebox(0,0)[lt]{\lineheight{1.25}\smash{\begin{tabular}[t]{l}$_V$\end{tabular}}}}%
    \put(0,0){\includegraphics[width=\unitlength,page=5]{actionFromHalfBraiding.pdf}}%
    \put(0.76268493,0.13489645){\color[rgb]{0,0,0}\makebox(0,0)[lt]{\lineheight{1.25}\smash{\begin{tabular}[t]{l}$(\forall \, K \in \mathcal{C}_{\mathrm{cp}})$\end{tabular}}}}%
    \put(0.34472746,0.00946109){\color[rgb]{0,0,0}\makebox(0,0)[lt]{\lineheight{1.25}\smash{\begin{tabular}[t]{l}$_{K^*}$\end{tabular}}}}%
    \put(0,0){\includegraphics[width=\unitlength,page=6]{actionFromHalfBraiding.pdf}}%
    \put(0.12748316,0.29226082){\color[rgb]{0,0,0}\makebox(0,0)[lt]{\lineheight{1.25}\smash{\begin{tabular}[t]{l}$_V$\end{tabular}}}}%
    \put(0,0){\includegraphics[width=\unitlength,page=7]{actionFromHalfBraiding.pdf}}%
    \put(0.0999192,0.19140125){\color[rgb]{0,0,0}\makebox(0,0)[lt]{\lineheight{1.25}\smash{\begin{tabular}[t]{l}$\lambda(t)$\end{tabular}}}}%
    \put(0,0){\includegraphics[width=\unitlength,page=8]{actionFromHalfBraiding.pdf}}%
    \put(0.0163226,0.00946117){\color[rgb]{0,0,0}\makebox(0,0)[lt]{\lineheight{1.25}\smash{\begin{tabular}[t]{l}$_{K^*}$\end{tabular}}}}%
    \put(0,0){\includegraphics[width=\unitlength,page=9]{actionFromHalfBraiding.pdf}}%
    \put(0.08546046,0.00946117){\color[rgb]{0,0,0}\makebox(0,0)[lt]{\lineheight{1.25}\smash{\begin{tabular}[t]{l}$_K$\end{tabular}}}}%
    \put(0,0){\includegraphics[width=\unitlength,page=10]{actionFromHalfBraiding.pdf}}%
    \put(0.18916726,0.00946121){\color[rgb]{0,0,0}\makebox(0,0)[lt]{\lineheight{1.25}\smash{\begin{tabular}[t]{l}$_V$\end{tabular}}}}%
    \put(0.27886216,0.13630085){\color[rgb]{0,0,0}\makebox(0,0)[lt]{\lineheight{1.25}\smash{\begin{tabular}[t]{l}$=$\end{tabular}}}}%
  \end{picture}%
\endgroup%

\end{equation}
\end{proof}

To conclude, take the category $\mathcal{C} = \mathrm{Comod}\text{-}\OO$ of comodules over a quasitriangular Hopf $k$-algebra (\S\ref{subsecComodH}). We recall from Lemma \ref{propLFPComod} and Prop.\,\ref{descriptionCoendComodH} that in this case the compact objects are the finite dimensional comodules and the coend $\mathscr{L}$ is $\OO$ as a coalgebra, but endowed with a ``braided product'' as an algebra.

\smallskip

\indent Item 1 in Lemma \ref{lemmaHBonCompacts} is easily established in this particular case. Indeed, assume that we have a half-braiding $u : V \otimes - \Rightarrow - \otimes V$ relative to the finite dimensional comodules. We define $(V,t) \in \mathcal{Z}(\mathcal{C})$ by $t_X(v \otimes x) = u_{F_x}(v \otimes x)$ for all $v \in V$ and $x \in X$, where $F_x \subset X$ is some finite-dimensional subcomodule containing $x$ (it exists by the fundamental theorem on comodules \cite[Th. 5.1.1]{Mon}). Naturality of $u$ makes this independent of the choice of $F_x$.

\smallskip

\indent Here the explicit form of Prop.\,\ref{propIsoZCandLModC} in the case $\mathcal{C} = \mathrm{Comod}\text{-}\OO$:

\begin{lemma}\label{lemmaIsoZCvsLModCForComodH}
1. For $(V,t) \in \mathcal{Z}(\mathcal{C})$ define a left action $\smallblacksquare$ of $\mathscr{L}$ on $V$ by
\[ \forall \, \varphi \in \mathscr{L}, \:\: \forall \, v \in V, \quad \varphi \smallblacksquare v = \mathcal{R}\bigl( \varphi_{(2)} \otimes v_{[1]} \bigr) \, (\varepsilon_{\OO} \otimes \mathrm{id}_V) \circ t_H\bigl(v_{[0]} \otimes \varphi_{(1)}\bigr) \]
where $\OO$ is viewed as a right $\OO$-comodule thanks to the coproduct $\Delta : \OO \to \OO \otimes \OO$ and $\varepsilon_{\OO} : \OO \to k$ is the counit. Then $\Upsilon(V,t) = (V,\smallblacksquare)$.
\\2. Conversely for $(M,\smallblacksquare) \in \mathscr{L}\text{-}\mathrm{mod}_{\mathcal{C}}$ and any $X \in \mathcal{C}$ define $t_X : M \otimes X \to X \otimes M$ by
\[ \forall \, m \in M, \:\: \forall \, x \in X, \quad t_X(m \otimes x) = \mathcal{R}^{-1}\bigl(x_{[2]} \otimes m_{[1]}\bigr)\, x_{[0]} \otimes \bigl( x_{[1]} \smallblacksquare m_{[0]} \bigr). \]
Then $\Upsilon^{-1}(M,\smallblacksquare) = (M,t)$.\footnote{In the formula defining $t_X$, note that $x_{[1]} \in \OO$ is viewed as an element in $\mathscr{L}$ in order to use the action $\smallblacksquare$. This is possible since $\mathscr{L}$ is $\OO$ as a vector space.}
\end{lemma}
\begin{proof}
1. Let $\varphi \in \mathscr{L}$. Note that $\mathscr{L} = \OO$ as vector spaces and choose a finite-dimensional subcomodule $F_{\varphi} \subset \OO$ containing $\varphi$. Then $F_{\varphi} \in \mathcal{C}_{\mathrm{cp}}$ and by \eqref{univDinatTransfoTransmutation} we have $i_{F_{\varphi}}(\varepsilon|_{F_{\varphi}} \otimes \varphi) = \varphi$ where $\varepsilon|_{F_\varphi} : F_\varphi \to k$ is the restriction of the counit to $F_\varphi$, because the comodule structure on $\OO$ is $\varphi_{[0]} \otimes \varphi_{[1]} = \varphi_{(1)} \otimes \varphi_{(2)}$. Now using \eqref{actionFromHalfBraiding} we find
\begin{align*}
\varphi \smallblacksquare v &= i_{F_\varphi}(\varepsilon|_{F_\varphi} \otimes \varphi) \smallblacksquare v = \bigl( \mathrm{ev}_{F_\varphi} \otimes \mathrm{id}_V \bigr) \circ \bigl( \mathrm{id}_{F_\varphi^*} \otimes t_{F_\varphi} \bigr) \circ \bigl( \mathrm{id}_{F_\varphi^*} \otimes c_{F_\varphi,V} \bigr)(\varepsilon|_{F_\varphi} \otimes \varphi \otimes v)\\
&= \mathcal{R}(\varphi_{(2)} \otimes v_{[1)}) \bigl( \mathrm{ev}_{F_\varphi} \otimes \mathrm{id}_V \bigr) \circ \bigl( \mathrm{id}_{F_\varphi^*} \otimes t_{F_\varphi} \bigr)(\varepsilon|_{F_\varphi} \otimes v_{[0]} \otimes \varphi_{(1)})\\
&=\mathcal{R}\bigl( \varphi_{(2)} \otimes v_{[1]} \bigr) \, (\varepsilon|_{F_\varphi} \otimes \mathrm{id}_V) \circ t_{F_\varphi}\bigl(v_{[0]} \otimes \varphi_{(1)}\bigr) =\mathcal{R}\bigl( \varphi_{(2)} \otimes v_{[1]} \bigr) \, (\varepsilon \otimes \mathrm{id}_V) \circ t_{\OO}\bigl(v_{[0]} \otimes \varphi_{(1)}\bigr)
\end{align*}
where for the last equality we used naturality of $t$ to commute it with the inclusion $F_\varphi \hookrightarrow \OO$ from which $\varepsilon|_{F_\varphi}$ is defined.
\\2. It suffices to check this formula when $X$ is a compact object: $X \in \mathcal{C}_{\mathrm{cp}} = \mathrm{comod}\text{-}\OO$. Then by \eqref{univDinatTransfoTransmutation} and the formula in Proposition \ref{propIsoZCandLModC} we get
\begin{align*}
t_X(m \otimes x) &= \bigl( \mathrm{id}_X \otimes \smallblacksquare \bigr) \circ \bigl( \mathrm{id}_X \otimes i_X \otimes \mathrm{id}_M \bigr) \circ \bigl( \mathrm{coev}_X \otimes c^{-1}_{M,X} \bigr)(m \otimes x)\\
&= \mathcal{R}^{-1}(x_{[1]} \otimes m_{[1]}) \, \bigl( \mathrm{id}_X \otimes \smallblacksquare \bigr) \circ \bigl( \mathrm{id}_X \otimes i_X \otimes \mathrm{id}_M \bigr)(x_i \otimes x^i \otimes x_{[0]} \otimes m_{[0]})\\
&= \mathcal{R}^{-1}(x_{[2]} \otimes m_{[1]}) \, x_i \otimes x^i(x_{[0]})x_{[1]} \smallblacksquare m_{[0]} = \mathcal{R}^{-1}(x_{[2]} \otimes m_{[1]}) \, x_{[0]} \otimes x_{[1]} \smallblacksquare m_{[0]}
\end{align*}
where $(x_i)$ is a basis of $X$ with dual basis $(x^i)$ and we used implicit summation.
\end{proof}

\section{Balance on \texorpdfstring{$\mathcal{Z}(\mathcal{C})$}{the Drinfeld center}}\label{appBalanceZC}
Here we assume that the category $\mathcal{C}$ satisfies the assumptions \eqref{assumptionsCategoryC} and is moreover cp-ribbon, which means that the subcategory of compact objects $\mathcal{C}_{\mathrm{cp}}$ has a ribbon structure $\theta$ (see \eqref{axiomsTwist}). Then right dual objects can be realized on left dual objects by means of the duality morphisms \eqref{dualityMorphByMeansOfTwist}.

\smallskip

\indent Let $(V,t) \in \mathcal{Z}(\mathcal{C})$ and take a colimit presentation $V = \mathrm{colim}\bigl( K : \mathcal{I} \to \mathcal{C}_{\mathrm{cp}} \bigr)$ with universal cocone $\phi$. There exists a unique $\Theta_{(V,t)} \in \Hom_{\mathcal{C}}(V,V)$ such that
\begin{equation}\label{defBalZC}
\xymatrix@C=5.5em{
K(X) \ar[d]_{\phi_X} \ar[r]^-{\widetilde{\mathrm{coev}}_{K(X)} \,\otimes\, \mathrm{id}} & K(X)^* \otimes K(X) \otimes K(X) \ar[r]^-{\mathrm{id} \,\otimes\, \phi_X \,\otimes\, \mathrm{id}} & K(X)^* \otimes V \otimes K(X) \ar[d]^{\mathrm{id} \,\otimes\, t_{K(X)}}\\
V \ar[r]_{\Theta_{(V,t)}} & V &K(X)^* \otimes K(X) \otimes V \ar[l]^-{\mathrm{ev}_{K(X)} \,\otimes\, \mathrm{id}}
} \end{equation}
commutes for all $X \in \mathcal{I}$. This is because the clockwise composition of arrows in this diagram is easily seen to be a cocone for $K$, and hence factors uniquely through $\phi$. Thanks to arguments which are completely similar to those in the proof of Lemma \ref{lemmaExtensionNat}, one can show that $\Theta_{(V,t)}$ does not depend on the choice of the colimit presentation.

\smallskip

\indent We recall from \S\ref{subsecZC} that the braiding in $\mathcal{Z}(\mathcal{C})$ is given by $T_{\mathbf{V}, \mathbf{W}} = t_W$ for $\mathbf{V} = (V,t)$ and $\mathbf{W} = (W,u)$.
\begin{proposition}\label{propBalanceZC}
1. $\Theta_{(V,t)}$ is in $\Hom_{\mathcal{Z}(\mathcal{C})}\bigl( (V,t), (V,t) \bigr)$, i.e. it commutes with $t$.
\\2. $\Theta_{(V,t)}$ is an isomorphism whose inverse is given by
\[ \xymatrix@C=5.5em{
K(X) \ar[d]_{\phi_X} \ar[r]^-{\mathrm{id} \,\otimes\, \mathrm{coev}_{K(X)}} & K(X) \otimes K(X) \otimes K(X)^* \ar[r]^-{\mathrm{id} \,\otimes\, \phi_X \,\otimes\, \mathrm{id}} & K(X) \otimes V \otimes K(X)^* \ar[d]^{t_{K(X)}^{-1} \,\otimes\, \mathrm{id}}\\
V \ar[r]_{\Theta_{(V,t)}^{-1}} & V &V \otimes K(X) \otimes K(X)^* \ar[l]^-{\mathrm{id} \,\otimes\, \widetilde{\mathrm{ev}}_{K(X)}}
} \]
3. The collection $\Theta = \bigl( \Theta_{\mathbf{V}} \bigr)_{\mathbf{V} \in \mathcal{Z}(\mathcal{C})}$ is a  natural isomorphism $\mathrm{Id}_{\mathcal{Z}(\mathcal{C})} \overset{\sim}{\implies} \mathrm{Id}_{\mathcal{Z}(\mathcal{C})}$.
\\4. $\Theta$ is a balance on $\mathcal{Z}(\mathcal{C})$, which means that $\Theta_{\mathbf{V}_1 \otimes \mathbf{V}_2} = T_{\mathbf{V}_2, \mathbf{V}_1} \circ T_{\mathbf{V}_1,\mathbf{V}_2} \circ \bigl( \Theta_{\mathbf{V}_1} \otimes \Theta_{\mathbf{V}_2} \bigr)$ for all $\mathbf{V}_1, \mathbf{V}_2 \in \mathcal{Z}(\mathcal{C})$, where $T$ is the braiding in $\mathcal{Z}(\mathcal{C})$ recalled above.
\end{proposition}
\begin{proof}
1. We have to show that $t_C \circ \bigl( \Theta_{(V,t)} \otimes \mathrm{id}_C \bigr) = \bigl(\mathrm{id}_C \otimes \Theta_{(V,t)}\bigr) \circ t_C$ for all $C \in \mathcal{C}$. By Lemma \ref{lemmaExtensionNat}, it suffices to prove this when $C$ is a compact object. Take a colimit presentation of $V$ as above \eqref{defBalZC}. By assumption on $\mathcal{C}$, see \eqref{assumptionsCategoryC}, the object $K(X) \otimes C$ is compact for all $X \in \mathcal{I}$. Hence we can apply Lemma \ref{lemmaFactoCompactObjects} to the functor $K \otimes \mathrm{id}_C : \mathcal{I} \to \mathcal{C}_{\mathrm{cp}}$ and to the morphism $f = t_C \circ (\phi_X \otimes \mathrm{id}_C) \in \Hom_{\mathcal{C}}\bigl( K(X) \otimes C, C \otimes V \bigr)$ for each $X \in \mathcal{I}$. This yields the existence of $Y \in \mathcal{I}$ and $\varphi \in \Hom_{\mathcal{C}}\bigl(  K(X) \otimes C, C \otimes K(Y)\bigr)$ such that $t_C \circ (\phi_X \otimes \mathrm{id}_C) = (\mathrm{id}_C \otimes \phi_Y) \circ \varphi$. The following computation, whose details are explained below, is guided by the principle that only compact objects can be dualized, and hence all manipulations involving cups and caps must necessarily be done with them.
\begin{center}
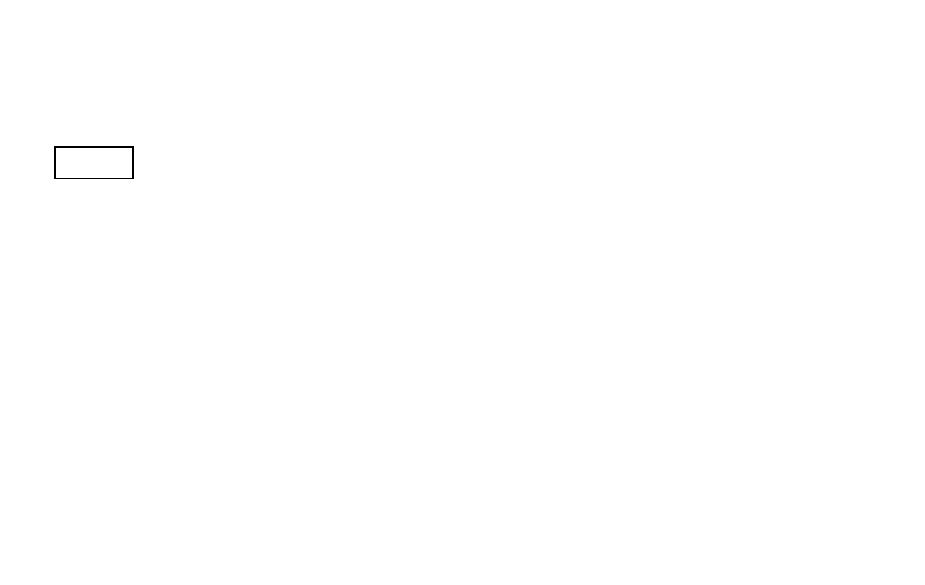
\end{center}
and the desired equality follows because $\phi \otimes \mathrm{id}_C$ is a universal cocone (recall from \eqref{assumptionsCategoryC} the assumption that $- \otimes C$ is cocontinuous) and the object $X \in \mathcal{I}$ is arbitrary. The first equality uses \eqref{defBalZC}, the second equality uses the half-braiding axiom \eqref{axiomHalfBraiding}, the third equality is a trick (nothing is changed after simplification), the fourth equality uses the morphism $\varphi$ introduced before the computation, the fifth equality is by naturality of cups/caps and of the half-braiding $t$, the sixth equality uses the fact that
\begin{equation}\label{inverseHBwithDuals}
t_{C^*} = \bigl( \widetilde{\mathrm{coev}}_C \otimes \mathrm{id}_V \bigr) \circ \bigl(\mathrm{id}_{C^*} \otimes t_C^{-1} \otimes \mathrm{id}_{C^*} \bigr) \circ \bigl( \mathrm{id}_V \otimes \widetilde{\mathrm{ev}}_V \bigr)
\end{equation}
which easily follows from the half-braiding axiom \eqref{axiomHalfBraiding} and naturality of $t$, the seventh equality uses the half-braiding property \eqref{axiomHalfBraiding}, the eighth equality uses naturality of $t$ and the zig-zag axiom between ev and coev, the ninth equality is by \eqref{defBalZC} and the last equality is by definition of $\varphi$.
\\2. Denote by $\overline{\Theta}_{(V,t)} : V \to V$ the proposed inverse. Take again a colimit presentation of $V$ made of compact objects as above \eqref{defBalZC}, and let $X \in \mathcal{I}$. By Lemma \ref{lemmaFactoCompactObjects}, there exists $Y \in \mathcal{I}$ and $g \in \Hom_{\mathcal{C}}(K(X),K(Y))$ such that $\overline{\Theta}_{(V,t)} \circ \phi_X = \phi_Y \circ g$. Then
\begin{center}
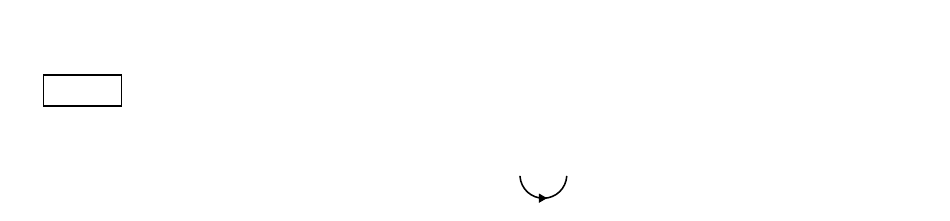
\end{center}
and the last term equals $\phi_X$. For the second equality we used \eqref{defBalZC}, for the third equality we used the naturality of $t$ and of cups/caps, the fourth equality is by definition of $g$ and of $\overline{\Theta}_{(V,t)}$, the fifth equality uses \eqref{inverseHBwithDuals}, the sixth equality uses the axiom of half-braidings \eqref{axiomHalfBraiding} and the last equality is by naturality of $t$. The last term equals $\phi_X$ and this is true for all $X \in \mathcal{I}$. We thus conclude that $\Theta_{(V,t)} \circ \overline{\Theta}_{(V,t)} = \mathrm{id}_V$ because $\phi$ is an universal cocone.
\\3. We have to prove naturality. Let $f \in \Hom_{\mathcal{Z}(\mathcal{C})}\bigl( (V_1,t^1), (V_2,t^2) \bigr)$ and take colimit presentations $V_i = \mathrm{colim}\bigl( K_i : \mathcal{I}_i \to \mathcal{C}_{\mathrm{cp}} \bigr)$ with universal cocones $\phi^i$ for $i=1,2$. Let $X \in \mathcal{I}_1$ and apply Lemma \ref{lemmaFactoCompactObjects} to the morphism $f \circ \phi^1_X : K_1(X) \to V_2$. This yields $Y \in \mathcal{I}_2$ and $g \in \Hom_{\mathcal{C}}\bigl( K_1(X), K_2(Y) \bigr)$ such that $f \circ \phi^1_X = \phi^2_Y \circ g$. We then have
\begin{center}
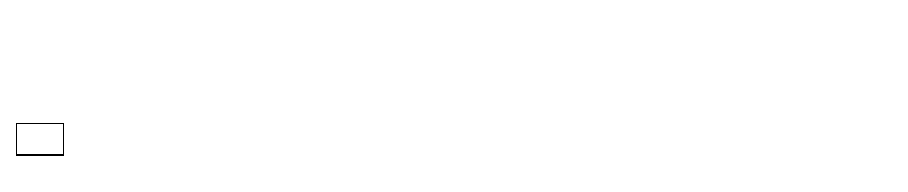
\end{center}
where the first and second equalities are by definition of $g$ and $\Theta$, the third equality is by naturality, the fourth is by definition of $g$, the fifth uses that $f$ is a morphism in $\mathcal{Z}(\mathcal{C})$ and the last is by definition of $\Theta$.
\\4. Write $\mathbf{V}_i = (V_i,t^i)$ and take colimit presentations of $V_1, V_2 \in \mathcal{C}$ as in the proof of the previous item. Consider the functor $K : \mathcal{I}_1 \times \mathcal{I}_2 \to \mathcal{C}_{\mathrm{cp}}$ given by $K(X,Y) = K_1(X) \otimes K_2(Y)$. Then by \eqref{colimPresOfMonProd} we have $V_1 \otimes V_2 = \mathrm{colim} \, K$, with universal cocone $\phi^1 \otimes \phi^2$ given by $(\phi^1 \otimes \phi^2)_{(X,Y)} = \phi^1_X \otimes \phi^2_Y$. Also let $\mathbf{V}_1 \otimes \mathbf{V}_2 = (V_1 \otimes V_2, t^{12})$ as defined in \eqref{defTensorProductInZC}. Then we have
\begin{center}
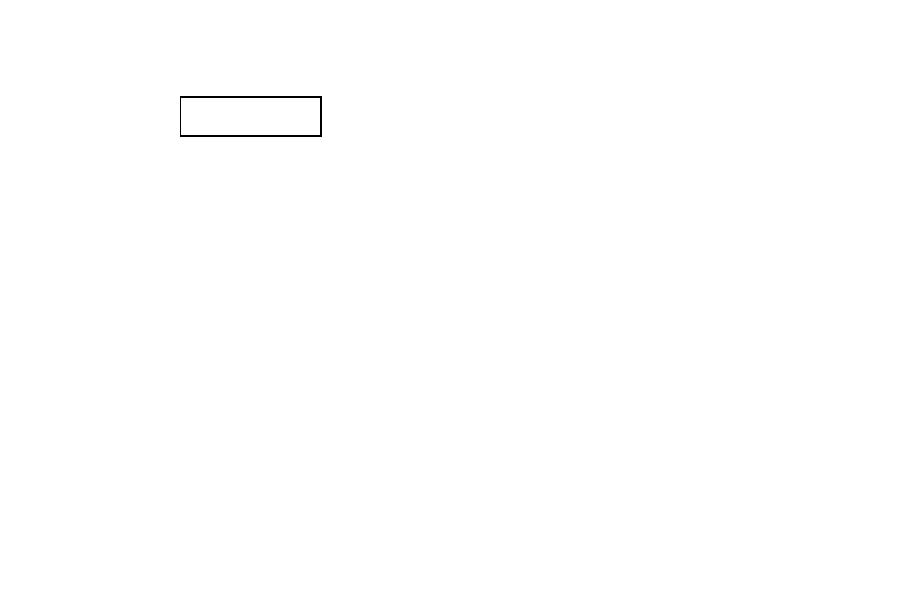
\end{center}
where the first equality is by definition of $\Theta$ in \eqref{defBalZC}, the second equality is by definition of $t^{12}$ in \eqref{defTensorProductInZC}, the third equality uses the half-braiding axiom \eqref{axiomHalfBraiding}, the fourth equality is by definition of $\Theta$, the fifth equality uses naturality of $t^2$, the sixth equality uses the half-braiding axiom and the last equality is by naturality $t^1$ and definition on $\Theta$. Since $\phi^1 \otimes \phi^2$ is a universal cocone the desired identity follows.
\end{proof}

\section{On \texorpdfstring{$\mathscr{L}$}{coend}-linear algebras in \texorpdfstring{$\mathrm{Comod}\text{-}\OO$}{categories of comodules}}\label{appendixLlinearComodH}
Here we give yet another equivalent definition for $\mathscr{L}$-linear algebras in $\mathrm{Comod}\text{-}\OO$ (recall that these have been described in \S\ref{subsec:BimodHcomod}). Hopefully this will facilitate the comparison with other works. We use notations introduced in \S\ref{subsecComodH} regarding right $\OO$-comodules.

\smallskip

\indent Recall that a (right-right) {\em Yetter--Drinfeld module}\footnote{Also known as a {\em crossed module}, e.g. in \cite[Prop.\,7.1.6]{Majid}.} over $\OO$ is a $k$-vector space $M$ which is both a right $\OO$-comodule and a right $\OO$-module and such that
\begin{equation}\label{axiomYDmodule}
\forall \, m \in M, \:\: \forall \, \varphi \in \OO, \quad (m \cdot \varphi)_{[0]} \otimes (m \cdot \varphi)_{[1]} = m_{[0]} \cdot \varphi_{(2)} \otimes S(\varphi_{(1)}) m_{[1]} \varphi_{(3)}.
\end{equation}
They form a monoidal category $\mathcal{Y}\mathcal{D}^{\OO}_{\OO}$, morphisms being linear maps which commute with both actions and coactions. There is a well-known monoidal isomorphism
\begin{equation}\label{isoYDZComodH}
\mathcal{Y}\mathcal{D}^{\OO}_{\OO} \cong \mathcal{Z}(\mathrm{Comod}\text{-}\OO)
\end{equation}
given by
\begin{equation}\label{FromYDHHToZComodH}
\mathcal{Y}\mathcal{D}^{\OO}_{\OO} \ni M \mapsto (M,t) \in \mathcal{Z}(\mathrm{Comod}\text{-}\OO) \quad \text{with } t_X(m \otimes x) = x_{[0]} \otimes m \cdot x_{[1]}
\end{equation}
for all $X \in \mathrm{Comod}\text{-}\OO$, $m \in M$ and $x \in X$. Note that $t_X^{-1}(x \otimes m) = m \cdot S^{-1}(x_{[1]}) \otimes x_{[0]}$. Conversely
\begin{equation}\label{FromZComodHToYDHH}
\mathcal{Z}(\mathrm{Comod}\text{-}\OO) \ni (V,t) \mapsto V \in \mathcal{Y}\mathcal{D}^{\OO}_{\OO} \quad \text{with } v \cdot \varphi = (\mathrm{id}_V \otimes \varepsilon_{\OO}) \circ t_{\OO}^{-1}\bigl(S(\varphi) \otimes v\bigr)
\end{equation}
for all $v \in V$, $\varphi \in \OO$ and where the $\OO$-comodule isomorphism $t_{\OO} : V \otimes \OO \to \OO \otimes V$ is obtained by endowing $\OO$ with the $\OO$-comodule structure coming from its coproduct (by definition we are given $t_X : V \otimes X \overset{\sim}{\to} X \otimes V$ for all $X \in \mathrm{Comod}\text{-}\OO$).

\smallskip

\indent Through the isomorphism \eqref{isoYDZComodH}, a half-braided algebra in $\mathrm{Comod}\text{-}\OO$ becomes a right $\OO$-comodule-algebra $A$ endowed with a right action $\cdot$ of $\OO$ such that $A$ is a Yetter--Drinfeld module and it holds
\begin{equation}\label{eq:crossedalgebra}
 \mathcal{R}\bigl(a_{[1]} \otimes \varphi_{(1)}\bigr) \, a_{[0]} (b \cdot \varphi_{(2)}) = (a b) \cdot \varphi = \mathcal{R}^{-1}\bigl( \varphi_{(2)} \otimes b_{[1]}  \bigr) \, (a \cdot \varphi_{(1)})b_{[0]} 
\end{equation}
for all $a,b \in A$ and $\varphi \in \OO$. Let us call this a {\em YD-comodule-algebra}.\footnote{Note that this notion is different from what is called a Yetter--Drinfeld algebra in the literature (see e.g. \cite[\S 3]{taipe}).} The next lemma gives an equivalent definition of these objects:

\begin{lemma}\label{charactYDComodAlgThroughPseudoQMM}
1. Let $A$ be a YD-comodule-algebra and consider the linear map $d : \OO \to A$ given by $d(\varphi) = 1_A \cdot \varphi$. It satisfies
\begin{align}
&d(\varphi)_{[0]} \otimes d(\varphi)_{[1]} = d(\varphi_{(2)}) \otimes S(\varphi_{(1)})\varphi_{(3)},\label{equivariancePseudoQMMforYDALg}\\
&d(\varphi\psi) = \mathcal{R}\bigl( S(\varphi_{(1)})\varphi_{(3)} \otimes \psi_{(1)} \bigr) \, d(\varphi_{(2)}) d(\psi_{(2)}),\label{OlinearPseudoQMM}\\
&\mathcal{R}\bigl( a_{[1]} \otimes \varphi_{(1)} \bigr) \, a_{[0]} d(\varphi_{(2)}) = \mathcal{R}^{-1}\bigl( \varphi_{(2)} \otimes a_{[1]} \bigr) \, d(\varphi_{(1)})a_{[0]} \label{exchangePseudoQMMforYDALg}
\end{align}
for all $\varphi,\psi \in \OO$ and $a \in A$.
\\2. Conversely let $A$ be a right $\OO$-comodule-algebra and let $d : \OO \to A$ be a linear map which satisfies \eqref{equivariancePseudoQMMforYDALg}, \eqref{OlinearPseudoQMM} and \eqref{exchangePseudoQMMforYDALg}. For all $a \in A$ and $\varphi \in \OO$ define
\begin{equation}\label{actionFromPseudoQMMComodAlg}
a \cdot \varphi = \mathcal{R}\bigl(a_{[1]} \otimes \varphi_{(1)}\bigr) \, a_{[0]}d(\varphi_{(2)}).
\end{equation}
Then $A$ is a YD-comodule-algebra.
\end{lemma}
\begin{proof}
1. \eqref{equivariancePseudoQMMforYDALg} is obtained by taking $m=1_A$ in \eqref{axiomYDmodule}. \eqref{OlinearPseudoQMM} is obtained as follows:
\begin{align*}
d(\varphi\psi) &= (1_A \cdot \varphi) \cdot \psi = d(\varphi) \cdot \psi = \bigl( d(\varphi)1_A \bigr) \cdot \psi \overset{\eqref{eq:crossedalgebra}}{=} \mathcal{R}\bigl( d(\varphi)_{[1]} \otimes \psi_{(1)} \bigr) \, d(\varphi)_{[0]} \, (1_A \cdot \psi_{(2)})\\
&= \mathcal{R}\bigl( d(\varphi)_{[1]} \otimes \psi_{(1)} \bigr) \, d(\varphi)_{[0]} \, d(\psi_{(2)}) \overset{\eqref{equivariancePseudoQMMforYDALg}}{=} \mathcal{R}\bigl( S(\varphi_{(1)}) \varphi_{(3)} \otimes \psi_{(1)} \bigr) \, d(\varphi_{(2)}) \, d(\psi_{(2)})
\end{align*}
For \eqref{exchangePseudoQMMforYDALg} it suffices to note from \eqref{eq:crossedalgebra} that
\[ \mathcal{R}\bigl( a_{[1]} \otimes \varphi_{(1)} \bigr) \, a_{[0]} \bigl(1_A \cdot \varphi_{(2)} \bigr) = (a1_A) \cdot \varphi = (1_Aa) \cdot \varphi = \mathcal{R}^{-1}\bigl( \varphi_{(2)} \otimes a_{[1]} \bigr) \, \bigl( 1_A \cdot \varphi_{(1)} \bigr) a_{[0]}. \]
2. Let us first prove that \eqref{actionFromPseudoQMMComodAlg} defines a right $\OO$-action:
\begin{align*}
a \cdot (\varphi \psi) &= \mathcal{R}\bigl(a_{[1]} \otimes \varphi_{(1)}\psi_{(1)}\bigr) \, a_{[0]}d(\varphi_{(2)}\psi_{(2)})\\
&\overset{\eqref{OlinearPseudoQMM} + \eqref{equivariancePseudoQMMforYDALg}}{=} \mathcal{R}\bigl(a_{[1]} \otimes \varphi_{(1)}\psi_{(1)}\bigr) \mathcal{R}\bigl( d(\varphi_{(2)})_{[1]} \otimes \psi_{(2)} \bigr) \, a_{[0]}d(\varphi_{(2)})_{[0]}d(\psi_{(3)})\\
&= \mathcal{R}\bigl( a_{[2]} \otimes \varphi_{(1)} \bigr) \mathcal{R}\bigl( a_{[1]}d(\varphi_{(2)})_{[1]} \otimes \psi_{(1)} \bigr) \, a_{[0]} d(\varphi_{(2)})_{[0]} d(\psi_{(2)})\\
&\overset{\eqref{actionFromPseudoQMMComodAlg}}{=} \mathcal{R}\bigl( a_{[1]} \otimes \varphi_{(1)} \bigr) \, \bigl( a_{[0]}d(\varphi_{(2)}) \bigr) \cdot \psi \overset{\eqref{actionFromPseudoQMMComodAlg}}{=} (a \cdot \varphi) \cdot \psi
\end{align*}
where the third equality uses co-quasitriangularity of $\mathcal{R}$, see \cite[Def.\,2.21]{Majid}. Let us now show that $A$ is a Yetter--Drinfeld module, \textit{i.e.} condition \eqref{axiomYDmodule}:
\begin{align*}
&(a \cdot \varphi)_{[0]} \otimes (a \cdot \varphi)_{[1]} \overset{\eqref{actionFromPseudoQMMComodAlg}}{=} \mathcal{R}\bigl(a_{[1]} \otimes \varphi_{(1)}\bigr) \, \bigl( a_{[0]}d(\varphi_{(2)}) \bigr)_{[0]} \otimes \bigl( a_{[0]}d(\varphi_{(2)}) \bigr)_{[1]}\\
&\overset{\eqref{axiomComodAlg}}{=} \mathcal{R}\bigl(a_{[2]} \otimes \varphi_{(1)}\bigr) \, a_{[0]}d(\varphi_{(2)})_{[0]} \otimes a_{[1]}d(\varphi_{(2)})_{[1]} \overset{\eqref{equivariancePseudoQMMforYDALg}}{=} \mathcal{R}\bigl(a_{[2]} \otimes \varphi_{(1)}\bigr) \, a_{[0]}d(\varphi_{(3)}) \otimes a_{[1]}S(\varphi_{(2)})\varphi_{(4)}\\
&= \mathcal{R}\bigl(a_{[1]} \otimes \varphi_{(2)}\bigr) \, a_{[0]}d(\varphi_{(3)}) \otimes S(\varphi_{(1)})a_{[2]}\varphi_{(4)} \overset{\eqref{actionFromPseudoQMMComodAlg}}{=} a_{[0]} \cdot \varphi_{(2)} \otimes S(\varphi_{(1)})a_{[2]}\varphi_{(3)}
\end{align*}
where the unlabelled equality uses properties of a co-R-matrix, more precisely eqs. (2.7) and (2.9) in \cite{Majid}. Finally, the first equality in condition \eqref{eq:crossedalgebra} is obtained as follows
\begin{align*}
&(ab) \cdot \varphi \overset{\eqref{actionFromPseudoQMMComodAlg}}{=} \mathcal{R}\bigl((ab)_{[1]} \otimes \varphi_{(1)}\bigr) \, (ab)_{[0]}d(\varphi_{(2)}) \overset{\eqref{axiomComodAlg}}{=} \mathcal{R}\bigl(a_{[1]}b_{[1]} \otimes \varphi_{(1)}\bigr) \, a_{[0]}b_{[0]}d(\varphi_{(2)})\\
&= \mathcal{R}\bigl(a_{[1]} \otimes \varphi_{(1)}\bigr) \mathcal{R}\bigl(b_{[1]} \otimes \varphi_{(2)}\bigr) \, a_{[0]}b_{[0]}d(\varphi_{(3)}) \overset{\eqref{actionFromPseudoQMMComodAlg}}{=} \mathcal{R}\bigl(a_{[1]} \otimes \varphi_{(1)}\bigr) \, a_{[0]} (b \cdot \varphi_{(2)}). 
\end{align*}
The second equality is obtained similarly, but now by writing the right action as $a \cdot \varphi = \mathcal{R}^{-1}\bigl( \varphi_{(2)} \otimes a_{[1]} \bigr) \, d(\varphi_{(1)})a_{[0]}$ thanks to \eqref{exchangePseudoQMMforYDALg}.
\end{proof}
\noindent We note that the commutation relation  \eqref{exchangePseudoQMMforYDALg} can be equivalently written as
\[ d(\varphi)a = \mathcal{R}\bigl(a_{[1]} \otimes \varphi_{(1)}\bigr) \mathcal{R}\bigl( \varphi_{(3)} \otimes a_{[2]} \bigr) \, a_{[0]}d(\varphi_{(2)}). \]

\indent Let $\mathbf{B} = (B,\smallblacktriangleright,\smallblacktriangleleft)$ be a $(A_2,A_1)$-bimodule over YD-comodule-algebras $A_1$, $A_2$ and for $i = 1,2$ denote by $d_i : \OO \to A_i$ the morphism of algebras defined in Lemma \ref{charactYDComodAlgThroughPseudoQMM}. The correspondence \eqref{FromYDHHToZComodH} gives half-braidings $t^i : A_i \otimes - \Rightarrow - \otimes A_i$ (for $i=1,2$) from which the half-braidings $\mathrm{hbl}^{\mathbf{B}}, \mathrm{hbr}^{\mathbf{B}} : B \otimes - \Rightarrow - \otimes B$ are defined in \eqref{halfBraidingInTermOfAction}. We get
\begin{align*}
\mathrm{hbl}^{\mathbf{B}}_X(b \otimes x) &= \mathcal{R}^{-1}\bigl( x_{[2]} \otimes b_{[1]} \bigr) \, x_{[0]} \otimes d_2(x_{[1]}) \smallblacktriangleright b_{[0]},\\
\mathrm{hbr}^{\mathbf{B}}_X(b \otimes x) &= \mathcal{R}\bigl( b_{[1]} \otimes x_{[1]} \bigr) \, x_{[0]} \otimes b_{[0]} \smallblacktriangleleft d_1(x_{[2]})
\end{align*}
for all $X \in \mathrm{Comod}\text{-}\OO$, $x \in X$ and $b \in B$. Using \eqref{INVERSEhalfBraidingInTermOfAction}, their inverses are given by $\mathrm{hbl}^{\mathbf{B}-1}_X(x \otimes b) = \mathcal{R}\bigl( x_{[1]} \otimes b_{[1]} \bigr) \, d_2\bigl( S^{-1}(x_{[2]}) \bigr) \smallblacktriangleright b_{[0]} \otimes x_{[0]}$ and $\mathrm{hbr}^{\mathbf{B}-1}_X(x \otimes b) = \mathcal{R}^{-1}\bigl( b_{[1]} \otimes x_{[2]} \bigr) \, b_{[0]} \smallblacktriangleleft d_1\bigl( S^{-1}(x_{[1]}) \bigr) \otimes x_{[0]}$. Now the correspondence  \eqref{FromZComodHToYDHH} gives two right actions of $\cdot^l, \cdot^r : B \otimes \OO \to B$ coming from $\mathrm{hbl}^{\mathbf{B}}$ and $\mathrm{hbr}^{\mathbf{B}}$ respectively. A short computation reveals that
\[ b \cdot^l \varphi = \mathcal{R}^{-1}\bigl( \varphi_{(2)} \otimes b_{[1]} \bigr) \, d_2(\varphi_{(1)}) \smallblacktriangleright b_{[0]}, \qquad b \cdot^r \varphi = \mathcal{R}\bigl( b_{[1]} \otimes \varphi_{(1)} \bigr) \, b_{[0]} \smallblacktriangleleft d_1(\varphi_{(2)}). \]
for all $b \in B$ and $\varphi \in \OO$. Hence the bimodule $\mathbf{B}$ is hb-compatible if and only if the two right $\OO$-actions above are equal. We could call this condition {\em YD-compatibility}.

\end{document}